\documentclass[11pt]{article}
\usepackage{bbm}
 \usepackage{amssymb}
\usepackage{amssymb, amsthm, amsmath, amscd}
\usepackage[usenames,dvipsnames]{color}

	\theoremstyle{plain}
	\newtheorem{thmintro}{Theorem}
	\theoremstyle{definition}
	\newtheorem{exmintro}[thmintro]{Example}
	\newtheorem{dfintro}[thmintro]{Definition}

\newtheorem{thm}{Theorem}[section]
\newtheorem{lem}[thm]{Lemma}
\newtheorem{prop}[thm]{Proposition}
\newtheorem{cor}[thm]{Corollary}

\theoremstyle{definition}
\newtheorem{NN}[thm]{}
\theoremstyle{definition}
\newtheorem{df}[thm]{Definition}
\theoremstyle{definition}
\newtheorem{rem}[thm]{Remark}

\theoremstyle{definition}
\newtheorem{exm}[thm]{Example}
\newcommand{\red}{\textcolor{red}}
\newcommand{\blue}{\color{blue}}
\newcommand{\green}{\color{green}}
\newcommand{\Green}{\color{Green}}
\newcommand{\ppl}{\color{purple}}

\renewcommand{\phi}{\varphi}

\definecolor{purple}{RGB}{150,10,200} %xuanlong
\newcommand{\CAs}{$C^*$-algebras}

\newcommand{\Pp}{\mathcal{P}}
\newcommand{\Hh}{\mathcal{H}}

\newcommand{\G}{\mathcal{G}}
\newcommand{\F}{\mathcal{F}}

\newcommand{\N}{\mathbb{N}}
\newcommand{\Z}{\mathbb{Z}}
\newcommand{\Q}{\mathbb{Q}}
\newcommand{\R}{\mathbb{R}}
\newcommand{\C}{\mathbb{C}}
\newcommand{\T}{\mathbb{T}}

\numberwithin{equation}{section}

\newcommand{\Hom}{\operatorname{Hom}}
\newcommand{\Ext}{\operatorname{Ext}}

\newcommand{\Aff}{\operatorname{Aff}}

\newcommand{\id}{\operatorname{id}}

\newcommand{\Her}{\mathrm{Her}}

\newcommand{\cpc}{completely positive contractive linear map}
\newcommand{\morp}{contractive completely positive linear map}

\newcommand{\hm}{homomorphism}
\newcommand{\dt}{\delta}
\newcommand{\ep}{\varepsilon}

\newcommand{\td}{\tilde}

\newcommand{\DT}{\Delta}

\newcommand{\la}{\langle}
\newcommand{\ra}{\rangle}
\newcommand{\andeqn}{\,\,\,{\rm and}\,\,\,}
\newcommand{\rforal}{\,\,\,{\rm for\,\,\,all}\,\,\,}
\newcommand{\CA}{$C^*$-algebra}
\newcommand{\SCA}{$C^*$-subalgebra}

\newcommand{\af}{{\alpha}}
\newcommand{\bt}{{\beta}}

\newcommand{\Tw}{\overline{T(A)}^w}
\newcommand{\wtd}{\widetilde}

\newcommand{\diag}{{\rm diag}}

\newcommand{\wilog}{without loss of generality}
\newcommand{\Wlog}{Without loss of generality}

\newcommand{\D}{\mathbb D}
\newcommand{\beq}{\begin{eqnarray}}
\newcommand{\eneq}{\end{eqnarray}}
\newcommand{\tforal}{\,\,\,\text{for\,\,\,all}\,\,\,}
\newcommand{\tand}{\,\,\,\text{and}\,\,\,}

\newcommand{\ro}{\rho}
\newcommand{\p}{\mathfrak{p}}
\newcommand{\q}{\mathfrak{q}}

\usepackage{amsfonts}
\usepackage{mathrsfs}
\usepackage{textcomp}

\usepackage[all]{xy}

\title{  Extensions of  C*-algebras}

\author{James Gabe, Huaxin Lin  and Ping Wong Ng
 }

\date{
}
\begin{document}

\maketitle

\begin{abstract}
Let $A$ be a separable amenable \CA, and $B$ a non-unital  and $\sigma$-unital simple 
\CA\, with continuous scale
($B$ need not be stable).  We classify, up to unitary equivalence, all essential extensions of the form $0 \rightarrow B \rightarrow 
D \rightarrow A \rightarrow 0$ using KK theory.      

There are characterizations of when the relation of weak unitary equivalence is the same as 
the relation of
unitary equivalence, and characterizations of when an extension is liftable (a.k.a.~trivial or split).   
In the case where $B$ is purely infinite, an essential extension $\rho : A \rightarrow M(B)/B$ 
is liftable if and only if 
$[\rho]=0$ in $KK(A, M(B)/B)$.   When $B$ is stably finite, 
the extension $\rho$ is often not liftable when $[\rho]=0$ in $KK(A, M(B)/B).$ 
%does not guarantee that  $\rho$ is liftable.

Finally, when $B$ additionally has tracial rank zero and when $A$ belongs to a sufficiently regular class of unital separable amenable C*-algebras, 
%then 
we have a
version of the Voiculescu noncommutative Weyl--von Neumann theorem:  Suppose that $\Phi, \Psi:  A \rightarrow M(B)$ are unital injective homomorphisms  such that $\Phi(A) \cap B = \Psi(A) \cap B = \{ 0 \}$ and $\tau \circ \Phi = \tau \circ \Psi$ for all $\tau \in T(B),$ {the tracial state space of $B.$}  Then there exists a sequence $\{ u_n \}$ of unitaries in $M(B)$ such that (i) $u_n \Phi(a) u_n^* -
\Psi(a) \in B$ for all $a \in A$ and $n \geq 1$,  (ii)  $\|   u_n \Phi(a) u_n^* -
\Psi(a) \| \rightarrow 0$ as $n \rightarrow \infty$ for all $a \in A$.

\end{abstract}

\section{Introduction}

The study of extensions of \CA s began with the classification of essentially normal 
operators, then Brown--Douglas--Fillmore (BDF) theory and finally developed into 
Kasparov's KK-theory.  Let $A$ be a separable amenable \CA\, in the UCT class
and ${\cal K}$ be the  \CA\, of compact operators on an infinite dimensional 
separable Hilbert space $H.$   BDF--Kasparov theory states that the non-unital essential 
extensions of the form
$$
0\to {\cal K}\to E\to A\to 0
$$ 
are uniquely determined by $KK^1(A, \cal K)$ up to unitary equivalence. 
Let $\tau_1, \tau_2: A\to C({\cal K}),$ where $C({\cal K})$ is the Calkin algebra, 
be two {{non-unital}} monomorphisms.\footnote{{In extension theory, the unital and non-unital cases
often need to be separated. We provide more details, as well as relevant background
material, in  the paper.}}  Then there exists a unitary  $u \in C({\cal K})$ such that
\beq
u^*\tau_1(a)u=\tau_2(a) \rforal a\in A
\eneq
if and only if $[\tau_1]=[\tau_2]$ in $KK(A, C({\cal K})) \cong KK^1(A,\mathcal K).$ Moreover 
$\tau_1$ is liftable (i.e., trivial -- meaning that
 there is a \hm\, $\phi: A\to M({\cal K})=B(H)$ {{for
which $\tau_1 = \pi \circ \phi$, where $\pi : B(H) \rightarrow B(H)/{\cal K}$ is the usual
quotient map}}) if and only 
if $[\tau_1]=0$ in $KK(A, C({\cal K})).$ 

The above mentioned result  fails if we replace ${\cal K}$ by a general separable simple \CA\, 
$B$ (when $M(B)/B$ is not simple -- see \cite{LinExtRR0III}, for example). 
One of the important features of the Calkin algebra is that it is simple and purely infinite.
It is known that a $\sigma$-unital {{non-unital}} simple \CA\, $B$ has a simple corona algebra 
$C(B)=M(B)/B$ if and only if {{either $B = {\cal K}$ or $B$ has continuous scale }} (see Definition \ref{d:contscale}).

%{\color{cyan}
\begin{dfintro} \label{df:May2420235AM}
Let $A$ be a separable amenable $C^\ast$-algebra and let $B$ be a non-unital but $\sigma$-unital simple $C^\ast$-algebra with 
continuous scale. Denote by {{$\Ext(A, B)$}} the set of 
%strong 
{{(strong) unitary equivalence}} classes 
of non-unital essential extensions of the 
form $0\to B\to E\to A \to 0$ 

\noindent
{{(see  Definition \ref{d:Ext} and Remark \ref{r33}).}}
\end{dfintro}
%}

We wish to emphasize that we are not assuming that $B$ is stable in the above definition, in contrast to  classical BDF-Kasparov theory.   An important feature about Kasparov's $KK$-groups $KK^1(A,B)$ is that it can be used to classify certain non-unital extensions $0 \to B\otimes \mathcal K \to E \to A \to 0$. Since $KK^1(A,B) \cong KK(A, C(B\otimes \mathcal K))$ by Proposition 4.2 in \cite{DadarlatTopExt}, the following is really an extension of this classical result for $B$ with continuous scale.

\iffalse
{\blue{We also note that while in Definition \ref{df:May2420235AM} and for
some of 
the main results stated in the rest of this introduction (e.g., see
Theorem \ref{thm:1}), the extensions are non-unital, we also
have corresponding results for the unital case. 
We here only state the results
for the non-unital case mainly because the theory is smoother for non-unital
extensions. For example, without appropriate assumptions, the set of
all unitary equivalence classes of
unital essential extensions (for given domain algebra and canonical ideal) 
need not even be a group.}}
\fi

\begin{thmintro}\label{thm:1}
Let $A$ be a separable amenable $C^\ast$-algebra and let $B$ be a non-unital but $\sigma$-unital simple $C^\ast$-algebra with continuous scale. Then $\Ext(A,B)$ is canonically isomorphic to the Kasparov group $KK(A, M(B)/B)$\, {{(see Theorem \ref{TH2}).}}
\end{thmintro}

%We show that $\Ext(A, B)$ is isomorphic to $KK(A, C(B))$ (as an abelian group). 
While this looks like the perfect generalization of the BDF--Kasparov theorem, there are several 
important differences. First, when $B$ is not purely infinite, an essential extension 
$\tau: A\to C(B)$ with $KK(\tau)=0$ may not be liftable (i.e., trivial), {{or the short exact sequence 
in Definition A  may not be split.}} In fact, in some cases, 
$\tau$ is never liftable whenever $KK(\tau)=0$, even when trivial extensions exist. 
Second, trivial extensions may not be unitarily equivalent. In fact, there may be 
infinitely many inequivalent extensions represented by different $KK$-elements 
which are liftable. 
Therefore, it is important to know which $KK$-elements represent  liftable extensions. 
  In special cases we will characterize in terms of $KK$-theory and traces when extensions are liftable. 
 {{Here we only mention the non-unital case in the introduction for convenience as the similar statement in the
 unital case  for non-stable \CA s $B$ would be somewhat complicated as  the unitary equivalence classes 
 of unital essential extensions may not form a group  (see Remark
\ref{rem:UnitalAndNotASemigroup}).}} 

 %(see ?).}} {\red{\bf Ping: Which reference is the best?}}

 In the following we let $QT(B)$ denote the convex set of quasitracial states on $B$
 and $\Aff(QT(B))$ the space of all real continuous affine functions on $QT(B).$
 {{We consider the case where $B$ has continuous scale, strict comparison and stable rank one (this will be handled with slightly weaker hypothesis). In this case we have  
  $K_0(M(B))\cong \Aff(QT(B))$ and  $K_0(M(B)/B)\cong K_1(B)\oplus \Aff(QT(B))/\rho_B(K_0(B))$
  (see Theorem \ref{thm:KComp}).
  Suppose that $\psi: A\to M(B)/B$ is the Busby invariant of an essential extension defined 
  in Definition A and there is a  \hm\, $\phi: A\to M(B)$ such that
  $\pi\circ \phi=\psi.$  Note that $\phi$ induces a \hm\, $\phi_{*0}: K_0(A)\to \Aff(QT(B))$
  and $\pi_{*0}\circ \phi_{*0}=\psi_{*0}.$ Hence 
  $\psi_{*0}$ maps $K_0(A)$ into $\Aff(QT(B))/\rho_B(K_0(B)).$  Also 
  $\psi_{*0}$  can be lifted to 
  an order preserving  \hm\, $\gamma: K_0(A)\to \Aff(QT(B)).$ Moreover  $\gamma$ needs 
  to factor through $\Aff(QT(A)),$ the space of all real continuous affine functions on $QT(A).$
  Furthermore, $\tau\circ \phi$ gives a faithful quasitrace (instead of an arbitrary quasitrace).
  Only these $\psi_{*0}$ could  possibly give  liftable extensions. We also discuss whether every 
  one of them is liftable. 
   Theorem \ref{Tlifable}  
  gives an elaborate  answer to the liftable question.}}

 \iffalse
 , and let $T_{<1,f}(A)$ denote the convex set of faithful positive tracial functionals of norm strictly less than $1$.   
If $B$ has continuous scale, then $QT(B) = QT(M(B))$ canonically since $M(B)/B$ is purely infinite, and all quasitraces on $M(B)$ are faithful. Therefore, if a Busby invariant $\tau \colon A \to M(B)/B$ lifts to a $\ast$-homomorphism $\phi \colon A \to M(B)$, then there is an induced affine continuous map $\gamma \colon QT(B) \to T_{<1,f}(A)$ which is compatible with the $K$-theory map $\phi_{\ast 0} \colon K_0(A) \to K_0(M(B))$ in the sence that for every $\tau \in QT(B)$ and $x\in K_0(A)$ one has $\gamma(\tau)_\ast(x) = \tau_\ast (\phi_{\ast 0}(x))$. It turns out that in nice cases, liftability can be characterised by witnessing this liftability on the level of $KK$-theory with a lift compatible with traces.

 The following result is a special case of a much more general result on liftability.

%{\ppl{Need a word for unital case???}}

{\color{cyan}
\begin{thmintro}\label{thm:2}
Let $A$ be a separable amenable $C^\ast$-algebra satsfying the UCT, and let $B$ be a non-unital, separable, simple $C^\ast$-algebra with continuous scale, stable rank one, and with strict comparison of positive elements (CHECK ASSUMPTIONS). Let $\tau \colon A \to M(B)/B$ be the Busby invariant of a non-unital essential extension. Then $\tau$ is liftable if and only if there is a lift of $KK(\tau)$ in $KK(A, M(B))$ which is compatible in $K$-theory with an affine continuous map $QT(B) \to T_{<1,f}(A)$.
\end{thmintro}}
\fi

\begin{exmintro}
Let $A$ be any separable amenable $C^\ast$-algebra satisfying the UCT, and let us consider the {{case,}} where $B = \mathcal W$ is the {{Razak}}
%-Jacelon 
algebra. Then $\mathcal W$ fits into  Theorem \ref{Tlifable},
%the above theorems, 
$K_\ast(\mathcal W) = 0$ and $\mathcal W$ has a unique (quasi)tracial state. Moreover, there are canonical isomorphisms $K_0(M(\mathcal W)) \cong K_0(C(\mathcal W)) \cong  {{\Aff (T(\mathcal W))}} \cong \mathbb R$ and $K_1(M(\mathcal W)) = K_1(C(\mathcal W)) = 0$. From this and the UCT, it is straightforward to check that there are canonical isomorphisms
\[
\Ext(A, \mathcal W) \cong KK(A,C(\mathcal W)) \cong KK(A,M(\mathcal W)) \cong \Hom(K_0(A), \mathbb R)
\]
by Theorem \ref{thm:1}. 
{{One should note that, if $A$ has no faithful tracial state, then there is no liftable essential extensions.
Suppose that $A$ has a faithful tracial state. Then
 by Theorem \ref{Tlifable} in section 6, an essential extension $\tau: A\to C(\mathcal W)=M(\mathcal W)/\mathcal W$ is liftable if and only if $\tau_{*0}=r\rho_A(t)$ for some faithful tracial state $t$ and $r\in (0,1],$
where $\rho_A: K_0(A)\to \Aff(T(A))$ is the usual pairing (see also Theorem 8.3 of \cite{LN}).
One then observes}} that a  non-unital extension $\tau$ with $KK(\tau) = 0$ is liftable if and only if $A$ has a faithful trace  and $K_0(A)={\rm ker}\rho_A.$
%which induces the zero pairing map $K_0(A) \to \mathbb R.$
{{In particular, if $A=C(X)$ for some compact metric space $X$, non-unital extensions $\tau: A\to C({\cal W})$
with
$KK(\tau)=0$ can never be liftable.}}
\end{exmintro} 

\iffalse
{\color{cyan}
Moreover, by Theorem \ref{thm:2}, the liftable non-unital extensions are exactly identified with the pairing maps $\tau_\ast \colon K_0(A) \to \mathbb R$ induced by  some $\tau\in T_{<1, f}(A)$.
%
In particular, observe that the non-unital extension with $KK(\tau) = 0$ is liftable if and only if $A$ has a faithful trace which induces the zero pairing map $K_0(A) \to \mathbb R$.}
\end{exmintro}
\fi

%{\bf{\red{If we bring this here, we might want to discuss the pairing a bit more, for example, when $A$ 
%is a simple (unital) AF-algebra.}}}

Recall that $\tau: A\to C({\cal K})$ is a liftable extension (or trivial extension)
if there is a monomorphism $\phi: A\to B(H)$ such that 
$\pi\circ \phi=\tau.$ Suppose that $\phi_1, \phi_2: A\to B(H)$ are      
two unital monomorphisms with $\phi_i(A)\cap {\cal K}=\{0\}$ ($i=1,2$). 
By Voiculescu's  theorem, there exists a sequence of unitaries $\{u_n\} \subset  B(H)$
such that
\beq\label{intr-1}
&&\lim_{n\to\infty}\|u_n^* \phi_1(a)u_n-\phi_2(a)\|=0
%\rforal a\in A
\andeqn\\\label{intr-2}
&& u_n^* \phi_1(a)u_n-\phi_2(a)\in \mathcal K
\eneq
for all $a\in A.$  Naturally, we also search for a generalization for 
this important feature of extensions.   

However, as we mentioned above, there may be many liftable 
extensions in general. Therefore one may not obtain the uniqueness part of the 
Voiculescu theorem.  Suppose that $\phi_1, \phi_2: A\to M(B)$ are 
two unital monomorphisims such that  $\pi_B\circ \phi_1=\pi_B\circ \phi_2,$
where $\pi_B: M(B)\to M(B)/B$ is the quotient map. 
One quickly realizes that one may not expect that they will be approximately unitarily
equivalent as in \eqref{intr-1}.  Much of the work of this research is to 
investigate this problem.    We restrict ourselves to those 
separable simple \CA\, $B$  which have strict comparison and stable rank one.
In order to have sufficiently many quasidiagonal extensions, we also assume that
$B$ has tracial rank zero. 

\begin{thmintro}\label{thm:3}
Let $A$ a separable amenable unital simple $\mathcal Z$-stable $C^\ast$-algebra satisfying the UCT (or a unital AH-algebra) and let $B$ be a non-unital, separable, simple $C^\ast$-algebra with tracial rank zero and continuous scale.
% (CHECK ASSUMPTIONS). 
Suppose that $\phi_1, \phi_2 \colon A \to M(B)$ are two unital $\ast$-homomorphisms such that $\pi_B \circ \phi_1 , \pi_B \circ \phi_2 \colon A \to M(B)/B$ are both injective. 

Then there exists a sequence $\{u_n\}$ of unitaries in $M(B)$ such that
 \beq
 &&\lim_{n\to\infty} \|u_n^*\phi_1(a)u_n-\phi_2(a)\|=0\andeqn\\
 && u_n^*\phi_1(a)u_n -\phi_2(a)\in B
 \eneq
 for all $a\in A$,  if and only if $\tau\circ \phi_1=\tau\circ \phi_2$ for all $\tau\in T(B)$ ($= T(M(B))$).
\end{thmintro}
%
%For  a \CA\, $A$ in a class of separable \CA s which includes all separable unital AH-algebras 
% with a faithful trace and all separable classifiable  simple \CA s,   we show that any two 
%{\blue{unital}}  monomorphisms $\phi_1, \phi_2: A\to M(B)$ with property 
% that $\phi_i(A)\cap B=\{0\}$ ($i=1,2$) there exists 
% a sequence of unitaries $\{u_n\}\subset M(B)$ such that
% \beq
% &&\lim_{n\to\infty} \|u_n^*\phi_1(a)u_n-\phi_2(a)\|=0\andeqn\\
% && u_n^*\phi_1(a)u_n -\phi_2(a)\in B
% \eneq
%for all $a\in A,$  if and only if $\tau\circ \phi_1=\tau\circ \phi_2$ for all $\tau\in T(B),$
%the tracial state space of $B$. 

\vspace*{3ex}

This article is organized as follows.

Section 2: We  give a list of notation and some {{definitions}}
that will be used throughout the article.

Section 3: We  present some clarification of certain notions around essential extensions and some 
more {{definitions and preliminary}} 
%definitions and
results focused on extensions.

Section 4. 
 We present one of the main results: Let $A$ be separable  amenable \CA\, which satisfies the UCT and $B$ be a non-unital and $\sigma$-unital 
simple \CA\, with continuous scale.
Then $\Ext(A, B)=KK(A, C(B))$.  There are also other results,
including some in the unital case. 

Section 5:  This section provides  {{some $K$-theory computation for 
the multiplier algebras
$M(B)$ and  corona algebras $M(B)/B$ for certain non-unital simple \CA s.}}
%restatement of extension of purely infinite simple \CA s.

Section 6: In this section, we introduce two classes of amenable \CA s ${\cal A}_1$ and ${\cal A}_2.$
We also discuss the liftable extensions (i.e., trivial extensions)
 in the case of stably finite \CA s. 

Section 7: In this section we briefly
discuss quasidiagonal, approximately trivial and
null extensions. We present cases where all null extensions 
are trivial, cases where there are always some null extensions that
 are not trivial and 
cases where all null extensions are not trivial.

Section 8: This section contains some technical lemmas about order
preserving homomorphisms between certain ordered groups.
This includes technical results which allow certain order preserving \hm s
on an ordered subgroup to be extended, as well as ``quasidiagonal-type"  
 decompositions
of certain order preserving \hm s between ordered groups which allow for certain
control over the ``size of the blocks".   

Section 9:  This section studies quasidiagonal extensions 
 and the technical results  are preparation for the proof of 
the main theorem in section  10. 

Section 10: This section 
introduces a class of separable amenable \CA s 
${\cal A}_0$ and presents one of the main theorems:  A Voiculescu--Weyl--von Neumann theorem
for separable \CA s in the class ${\cal A}_0.$ 

Section 11:  In this section, we show that large classes of \CA s 
belong to ${\cal A}_0$.  Among other things, 
all  unital separable simple amenable ${\cal Z}$-stable \CA s 
satisfying the UCT 
and all unital AH-algebras with faithful traces are in ${\cal A}_0.$ 

%Section 12:  This section serves as an appendix. 

{\bf Acknowledgements}

J. Gabe has been supported by the Independent Research Fund Denmark through the Sapere Aude: DFF Starting Grant 1054-00094B, and by the Australian Research Council grant DP180100595. 
H. Lin is partially supported by a grant of NSF (DMS 1954600).   H. Lin and P. Ng acknowledge the support from 
the Research Center for Operator Algebras in East China Normal University
 which is in part supported  by NNSF of China  (12271165 and 12171165)  and Shanghai Science and Technology
 Commission (22dz2229014)
 and  Shanghai Key Laboratory of PMMP.

%{\red{The main results:}}

%{\red{ 1) Theorem \ref{TH2},}}

%{\red{2) Theorem \ref{TembeddingAH}.}}

%{\red{3) Theorem \ref{Tlifable},}}

%{\red{4) Corollary  \ref{Cor9}.  {\red{Voiculescu}}}}
%...
   
   \section{Preliminaries}

We provide some preliminary definitions and notation.  Other notation will be introduced in the 
text 
%body of the
%paper 
as the need arises.
 
Let $A$ and $B$ be \CA s. Throughout this article by a \hm\, 
 $h: A\to B$  we mean a $*$-\hm\, $h: A\to B.$
 {{Denote by ${\cal K}$ the \CA\, of compact operators on the separable infinite dimensional Hilbert space 
 $l^2.$}}
 
\begin{df}\label{D1}
Let $A$ be a C*-algebra.  
Denote by $A^{\bf 1}$ the closed unit ball of $A$ and 
$A_+$ the positive part of $A.$
If $x\in A,$ denote by $\Her(x)={\overline{xAx^*}},$
the hereditary \SCA\, of $A$ generated by $x.$

{{Suppose that $V \in M(A)$, the multiplier algebra of $A$. Denote by
${\rm Ad}V : A \rightarrow A$ the completely positive map given by
${\rm Ad}V (a)=V^*aV$, for all $a \in A$.}}  
{{Denote by ${\rm Ped}(A)$ the Pedersen ideal of $A.$ Denote  
${\rm Ped}(A)_+={\rm Ped}(A)\cap A_+.$}}

If $A$ is non-unital, $\tilde{A}$ denotes the {{minimal}} 
unitization of $A$
(\cite{WeggeOlsen} Prop. 2.1.3).  
For general $A$,
$A^+$ denotes the \emph{forced unitization} of $A$, i.e., 
$A^+ := \tilde{A}$ when $A$ is non-unital, and $A^+ := A \oplus \mathbb{C}$ when $A$ is unital.  
If $A$ and $C$ are C*-algebras with $C$ unital and $\phi : A \rightarrow C$
is a \hm,
%*-homomorphism, 
the \emph{forced unitization} of $\phi$ is the homomorphism
$\phi^+ : A^+ \rightarrow C$ for which $\phi^+(1_{A^+}) = 1_C$ and $\phi^+|_A = \phi$.

\end{df}

\begin{df}\label{Dtrace}

{\rm Let $A$ be a \CA.  Denote by $T(A)$ the tracial state space of $A$ {{(which could be an empty set)}}
given the weak* topology.
Denote by $T_f(A)$ the set of all faithful
tracial states of $A$. $T_f(A)$  is a convex subset of $T(A).$ 
Note that when $A$ is 
a separable simple C*-algebra such that either $A$ is unital or $A$ has continuous scale
(see Definition \ref{df:ContinuousScale}), then $T(A) = T_f(A)$ is a (weak*) compact convex set. 
{{We extend $T(A)$ to $A\otimes {\cal K}$ as follows:
if $a\in {\rm Ped}(A\otimes {\cal K}),$ $\tau(a)=(\tau\otimes Tr)(a)$ for all $\tau\in T(A),$
where $Tr$ is the standard (non-normalized) trace on ${\rm Ped}({\cal K}).$}}

Let ${\tilde{T}}(A)$ denote the cone of densely defined,
positive, 
%(norm) 
lower semi-continuous traces on $A$, equipped with the topology
of pointwise convergence on elements of the Pedersen ideal  ${\rm Ped}(A)$ of $A.$
Denote by $\td T_f(A)$ the subcone of $\td T(A)$ consisting of the faithful traces (i.e.,
strictly positive on ${\rm Ped}(A)_+ \setminus \{0\}$).
Set  $T_0(A)=\{r\cdot \tau: r\in [0,1], \,\, \tau\in T(A)\}.$
Let $\overline{T(A)}^w$ the weak*-closure of $T(A)$ in ${\tilde T}(A).$ 
Note that $\overline{T(A)}^w\subset T_0(A).$

{{
Let $B$ be another \CA\, with $T(B)\not=\emptyset$ and let $\phi: A\to B$ be a 
homomorphism.  
We will then use  $\phi_T:\td T(B)\to \td T(A)$ to denote the induced 
affine map.
If, in addition, $\phi$ brings an approximate unit of $A$ to an approximate
unit of $B$, the restriction of $\phi_T$ to $T(B)$
gives an affine continuous map
$T(B) \rightarrow T(A)$, which we still denote by $\phi_T$.}} 

Let $I$ be a  (closed two-sided) ideal of $A$ and $\tau\in T(I).$
It is well known that $\tau$ can be uniquely extended to a  tracial state of $A$
(by taking $\tau(a)=\lim_\af \tau(ae_\af)$ for all $a\in A,$ where $\{e_\af\}$ is a
{{quasi-central}} approximate identity for $I$; recall that since $\tau \geq 0$, $1 = \| \tau \| = \lim_{\alpha} \tau(e_{\alpha})$ {{ -- see Definition 2.5 of \cite{Lincrell}).}} In what follows, we will continue to use $\tau$ for the extension.
{{Also,}} when $A$ is not unital and $\tau\in T(A),$ we will use
$\tau$ for the extension to ${\widetilde{A}}$ as well as to $M(A),$
the multiplier algebra of $A.$}
\end{df}

\begin{df}\label{Dppp}
Let $S$ be a convex subset of a locally convex topological vector space. 
We denote
\beq\nonumber
&&\Aff(S) := \{ f : S \rightarrow \mathbb{R} : f \makebox{  is 
%bounded,
affine and continuous, and } f(0)=0 \},\\ 
&&\Aff_+(S):=\{f\in \Aff(S): f(t)>0\rforal t\in S\setminus \{0\}\}\cup \{0\},\\  
%${\rm LAff}(S) := \{ f : S \rightarrow (-\infty, \infty] : f   
%\makebox{  is affine and lower semi-continuous  }\}$, and\\
 %$
&&{\rm LAff}_+(S)  :=
\{f: S \to [0, \infty]:  \exists  f_n\in \Aff_+(S), \, 
f_n\nearrow f  \makebox{  pointwise  }\}\cup \{0\}.
\eneq
(If $0\not\in S,$ we ignore the condition that $f(0)=0.)$
%Note that ${\rm LAff}_+(S) \subset {\rm LAff}(S)$. 
Denote by $\Aff^b(S)$ the set of all 
bounded continuous functions in $\Aff(S).$  If $g\in \Aff^b(S),$ {{let}} $\|g\|=\sup\{|g(s)|: s\in S\}.$

{{In what follows, when we say $\Delta$ is a compact metrizable Choquet simplex and write 
$\Aff(\Delta),$ then 
$\Aff_+(\Delta)=\{f\in \Aff(\Delta): f(t)>0\rforal t\in \Delta\}\cup\{0\}.$ In other words, 
in this case, we assume there is no $0$ in $\Delta.$}}
\end{df}

\begin{df}\label{DefAq}
{{Let $A$ be a \CA\, with $T(A)\not=\emptyset$ and let $a\in {\rm{Ped}}(A\otimes {\cal K}).$
Define $\hat{a}(\tau)=\tau(a)$ for all $\tau\in T(A).$}}

{{Let $A_0$ be the set of those elements with the form 
$\sum_i x_ix_i^*-\sum_i x_i^*x_i.$
Let $A_{s.a.}^q$ be the quotient $A_{s.a}/A_0.$
It is a Banach space (see,  \cite{CP} and
Definition 2.2 of \cite{GLNI}). 
Denote by $A_{s.a.}^{q, {\bf 1}}$  the image of $A_{s.a.}^{\bf 1}$ in $A_{s.a.}^q.$}}
\end{df}

\begin{df}\label{Dfdt}
For all $0 < \delta < 1$,  define 
$f_{\delta} :[0, \infty) \rightarrow [0,1]$ to be  the unique continuous function such that
$$f_{\delta}(t) =
\begin{cases}
1 & t \in [\delta, \infty)\\  
0 & t \in [0, \frac{\delta}{2}]\\
\makebox{linear on  } & [\frac{\delta}{2}, \delta].
\end{cases}
$$
\end{df}

 \begin{df}\label{Drho}
 Let $A$ be a $\sigma$-unital \CA\, with $\tilde T(A)\not=\{0\}.$ 
 It follows from Definition 2. 6 of \cite{GLIIIrange} that there is a \hm\, $\rho_A: K_0(A)\to \Aff(\tilde T(A)).$ 
In the case that $A={\rm Ped}(A)$ (i.e., $A$ is compact, see Theorem 4.17 of \cite{eglnp}), 
% {\blue{(e.g., when $A$ is unital, and when 
%$A$ is ),}} 
every trace in  $\tilde T(A)$ is bounded on $A,$ and so 
$\rho_A([p]-[q])(\tau) =\tau(p)-\tau(q)$ for all $\tau\in \td T(A)$
and for projections $p, q\in M_n(\td A)$ (for some integer $n\ge 1$)
{{with}} $\pi_\C^A(p)=\pi_\C^A(q),$  {{when $A$ is not unital and}}  where $\pi_\C^A: \td A\to \C$ is the quotient map. 
{{Note that}} $p-q\in  M_n(A).$
Therefore, $\rho_A([p]-[q])$ is continuous on $\td T(A).$
%In the case that $\td T^b(A)=\td T(A),$ for example, $A={\rm Ped}(A),$
%we write $\rho_A:=\rho_A^b.$

Define $K_0(A)_+^\ro=\{x\in K_0(A): \rho_A(x)>0\}\cup \{0\}.$ It is a subsemigroup
of $K_0(A).$ 
 \end{df}
  
\begin{df}\label{Dproprho}
Let $A$ be a $\sigma$-unital stably finite \CA.
%such that $K_0(A)=G\oplus {\rm ker}\rho_A.$
We say $K_0(A)$ has property 
%${\Lro}$
({\Large{$\varrho$}})
 if $x\in K_0(A)_+$ and $y\in {\rm ker}\rho_A$ implies that 
 $x+y\in K_0(A)_+.$  If $K_0(A)_+^\ro=K_0(A)_+,$ then $K_0(A)$ has property ({\Large{$\varrho$}}).
\end{df}

 \begin{rem}\label{Rrho-1}
 
(1)  Let $A$ be a unital stably finite \CA\, with $\td T_f(A)\not=\emptyset.$  Denote 
${\rm ker}\rho_{A,f} := \{x\in K_0(A): \rho_A(x)(t)=0\rforal t\in \td T_f(A)\}.$ 
Then ${\rm ker}\rho_{A,f}={\rm ker}\rho_A.$ 
It is clear that ${\rm ker}\rho_{A}\subseteq {\rm ker}\rho_{A, f}$.   To see the other inclusion, 
let $x\in {\rm ker}\rho_{A,f}.$  For each $\tau\in \td T(A),$ choose $\tau_f\in \td T_f(A).$ 
Then $t = \tau+ \tau_f \in \td T_f(A)$.%$t=(1/2)(\tau+\tau_f)\in \td T_f(A).$ 
Therefore $\rho_A(x)(t)=\rho_A(x) (\tau_f) = 0$. It follows that $\rho_A(x)(\tau)=0.$
%Then $t=(1/2)(\tau+\tau_f)\in \td T_f(A).$ Therefore $\rho_A(x)(t)=0.$ It follows that $\rho_A(x)(\tau)=0.$

(2)  Let $A$ be a $\sigma$-unital stably finite \CA\, such that $K_0(A)=G\oplus {\rm ker}\rho_A,$
where $G$ is a subgroup.  Suppose that $K_0(A)$ has property ({\Large{$\varrho$}}) and 
% We also assume that $K_0(A)$ has the property: if $x\in K_0(A)_+$ and $y\in {\rm ker}\rho_A,$ then 
 %$x+y\in K_0(A)_+.$
 %$x\in K_0(A)_+\setminus \{0\}$ if and only if $\rho_A(x)(\tau)>0$ for all 
 %$\tau\in \td T(A).$ 
 %
 %suppose 
 that $K_0(A)_+\cap G$ is generated by a finite  subset $K_p$  as a semigroup. 
  Let $(g, \zeta)\in G\oplus {\rm ker}\rho_A,$ where $g\in G$ and $\zeta\in {\rm ker}\rho_A.$
 Suppose that $(g, \zeta)\in K_0(A)_+.$ Then $(g,0)=(g, \zeta)-(0, \zeta)\in K_0(A)_+.$ 
 It follows that, for any $x\in K_0(A)_+,$ $x=(g, \zeta)$ for some $g\in K_0(A)_+\cap G$ and $\zeta\in {\rm ker}\rho_A.$
 Hence, there are $r_1, r_2,...,r_n\in \Z_+$ and $g_1, g_2,...,g_n\in K_p$ and $\zeta\in {\rm ker}\rho_A$ such that 
 $x=\sum_{i=1}^n r_i g_i+\zeta.$ Conversely, all $x$ with such a decomposition (and with
at least one $r_i g_i$ nonzero) is an element of $K_0(A)_+.$

 \end{rem}

\begin{df}\label{Dcuntz}
Denote by $\{e_{i,j}\}$ a system of matrix unit for ${\cal K}.$ 
Let $A$ be a \CA. 
Consider $A\otimes {\cal K}.$  In  what follows, we identify $A$ with $A\otimes e_{1,1}$ 
as a hereditary \SCA\, of $A\otimes {\cal K},$ whenever it is convenient. 
Let 
%$
%A$ be a \CA\, and 
$a,\, b\in (A\otimes {\cal K})_+.$ 
We 
write $a \lesssim b$ if there is 
$x_i\in A\otimes {\cal K}$  for all $i\in \N$ 
such that
$\lim_{i\rightarrow\infty}\|a-x_i^*bx_i\|=0$.
We write $a \sim b$ if $a \lesssim b$ and $b \lesssim a$ both  hold.
%We also write $a\lesssim b$ and $a\sim b,$
%when {{there is no confusion.}}
The Cuntz relation $\sim$ is an equivalence relation.
Set ${\rm Cu}(A)=(A\otimes {\cal K})_+/\sim.$
%W(A):=M_{\infty}(A)_+/\sim$.
Let $[a]$ denote the equivalence class of $a$. 
We write $[a]\leq [b] $ if $a \lesssim  b$.
%We write $a\simle b,$ if there exists $x\in A\otimes {\cal K}$ such that
%$x^*x=a$ and $xx^*\in \Her(b).$
\end{df}

%{\color{cyan}REMARK: I'm changing ``quasi-trace'' to ``quasitrace''}

\begin{df}\label{Dqtr}
Let $A$ be a $\sigma$-unital \CA. 
A densely  defined 2-quasitrace  is a 2-quasitrace defined on ${\rm Ped}(A)$ (see  Definition II.1.1 of \cite{BH}). 
Denote by ${\widetilde{QT}}(A)$ the set of densely defined 2-quasitraces 
on %of 
$A\otimes {\cal K}.$  
%Suppose, for the convenience of this paper, that all 2-quasi-traces defined on ${\rm Ped}(A)$ 
%are traces. 
%
%Denote by ${\widetilde{T}}(A)$ the set of densely defined traces 
%on %of 
%$A\otimes {\cal K}.$  
 %In what follows we will identify 
 Recall that we identify 
$A$ with $A\otimes e_{1,1}.$
%$ whenever it is convenient. 
Let $\tau\in {\widetilde{QT}}(A).$  Note that $\tau(a)\not=\infty$ for any 
$a\in {\rm Ped}(A)_+\setminus \{0\}.$
% {\color{cyan} REMOVED: $\setminus \{0\}$}
%Let $\tau\in {\widetilde{QT}}(A).$  Note that $\tau(a)\not=\infty$ for any $a\in {\rm Ped}(A)_+\setminus \{0\}.$
%

We endow ${\widetilde{QT}}(A)$ 
{{with}} the topology  in which a net 
%$\tau_i$
${{\{}}\tau_i{{\}}}$ 
 converges to $\tau$ if 
 %$\tau_i(a)$ 
${{\{}}\tau_i(a){{\}}}$ 
 converges to $\tau(a)$ for all $a\in 
 {\rm Ped}(A)$ 
 (see also (4.1) on page {{985 of \cite{ERS}).}}
  Denote by $QT(A)$ the set of normalized 2-quasitraces of $A$ ($\|\tau|_A\|=1$).
  
%{\color{cyan}For $\ep >0$ we let $f_\ep\colon [0,\infty] \to [0,1]$ be the continuous function which is 0 on $[0,\ep/2]$, 1 on $[\ep, \infty]$, and which is affine on $[\ep/2, \ep]$.} 
Note that, for each $a\in ({{A}}%A_+
\otimes {\cal K})_+$ and $\ep>0,$ $f_\ep(a)\in {\rm Ped}(A\otimes {\cal K})_+,$  where $f_\epsilon$ is as in Definition \ref{Dfdt}. 
Define 
\beq
\widehat{[a]}(\tau):=d_\tau(a)=\lim_{\ep\to 0}\tau(f_\ep(a))\rforal \tau\in {\widetilde{QT}}(A).
\eneq

Let $A$ be a simple \CA\,with $\wtd{QT}(A)\setminus \{0\}\not=\emptyset.$
% whose  2-quasi-traces are all traces.   
Then
$A$
is said to have (Blackadar's) strict comparison, if, for any $a, b\in (A\otimes {\cal K})_+,$ 
the condition 
\beq
d_\tau(a)<d_\tau(b)\rforal \tau\in {\widetilde{QT}}(A)\setminus \{0\}
\eneq
implies that 
$a\lesssim b.$ 
%(see the paragraph after  Definition  \ref{DpureW}). 
%
%
%Except in subsection \ref{sub2}, we
%
%
If $A$ is exact, then every 2-quasitrace is a trace (\cite{Ha}).
\end{df}

\begin{df}\label{Dunitary}
Let $A$ be a unital \CA. Denote by $U(A)$ the unitary group of $A.$ and 
by $U_0(A)$ the path connected component of $U(A)$ containing $1_A.$
\end{df}

\section{Extension groups}

\begin{df}\label{Dmul}
For each \CA\, $B$, $M(B)$ denotes the multiplier algebra of
$B$, and  $C(B) := M(B)/B$ denotes
the corona algebra of $B$.   
$\pi_B : M(B) \rightarrow C(B)$ denotes the 
usual quotient map.  Sometimes we 
write $\pi$ instead of $\pi_B$. 
%In fact, we usually write $\pi$.}} 

For each  \CA\,
%C*-algebra 
extension 
\begin{equation} 0 \rightarrow B \rightarrow D \rightarrow  C \rightarrow 0  \label{Aug320215PM} \end{equation}
(of $C$ by $B$)\footnote{In the literature, the terminology is
sometimes reversed and this is sometimes called an ``extension of
$B$ by $C$".},
we will work with the corresponding \emph{Busby invariant}
which is the induced  homomorphism
 %homomorphism
$\phi : C \rightarrow C(B)$.    Recall that (\ref{Aug320215PM}) is essential if and only if $\phi$ is injective.  We will
mainly  work
%be working 
with essential extensions.

An essential extension is \emph{unital} exactly when the corresponding Busby invariant is a unital
map, i.e., if $\phi : C \rightarrow C(B)$ is the Busby invariant, then
$C$ is unital and $\phi(1_C) = 1_{C(B)}$.  
Hence, an essential extension with Busby invariant $\phi : C\rightarrow C(B)$ is 
\emph{non-unital} if $\phi$ is non-unital, i.e., either $C$ is not unital or
$C$ is unital but $\phi(1_C) \neq 1_{C(B)}$.                          
\end{df}

We next recall   the definitions of the basic equivalence relations of extensions (in terms of their Busby invariants), as well as the objects that
we will
% (with additional assumptions) 
{{compute.}}

\begin{df}\label{d:Ext}
Let $A, B$ be C*-algebras with $B$ non-unital, and let
$\phi, \psi : A \rightarrow C(B)$ be extensions.
\begin{enumerate}
\item $\phi$ and $\psi$ are \emph{unitarily equivalent} ($\phi \sim \psi$) if there exists a unitary $u \in M(B)$ such 
that $$\pi(u) \phi(.) \pi(u)^* = \psi(.).$$
\item $\phi$ and $\psi$ are \emph{weakly unitarily equivalent} ($\phi \sim_{wu} \psi$) if there exists a unitary
$v \in C(B)$ such that
$$v \phi(.) v^* = \psi(.).$$
\item $\phi$ and $\psi$ are \emph{weakly equivalent} ($\phi \sim_w \psi$) if there exists a partial isometry
$w \in C(B)$ for
which $w^* w \phi(a) =  \phi(a)$ for all $a \in A$, and
$$w \phi(.) w^* = \psi(.).$$
\item $\phi$ and $\psi$ are \emph{Murray--von Neumann equivalent} ($\phi \sim_{MvN} \psi$) if there exists a contraction $c \in C(B)$ such that
$$c\phi(.)c^\ast = \psi(.) \quad \textrm{and} \quad c^\ast \psi(.)c = \phi(.).$$
\item $\Ext(A, B)$ is the set of unitary equivalence classes of non-unital essential extensions
$A \rightarrow C(B)$.
\item $\Ext^w(A, B)$ is the set of weak equivalence classes of essential extensions 
$A \rightarrow C(B)$.
\item $\Ext^w(A, B)_1$ is the set of weak (unitary) equivalence classes of essential unital 
extensions $A  \rightarrow C(B)$. 
\end{enumerate}
\end{df}

\begin{rem}\label{r33}
In the literature, ``unitary equivalence" is sometimes called ``strong unitary equivalence". 
\end{rem}

%Much 
Many of the  results stated
% theory
 in this paper can be generalized to full extensions into properly infinite corona algebras.
However, we will focus on the case of simple purely infinite corona algebras, the most important and elegant
case of the theory.  We will investigate full extensions and more in a future paper.

The definition of a continuous scale C*-algebra was first introduced in \cite{LinContScaleI}.

\begin{df}\label{d:contscale}
Let $B$ be a $\sigma$-unital simple C*-algebra.  $B$ has \emph{continuous scale} if
$B$ has a sequential approximate unit $\{ e_n \}$ such that 
\begin{enumerate}
\item $e_{n+1} e_n = e_n$ for all $n$, and 
\item for all $a \in B_+ \setminus \{ 0 \}$, there exists $N \geq 1$ for which
for all $m > n \geq N$, $$e_m - e_n 
%\preceq 
\lesssim a.$$
\end{enumerate}

{{In the above, $\lesssim$ is the subequivalence relation between positive elements,
% (generalizing
%Murray-von Neumann subequivalence of projections) 
where for all $c, d \in B_+$,
$c \lesssim d$ means that there exists a sequence $\{ x_k \}$ in $B$
such that
% for which
$\lim_{k\to\infty}\|x_k d x_k^*-c\|=0.$}}
% in norm as $k \rightarrow \infty$.}} 
\label{df:ContinuousScale}
\end{df}

The main thing about continuity of the scale that concerns us here is that it characterizes simple pure infiniteness 
%{\color{cyan} 
of corona algebras
%.}

\begin{thm}\label{Tcontscal}
Let $B$ be a  non-unital but $\sigma$-unital 
simple C*-algebra.  Then the following
statements are equivalent:
\begin{enumerate}
\item  $B$ is either elementary or has continuous scale.
\item $C(B)$ is simple.
\item $C(B)$ is simple purely infinite.
\end{enumerate}
Suppose, in addition, that $B$ is stable.
Then the above statements are all equivalent to the following:
\begin{enumerate}
\item[(4)] Either $B$ is elementary or $B$ is simple purely infinite.
\end{enumerate}  
\end{thm}

\begin{proof}
See \cite{LinContScaleI} and \cite{LinSimpleCorona}.
See also \cite{RorMult} and \cite{ZhangRiesz}. 
\end{proof}

As indicated earlier, perhaps one reason for the success of  BDF theory is that their multiplier and corona algebras have particularly nice 
structure.  
For example,  the BDF--Voiculescu theorem, which roughly says that all essential extensions are absorbing, would not be true if the Calkin algebra $B(H)/\mathcal K$ were not simple.  
Thus, the continuous scale case is the basic test case for understanding extension theory, and as discussed in \cite{LinExtRR0III}, it is also the case that one must understand first in order to proceed to the general case.
It is also the most elegant case with the nicest theory, and thus is 
the focus of this paper. 

\vspace{0.1in}

%\subsection{
{\bf Algebraic structure}
%}
\vspace{0.1in}

To investigate the group structure on $\Ext(A, B)$ and other 
objects, we will need  a preliminary investigation into the connections 
between $\sim$, $\sim_w$, $\sim_{wu}$ and $\sim_{MvN}$.  
Note that it is immediate that
\[
\phi \sim \psi \quad \Rightarrow \quad \phi \sim_{wu} \psi \quad \Rightarrow \quad \phi \sim_w \psi \quad \Rightarrow \quad \phi \sim_{MvN} \psi. 
\]
It turns out that all four %three
relations are the same in the non-unital case.

{{For a {{\CA}}\,
%C*-algebra 
$C$ and a subset $E \subseteq C$,
we denote 
$$E^{\perp} := \{ x \in C : x y = yx = 0 \makebox{  for all } y \in E \}.$$}}

%{\blue{
%Finally, we introduce one more notation.  For a C*-algebra $C$ and $x \in C$,
%we let ${\rm Her}(x)$ denote the hereditary C*-subalgebra of $C$ given
%by ${\rm Her}(x) := \overline{x C x^*}$.}}     

The following dilation lemma, as well as similar results, are surely known to experts.

\begin{lem}\label{lem:dilation}
Let $D$ be a unital properly infinite $C^\ast$-algebra. Suppose that $p,q\in D$ are properly infinite full projections and that $x\in D$ is a contraction such that $x = (1-p)x(1-q)$. Then there exists a unitary $u\in D$ with $[u]_1 = 0 \in K_1(D)$  such that 
$$ au = ax \qquad \textrm{and} \qquad ub = xb $$
for all $a\in \{1-xx^\ast\}^\perp$ and $b\in \{1-x^\ast x\}^\perp$. 
\end{lem}
\begin{proof}
First, by a standard application of a theorem of Cuntz \cite{CuntzAnnals} we may find properly infinite full projections $p_2,p_3\in pDp$ such that $p_1 = 1-(p_2+p_3)$ is full and properly infinite, $[p_2]_0 = 0$ in $K_0(D)$ and $[p_3]_0 = [1_D]_0$. In particular, $p_1 \sim p_2$ and $p_1 \geq 1-p$. 

Similarly, we find a properly infinite full projection $q_1 \geq 1-q$ such that $q_1 \sim p_1$ and $1-q_1 \sim 1-p_1$. Hence there exists a unitary $w\in D$ such that $wp_1  = q_1 w$. 

Let $y = x w$. As $x = p_1 x q_1$, it follows that $y=p_1 x q_1w = p_1 x w p_1$ is a contraction in $p_1 Dp_1$, and therefore
\[
U = \left(\begin{array}{cc} y & (p_1-yy^\ast)^{1/2} \\ (p_1-y^\ast y)^{1/2} & -y^\ast \end{array} \right) \in M_2(p_1 D p_1)
\]
 is a unitary. Since $p_1 \sim p_2$ we let $s\in D$ such that $ss^\ast = p_1$ and $s^\ast s = p_2$. As $p_1 \perp p_2$, there is an isomorphism $\Phi \colon M_2(p_1 Dp_1) \xrightarrow \cong (p_1+p_2)D(p_1+p_2)$ given by
 \[
 \Phi\left(\begin{array}{cc} d_{11} & d_{12} \\ d_{21} & d_{22} \end{array} \right) = d_{11} + d_{12} s + s^\ast d_{21} + s^\ast d_{22} s
 \]
 for $d_{ij} \in p_1Dp_1$. 
 
  Let $v = \Phi(U) + p_3$ which is a unitary in $D$. Since $\{ 1-xx^\ast\}^\perp$ consists of all elements for which $xx^\ast$ acts as a unit, and since $p_1 \geq 1-p \geq xx^\ast$, we have $\{1-xx^\ast\}^\perp \subseteq p_1 D p_1$. Hence $a (p_1 -yy^\ast)^{1/2} = a (p_1 - xx^\ast)^{1/2} = 0$ for all $a\in \{ 1-xx^\ast\}^\perp$. In particular, for any $a\in \{1-xx^\ast\}^\perp$ we have
  $$ a v = a y + a(p_1 - yy^\ast)^{1/2} s = ay.$$
  It follows that $vw^\ast \in D$ is a unitary such that $avw^\ast = ayw^\ast = ax$ for all $a\in \{1-xx^\ast\}^\perp$. 
  
  As above, $\{ 1- x^\ast x\}^\perp \subseteq q_1 D q_1$ and $(q_1 - x^\ast x) b = 0 $ for $b\in \{1-x^\ast x \}^\perp$. Hence, for such $b$ we have
  \[
   (p_1-y^\ast y)^{1/2} w^\ast b = w^\ast (q_1 - x^\ast x)^{1/2} w w^\ast b = 0.
  \]
  Thus for $b\in \{1-x^\ast x\}^\perp$, we have
  \[
  vw^\ast b = v w^\ast q_1 b = vp_1 w^\ast b = (y + s^\ast (p_1 - y^\ast y)^{1/2})w^\ast b = xww^\ast b = xb.
  \]
So it remains to correct the $K$-theory class of $vw^\ast$. For this, we use that $1-q_1$ is a full properly infinite projection together with a theorem of 
\cite{CuntzAnnals} to find a unitary $u' \in (1-q_1) D (1-q_1)$ such that $[q_1+u']_1 = - [vw^\ast]_1$ in $K_1(D)$. Then $u = vw^\ast (q_1+u')$ is a unitary in $D$ with $[u]_1 = 0$ and for $a\in \{1-xx^\ast\}^\perp$ we have
\[
au = avw^\ast (q_1 + u') = ax(q_1 + u') = ax (1-q) (q_1 + u') = ax(1-q) = ax
\]
and for $b\in \{1-x^\ast x\}^\perp$ we get
\[
ub = vw^\ast (q_1 + u') b = vw^\ast b = xb
\]
as desired.
\end{proof}

Let $\phi : A \rightarrow C$ be a homomorphism between C*-algebras with
$C$ unital and properly infinite.  We say that $\phi$ has \emph{large complement} if there exists
a projection $p \in \phi(A)^{\perp}$ for which $p \sim 1_{C}$.  

%{\color{cyan}
Recall that a unital $C^\ast$-algebra $D$ is \emph{$K_1$-injective} if whenever $u\in D$ is a unitary with trivial $K_1$-class, then {{$u\in U_0(D).$}}
%$ is homotopic to $1$ in the unitary group of $D$. 
By a theorem of \cite{CuntzAnnals}, all simple unital purely infinite $C^\ast$-algebras are $K_1$-injective. It is an open problem whether every unital properly infinite $C^\ast$-algebra is $K_1$-injective.

%{\color{cyan}
\begin{prop}\label{p:largecompMvN}
Let $A$ and $B$ be $C^\ast$-algebras and suppose that $C(B)$ is properly infinite and $K_1$-injective. If two extensions $\tau_1, \tau_2 \colon A \to C(B)$ have large 
complement and are Murray--von Neumann equivalent, then they are unitarily equivalent.
\end{prop}
\begin{proof}
Let $c\in C(B)$ be a contraction such that $c^\ast \tau_1(a)c = \tau_2(a)$ and $c \tau_2(a) c^\ast = \tau_1(a)$ for all $a\in A$. As $\tau_i$ has large complement, we pick a projection $p_i \in \tau_i(A)^\perp$ such that $p_i \sim 1_{C(B)}$. Hence $p_i$ is full and properly infinite for $i=1,2$. By replacing $c$ with $(1-p_1) c (1-p_2)$, we may assume that $(1-p_1) c (1-p_2) = c$. By Lemma 3.8 of \cite{GabeCrelle} (which is a standard argument) we have $\tau_1(a) c c^\ast = \tau_1(a)$ for all $a\in A$, and thus $\tau_1(A) \subseteq \{1-cc^\ast\}^\perp$.  Similarly $\tau_2(A) \subseteq \{1-c^\ast c\}^\perp$. Applying Lemma \ref{lem:dilation}, there is a unitary 
$w\in {{C(B)}}$
%D$ 
such that $\tau_1(a)w = \tau_1(a) c$ and $w \tau_2(a) = c \tau_2(a)$ for all $a\in A$ and $[w]_1 = 0$ in $K_1(D)$. As {{$C(B)$}} 
%$D$ 
is $K_1$-injective, $w$ is homotopic to $1$ and thus it lifts to a unitary 
$u\in {{M(B).}}$
%M(D)$. 
Moreover, for all $a\in A$
$$ \tau_1 (a)\pi(u) = \tau_1(a) c = cc^\ast \tau_1(a) c = c \tau_2(a) = \pi(u) \tau_2(a)$$
and therefore $\tau_1$ and $\tau_2$ are unitarily equivalent. 
\end{proof}

%}

\begin{prop}\label{TH1Pnonunital}
Let $B$ be a non-unital and $\sigma$-unital simple \CA\, with  continuous scale,
%C*-algebra, 
let $A$ be a $\sigma$-unital  \CA, and let 
$\tau_1, \tau_2: A\to C(B)$ be two 
non-unital essential extensions.

Then 
$\tau_1$ (and also $\tau_2$) has large complement, and 
$$\tau_1 \sim \tau_2 \quad \Leftrightarrow \quad \tau_1 \sim_{wu} \tau_2 \quad  \Leftrightarrow \quad \tau_1 \sim_w \tau_2 \quad \Leftrightarrow \quad \tau_1 \sim_{MvN} \tau_2.$$
\end{prop}

\begin{proof}
We assume that $A$ is not zero. 
Let $e_A$ be a strictly positive element of $A$ and set $b:= \tau_1(e_A)$.
%Suppose that $A$ is not unital. 
%Then $0$ is not an isolated point of ${\rm sp}(e_A)$.   
%Then $0$ is not an isolated point of ${\rm sp}(b).$
%Therefore 
Then ${\rm Her}(b)$ is a 
%non-unital and 
$\sigma$-unital hereditary \SCA\, of $C(B).$   Since $\tau_1$ is not unital, $\Her(b)\not= C(B).$
%by
It follows from 
Theorem 15 
 of \cite{PedSAW} that ${\rm Her}(b)^\perp\not=\{0\}.$ Since ${\rm Her}(b)^\perp$
is also a hereditary \SCA\, of $C(B)$ which is purely infinite and simple, ${\rm Her}(b)^{\perp}$ is also 
purely infinite and simple. 
By \cite{sZhang}, it has real rank zero. 
Let $p\in {\rm Her}(b)^\perp$ be a nonzero projection. 
Similarly, we can find a nonzero projection $q \in {\rm Her}(b')^{\perp}$,
where $b' =_{df} \tau_2(e_A)$. 
{{Since again $C(B)$ is simple purely infinite, $1_{C(B)} \preceq p$ and
$1_{C(B)} \preceq q$.}}  Hence, both $\tau_1$ and $\tau_2$ have large complement. The result on the equivalences of $\tau_1$ and $\tau_2$ follows from Proposition \ref{p:largecompMvN}.
\end{proof}

\begin{rem}
In connection with Proposition \ref{TH1Pnonunital}, there are also cases where
for unital extensions, unitary and weak equivalence are always the same.  For example,
this is the case with domain algebra $A = C(X)$, where $X$ is a compact metric
space (the commutative case).  In Theorem \ref{TH1}, we give some necessary and
sufficient K-theory conditions  for when $\sim$ is always the same as $\sim_{wu}$.    
See also Remark \ref{RH1} which uses Theorem \ref{TH1}  to explain why, in many cases (including the 
commutative case and the non-unital case), $\sim$ agrees with $\sim_{wu}$.   
\label{rem:WeakAndStrongEquiv}
\end{rem}

We next discuss algebraic structure on the $\Ext$-groups from Definition \ref{d:Ext}.
%$Ext(A, B)$ and others. 
\begin{df}\label{DefBDF-1}
Let $B$ be a non-unital and $\sigma$-unital simple 
%continuous scale
C*-algebra with continuous scale. 
Let $A$ be a separable C*-algebra and 
$\phi, \psi : A \rightarrow C(B)$ two extensions.   
The \emph{BDF sum} of  $\phi$ and $\psi$ is defined to be
$$(\phi \oplus \psi)(.) =_{df} S \phi(.) S^* +  T \psi(.) T^*$$
where $S, T \in C(B)$ are two isometries for which
$$SS^* + TT^* \leq 1.$$
It is clear that the BDF sum is well-defined (independent of 
choices of $S$ and $T$) up to weak equivalence.
%{\red{\bf References?}}
{{In addition, by Proposition \ref{TH1Pnonunital},
when both $\phi$ and $\psi$ are essential and non-unital, the 
BDF sum is well-defined up unitary equivalence (and hence weak  
unitary equivalence).}} 
Thus, the BDF sum induces sums on
$\Ext(A, B)$ and $\Ext^w(A, B)$, which we also call the \emph{BDF sum}.  
\end{df}

The following is straightforward to check, and is left for the reader.

\begin{prop}
Let $A$ be a separable C*-algebra and $B$ a non-unital  but $\sigma$-unital
simple  C*-algebra with continuous scale.

Then $\Ext(A, B)$ and $\Ext^w(A, B)$ are  abelian semigroups.
\label{prop:Ext1Semigroup}
\end{prop} 

With $A$ and $B$ as in the above proposition, there is an obvious semigroup homomorphism $\Ext(A,B) \to \Ext^w(A,B)$.

\begin{thm}
Let $A$ be a separable nuclear C*-algebra and $B$ a non-unital  but
$\sigma$-unital simple 
%continuous scale
C*-algebra with continuous scale.

Then $\Ext(A, B)$ and $\Ext^w(A, B)$ are abelian groups, and the canonical homomorphism 
$$\Ext(A,B) \to \Ext^w(A,B)$$
is an isomorphism.
%In fact,
%$$\Ext(A, B) \cong \Ext^w(A, B).$$
\label{thm:Ext1Group}
\end{thm} 

\begin{proof}
That $\Ext(A, B)$ is a group follows from {{\cite[Theorem 2.10]{NgRobin}.}} 
 So it suffices to show that the  map $\Ext(A,B) \to \Ext^w(A,B)$ is bijective.

The map is injective by Proposition \ref{TH1Pnonunital}. For surjectivity, let $\phi \colon A \to C(B)$ be a (possibly unital) essential extension. Let $S\in C(B)$ be a proper isometry. Then $[\phi]_w = [S\phi(.)S^*]_w$, and $[S\phi(.)S^\ast] \in \Ext(A,B)$, so the map $\Ext(A,B) \to \Ext^w(A,B)$ is surjective.

\end{proof}

\vspace{0.1in}

%\subsubsection{
{\bf The unital case}

%}
\vspace{0.1in}

\begin{df}\label{DefBDF-2}
Suppose that $B$ is a non-unital and $\sigma$-unital simple 
\CA\, with continuous
scale
% C*-algebra
for which 
$[1_{C(B)}] = 0$ in $K_0(C(B))$. 
For any unital separable C*-algebra $A$ and
any two unital extensions $\phi, \psi : A \rightarrow C(B)$, we define
the \emph{BDF sum} to be 
$$(\phi \oplus \psi)(.) = S \phi(.) S^* + T \psi(.)T^*$$
where $S, T \in C(B)$ are isometries for which
$$SS^* + TT^* = 1.$$
In what follows in the case that $[1_{Q(B)}]=0,$ when we write $\Ext^w(A, B)_1,$ we assume 
this BDF sum is defined. 
\end{df}

This BDF sum is well-defined up to weak unitary equivalence, and thus
we have the following:

\begin{prop}
Let $B$ be a non-unital but $\sigma$-unital simple \CA\, with continuous
scale
%C*-algebra
such that    
$[1_{C(B)}] = 0$ in $K_0(C(B))$. 
Let $A$ be a unital separable C*-algebra.

Then $\Ext^w(A, B)_1$ is an abelian semigroup.
\end{prop}

\begin{thm}
Let $B$ be a  non-unital $\sigma$-unital simple continuous
scale C*-algebra
such that 
$[1_{C(B)}] = 0$ in $K_0(C(B))$.
Let $A$ be a unital separable amenable %nuclear
C*-algebra.

Then $\Ext^w(A, B)_1$ is an abelian group.
\label{thm:Extw1Group}
\end{thm}

\begin{proof}
The proof is exactly the same as that
of \cite[Theorem 2.13]{NgRobin}, except that, unitary equivalence
is replaced with weak unitary equivalence, since we do not assume
that $A$ has any one-dimensional representations.
\end{proof}

For the benefit of the reader, 
let us mention
% we end this section by mentioning 
%(without full 
%explanation) 
the theory for simple 
corona algebras of a stable canonical ideal. 
\begin{thm}
Let $B$ be a 
$\sigma$-unital and stable C*-algebra.  

Then the following statements are equivalent.
\begin{enumerate}
\item $B \cong {\cal K}$ or $B$ is simple purely infinite. 
\item For every separable C*-algebra $A$, every (unital or non-unital)
essential extension $A \rightarrow C(B)$ is nuclearly absorbing (in the unital or 
non-unital senses respectively).   
\end{enumerate}

As a consequence, when $A$ is a separable amenable %nuclear
C*-algebra, and $B$
satisfies the above conditions, then 
$$\Ext(A, B) = KK^1(A, B).$$
\end{thm}

\begin{proof}
This is essentially contained in \cite{ElliottKucerovsky}, 
\cite{VoiculescuWvN}, \cite{KasparovKK}, and 
\cite{GiordanoNg}.

Reference \cite{ElliottKucerovsky} is for separable canonical ideals $B$, but 
the case where $B$ is $\sigma$-unital is an easy and standard reduction.
\end{proof}

We note that, if in the above, $A$ also satisfies the Universal Coefficient
Theorem, then $KK^1(A, B)$ is in principle ``computable" from K-theory.

In the following section we will obtain similar results for computing the $\Ext$-groups using $KK$-theory when $B$ is not assumed to be stable but instead is assumed to have continuous scale.

{{
\begin{rem} \label{rem:UnitalAndNotASemigroup} 

{{Let $A$ be a unital separable \CA.

(1) We first note that, in order to  
%have the sum of  unital essential extensions 
%$\phi_1, \phi_2: A\to M(B)/B$  to  be weakly unitarily equivalent to a unital 
%essential; extensions, one needs a partial iso
 %one requires
have the BDF sum defined in \ref{DefBDF-2}  to work
for $\Ext^w(A,B)_1,$ 
we need to assume that there exists $V\in M_k(M(B)/B)$ such that
$V^*V=1_{M(B)/B}$ and $VV^*=1_{M_k(M(B)/B)}$ for every $k\in \N.$  
In other words, 
%%In other words, without 
the assumption that  
$[1_{C(B)}]=0$ is required.
% $Ext^w(A, B)_1$ is not even a semigroup.

(2)  One may define the sum of two maps $\phi_1,\phi_2: A\to C(B)$ in 
$C(B)\otimes {\cal K}.$ Since we assume that $B$ has continuous scale, 
$C(B)$ is purely infinite simple, this sum actually works well for $\Ext(A,B)$
and is equivalent to the BDF sum defined in \ref{DefBDF-1}.  However, without 
assuming $[1_{C(B)}]=0,$ the sums of unital essential extensions are not weakly unitarily 
equivalent to unital essential extensions. In other words, without assuming 
$[1_{C(B)}]=0,$ the addition 
could not be defined for $Ext^w(A, B)_1$ so that it becomes 
a semigroup.

 (3) However, by Theorem \ref{thm:ClassifyUnitalExt}
below,  we do have a desired classification for $\Ext^w(A,B)_1$
without assuming  that $[1_{C(B)}]=0$ (and without group structure). 
%does not prevent us to have the classification 
%of $Ext^w(A, B)_1.$ 

(4)  Let us mention that, in general, there is no reason to believe that
$[1_{Q(B)}]=0$ 
should be automatic.   Suppose that $B$ is a non-unital separable ${\cal Z}$-stable 
simple \CA\,   with continuous scale. Let $b$ be a strictly positive element $b.$ If $\widehat{\la b\ra}\not\in \rho_B(K_0(B)),$
then $[1_{M(B)}]\not\in \rho_B(K_0(B)).$ Consequently $[1_{C(B)}]\not=0$ in $K_0(C(B))$ (see section 5).
}}

Here is an example:  Let $D$ be the UHF-algebra with $K_0(D)=\Q,$ 
%group being
the rational numbers $\mathbb{Q}$ and unique tracial state $\tau$.  
Let $\{ p_n \}$ be a sequence of
pairwise orthogonal projections in $D \otimes {\cal K}$ such that
$\sum_{n=1}^{\infty} p_n$ converges strictly in $M(D \otimes {\cal K})$ and
$\sum_{n=1}^{\infty} \tau(p_n)$ is an irrational number in $(0, 1)$
(e.g., $\sum_{n=1}^{\infty} \tau(p_n)$ can be equal to $\frac{\pi}{5}$).
Let $P \in M(D \otimes {\cal K}) \setminus (D \otimes {\cal K})$ be the projection
defined by $P := \sum_{n=1}^{\infty} p_n$ and 
$B := P (D \otimes {\cal K}) P$.  Then $B$ has continuous scale and since
$\widehat{P}$ is not in the image of $K_0(B)$ ($= \mathbb{Q}$) in 
$\Aff(T(D))$, $[1_{C(B)}] = [\pi(P)] \neq 0$ in $K_0(C(B))$.
(E.g., see the computations in Theorem \ref{thm:KComp} and Proposition
\ref{prop:K0M(B)}.  See also \cite{LinExtContScale} 1.7 and the references
therein.) 

(5) 
We note that the above phenomena have been
 well-known for a long time, and hence
we require the assumption that $[1_{C(B)}] = 0$ in $K_0(C(B))$ in the unital
case.  See, for example, \cite{LinExtContScale} 1.10.

(6) The requirement for replacing unitary equivalence with weak unitary
equivalence, for a reasonable classification in the unital case,  
 has also been known for a long time.  (E.g., see
\cite{NgNonstableAbsorb} the remark before Theorem 3.6.) And this
requirement is necessary in order to have a reasonable classification
of extensions as in Theorem \ref{thm:ClassifyUnitalExt}.  As already
noted previously, there are nonetheless many special cases (like
when the domain of the Busby invariant is commutative) where this 
requirement is not present.  (See Theorem \ref{TH1} and the remark 
after it.)

\end{rem}
}}

\section{Classification of extensions with simple corona algebras}
%{Existence and uniqueness}

%{\color{cyan}NOTE: I REMOVED THE WORD ``essential'' IN THE SECTION TITLE}

Since $KK$-theory is invariant under unitary conjugation, there is a canonical homomorphism $\Ext(A,B) \to KK(A,C(B))$ which sends a Busby invariant to its class in $KK$-theory. The aim of this section is to show that this map is an isomorphism when $B$ is not assumed necessarily
% assumed 
to be stable but instead is assumed to have continuous scale. 

Recall that when $B$ is $\sigma$-unital and has continuous scale, then  the corona algebra $C(B)$ is simple and purely infinite.

\begin{df}
Let $\phi, \psi : A \rightarrow C$ be two homomorphisms between C*-algebras where $A$ is separable.
%\begin{enumerate}
%\item \label{df:AsympMvN} 
Then $\phi$ and $\psi$ are \emph{asymptotically 
Murray--von Neumann equivalent} if there exists a norm continuous path
$\{ c_t \}_{t \in [0,\infty)}$ of contractions in $M(C)$ such that 
$$\| c_t^* \phi(a) c_t - \psi(a) \| \rightarrow 0 \qquad \makebox{   and   } \qquad
\| c_t \psi(a) c_t^* - \phi(a) \| \rightarrow 0$$
as $t \rightarrow \infty$, for all $a \in A$. 
%\item  If, in (\ref{df:AsympMvN}), each $v_t$ can be chosen to be a unitary,
%then $\phi$ and $\psi$ are \emph{asymptotically unitarily equivalent}.  
%\end{enumerate}
\end{df}

The role of the above equivialence relation is captured by the following result. It is an immediate consequence of Theorems A, B and Corollary 9.8 of \cite{GabeClassifyOinftyStable} from the first named author's new proof of the Kirchberg--Phillips theorem.

%{\color{cyan}
%{
\begin{thm}\label{t:KP}
Let $A$ be a separable amenable C*-algebra and let $D$ be a simple unital purely infinite C*-algebra. Then the group $KK(A,D)$ is (canonically) in one-to-one correspondence with the asymptotic Murray--von Neumann equivalence classes of injective 
%$\ast$-
homomorphisms $A\to D$.
\end{thm}

The following relies heavily on a result of Phillips and Weaver (\cite{PhillipsWeaver}; using 
ideas from \cite{ManThomsen}).

\begin{lem}\label{TWP}
Let $A$ be a separable \CA\, and $B$ be a non-unital and $\sigma$-unital \CA.
Then two \hm s $\phi_1, \phi_2 \colon A\to C(B)$ are asymptotically Murray--von Neumann equivalent if and only if they are Murray--von Neumann equivalent (see Definition \ref{d:Ext}).
\end{lem}
\begin{proof}
Clearly Murray--von Neumann equivalence implies asymptotic Murray--von Neumann equivalence. For the other implication, suppose $\phi_1$ and $\phi_2$ are asymptotically Murray--von Neumann equivalent. For simplicity, we write $C(B)_{\mathrm{as}}$ to denote $C_b([0,\infty), C(B)) / C_0([0,\infty), C(B))$, and $\iota \colon C(B) \to C(B)_{\mathrm{as}}$ to denote the canonical inclusion. By asymptotic Murray--von Neumann equivalence, there is a contraction $c\in C(B)_{\mathrm{as}}$ such that $c^\ast \iota\phi_1(.)c = \iota \phi_2(.)$ and $c\iota\phi_2(.)c^\ast = \iota \phi_1(.)$. 

Let $D$ be the universal $C^\ast$-algebra generated by a contraction $x$, and let $\eta \colon D \to C(B)_{\mathrm{as}}$ be the $\ast$-homomorphism satisfying $\eta(x) = c$. Let 
\[
\Phi := \iota\phi_1 \star \iota\phi_2 \star \eta \colon A \star A \star D \to C(B)_{\mathrm{as}}
\]
where $A \star A \star D$ is the universal free product $C^\ast$-algebra. Write $A_1$ and $A_2$ for the first and second copy respectively of $A$ inside $A\star A \star D$. Let $E \subseteq A \star A \star D$ be the $C^\ast$-subalgebra generated by $A_1$, $A_2$, $x^\ast A_1 x$ and $xA_2 x^\ast$. Clearly $\Phi(E) \subseteq \iota(C(B))$ and therefore, by Proposition 1.4 of \cite{PhillipsWeaver}, there is a $\ast$-homomorphism $\Psi \colon A\star A \star D \to C(B)$ such that $\iota(\Psi(y)) = \Phi(y)$ for all $y\in E$. In particular, $\Psi(x)\in C(B)$ is a contraction such that for all $a\in A$ (with $a_1 \in A_1$ the corresponding element) we have
\[
\iota(\Psi(x)^\ast \phi_1(a) \Psi(x)) = \iota(\Psi(x^\ast a_1 x)) = \Phi(x^\ast a_1 x) = c^\ast \iota \phi_1(a) c = \iota\phi_2(a).
\]
As $\iota$ is injective, it follows that $\Psi(x)^\ast \phi_1(a) \Psi(x) = \phi_2(a)$ for all $a\in A$ and similarly we obtain $\Psi(x) \phi_2(a) \Psi(x)^\ast = \phi_1(a)$ for all $a\in A$. Hence $\phi_1 \sim_{MvN} \phi_2$ as desired.
\end{proof}

{The following is Theorem \ref{thm:1} from the introduction.}

\begin{thm}{{[cf.  Theorem 3.7 of \cite{LinExtQuasidiagonal}]}}\label{TH2}
Let $A$ be a separable  amenable \CA\, and let $B$ be a non-unital and $\sigma$-unital 
simple \CA\, with continuous scale. Then the canonical homomorphism $\Ext(A,B) \to KK(A,C(B))$ is an isomorphism.

%Then $\Ext(A, B)=KK(A, C(B))$
%as abelian groups.
%In particular, 
%if  $\tau_1, \tau_2: A\to M(B)/B$ are  two \color{cyan} non-unital} monomorphisms and  either
%both are unital or both are non-unital, 
%then $$\tau_1\sim_{wu} \tau_2 \makebox{   if and only if  } 
%KK(\tau_1)=KK(\tau_2).$$ 
%If both $\tau_1$ and $\tau_2$ are non-unital, then
%$$\tau_1 \sim \tau_2 \makebox{   if and only if  } KK(\tau_1) 
%= KK(\tau_2).$$ 
\end{thm}
\begin{proof}
Surjectivity of the map follows from Theorem \ref{t:KP}, and injectivity is obtained by combining Proposition \ref{TH1Pnonunital}, Theorem \ref{t:KP} and Lemma \ref{TWP}.
\end{proof}
We now consider the unital case. If $A$ is unital we let 
$$KK(A, C(B))_1 =_{df} \{ x \in KK(A, C(B)) : 
{{x}}([1_A]_{K_0(A)}) = [1_{C(B)}]_{K_0(C(B))} \}.$$
As with the map $\Ext(A,B) \to KK(A,C(B))$, there is an obvious map of sets $\Ext^w(A,B)_1 \to KK(A,C(B))_1$ given by mapping a unital Busby invariant to its $KK$-class. This is a homomorphism provided $[1_{C(B)}]_0 = 0$ in $K_0(C(B))$.
%}

\begin{thm}\label{thm:ClassifyUnitalExt}
Let $A$ be a unital separable amenable C*-algebra and let $B$ be a non-unital but $\sigma$-unital simple C*-algebra
with continuous scale.  

(1) Then the map $\Ext^w(A, B)_1  \to KK(A, C(B))_1$ is a bijection. 

{{(2) Let $\phi_1, \phi_2: A\to M(B)/B$ be two unital essential extensions. Then 
$\phi_1, \phi_2$ are weakly unitarily equivalent if and only if $[\phi_1]=[\phi_2].$}}

(3)  if, in addition,  $[1_{C(B)}]_0 = 0$ in $K_0(C(B))$ then this is an isomorphism of groups.
%If, in addition, $[1_{C(B)}]_{K_0(C(B))} = 0$, then they are the
%same as abelian groups.  
%
%Here, $KK(A, C(B))_1 =_{df} \{ x \in KK(A, C(B)) : 
%x_0([1_A]_{K_0(A)}) = [1_{C(B)}]_{K_0(C(B))} \}$.  
%\label{thm:ClassifyUnitalExt}
\end{thm}

\begin{proof}
Murray--von Neumann equivalence of unital Busby invariants is exactly weak unitary equivalence, so the injectivity of the map follows from Theorem \ref{t:KP} and Lemma \ref{TWP}. For surjectivity, let $x\in KK(A,C(B))_1$, and use Theorem \ref{t:KP} to find an injective homomorphism $\phi \colon A \to C(B)$ with $KK(\phi) = x$. Since $x\in KK(A,C(B))_1$ we have $[\phi(1_A)]_0 = [1_{C(B)}]_0$ in $K_0(C(B))$. As $C(B)$ is simple purely infinite, it follows that there is an isometry $v\in C(B)$ with $vv^\ast = \phi(1_A)$. Then $\psi := v^\ast \phi(.) v\colon A \to C(B)$ is a unital homomorphism with $KK(\psi)= KK(\phi) = x$, and this $\psi$ induces an element in $\Ext^w(A,B)_1$. Hence the map $\Ext^w(A,B)_1 \to KK(A,C(B))_1$ is a bijection. {{This proves (1), and (2) also follows.}}
The map  is clearly an isomorphism
of groups when $[1_{C(B)}]_0 = 0$ in $K_0(C(B)).$ 
%By an argument similar to that of Theorem \ref{TH2}, one can
%show that the map 
%$$\Ext^w(A, B)_1 \rightarrow KK(A, C(B))_1
%: [\phi]_w \mapsto [\phi]_{KK}$$ 
%is bijective.
%Moreover, it is easy to see that, when $[1_{C(B))}]_{K_0(C(B))} = 0$,
%the map is a group homomorphism. 
\end{proof}

We now give K-theory conditions on the relevant algebras that
are necessary and sufficient for all weakly unitarily equivalent
essential extensions to be unitarily equivalent.

Towards these K-theory conditions, 
recall that, for any 
%C*-algebras 
\CA s $A$ and $C$, with $A$ unital,  
\beq
\label{equdf:H1}
H_1(K_0(A), K_1(C)) :=  \{x \in K_1(C) :  x = \lambda([1_A]) \makebox{ for some  }  \lambda\in {\rm Hom}(K_0(A), K_1(C))\}.
\eneq
{{(See \cite{Linproc09} and \cite[Definition 12.1]{Linamj09}.)}}

\iffalse
{\red{``Strong"  may not be a good word as there is another choice of equivalence class which is even 
stronger.}}

{\green{Throughout, I have replaced ``strong unitary equivalence" with ``unitary equivalence".  See
also Remark 2.2.}}
\fi

%{\color{cyan}COMMENT: changed the statement below to assume that there exists a unital extension. Otherwise, condition 1 would always be satisfied but 2 could fail (e.g.~if $[1_A] = 0$, $[1_{C(B}] \neq 0$ then no unital extension exists (so 1 always holds), and if in addition $K_1(M(B)) = 0$ and $K_1(C(B)) \neq 0$, then 2 fails.

%COMMENT: there's still one issue in the proof! We need $M(B)$ is $K_1$-surjective (or some additional argument).}

\begin{thm}\label{TH1}
Let $B$ be a non-unital but $\sigma$-unital simple \CA\, with continuous scale.
Let $A$ be a separable unital amenable \CA\, which satisfies 
the UCT. Suppose that $\Ext^w(A, B)_1  \neq \emptyset$, or in other words, suppose that there exists an essential unital extension $0 \to B \to E \to A \to 0$.
Then the following statements are equivalent:
\begin{enumerate}
\item Any two unital extensions of $A \rightarrow C(B)$
which are weakly unitarily equivalent  
are unitarily equivalent, i.e., $\sim$ is the same as
$\sim_{wu}$.
\item \beq\nonumber
&&\hspace{-0.8in}H_1(K_0(A), K_1({C(B)}))/\pi_*(U(M(B))/U_0(M(B)))\\\label{TH1-1}
%
%\pi_{*1}(K_1(M(B))\\
&&\hspace{0.6in}=K_1({C(B)})/{{\pi_*(U(M(B))/U_0(M(B))),}}
%\pi_{*1}(K_1(M(B))),
\eneq
\noindent
where $\pi: M(B)\to C(B)$  is the quotient map and 
$\pi_*: U(M(B))/U_0(M(B))\to K_1(C(B))$ is the induced \hm.
\end{enumerate}
 \end{thm}

\begin{proof}
We first prove that $(2)$ implies $(1)$.
Let {{$\tau_1, \tau_2: A\to C(B)$ }} be 
two essential {unital} extensions  such that
there is a unitary $u\in M(B)/B$ where  
$u^*\tau_1(a)u=\tau_2(a)$ for all $a\in A.$ 
%By {\blue{Proposition \ref{TH1Pnonunital},}} we may assume that $\tau_1$ (and
%hence $\tau_2$) is unital.  

By (\ref{TH1-1}), there is a \hm\, $\xi: K_0(A)\to K_1(M(B)/B)$
such that $\xi([1_A])+[u]\in {{\pi_*(U(M(B))/U_0(M(B)))}}.$
%\pi_{*1}(K_1(M(B)).$
Let $j:K_0(A)\to  K_1(A\otimes C(\T))=K_1(A)\oplus K_0(A)$ 
be the identification with the second summand. 
In particular, $j([1_A])=[1\otimes z]$ in $K_1(A\otimes C(\T)),$
where $z\in C(\T)$ is  the identity function from the unit circle to itself. 
By the UCT, choose  an element $x\in KK^1(A, M(B)/B)$
such that $x$  induces $\xi$ as a \hm\, from $K_0(A)$ to $K_1(M(B)/B).$
Let $$y:=KK(\tau_1)\oplus x\in KK(A\otimes C(\T), M(B)/B)=KK(A, M(B)/B)\oplus KK^1(A, M(B)/B)$$
(see, for example, Theorem 19.6.1 of 
\cite{BlackadarBook} with trivial $\af$). 
By Theorem \ref{thm:ClassifyUnitalExt}, %the UCT and Theorem \ref{TH2}, 
there is an essential {unital} extension 
$\tau_3: A\otimes C(\T)\to M(B)/B$ such that
$KK(\tau_3)=y.$ Let $\tau_4:={\tau_3}|_{A\otimes 1}: A
\to M(B)/B.$
Then $KK(\tau_4)=KK(\tau_1)$,  
%Hence, after conjugating by an appropriate isometry,
%we may assume that $\tau_3$ is unital.  
%Moreover, 
{and therefore} by Theorem {\ref{thm:ClassifyUnitalExt}} %\ref{TH2} and \ref{thm:Ext1Group}, 
there exists a unitary $W\in M(B)/B$
such that $W^*\tau_4(a)W=\tau_1(a)$ for all $a\in A.$
Note that 
\beq
[\tau_3(1\otimes z)]+[u]=\xi([1_A])+[u]\in \pi_*(U(M(B))/U_0(M(B))).
%\pi_{*1}(K_1(M(B)).
\eneq
Also, for all $a\in A,$
\beq
\tau_3(1_A\otimes z)\tau_4(a)&=&\tau_3(1_A\otimes z)\tau_3(a\otimes 1_{C(\T)})
\\
&=&\tau_3(a\otimes 1_{C(\T)})\tau_3(1\otimes z)=\tau_4(a)\tau_3(1\otimes z).
\eneq
Put $v=W^*\tau_3(1\otimes z) Wu.$ Then $[v]\in \pi_*(U(M(B))/U_0(M(B))).$
%\pi_{*1}(K_1(M(B)).$
In other words, there is a unitary $V\in M(B)$ such that  $[\pi(V)]=[v]$ in $K_1(M(B)/B).$ 
% {\color{cyan}COMMENT: We need $M(B)$ is $K_1$-surjective for this!!!}
Therefore, since $M(B)/B$ is purely infinite and simple,  $v\pi(V)^*\in U_0(M(B)/B).$
This implies that there is $V_0\in M(B)$ such that $\pi(V_0)=v\pi(V)^*.$
Let $V_1=V_0V.$ Then $\pi(V_1)=v\pi(V)^*\pi(V)=v.$ 
Moreover
\beq
v^*\tau_1(a)v&=&u^*W^*\tau_3(1\otimes z)^*W\tau_1(a)W^*\tau_3(1\otimes z)Wu\\
&=&u^*W^*\tau_3(1\otimes z)^*\tau_4(a)\tau_3(1\otimes z)Wu\\
&=&u^*W^*\tau_4(a)Wu=u^*\tau_1(a)u=\tau_2(a)\rforal a\in A.
\eneq

For {$(1)\Rightarrow (2)$}, 
let $u\in M(B)/B$ be a unitary.
{Let} $\tau: A\to M(B)/B$ {be an essential unital extension.} %is an essential extension.
Put $\tau_1:={\rm Ad}\, u\circ \tau.$ Then $\tau_1$ and $\tau$ are weakly unitarily equivalent.
If $\tau_1$ and $\tau$ are strongly unitarily equivalent, then there is a unitary 
$V\in M(B)$ such that
$\pi(V)^*\tau_1\pi(V)=\tau.$
Put $w=u \pi(V).$ 
Then 
\beq
w^*\tau(a)w=w^*(u\tau_1(a)u^*)w=\pi(V)^*\tau_1(a)\pi(V)=\tau(a)\rforal a\in A.
\eneq
Define $\phi: A\otimes C(\T)\to M(B)/B$ by 
\beq
\phi(a\otimes f(z))=\tau(a)f(w)\rforal a\in A\andeqn f\in C(\T).
\eneq
Then $\phi$ is an extension.
Consider \hm\, $\phi_{*1}: K_1(A\otimes C(\T))\to K_1(M(B)/B).$
By identifying $K_0(A)$ with the summand of $K_1(A\otimes C(\T))=K_1(A)\oplus K_0(A),$
one obtains the \hm\, $\lambda:={\phi_{*1}}|_{K_0(A)}  \colon  K_0(A)\to K_1(M(B)/B)$
which maps $[1_A]$ to $[w].$
Denote by $q: K_1(M(B)/B)\to K_1(M(B)/B)/{{\pi_*(U(M(B))/U_0(M(B)))}}$
%\pi_{*1}(M(B))$ 
the quotient map. 
Since $w=u\pi(V),$
one concludes that
\beq
q([u])\in H_1(K_0(A), K_1(M(B)/B))/{{\pi_*(U(M(B))/U_0(M(B))).}}
%\pi_{*1}(M(B)).
\eneq
Since $u$ is arbitrarily chosen and since $M(B)/B$ is purely infinite and thus $K_1$-surjective, this implies that
\beq\nonumber
H_1(K_0(A), K_1(M(B)/B))/\pi_*(U(M(B))/U_0(M(B)))
%\pi_{*1}(M(B))
=K_1(M(B)/B)/{{\pi_*(U(M(B))/U_0(M(B))).}}
%\pi_{*1}(M(B)).
\eneq
\end{proof}

\begin{rem}\label{RH1}
In the case $K_1(M(B)/B)=\pi_{*1}(K_1(M(B))),$ of course,
the relations of  unitary and  weak 
unitary equivalence, {{on extensions of $A$ by $B$, }} are the same.

In the case $K_1(M(B))=\{0\},$ the assumption \eqref{TH1-1} means\\ 
$H_1(K_0(A), K_1(M(B)/B))=K_1(M(B)/B),$ which is prima facie 
a stronger condition. 
(This case occurs often;  e.g., see Theorem \ref{thm:K1M(B)}; see
also  
\cite{BlackadarBook} Proposition 12.2.1, \cite{LinCERRR0} and
\cite{LinWvNII}.) 

Suppose that $A=C(X)$ for some compact metric space $X.$
Then $K_0(A)=\Z\cdot [1_A]\oplus G$ for some subgroup $G\subset K_0(A).$
Then any \hm\, $\lambda: \Z\cdot [1_A]\to K_1(M(B)/B)$ can be extended to a \hm\, 
in ${\rm Hom}(K_0(A), K_1(M(B)/B)).$ It follows 
that $H_1(K_0(A), K_1(M(B)/B))=K_1(M(B)/B).$
By Theorem \ref{TH1}, 
if $A$ is commutative, 
unitary and weak unitary equivalence are the same.
This of course covers all cases that $K_0(A)=\Z\cdot [1_A]\oplus G$
(not just the commutative case).  

If $A={\tilde C}$ for some non-unital \CA\, $C,$ then $K_0(A)=\Z\cdot [1_{\tilde C}]\oplus K_0(C).$
As above, $H_1(K_0(A), K_1(M(B)/B))=K_1(M(B)/B).$
Theorem \ref{TH1} also explains 
why all weak unitary equivalences among non-unital essential extensions
are in fact unitary equivalences.

%{\color{cyan}COMMENT: the following is only true if $[1_A]$ is non-zero and not a torsion element!}
{{If $m\cdot [1_A]\not=0$ for any $m\in \N$ and 
if}}  $K_1(M(B)/B)$ is divisible, unitary and weak unitary equivalence are the same 
as
%{\color{cyan} 
\beq
 H_1(K_0(A), K_1(M(B)/B))=K_1(M(B)/B).
 \eneq 
 %}
   In fact, one only needs that 
$K_1(M(B)/B)/{{\pi_*(U(M(B))/U_0(M(B)))}}$
%\pi_{*1}(K_1(M(B)/B))$ 
is divisible.
Finally, let us point out, there are many pairs $A$ and $B$
for which
\beq
H_1(K_0(A), K_1(M(B)/B))/{{\pi_*(U(M(B))/U_0(M(B))}}\\
%\pi_{*1}(K_1(M(B)/B))\\
=K_1(M(B)/B)/{{\pi_*(U(M(B))/U_0(M(B)))}}.
%\pi_{*1}(K_1(M(B)/B)).
\eneq
This explains the reason why often unitary and weak unitary equivalence are the same for many classical cases.

\end{rem}

%\subsection{
{\bf{Extensions with purely infinite ideals}}

\vspace{0.2in}

%{{\color{cyan}
We will end this section by considering the case where the $C^\ast$-algebra $B$ in the above results is $\sigma$-unital, non-unital, simple and purely infinite. In this case $B$ is stable by Zhang's dichotomy \cite{Zhangdichotomy} and has continuous scale since $a \lesssim b$ for all non-zero $a,b\in B_+$. We will observe a  result about liftability of such extensions which is surely known to experts (it can also be proved by classical BDF-Kasparov theory). We include it since we do not believe it appears explicitly in this form in the published literature.

%{\color{cyan}
We emphasize that in the unital case below, the lifting homomorphism is not assumed to be unital.
%}
%}}

\begin{thm}\label{Tpurelyinfinite}
Let $B$ be a non-unital and $\sigma$-unital purely infinite simple C*-algebra, and let $A$ be a separable  amenable 
\CA.
Then an essential
extension $\tau: A\to M(B)/B$ (unital or non-unital)
 is liftable  if and only if $KK(\tau)=0.$ 
\end{thm}
\begin{proof}
Since $B$ is stable we have $KK(A,M(B)) = 0$ by Proposition 4.1 in \cite{DadarlatTopExt}. Hence $KK(\tau) = 0$ whenever $\tau$ is liftable. 

Conversely, suppose $KK(\tau) = 0$.  Suppose first that $\tau$ is non-unital. We may pick 
%$\ast$-
{{a \hm}} $\phi \colon A \to M(B)$ such that $\pi_B \circ \phi$ is a non-unital embedding. As $KK(\phi) = 0$ it follows that $KK(\pi_B \circ \phi) = 0 = KK(\tau)$, so by Theorem \ref{TH2} there exists a unitary $u\in M(B)$ such that $\pi_B(u\phi(.)u^\ast) = \tau$, and hence $u\phi(.) u^\ast$ is a lift of $\tau$.

Suppose now that $\tau$ is unital, and pick a unital 
%$\ast$-homomorphism 
{{\hm\,}} $\phi_1 \colon A \to M(B)$ such that $\pi_B \circ \phi_1$ is an embedding. As above, $KK(\pi_B \circ \phi_1) = 0 = KK(\tau)$ and by Theorem \ref{thm:ClassifyUnitalExt} there is a unitary $w\in C(B)$ such that $\tau = w\pi_B(\phi_1(.))w^\ast$. Since every element in $K_0(B)$ is positive \cite{CuntzAnnals}, it follows from Corollary 1 in \cite{Nagy} that there exists an isometry $v\in M(B)$ such that $\pi_B(v) = w$. Hence $v\phi_1(.)v^\ast$ is a $\ast$-homomorphism which lifts $\tau$.
\end{proof}

For completeness and since we will 
%be proving 
{{prove}} an analogous result in the 
finite case, we here mention the following well-known noncommutative 
Weyl--von Neumann theorem.  (See, for example, \cite{ElliottKucerovsky}.)

\begin{thm}\label{Tinfinteuniq2}
Let $A$ be a separable amenable \CA\, 
and $B$ be a non-unital and $\sigma$-unital purely infinite simple \CA.
Suppose that $\phi_1, \phi_2: A\to M(B)$ are two injective \hm s  
such that either both $\phi_1, \phi_2$ are unital, or both
$\pi \circ \phi_1, \pi \circ \phi_2$ are non-unital, and 
$\phi_i(A)\cap B=\{0\},$ $i=1,2$. Then there exists a sequence of unitaries $\{u_n\}\subset M(B)$ 
such that
\beq
\lim_{n\to\infty}\|u_n^*\phi_1(a)u_n-\phi_2(a)\|=0\andeqn
u_n^*\phi_1(a)u_n-\phi_2(a)\in B \rforal a\in A. 
\eneq

\end{thm}

\section{Some K theory computations} 

\label{KComputations}

%{\blue{This section may {{serve}} as an appendix to the paper.}}  
%%%%%%%%%%%%%%%%%%%%%%%%%%%%%%%

Recall that for a \CA\, $A$, ${\rm Cu}(A)$ denotes the Cuntz semigroup
of $A$.

  \begin{df}\label{Dgamma}
  Let $A$ be a simple \CA. 
  Denote by $\Gamma: {\rm Cu}(A)\to {\rm LAff}_+(\wtd{QT}(A))$
  the canonical order preserving \hm\, defined by
  $\Gamma([x])(\tau)=\widehat{[x]}(\tau) =d_\tau(x)$ for all $x\in (A\otimes {\cal K})_+$ and $\tau \in QT(A)$
  (see  \ref{Dcuntz}  and  \ref{Dqtr}).
  \end{df}

Let us also briefly recall some other definitions.
Let $A$ be a C*-algebra and let ${\rm Cu}(A)$ be the Cuntz semigroup
of $A$.  For all $x, y \in {\rm Cu}(A)$, we write $x \ll y$ if whenever
$y \leq \sup z_n$ for an increasing sequence $\{ z_n \}$ in ${\rm Cu}(A)$,
there exists an $n_0 \in \mathbb{N}$ such that $x \leq z_{n_0}$.
We say that ${\rm Cu}(A)$ is 
\emph{almost divisible} if for all $x, y \in {\rm Cu}(A)$ for
which $x \ll y$, for all
$k \in \mathbb{N}$, there exists a $z \in {\rm Cu}(A)$ such that 
$$kz \leq y \makebox{  and  } x \leq (k+1)z.$$
We say that ${\rm Cu}(A)$ is \emph{almost unperforated} if for all $x, y \in 
{\rm Cu}(A)$, for all $k \in \mathbb{N}$, that $(k+1)x \leq k y$ implies that
$x \leq y$.

We will consider only separable simple \CAs\, with strict comparison. Let us quote the following 
statement (see also \cite{Th}).

\begin{thm}[{\cite[Corollary 1.2, 1.3]{Linossr=1}}]\label{stst=1}
Let $A$ be a separable simple \CA\, with strict comparison for positive elements
which is not purely infinite.
Then the following statements are equivalent:

(1) $\Gamma$ is surjective.  

(2) ${\rm Cu}(A)$ is almost divisible. 

(3) $A$ has tracial approximate oscillation zero. 

(4) $A$ has stable rank one.
\end{thm}

We refer to \cite{FLosc} for the definition of tracial approximate oscillation zero.

The following is a folklore.

\begin{lem}\label{strictpe}  
Let $A$ be a unital \CA\, and $a\in A$ be a strictly positive element.
Then $a$ is invertible. If $a=b+c,$ where $b, c\in A_+$ and $bc=cb=0,$
then $\overline{cAc}$ is unital.
\end{lem}

\begin{proof}
Replacing $a$ by $f(a)$ for some $f\in C_0((0, \|a\|],$
 we may assume that
$0\le a\le 1.$
Put $X={\rm sp}(a)\subset [0,1].$ 
If $a$ is not invertible, then $0\in {\rm sp}(a).$ 
Then the \SCA\, $C$ generated by $1$ and $a$ is isomorphic 
to $C(X).$  Let $\tau: C\to \C$ be the state defined by 
$\tau(g)=g(0)$ for all $g\in C(X)$ (so $\tau\not=0$).  Note that
$\tau(a)=0.$ Extend $\tau$ to a state of $A$ denoted again by $\tau.$
Then we still have $\tau(a)=0$ but $\tau(1)=1.$ This implies that $a$ is not strictly positive.
This shows that $0\not\in X,$ or $a$ is invertible. 

If $a=b+c,$ where $b, c\in A_+$ and $bc=cb=0,$ as $0\not\in {\rm sp}(a),$ 
there is $f\in C_0(0, \| a \|]_+$ such that $f(a)=1.$
In particular (recall that $b$ and $c$ are mutually orthogonal)
\beq
f(a)=f(b+c)=f(b)+f(c)=1.
\eneq
Therefore,  
\beq
f(a)^2=f(b)^2+f(c)^2=1=f(a)=f(b)+f(c).
\eneq
It follows that $f(c)=f(c)^2.$ Hence $f(c)$ is a projection and 
%$\overline{cAc}
$f(c)$ is strictly positive element of $\overline{cAc}.$  Consequently $\overline{cAc}$
is unital.
\end{proof}

\begin{NN}\label{12NN}
Let $B$ be a  finite separable simple \CA\, with almost unperforated  {{and almost divisible  ${\rm Cu}(B)$.}}  
%with continuous scale. 
By Theorem \ref{stst=1}, $\Gamma$ is surjective. 
%strict comparison and stable rank one. 
%It follows from Theorem 7.14 of \cite{APPT} that 
%{\blue{I.e., the map ${\rm Cu}(B)\to {\rm LAff}_+(\wtd{QT}(B)): 
%[x]\mapsto \widehat{[x]}$ is surjective.}} 
Choose 
%
%Let $e\in {\rm Ped}(B\otimes {\cal K})_+\setminus \{0\}.$ Then, by \cite{Br}, 
%$\overline{eBe}\otimes {\cal K}\cong B\otimes {\cal K}.$ 
%It follows from Theorem 7.14 of \cite{APRT} ({\red{Duke?}}) that 
%Since $\Gamma$ is surjective, one may choose 
$e\in (B\otimes {\cal K})_+\setminus \{0\}$ such that  $\widehat{[e]}\in \Aff_+(\wtd{QT}(B)) \setminus\{0\}.$
Then
$\overline{eBe}$ has continuous scale.  In what follows in this section, 
\wilog, we assume 
that $B$ has continuous scale to make the statements easier to state.

We use the fact that, if $A$ is a $\sigma$-unital simple \CA\, 
with continuous scale, then $QT(A)$ is compact (see 
Theorem 2.19 of \cite{FLosc} and Theorem 5.3 and Proposition 5.4 of [10]).
\end{NN}

\iffalse
{\blue{For a separable simple C*-algebra $B$ and for an element
$e \in Ped(B)_+ \setminus \{ 0 \}$, let 
$T_e(B)$ be the collection of all norm lower semicontinouous
traces $\tau : (B \otimes \mathcal{K})_+ \rightarrow [0, \infty]$
for which $\tau(e) = 1$.  Recall that $T_e(B)$, with the topology of 
pointwise convergence on $Ped(B)_+$, is a Choquet simplex.  Recall
also that each $\tau \in 
T_e(B)$ extends uniquely to a strict topology
lower semi-continuous trace $M(B \otimes \mathcal{K})_+ 
\rightarrow [0, \infty]$ 
which we also denote by $\tau$.}}

\iffalse
{\blue{The results that follow do not depend on the choice of 
Pedersen ideal element $e \in Ped(B)_+ \setminus \{ 0 \}$, and we will 
freely write
$T_e(B)$ without specifying which $e$ beforehand.}}  
\fi

{\blue{Finally, recall that for all $0 < \delta < 1$, 
$f_{\delta} :[0, \infty) \rightarrow [0,1]$ is the unique continuous function such that
$$f_{\delta}(t) =
\begin{cases}
1 & t \in [\delta, \infty)\\  
0 & t \in [0, \frac{\delta}{2}]\\
\makebox{linear on  } & [\frac{\delta}{2}, \delta].
\end{cases}$$}}  
 \fi
%{\red{In the next 2 results, is $\sigma$-unitality enough?}}\\

\begin{thm}\label{125-2331}
 Let $B$ be a 
%$\sigma$-unital 
separable simple 
\CA\, with continuous scale, strict comparison for positive elements and  satisfying one of the conditions in 
Theorem \ref{stst=1}.
%stable rank one.}}

 Then, for any pair of projections 
$p, q\in M(B \otimes {\cal K}) \setminus B \otimes {\cal K},$
\begin{enumerate}
%{{\item  for any pair of projections 
%$p, q\in M(B \otimes {\cal K}) \setminus B \otimes {\cal K},$
%$B$ has \emph{projection injectivity} if
%for all projections $P, Q \in M(B \otimes K) \setminus B \otimes K$,
{\item if $\tau(p) = \tau(q)$
for all $\tau \in QT(B)$, then
$$ p \sim q \,\,\,{\rm in}\,\,\, M(B\otimes {\cal K});$$}
\item  $p\lesssim q$ in $M(B\otimes {\cal K})$ if and only if 
there exists $f\in {\rm LAff}_+(QT(B))$ such that
$\widehat{p}+f=\widehat{q};$ and,  
\item if
$\tau(p)<\infty$ for all $\tau\in QT(B)$ and  $\widehat{q}-\widehat{p}\in {\rm LAff}_+(QT(B)),$
 then 
there exists a partial isometry $v\in M(B\otimes {\cal K})$ such that  
$$
v^*v=p\andeqn vv^*\le q.     
$$

Moreover, 
{{\item for every {{$f \in {\rm LAff}_+(QT(B)) \setminus \{ 0 \}$,}} there exists a projection
$p \in M(B \otimes {\cal K}) \setminus B \otimes {\cal K}$ such that
$$\widehat{p}=f,\,\,\,{\rm i.e.},\,\,\, \tau(p) = f(\tau)
\rforal \tau \in QT(B).$$}}

\end{enumerate}
\label{thm:Aug12022}
\end{thm}

\begin{proof}

For (1), let $p, q\in M(B\otimes {\cal K})\setminus B\otimes {\cal K}$ such that $\tau(p) = \tau(q)$ for all $\tau\in QT(B)$. By viewing $M(B\otimes  \mathcal K) = \mathcal B( B\otimes \mathcal K)$ as the adjointable operators on the Hilbert $(B\otimes \cal K)$-module $B\otimes \cal K$, it follows that $p\sim q$ if and only if the right Hilbert $(B\otimes \cal K)$-modules $X_p:=p(B\otimes \cal K)$ and $X_q:=q(B\otimes \cal K)$ are isomorphic. % {\ppl{HL: It seems to me that one direction needs a justification.?}}  
In fact, $\mathcal B(X_p,X_q) = qM(B \otimes \mathcal K)p$ and since an isomorphism $X_p \cong X_q$ is implemented by a unitary in $\mathcal B(X_p, X_q)$  (see, for instance, Theorem 3.5 in \cite{Lance}), this unitary is the same as an element $v\in qM(B \otimes \mathcal K)p$ with $v^\ast v = p$ and $vv^\ast =q$.
Hence, it suffices to show that $X_p \cong X_q$.
%In fact, a partial isometry $v\in \mathcal B(B\otimes \mathcal K)$ with $vv^\ast = p$ and $v^\ast v = q$ restricts to a unitary in $\mathcal B(p(B\otimes \mathcal K), q(B\otimes \mathcal K))$. Conversely, if $u$ is such a unitary, it extends canonically to a partial isometry $v$ as above since 
%\[
%B\otimes \mathcal K= p(B\otimes \mathcal K) \oplus (1-p)(B\otimes \mathcal K) = q(B\otimes \mathcal K) \oplus (1-q)(B\otimes \mathcal K)
%\]
%and since the units in $\mathcal B(p(B\otimes \mathcal K))$ and $\mathcal B(q(B\otimes \mathcal K))$ are $p$ and $q$ respectively. 

Let $a\in p(B\otimes \mathcal K)p$ and $b\in q(B\otimes \mathcal K)q$ be strictly positive elements. Since $B$ has stable rank one, it follows from Theorems 4.29 and 4.33 of \cite{AraPereraToms} that $X_p \cong X_q$ if and only if $[a]= [b]$ in $\mathrm{Cu}(B)$.  

Note that $d_\tau(a) = \tau(p) = \tau(q) = d_\tau(b)$ for all $\tau\in QT(B)$. Since $p$ is not in $B\otimes \mathcal K$, $p(B\otimes \mathcal K)p$ is not unital, and therefore $0$ is in the spectrum of $a$. Hence $d_\tau(f_{1/n}(a)) < d_\tau(a) = d_\tau(b)$ for all $\tau\in QT(B)$ (since each $\tau$ is faithful) and $n\in \mathbb N$ (here $f_{1/n}$ is as  in Definition \ref{Dfdt}). By strict comparison we get $[f_{1/n}(a)] \leq [b]$ in $\mathrm{Cu}(A)$ for all $n\in \mathbb N$. As $[a] = \sup_n [f_{1/n}(a)]$ it follows that $[a] \leq [b]$. Similarly, $[b] \leq [a]$ and thus, by what was shown above, $p\sim q$.

To see (4), let {{$f\in {\rm LAff}_+(QT(B)) \setminus \{ 0 \}.$}}  
Since $QT(B)$ is compact (as $B$ has continuous scale), 
choose $s > 0$ so that $f(\tau) > s$ for all $\tau \in QT(B)$.
Note {{$f-s\in {\rm LAff}_+(QT(B)) \setminus \{ 0 \}.$}} 
There exists a sequence 
{{$f_n'\in {\rm Aff}_+(QT(B)) \setminus \{ 0 \}$}} such that 
$f_n'\nearrow f - s$ pointwise on $QT(B)$.
For each $n$, let $f_n := f'_n + \sum_{j=1}^n \frac{s}{2^j}$.     
Hence, the sequence $\{ f_n \}$ is strictly increasing (i.e., $f_n(\tau) 
< f_{n+1}(\tau)$ for all $n$, $\tau$) and $f_n \nearrow f$.  
For each $n$, let $g_n := f_n - f_{n-1} \in {{{\rm Aff}_+(QT(B))\setminus \{ 0 \}}}$ 
{{(where $f_0 := 0$).}} 

\iffalse
Choose $g_1=(1/2)f_1'$ and $g_2={\blue{(1/2+1/4)f_2'-g_1}}=(1/2)f_2'-(1/2)f_1'+(1/4)f_2'.$
Then $g_2(t)>0$  for all $t\in T_e(B).$  Put $f_1=g_1$ and $f_2=g_1+g_2=(1/2+1/4)f_2.$

If $g_1, g_2,...,g_n$ are chosen
such that $f_k=\sum_{j=1}^k g_k
=(\sum_{j=1}^{k+1}1/2^j)f_k',$ 
 define 
\beq
g_{n+1}=(\sum_{j=1}^{n+2}1/2^j)f_{n+1}'-f_n=(\sum_{j=1}^{n+1}(1/2^j)(f_{n+1}'-f_n')+(1/2^{n+2})f_{n+1}.
\eneq
Hence $g_{n+1}(t)>0$ for all $t\in QT(B).$
Then  (using the inductive assumption)
\beq
\sum_{k=1}^{n+1}g_k&=&g_{n+1}+\sum_{k=1}^n g_n
=(\sum_{j=1}^{n+1}1/2^j)(f_{n+1}'-f_n')+(1/2^{n+2})f_{n+1}'+f_n\\
&=&(\sum_{j=1}^{n+2}1/2^j)f_{n+1}'.
\eneq
By induction 
\fi
Thus,
we obtain a sequence $\{g_n\}\subset  \Aff_+(QT(B))\setminus \{ 0 \}$ such that
\beq
f=\sum_{j=1}^{\infty} g_j.  %=(\sum_{j=1}^{n+1}1/2^j)f_n'.
\eneq

%Note that $f_n\nearrow f.$
Since $B$ has stable rank one,  by Theorem 7.14 of \cite{APRT},
there is, for each $n\in \N,$ a positive contraction $a_n\in M_{r(n)}(B)$ such 
that 
$d_\tau(a_n)=g_n(\tau)$ for all $\tau\in QT(B).$

We may view 
\beq
c:=\diag(a_1, {\tfrac{1}{2}}a_2,..., {\tfrac{1}{n}} a_n,...)
\eneq
as an element in $B \otimes {\cal K}$.
We compute that 
{{\beq
d_{\tau}(c)=f(\tau)\rforal \tau\in QT(B). 
\eneq}}  
Since each $g_n \neq 0$, it follows that $a_n \neq 0$ for all $n\in \mathbb N$. Hence $0$ is in the spectrum of $c$ and thus $\overline{c(B\otimes \mathcal K)c}$ is not unital.

By Kasparov's stabilization theorem there is a projection $p\in M(B\otimes \mathcal K)$ such that $p(B\otimes \mathcal K) \cong \overline{c(B\otimes \mathcal K)}$. Hence $\tau(p) = d_\tau(c) = f(\tau)$ for all $\tau\in QT(B)$. Furthermore, as $\overline{c(B\otimes \mathcal K)c} \cong p(B\otimes \mathcal K)p$ is non-unital it follows that $p\notin B\otimes \mathcal K$.

Next, we prove (2).  For the ``only if" direction of (2), suppose that
$p,q \in M(B \otimes {\cal K}) \setminus B \otimes {\cal K}$ are projections such that
$p \lesssim q$.  Hence, let $p_1 \in M(B \otimes {\cal K})$ be a 
projection with $p \sim p_1$ and $p_1 \leq q$.
$$\widehat{q} = (p_1 + (q - p_1))^{\widehat{}} = 
\widehat{p_1} + \widehat{q - p_1} = \widehat{p} + \widehat{q - p_1}.$$
Take $f := \widehat{q - p_1} \in {\rm LAff}_+(QT(B))$.

For the ``if" direction of (2),  suppose that $f \in {\rm LAff}_+(QT(B))$ 
is such that $\widehat{p} + f = \widehat{q}$. 
If $f = 0$, then by (1), $p \sim q$ and we are done.  So suppose that
$f \neq 0$.
By (4), let $q_1 \in M(B \otimes {\cal K})
\setminus B \otimes {\cal K}$ be a projection such that 
$\widehat{q_1} = f$.   Hence,
$$\widehat{q} = \widehat{p} + f = \widehat{p} + \widehat{q_1} =
(p \oplus q_1)^{\widehat{}}.$$
Hence, by (1), $q \sim p \oplus q_1$.  Hence, $p \lesssim q$.

(3) follows from (2).

\end{proof}

%{\ppl{
\begin{cor}\label{c123}
{{Let $B$ be a 
%$\sigma$-unital 
separable simple 
\CA\, with continuous scale, strict comparison for positive elements and stable rank one.}}

\begin{enumerate}
\item  Then, for any pair of projections 
$p, q\in M(B) \otimes {\cal K} \setminus B \otimes {\cal K},$
%$B$ has \emph{projection injectivity} if
%for all projections $P, Q \in M(B \otimes K) \setminus B \otimes K$,
if $$\tau(p) = \tau(q)$$ 
for all $\tau \in QT(B)$, then
$$p \sim q\,\,\, {\rm in}\,\,\, M(B)\otimes {\cal K}$$
and, if $\tau(p)<\tau(q)$ for all $\tau\in QT(B),$ then 
\beq
p\lesssim q \,\,\, {\rm in}\,\,\, M(B)\otimes {\cal K}.
\eneq

{{\item For every $f \in {\rm Aff}_+(QT(B)) \setminus \{ 0 \}$, there exists a projection
$p \in M(B)\otimes {\cal K} \setminus B \otimes {\cal K}$ such that
$$\tau(p) = f(\tau)$$
for all $\tau \in QT(B).$ Moreover, for every $p\in M(B)\otimes {\cal K},$
$\widehat{p}\in \Aff_+(QT(B)).$}}

\end{enumerate}
%\label{thm:Aug12022}
\end{cor}%}}

\begin{proof}
We first show that $M(B) \otimes \mathcal K$ is a hereditary \SCA\, of 
$M(B\otimes {\cal K}).$  {{This fact  is known, but we provide the argument for the convenience
of the reader.}}   Write $\mathcal K ={{\overline{ \bigcup_{n\in \mathbb N} M_n(\mathbb C)}}}$ and let $p_n \in \mathcal K$ denote the unit in $M_n(\mathbb C)$. Since $M(B)\otimes \mathcal K = \overline{\bigcup_n M_n(M(B))}$ it suffices to show that each $M_n(M(B)) = M(M_n(B))$ is hereditary in $M(B\otimes \mathcal K)$. Define $P_n = 1_{M(B)}\otimes p_n\in M(B\otimes \mathcal K)$. Observe that $M_n(B) = P_n (B\otimes \mathcal K) P_n$. It is well-known that $P_nM(B\otimes \mathcal K) P_n \cong M(P_n (B\otimes \mathcal K)P_n)$ canonically (since multipliers in $M(P_n(B\otimes \mathcal K)P_n)$ extend in an obvious way to multipliers in $P_n M(B \otimes \mathcal K)P_n$), and thus $P_nM(B\otimes \mathcal K) P_n = M(P_n(B\otimes \mathcal K)P_n) = M_n(M(B))$ is a corner in $M(B\otimes \mathcal K)$, and thus hereditary.

Thus the first part (1) follows from  (1) of  Theorem \ref{thm:Aug12022}.

%We first show that $M(B)$ is a unital hereditary \SCA\, of 
%$M(B\otimes {\cal K}).$  {{This fact  is known, but we provide the argument for the convenience
%of the reader.}}   Let $\{e_{i,j}\}$ be a system of matrix units for ${\cal K}.$ 
%We identify $B$ with $B\otimes e_{1,1}.$  We claim that {{$1_{M(B)} \otimes e_{1,1}\in M(B\otimes {\cal K})$}}   
%as a projection.  Let $(b_{i,j})$ be an infinite matrix representation of any
%element in $B\otimes {\cal K}.$ Then {{$(1_{M(B)} \otimes e_{1,1})(b_{i,j})=(c_{i,j}),$}} 
%where $c_{1,j}=b_{1,j},$ $j\in \N$ and $c_{i,j}=0$ if $i>1.$ 
%One then verifies that $(c_{i,j})\in B\otimes {\cal K}.$ Similarly
%{{$(b_{i,j})(1_{M(B)} \otimes e_{1,1})=(d_{i,j}),$}} where $d_{i,1}=b_{i,1}$ and $d_{i,j}=0$ if $j>1,$ $i\in \N.$
%So {{$1_{M(B)} \otimes e_{1,1}$}} is a projection in $M(B\otimes {\cal K}).$ 
%One then identifies $M(B)$ with {{$(1_{M(B)} \otimes e_{1,1})M(B\otimes {\cal K}) (1_{M(B)} \otimes 
%e_{1,1}).$}} 
%Thus the first part (1) follows from  (1) of  Theorem \ref{thm:Aug12022}.

For the second part of (1), {{we prove the following claim:}} 

Claim:  {{For every projection $p\in M(B)\otimes {\cal K}\setminus B\otimes {\cal K}$,}} 
$$\widehat{p}\in \Aff_+(QT(B)).$$

Note that  $p$ is unitarily equivalent to a projection $p_1\in M_k(M(B))$ for some integer $k.$
\Wlog, we may assume that $p\in M_k(M(B)).$ 
It follows that $1_{M_{k+1}(M(B))}-p$ is a projection which is not in $M_{k+1}(B).$
It follows that $(k+1)-\widehat{p}\in {\rm LAff}_+(QT(B)).$ Since $\widehat{p}\in {\rm LAff}_+(QT(B)),$
this implies that $\widehat{p}\in \Aff_+(QT(B)).$ Hence the claim holds. 

Now let us assume that $\tau(p)<\tau(q)$ for all $\tau\in QT(B).$ 
By the claim above, both $\widehat{p}$ and $\widehat{q}$ are continuous. 
It follows that $\widehat{q}-\widehat{p}\in \Aff_+(QT(B)).$ By (3) of Theorem \ref{thm:Aug12022},
$p\lesssim q$ in $M(B)\otimes {\cal K}.$

To see (2) holds, let $f\in \Aff_+(QT(B)).$ By (2) of  Theorem \ref{thm:Aug12022},
there exists  a projection  $q\in M(B\otimes {\cal K})\setminus B\otimes {\cal K}$ such that
\beq
\tau(q)=f(\tau)\rforal \tau\in QT(B).
\eneq
There exists an integer $k\in \N$ such that $\tau(q)<k$ for all $\tau\in QT(B).$
Hence $k-\widehat{q}\in \Aff_+(QT(B)).$ Applying (3) of Theorem \ref{thm:Aug12022},
we may assume that 
 $q\in M_{k+1}(M(B))\setminus M_{k+1}(B).$
 This proves the first part of (2).  The second part of (2) follows from the claim above. 
\end{proof}

\iffalse
\begin{df}
Let $B$ be a $\sigma$-unital simple continuous scale C*-algebra.
\begin{enumerate}
\item $B$ has \emph{projection injectivity} if
for all projections $P, Q \in M(B \otimes K) \setminus B \otimes K$,
if $$\tau(P) = \tau(Q)$$ 
for all $\tau \in T(B)$, then
$$P \sim Q.$$ 
\item  $B$ has \emph{projection surjectivity} if
for every $f \in LAff(T(B))_{++}$, there exists a projection
$P \in M(B \otimes K) \setminus B \otimes K$ such that
$$\tau(P) = f(\tau)$$
for all $\tau \in T(B)$.  
\item $B$ has \emph{standard regularity} if 
$B$ has strict comparison of positive elements, stable rank one,
projection injectivity, projection surjectivity, and metrizable
Choquet simplex.
\end{enumerate}
\end{df}
\fi

Many C*-algebras satisfy the hypotheses of Theorem
\ref{thm:Aug12022}.  For example,
all separable simple Jiang-Su stable C*-algebras 
%with stable rank one 
{{satisfy the hypotheses of Theorem \ref{thm:Aug12022} (see Theorem 6.7 and 
Corollary 6.8 \cite{FLL}).}}  
In fact, there are also many
non-Jiang--Su stable C*-algebras, like non-unital hereditary $C^\ast$-subalgebras of $C^*_r(\mathbb{F}_{\infty})$, which also 
satisfy the hypotheses of Theorem \ref{thm:Aug12022}. See also \cite{FLsigma}.

The following is quoted from \cite{ChandNgSutradhar}.

\begin{thm}  
Let $B$ be a {{non-unital but}} $\sigma$-unital simple \CA\, with continuous scale, strict comparison for positive elements and 
stable rank one. Then $$K_1(M(B)) = 0.$$
\label{thm:K1M(B)}
\end{thm}

\iffalse
\begin{proof}
This is in \cite{ChandNgSutradhar}.  The proof, in \cite{ChandNgSutradhar}, is for the separable case, but 
{\blue{the same proof works wwith separability replaced with $\sigma$-unitality, as long as
the standing assumptions I., II. and III. are replaced with Theorem
\ref{125-2331},   Corollary \ref{c123} and \cite{KaftalNgZhangStrictComp} Theorem
6.6 respectively.}}   
\end{proof} 
\fi
%{\textbf{I am confident that the proof in \cite{ChandNgSutradhar} holds with
%only $\sigma$-unitality, as long as Theorem \ref{125-2331} and
%Corollary \ref{c123} only needs $\sigma$-unitality.}}

%{\red{\bf{The reduction seems somewhat  different from the standard one.  I think that there would be some considerable amount labor----- but a modification of the proof from \cite{ChandNgSutradhar} might be shorter-- HL}}

The following is from Lemma 7.2 of \cite{LN}.  It also follows from Theorem \ref{125-2331}
and Corollary \ref{c123}.

%{\red{In the next two results, if we only assume $B$ is $\sigma$-unital, do we need for $QT(B)$ to
%be a metrizable Choquet simplex?}}

\begin{prop}
{{Let $B$ be a non-unital separable simple \CA\, with continuous scale, strict comparison for positive elements
and stable rank one.}}

Then 
$$(K_0(M(B)), K_0(M(B))_+) = {{(\Aff(QT(B)), \Aff_+(QT(B))).}}$$ 
\label{prop:K0M(B)}
\end{prop}

Let $B$ be as in the above proposition. 
Recall that there is a natural group homomorphism
$$\rho_B : K_0(B) \rightarrow {{\Aff(QT(B))}}$$
which is given by
$$\rho_B([p]_0 - [q]_0)(\tau) =_{df} \tau(p)- \tau(q)$$
for all projections $p,q\in M_n(\tilde B)$ with $p-q\in M_n(B)$ and $\tau\in QT(B)$ (which extends uniquely to a non-normalized quasitrace on $M_n(B)$). %$x \in K_0(B)$ and $\tau \in QT(B)$.
However, projections in $M(B)$ do not have the cancellation property. 
In fact, one may have  projections $p\in  M_k(B)$ and $q\in M_k(M(B))\setminus M_k(B)$ 
for some $k\in \N$ such that $\tau(p)=\tau(q)$
for all $\tau\in QT(B).$ So, in particular,  $p$ and $q$ are not Murray-von Neumann equivalent.
Nevertheless, $[p]=[q]$ in $K_0(M(B)).$

%{\blue{The proof of the next result is already, to a certain extent, contained in
%the case of real rank zero canonical ideals, whose argument is due to Elliott,
%Lin and Zhang. (E.g., see \cite{LinExtContScale} 1.7.)}} 

The proof of the next result is already contained in
the case  that the canonical ideal has real rank zero, which was computed  some 
%a long
time ago. 
(E.g., see \cite{LinExtContScale} 1.7.)

\begin{thm} \label{thm:KComp}  
Let $B$ be {{a non-unital but}} $\sigma$-unital simple \CA\, with continuous scale, strict comparison for positive elements
{{and stable rank one.}}
Then we have the following:
\begin{enumerate}
\item $K_0(C(B)) \cong K_1(B) \oplus {{\Aff(QT(B))}}/\rho_B(K_0(B))$.   
\item $K_1(C(B)) = \ker(\rho_B)$. 
\end{enumerate}
\end{thm}

\begin{proof}
The exact sequence
$$0 \rightarrow B \rightarrow M(B) \rightarrow C(B) \rightarrow 0$$
induces a six term exact sequence in K theory which, together 
with Theorem \ref{thm:K1M(B)} and Proposition
\ref{prop:K0M(B)}, gives an exact sequence
$$0 \rightarrow K_1(C(B)) \rightarrow K_0(B) \rightarrow {{\Aff}}(QT(B))
\rightarrow K_0(C(B)) \rightarrow K_1(B) \rightarrow 0.$$
Then we obtain 
\beq
&&K_1(C(B))={\rm ker}\rho_B\andeqn\\
 &&0\to \Aff(QT(B))/\rho_B(K_0(B)) {\stackrel{j}{\longrightarrow}} K_0(C(B)){\stackrel{s}{\longrightarrow}} K_1(B)\to 0,
\eneq
where $j$ is an injective \hm\, and $s$ is a surjective \hm. 
In particular (2) holds.
Put $G_0= \Aff(QT(B))/\rho_B(K_0(B)).$
Note that $\Aff(QT(B))/\rho_B(K_0(B))$ is divisible and hence injective.
In other words, there is $\psi: K_0(C(B))\to \Aff(QT(B))/\rho_B(K_0(B))$
such that $\psi\circ j={\rm id}_{G_0}.$ It follows that $G\cong G_0\oplus {\rm ker}\psi.$ 
Therefore, $G/G_0\cong {\rm ker}\psi.$ It follows that ${\rm ker}\psi\cong K_1(B).$

\end{proof}

\begin{rem}\label{RTraceMB}
Let $B$ be a non-unital separable simple \CA\, with strict comparison and stable rank one.
In the same spirit as in \ref{12NN}, let us assume that $B$ has continuous scale.
%%%%%

(1) If
$p, q\in M(B\otimes {\cal K})\setminus B\otimes {\cal K},$ $\tau(p)<\tau(q)$
and $\tau(p) < \infty$  for $\tau\in QT(B)$ but {{$\widehat{q}-\widehat{p}\not\in {\rm LAff}_+(QT(B)),$}} 
then, by (2) of Theorem \ref{thm:Aug12022},  $p\not\lesssim q.$

(2) On the other hand, if $p, q\in M(B\otimes {\cal K})\setminus B\otimes {\cal K}$ are projections 
such that $\tau(p)<\tau(q)$ for all $\tau\in QT(B)$ and $\widehat{p}\in \Aff_+(QT(B)),$
then 
\beq
p\lesssim q\,\,\,{\rm in}\,\, M(B\otimes {\cal K}).
\eneq
(compare  Theorem 3.5 of \cite{Linossr=1}).

To see this, let $f=\widehat{q}-\widehat{p}.$ Since $\widehat{p}\in \Aff_+(QT(B)),$
$f\in {\rm LAff}_+(QT(B)).$ By (3) of Theorem \ref{thm:Aug12022}, 
\beq
p\lesssim q \,\,\,{\rm in}\,\, M(B\otimes {\cal K}).
\eneq
%%%%%%%%%%%

(3)  Suppose that there exists $h\in {\rm LAff}_+(QT(B))\setminus \Aff_+(QT(B))$ which is bounded. 
By (4) of Theorem \ref{thm:Aug12022}, there is a {{projection $q\in M(B\otimes {\cal K}) 
\setminus B \otimes {\cal K}$}} 
such that ${\widehat{q} = h}$.
%We may choose $q$ so that $\widehat{q}$ is bounded.
Put $B_1=q(B\otimes {\cal K})q.$ Then $B_1$ does not have continuous scale.
Moreover, $M(B_1)=qM(B\otimes {\cal K})q.$ Therefore, by (2) and (4) of Theorem \ref{thm:Aug12022}, $K_0(M(B_1))_+$ can be identified with
\beq\nonumber 
\{f\in {\rm LAff}_+(QT(B)): \makebox{{\rm there is}}\,\,k\in \N \,\,{\rm such\,\, that}\,\, f\leq k\hat q \textrm{ and } k \widehat{q}-f\in {\rm LAff}_+(QT(B))\}.
\eneq
%If $f\in \Aff_+(QT(B)),$ then there exists $k\in \N$ such that
%$k\widehat{q}-f\in {\rm LAff}_+(QT(B)).$ 
%It follows that $f\in K_0(M(B_1)).$
For any $g\in \Aff(QT(B)),$ choose $g_+\in \Aff_+(QT(B))$ such 
that $g(\tau)+g_+(\tau)>0$ for all $\tau\in QT(B).$ 
Then $g+g_+\in K_0(M(B_1)).$ It follows that $g\in K_0(M(B_1)).$
In other words, 
$\Aff(QT(B))\subset  K_0(M(B_1)).$ 
%$K_0(M(B_1))$ may be identified with a subgroup of ${\rm LAff}_+(QT(B)).$
In the light of (2) of 
Theorem \ref{thm:Aug12022}, one sees that  $K_0(M(B_1))$ has somewhat complicated order.

This further justifies our choice of non-unital \CA\, $B$ with continuous scale. 
\end{rem}

Many of the results stated in this section do not require the \CA\, $B$ to be separable.
The condition that $B$ is separable may be replaced by 
$B$ is $\sigma$-unital and ${\wtd QT}(B)$ has a metrizable Choquet simplex as its basis. 

\iffalse
{\color{cyan}COMMENT: Aren't the results of this section only true in the $\sigma$-unital case only if we restrict to projection $p\in M(B\otimes \mathcal K)$ such that $p(B\otimes \mathcal K)p$ is $\sigma$-unital???}
\fi

\section{Liftable extensions of stably finite \CA s}
%Finite canonical ideals}
%%%%%%%%%%%
%
In this section we investigate when an extension $0 \to B \to E \to A \to 0$ is liftable, in the sense that there is a homomorphism $A \to E$ which makes the extension split. The main theorem of this section is  Theorem \ref{Tlifable}.
%\ref{df:ClassCalA}.
%, from which Theorem \ref{thm:2} in the introduction follows. 
This is all done under certain assumptions on $A$ and $B$. The standing assumptions for $B$ are introduced in \ref{58} below, and the C*-algebra $A$ will be assumed to be in a class $\mathcal A_1$ introduced in Definition \ref{df:ClassCalA}. We explain in Remark \ref{rem:CGSTW} why it will follow from future work \cite{CGSTW} that $\mathcal A_1$ contains every separable amenable C*-algebra satisfying the UCT and which has a faithful amenable trace.

\iffalse
For a (possibly non-unital) C*-algebra $B$, we let $T(B)$ denote the 
tracial state space of $B$, with the weak* topology.  
Recall also that for a convex set $S$ in a locally convex space,
$\Aff(S)$ denotes the collection of affine continuous functions from
$S$ to $\mathbb{R}$, and 
we let
$$\Aff_+(S) = \{ f \in \Aff(S) : f(\tau) > 0 \makebox{  for all  }
\tau \in S \setminus \{ 0 \} \} \cup \{ 0 \}.$$

{\Green{Please compare \ref{58} with section 11}}
\fi
%%%%%%%%%%%%%%%%%%%%%%%%%%
%
\begin{NN}\label{58}
%{\bf{\red{Ping: Do we want to assume that $B$ is separable ?}}}
We will study the case that $B$ is a non-unital separable simple \CA\, with continuous scale such 
that the following statements hold: 
\begin{enumerate}
\item[I.] (Projection injectivity)  
If $p, q\in  M(B) \otimes {\cal K} \setminus B \otimes {\cal K}$ are two projections such that 
$\tau(p)=\tau(q)$
for all $\tau \in QT(B)$,
then $p \sim q$.  
(Recall that $QT(B)$ is compact, since $B$ has continuous scale.)
\item[II.] $(K_0(M(B)), K_0(M(B))_+)=(\Aff(QT(B)), \Aff_+(QT(B)))$.
Moreover, every element of $\Aff_+(QT(B)) \setminus \{ 0 \}$ can be realized by
a projection in $M(B) \otimes {\cal K} \setminus B \otimes {\cal K}$.
\item[III.]  $K_1(M(B))=\{0\}.$  
\end{enumerate}

\vspace*{2ex}

\textbf{Throughout this section, we will make the above \emph{standing 
assumptions} on
$B$.}\\

We note that the above assumptions are not too restrictive, and are possessed
by many C*-algebras, including C*-algebras
that are sufficiently regular.  For instance, by the results of Section \ref{KComputations}, any non-unital, separable, simple C*-algebra with continuous scale, strict comparison for positive elements, and stable rank one satisfies this standing assumption.
%(See, for example, section
%\ref{KComputations}.)\\  

Finally, from the computations in Section \ref{KComputations} and by our standing assumptions
on $B$, we have the following statements:
\begin{enumerate}
\item[IV.]  $K_0(C(B)) = K_1(B) \oplus {{\Aff}}(QT(B))/\rho_B(K_0(B))$.
\item[V.]  $K_1(C(B)) = \ker(\rho_B)$.\\
\end{enumerate}
\end{NN}

We continue by introducing more notation (some of which were already
introduced in earlier sections, but we present again for the convenience
of the reader).
For a (possibly non-unital) C*-algebra $C$, we let $T_f(C)$ denote the set 
of faithful tracial
states on $C$.  {{Note that}} $T_f(C)$ is a convex subset of $T(C)$.

Let $A$ be a separable amenable \CA\,  with $T_f(A)\not= \emptyset.$ %\{0\}.$ 
Let $$T_{f,0}(A)=\{r\tau: \tau\in T_f(A): 0\le r\le 1\}.$$ 
Let $\rho_{A,f}: K_0(A)\to \Aff(T_{f,0}(A))$ be 
the order  preserving \hm\, defined by
$$\rho_{A,f}([p]-[q])(\tau) =\tau(p)-\tau(q)$$ 
for all pairs of projections 
$p, q \in M_\infty(A^+)$ such that $\pi_\C^A(p)=\pi_\C^A(q)$ and $\tau\in T_{f,0}(A).$
Here, $\pi_{\C}^A : A^+ \rightarrow A^+/A = \C$ is the usual 
quotient map.  Also, if, in the above, $T_{f,0}(A)$ is replaced with
$T(A)$, then $\rho_{A, f}$ is replaced with $\rho_A$.

{{  
\begin{df}\label{Deflambda}
Let $A$ be a \CA\, with $T_f(A)\not=\emptyset.$
Recall that $\overline{T(A)}^w$  is the weak *- closure of $T(A)$ (see \ref{Dtrace}).
If $A$ is unital, $\overline{T(A)}^w=T(A).$
Then
the restriction map $\lambda : \Aff(\overline{T(A)}^w) \rightarrow 
\Aff^b(T_f(A))$ is strictly positive.
Note that $\|\lambda(f)\|\le \|f\|$ for all $f\in \Aff(T(A))$ (see \ref{Dppp}).
Denote by $\Aff^\lambda(T_f(A))=\lambda(\Aff(\overline{T(A)}^w))$ the linear subspace.
By $\Aff^\lambda_+(T_f(A)),$ we mean $\lambda(\Aff_+(T(A))).$
Note that if $A$ is simple, $T_f(A)=T(A).$

Let $\rho_A^w: K_0(A)\to \Aff(\Tw)$ be defined by
$$
\rho_A^w([p]-[q])(\tau)=\tau(p)-\tau(q)\rforal \tau\in  \Tw
$$
for any pair of projections $p, q\in  M_m(\wtd A)$ such that $p-q\in M_m(A)$ for all $m\in \N.$

Let $\rho_{A,f}^w=\lambda\circ \rho_A^w: K_0(A)\to \Aff^\lambda(T_f(A)).$ 
%\end{df}}}

%{\ppl{
Suppose that every $2$-quasitrace of $A$ is a trace.
If $C$ is another \CA\, with $QT(C)\not=\emptyset$ and $\phi: A\to C$ is a \hm,
%such that $\tau\circ \phi$ is a faithful trace for all $\tau\in QT(C),$ 
then one obtains a  continuous affine map $\phi_T: QT(C)\to T_0(A).$
Recall that $T_0(A)=\{ t\cdot s: t\in [0,1],\, s\in T(A)\}.$
This induces a continuous positive linear map
$L: \Aff(T_0(A))\to \Aff(QT(C)).$ 
Recall that $f(0)=0$ if $f\in \Aff(T_0(A)).$
Let $r: \Aff(T_0(A))\to \Aff(\Tw)$ be defined by
$r(f)=f|_{\Tw}$ for all $f\in \Aff(T_0(A)).$ 
It is an isometric isomorphism. Let 
$L_\phi: \Aff(\Tw)\to \Aff(QT(C))$  be defined by
$L\circ r^{-1}.$ 
%defined by $L_\phi(a)(\tau)=\tau\circ \phi(a)$ for all $a\in A_{s.a.}$ (see the first 
%part of the proof of Proposition 3.3 of \cite{LinNTZ}). 

Now let us assume that $\tau\circ \phi$ is a faithful trace for each $\tau\in QT(C).$
Note that if $f(s)=s(a)$ for some $a\in A_+$ and for all $s\in T(A),$ 
then $L_\phi(f)(\tau)\ge 0$ for all $\tau\in QT(C).$
Suppose that $f_1, f_2\in \Aff(T(A))$ such that $\lambda(f-f')=0.$ Let $a,b\in A_{s.a.}$
such that $s(a)=f_1(s)$ and $s(b)=f_2(s)$ for all $s\in  T(A)$ (see 2.7 and 2.8 of \cite{CP}).
It follows that $t(a-b)=0$ for all $t\in T_{f,0}(A).$
Then 
$\tau(\phi(a-b))=0$ since $\tau\circ \phi=r \cdot t$ for some $0\le r\le 1$ and $t\in T_f(A).$
This provides a linear map 
$\bar L_\phi: \Aff^\lambda(T_f(A))\to \Aff(QT(C)).$
As $\tau\circ \phi\in T_{f,0}(A),$ we conclude that $\|\bar L_\phi\|\le 1.$
\end{df}}}

{{
\begin{prop}\label{P230604}
Let $A$ be a non-unital \CA\, with $T_f(A)\not=\emptyset.$

(1) Then 
$\tau(a)<1$ for all $a\in A_+$ with $0\le a\le 1$ for all $\tau\in T_f(A).$

(2) If every 2-quasitrace of $A$ is a trace,  $C$ is another \CA\, and 
$\phi: A\to C$ is a \hm,  then ${\bar L_\phi}(1_{T_f(A)})(\tau)>\tau(\phi(a))$
for all $0\le a\le 1$ and $\tau\in QT(C)$ whenever $\tau\circ \phi$ is a faithful trace.
\end{prop}
\begin{proof}
For part (1), let 
$0\le a\le 1$ be in $A_+$ such that
$\tau(a)=1$ for some $\tau\in T_f(A).$  If there exists $n_0\in \N$ such that 
$a^{1/n}=a^{1/n_0}$ for all $n\ge n_0,$ put $e=a^{1/2n_0}.$ Then $e^2=a^{1/4n_0}=a^{1/n_0}\ge e.$ 
Since $0\le e\le 1,$ this implies that $e^2=e.$
So $e$ is a projection and $e\ge a.$ 
It follows that $\tau(e)=1.$ But $(1-e)b\not=0$ for some $\|b\|\le 1.$
Otherwise $e$ would be the unit for $A.$ Then $x=(1-e)bb^*(1-e)\not=0.$
It follows that $0\le e+x\le 1$ is in $A_+$ 
Hence $\tau(e+x)\le 1.$ But $\tau(e)=1.$ We conclude that $\tau(x)=0.$ This contradicts the fact 
that $\tau\in T_f(A).$ 

The remaining case is the case that $a^{1/n}-a\not=0$ for some 
$n\in \N.$ Since $0\le a^{1/n}-a\le 1,$ we must have $\tau(a^{1/n}-a)>0.$ 
This is not possible as $\tau(a)=1.$ Thus part (1)  holds.

Now consider part (2).  From what has be proved, 
we may assume that $a\not=a^{1/n}$ for all $n\in \N.$
Therefore $\tau\circ \phi(a^{1/n}-a)>0$ for all $\tau\in QT(C)$ such that 
$\tau\circ \phi$ is a faithful trace.
Since ${\bar L_\phi}(1_{T_f(A)})\ge \tau\circ \phi(a^{1/n})$ for all $n\in \N,$ 
we conclude that ${\bar L_\phi}(1_{T_f(A)})(\tau)>\tau(\phi(a))$
for all $0\le a\le 1$ and $\tau\in QT(C)$ when $\tau\circ \phi$ is a faithful trace.
\end{proof}
}}

{{
\begin{df}\label{DCAc}
Let $CA_+^{\bf 1}(\Aff^\lambda(T_f(A)), \Aff(QT(B)))$ be the set of 
%continuous
linear maps  $L$ 
from $\Aff^\lambda(T_f(A))$ to $\Aff(QT(B))$ 
which maps $\Aff^\lambda_+(T_f(A))\setminus \{0\}$ into $\Aff_+(QT(B))\setminus \{0\},$
and which maps $1_{T_f(A)}$ 
to either $1_{QT(B)}$ or an affine function strictly less than $1_{QT(B)}$.  Here ``strictly less"  means strictly less at every
point in $QT(B)$. 

Set{\small{\beq\nonumber
H_{A, B,\rho}:=\{\eta\in {\rm Hom}(K_0(A), \Aff(QT(B)))_+: \eta=L\circ \rho^w_{A,f}\andeqn L\in 
CA^{\bf 1}_+(\Aff^\lambda(T_f(A)), \Aff(QT(B)))\}.
\eneq
}}

\end{df}
}}
\iffalse
Let $CA_+^{\bf 1}({\ppl{\Aff^b}}_+(T_f(A)), \Aff_+(QT(B)))$ be the set of continuous
{\color{cyan}linear} maps 
from ${\ppl{\Aff^b}}(T_f(A))$ to $\Aff(QT(B))$ 
which maps ${\ppl{\Aff^b}}_+(T_f(A))\setminus \{0\}$ into $\Aff_+(QT(B))\setminus \{0\}$, 
and which maps $1_{T_f(A)}$ 
to either $1_{QT(B)}$ or an affine function strictly less than $1_{QT(B)}$.  ``Strictly less"  means strictly less at every
point in $QT(B)$. 
\fi

%Let $L_f\in CA_+(\Aff_+(T_f(A)), \Aff_+(T(B))).$ 
%Then $L: \Aff(T(A))\to \Aff(T(B))$
%defined by 

Note that if $T_f(A) \neq \emptyset$ then
 $CA^{\bf 1}_+({{\Aff^\lambda}}(T_f(A)), \Aff(QT(B)))\not=\emptyset.$ 
To see this, fix $\tau_0\in T_f(A).$
Define $j_f: QT(B)\to T_{f}(A)$ by $j_f(\tau)=\tau_0$
for all $\tau\in QT(B).$ Then $j_f$  induces an element in $CA^{\bf 1}_+({{\Aff^\lambda}}(T_f(A)), \Aff(QT(B))).$

%is an affine continuous map. 
%Since $T_{f,0}(A)\subset T_0(A),$ there is also $j: T(B)\to T_{f,0}(A)\subset T(A).$
% If $L_f\in CA_+^{\bf 1}({\ppl{\Aff^\lambda}}(T_f(A)), \Aff(QT(B))),$ define 
% $L: \Aff(T(A))\to \Aff(QT(B))$ by
%$$L(g)=L_f(g|_{T_f(A)})\rforal g\in \Aff(A).$$

 % If $L_f\in CA(T(B), T_{f,0}(A)),$ define 
%$j_\sharp: \Aff(T(A))\to \Aff(T_{f,0}(A))\to \Aff(T(B))$ by $j_\sharp(f)=f\circ j,$
%where the map $ \Aff(T(A))\to \Aff(T_{f,0}(A))$ is the restriction.

\iffalse
%Let ${\rm Hom}_\pi(K_0(A), \Aff(T(B)))_+$ be the subset of those  $\lambda\in {\rm Hom}(K_0(A), \Aff(T(B)))_+$ such that
%$\lambda(K_0(A)_+)\cap \rho_{B}(K_0(B))=\{0\}.$ 
\fi
\vspace*{2ex}
%{\blue{We needed to revise the definition of $H_{A, \rho}$ by replacing
%$Hom_{\pi}$ with $Hom$.  Please
%see (\ref{equ:Feb20Counterexample}).}}\\  

%Set
%{\small{\beq\nonumber
%H_{A, B,\rho}:=\{\lambda\in {\rm Hom}(K_0(A), \Aff({\color{cyan}Q}T(B)))_+: \lambda=L\circ \rho_{A,f}\andeqn L\in 
%CA^{\bf 1}_+({\ppl{\Aff^\lambda}}(T_f(A)), \Aff(QT(B)))\}.
%\eneq
%}}
{{Let}} $H^0_{A, B, \rho}$ {{be the subset of $H_{A,B,\rho}$ consisting}} of those ${{\gamma}}= L \circ \rho^w_{A, f} \in H_{A, B,\rho}$ for which
$L$ is strictly non-unital, i.e., $L(1_{T_f(A)})(\tau) < 1$ for all $\tau \in QT(B)$.
%{\ppl{We may use ${\cal A}_1,$ ${\cal A}_2$ and ${\cal A}_3.$}}

%Suppose that $K_0(A)$ is finitely generated as an abelian group. Then 
%$K_0(A)\cong \Z^k$ for some $k\in \Z.$  Let $\lambda\

\begin{rem}\label{RQTMB}
Let $B$ be as above.  Since $B$ has continuous scale, $M(B)/B$ is a purely infinite simple \CA\, (see \ref{Tcontscal}).
Therefore $QT(M(B))=QT(B).$ 
By \cite[II.4.4]{BH}, $QT(M(B))$ is a Choquet simplex.  Since $B$ is separable, this also implies that
$QT(B)$ is a metrizable Choquet simplex.  It follows that there is a unital simple AF-algebra 
$F$ (by \cite{BlackadarTraces})  for which there exists an affine homeomorphism $\iota : QT(B) \to T(F).$
It induces an order preserving \hm\, $\Gamma: \rho_F(K_0(F))\to \Aff(QT(B))$ such that
$\Gamma([1_F])=1_{QT(B)}.$ It follows from  \cite[Lemma 4.2] 
{PR}
that  there is a unital injective \hm\, $j_F: F\to M(B)$ such that 
$(j_F)_{*0}=\Gamma\circ \rho_F.$ Since $F$ is an AF-algebra, $\tau(j_F(a))=\iota(\tau)(a)$
for all $a\in F_{s.a.}$ and $\tau\in QT(B).$
\end{rem}

\begin{df}[The classes $\mathcal A_1$ and $\mathcal A_2$]\label{df:ClassCalA}

Let ${\cal A}_1$ be the class of all separable amenable \CA s $A$ such that,
for any unital simple infinite-dimensional AF-algebra $C$ with ${\rm ker}\rho_C=\{0\},$ 
and 
any order preserving \hm\, $\af\in {\rm Hom}(K_0(\wtd A), K_0(C))$ with 
$\af([1_{\wtd A}])\le [1_C],$   if there is a strictly positive 
linear continuous map  
$\Gamma: {{\Aff^\lambda}}(T_f(A))\to \Aff(T(C))$   such that
$ \Gamma\circ \rho^w_{A,f}= \rho_C \circ \af,$
%and any $\af\in H_{\wtd A, C,\rho}$ 
%
then 
there 
exists an embedding $h: A\to C$ such that 
$h_{*0}=\af|_{K_0(A)}$ (recall that $\wtd A$ is the minimal unitization of $A$)
and $\tau(h(a))=\Gamma(\hat{a})(\tau)$ for all $a\in A_{s.a.}$ and $\tau\in T_f(A).$

Let ${\cal A}_2$ be the class of all separable amenable \CA s $A$   such that, 
for any $B$ satisfying the standing assumptions {{of this section,}} {{and,}} for any $\lambda\in H_{A,B,\rho},$ there exists an embedding 
$h: A\to M(B)$ such that   $h(A)\cap B=\{0\}$ and
$h_{*0}=\lambda.$  Moreover, if $A$ is unital and $\lambda([1_A])=[1_{M(B)}],$ 
then we may choose $h(1_A)=1_{M(A)},$    and, if $\lambda([1_A])\not=[1_{M(B)}],$ we may choose
$1_{M(B)}-h(1_A)\not\in B.$ 

%\end{NN}

\end{df}
%-----we will discuss which \CA s in ${\cal A}$---Theorems have been proved.

%\end{NN}

{\begin{rem}\label{rem:CGSTW}
It will follow from upcoming work of the first named author, in collaboration with Carrión, Schafhauser, Tikuisis, and White \cite{CGSTW}, that every separable amenable C*-algebra satisfying the UCT and with a faithful amenable tracial state is in $\mathcal A_1$. 

Let $\mathcal A_1^{UCT}$ and $\mathcal A_2^{UCT}$ be the classes of C*-algebras in $\mathcal A_1$ and $\mathcal A_2$ respectively which also satisfies the UCT. Since every C*-algebra in $\mathcal A_2$ has a faithful trace, it follows from Proposition \ref{PA1A2} below, that $\mathcal A_1^{UCT} = \mathcal A_2^{UCT}$ and that these classes are exactly the class of all separable amenable C*-algebras satisfying the UCT which have a faithful tracial state. 

In order to not have this paper rely on forthcoming work, we will work with the classes $\mathcal A_1$ and $\mathcal A_2$ instead, and we will provide large classes of examples in Theorem \ref{TembeddingAH} which are in $\mathcal A_1$ and $\mathcal A_2$. 

But we emphasize that all results where we consider $A\in \mathcal A_i$ satisfying the UCT will hold for every separable amenable C*-algebra satisfying the UCT which has a faithful tracial state.
\end{rem}}

\begin{prop}\label{PA1A2}
If $A\in {\cal A}_1,$ then $A\in {\cal A}_2.$
\end{prop}

\begin{proof}
Suppose that $B$ satisfies the standing assumptions of this section.
Let $\lambda\in H_{A, B, \rho}$ and $\lambda=L\circ \rho^w_{A,f}$ for $L \in CA^{\bf 1}_+({{\Aff^\lambda}}(T_f(A)), \Aff(QT(B)))$.

Choose a countable dense ordered subgroup $H\subset \Aff(QT(B))$ such that   $1_{T(B)}\in H$ and 
\beq\label{monoid-1}
\lambda(K_0(\wtd A))_+\subseteq H_+:=H\cap \Aff_+(QT(B)).
%\andeqn 
%\rho_B(K_0(B))\cap H_+=\{0\}.
\eneq
%where 
There exists a unital simple 
separable  AF-algebra $F$ such that $T(F)={{Q}}T(B),$ 
${\rm ker}\rho_F=\{0\},$
 $K_0(F)=H$,  $V(F) = K_0(F)_+ = H_+$,
and $[1_F]=1_{T(B)}.$

Let $\iota: K_0(F)\to \Aff(QT(B))$ be the embedding.
We will construct an embedding  $\psi: F\to M(B).$  To do this, we  apply the same argument as that  of Lemma 8.1 of \cite{LN} 
(see also Lemma 5.2 of \cite{LinExtQuasidiagonal}).
It follows from 7.2 of \cite{LN} that 
\beq
V(M(B))=V(B)\sqcup (\Aff_+(QT(B))\setminus \{0\})
\eneq
(with the mixed sum  and order as follows: if $x\in V(B)$ and $y\in \Aff_+(QT(B))\setminus \{0\},$
$x+y=\rho_B(x)+y$;  $x\le y,$ if $\rho_A(x)(\tau) < y(\tau)$ for
all $\tau \in QT(B)$).  
Note that $\Aff_+(QT(B))$ is an ordered semigroup (monoid).
Recall that $\iota: V(F)\to \Aff_+(QT(B))$ is   the inclusion $H_+\subset \Aff_+(QT(B)).$
It follows from  Lemma 4.2 of 
\cite{PR} that there exists a   unital  \hm\, 
%Then, by the proof of Lemma 8.1 of \cite{LN}, there exists 
%an embedding 
$\psi: F \to M(B)$ such that   $\psi$ induces the inclusion $\iota.$
Note that, since $F$ is simple and since $\psi$ is unital,
$\psi$ is an embedding and $\psi(A) \cap B = \{ 0 \}$.  

%Note also, if $p\in F$ is a non-zero projection, 
%then $\psi(p)\in \Aff_+(T(B))$ (as a sub-semigroup 
%of $V(M(B))$ which has empty intersection with $V(B)$). 
%So $\psi(p)\not\in B.$ It follows that $\psi(F)\cap B=\{0\}.$
%{\green{$\psi(eFe) \cap B = \{ 0 \}$ and}}   

Next consider the linear map $L\in {{CA_+^{\bf 1}({{\Aff^\lambda}}(T_f(A)), \Aff(QT(B)))}}.$
%: {\ppl{\Aff^\lambda}}(T_f(A))\to \Aff(QT(B)).$
Then $L\circ \rho^w_{A,f}(K_0(A))\subset \rho_C(K_0(F))=K_0(F).$ 
Define $\af: K_0(A)\to K_0(F)$ by $\af(x)=L\circ \rho^w_{A,f}(x).$
Since $$L\in CA_+^{\bf 1}({{\Aff^\lambda}}(T_f(A)), \Aff(QT(B))),$$
$L$ is strictly positive. Since $A\in {\cal A}_1,$ 
it follows that there is an injective \hm\, $\phi: A\to F$ such that
\beq
\phi_{*0}=\alpha.
\eneq
Define $h: A\to M(B)$ by $h=\psi\circ \phi.$
Note that 
\beq
h_{*0}=\psi_{*0}\circ \phi_{*0}=\iota\circ \alpha =\lambda.
\eneq
If $A$ is unital, and $\af([1_A])=[1_{M(B)}]=[1_F],$ then $h(1_A)=1_{M(B)}.$
If $\af([1_A])\not=[1_{M(B)}],$ then $\af([1_A])<[1_F].$ Hence $1_F-\phi(1_A)\not=0$ and 
$1_F-\phi(1_A)\in \psi(C).$
It follows that $1_{M(B)}-h(1_A)\not\in B.$
\end{proof}

Recall that when $A$ satisfies the UCT and $B$ satisfies the standing assumptions of this section, then $KK(A, M(B))={\rm Hom}(K_0(A), \Aff(QT(B)))$. 
Therefore we may view $H_{A, B, \rho}$ as a subgroup of $KK(A, M(B)).$

{{
Let $A$ be a unital, simple, separable, stably finite C*-algebra.
Let $S_1(K_0(A))$ be the state space of $(K_0(A), K_0(A)_+, [1_A])$, i.e., 
$S_1(K_0(A)) := \{ f \in Hom(K_0(A), \mathbb{R})_+ : f([1_A]) = 1 \}$. 
Let $r_A : T(A) \rightarrow S_1(K_0(A))$ be the map given by 
$r_A(t)([p]) = t(p)$ for all projections $p \in {{A \otimes {\cal K}}}$ and for
all $t \in T(A)$.  The
\emph{Elliott invariant} of $A$ is defined to be the sextuple 
$$Ell(A) := (K_0(A), K_0(A)_+, [1_A], K_1(A), T(A), r_A).$$
Let $B$ be another unital, simple, separable, stably finite C*-algebra.  
A \emph{homomorphism} $\gamma : Ell(A) \rightarrow Ell(B)$ between Elliott
invariants is a triple $\gamma = (\kappa_0, \kappa_1, \kappa_T),$ where
$\kappa_j : K_j(A) \rightarrow K_j(B)$ ($j=0,1$) are group homomorphisms for
which $\kappa_0$ is order preserving and $\kappa_0([1_A]) = [1_B]$,
$\kappa_T : T(B) \rightarrow T(A)$ is an affine continuous map, and 
\beq\label{ellk}
r_A(\kappa_T(t))(x) =   r_B(t)(\kappa_0(x)) \makebox{  for all  }
x \in K_0(A) \makebox{  and }   t \in T(B).
\eneq
If $\alpha \in KK(A, B)$, then $\alpha$ is said to be \emph{compatible}
with $\gamma$ if $\alpha_j = \kappa_j$ for $j = 0,1$.
%{\red{There is no need to say that (\eqref{ellk}) holds, since that is
%part of the definition of $\gamma$, a morphism between
%Elliott invariants.}}% {\ppl{(and \eqref{ellk} holds).}}

For $\gamma$ as in the previous paragraph,
if, in addition, $\kappa_0$ is an ordered group isomorphism, $\kappa_1$
is a group isomorphism and $\kappa_T$ is an affine homeomorphism, then we 
say that $\gamma$ is an \emph{isomorphism} of Elliott invariants and we
write $Ell(A) \cong Ell(B)$.}}

The following is a variation of   Theorem 29.5 of \cite{GLNII} 
(see \cite{GLNhom2}).

\begin{lem}\label{Hom1}
Let $A$ and $B$ be two  finite  unital simple  
separable amenable ${\cal Z}$-stable \CA s such that $A$ satisfies the UCT.

Suppose that there is a \hm\, $\gamma:
{\rm Ell}(A)\to {\rm Ell}(B)$   and $\kappa\in KL(A, B)$
which is compatible with $\gamma.$ Then there is a \hm\, 
$\phi: A\to B$ such that $\phi$ induces $\gamma$ and $\kappa.$

\end{lem}
\begin{rem} \label{rmk:AffToT}
Let $A$ and $C$ be  unital separable C*-algebras 
with $T_f(A) \neq \emptyset \neq T(C)$
and $C$ simple.  
If we have a strictly positive affine continuous map
$\Gamma : {{\Aff^\lambda}}(T_f(A)) \rightarrow \Aff(T(C))$, then $\Gamma$ induces 
an affine continuous map $\gamma : T(C) \rightarrow T_f(A)$.  Here is a  
short argument:  The restriction map $\lambda : \Aff(T(A)) \rightarrow 
{{\Aff^\lambda}}(T_f(A))$ is strictly positive, and so we have a strictly positive map
$\Gamma \circ \lambda : \Aff(T(A)) \rightarrow \Aff(T(C))$.  This induces 
an affine continuous map $\gamma : T(C) \rightarrow T(A)$.  It remains to check
that the range of $\gamma$ is in $T_f(A)$.      
Note that for all $a \in A_+ \setminus \{ 0 \}$, for all $\tau \in T_f(A)$,
$\tau(a) > 0$, and so $\lambda(\hat{a})  \in {{\Aff_+}}(T_f(A)) \setminus \{ 0 \}$,
and hence (since $\Gamma$ is strictly positive) 
$\Gamma \circ \lambda(\hat{a})
\in {{\Aff_+}}(T(C)) \setminus \{ 0 \}$.   Hence, for all $\tau \in T(C)$,
for all $a \in A_+ \setminus \{ 0 \}$,
$\gamma(\tau)(a) = \Gamma \circ \lambda(\hat{a})(\tau) > 0$.  Hence,
the range of $\gamma$ is in $T_f(A)$.   
\end{rem}
%}}

{{
\begin{thm}\label{TembeddingAH}
(1) If $A$ is a unital AH-algebra with $T_f(A)\not=\emptyset$, then $A\in {\cal A}_1\subset {\cal A}_2.$

(2) If $A$ is a unital simple separable finite classifiable 
amenable C*-algebra,  then $A \in {\cal A}_1\subset {\cal A}_2.$
\end{thm}
}}

\begin{proof}
%{\ppl{Note that (2) follows from Lemma \ref{Hom1} and Proposition \ref{PA1A2}.

{{To see (1),  let $C$ be a unital simple}} {AF-algebra } {{with ${\rm ker}\rho_C=\{0\}.$ 
Suppose $\kappa_0: K_0(A)\to K_0(C)=\rho_C(K_0(C))$  is an order preserving \hm\, and 
$\Gamma: {{\Aff^\lambda}}(T_f(A))\to \Aff(T(C))$  is  a strictly positive affine continuous map }}  
such that $\Gamma\circ \rho^w_{A,f}=\rho_C\circ \kappa_0.$  
Let $e\in C$ be a projection such that $[e]=\kappa_0([1_{ A}]).$  
Put $F=eCe.$ 
The map $\Gamma$  induces
an affine continuous map $\gamma : T(C) \to T_f(A)$ (see 
Remark \ref{rmk:AffToT}) which
{{ is compatible with $\kappa_0.$ 
Note that $K_1(F)=\{0\}.$
By the UCT, there exist $\kappa\in KL_e(C, F)^{++}$ (see \cite[Theorem 6.11]{LinAH})
which induces $\kappa_0$ on $K_0(F)$ and zero map on $K_1(A).$ 
Let  $\af:  U(M_\infty(A))/CU(M_\infty(A))\to U(M_\infty(F))/CU(M_\infty(F))=0$ be zero.
Then $\kappa,$ $\gamma$ and $\af$ are compatible. 
By \cite[Theorem 6.11]{LinAH},   there is a  unital monomorphism 
$\phi: A\to F\subset C$ such that $\phi_{*0}=\kappa_0$ and $\phi_\sharp=\gamma.$
}}

The proof of statement {{ (2) follows the same argument using Proposition \ref{PA1A2}
but  also applying Lemma \ref{Hom1}  instead of \cite{LinAH}.}}
\end{proof}
\begin{cor}
Every separable commutative \CA\, is in ${\cal A}_1\subset {\cal A}_2.$ 
\end{cor}
 
\begin{proof}
Let $A$ be a separable commutative \CA. Then $\wtd A=C(X).$ for some 
compact metric space.  Let $\{t_n\}$ be a countable dense subset of $X.$ 
Define $\tau(f)=\sum_{n=1}^\infty f(t_n)/2^n.$ 
Then, if $f\not=0$ and $f\ge 0,$ then $\tau(f)>0.$ It follows that 
$T_f(\wtd A)\not=\emptyset.$  
\end{proof}

%In fact, we have been {\textbf{REFERENCE NEEDED}} 
%recently informed (by ...) that the following general
%result is true, though it has not appeared yet.  We state it here for
%completeness.

The following follows from the work in \cite{CGSTW}, see Remark \ref{rem:CGSTW}.

\begin{thm}
  Every separable amenable \CA\, in the UCT class and with a  faithful 
tracial state is in ${\cal A}_1$.
\end{thm}

\begin{lem}

Let $A \in \mathcal{A}_2$ be unital.
Say that $L \in CA^{\bf 1}_+({{\Aff^\lambda}}_+(T_f(A)), \Aff_+(QT(B)))$ and $p \in M(B) \setminus B$ is a
projection such that
$$L(1_{T_f(A)}) = [p].$$

Then there exists a unital  embedding $h : A \rightarrow pM(B)p$
such that $h(A) \cap B = \{ 0 \}$ and (when $h$ is viewed as a map into $M(B)$)
$$h_{*0} = L  \circ {{\rho^w_{A,f}}}.$$ 
\label{lem:Feb22202111:45PM}
\end{lem}

\begin{proof}

{{Since $A\in {\cal A}_2,$ there exists an}} {{injective}} {{ \hm\, 
$\phi: A\to M(B)$ such that 
$\phi_{*0}=L\circ \rho_{A,f}$ with $\phi(A)\cap B=\{0\}.$   It follows that $[\phi(1_A)]=[p].$ 
By the assumption on $B,$ there is a partial isometry $v\in M(B)$}} {{with
$vv^* = \phi(1_A)$}} {{such that 
$v^*\phi(1_A)v=p.$ Define $h: A\to pM(B)p$ by 
$h(a)=v^* \phi(a)v$ for all $a\in A.$ }}
\vspace{0.2in}

\iffalse  
{\red{Old proof}.}

From Lemma \ref{lem:Feb22202110PM}, {\Green{Is this reference different from something I mentioned? Please define the map}}
 let 
$$\widetilde{\gamma}:   \Aff(T(B)) \rightarrow \Aff(T(P B P))$$
be an order-preserving, continuous, linear surjection such that 
for all $a \in {\red{(pM(A)p)_{s.a.}}},$
%p M(B)_{SA} p$, 
viewing $a$ as an element of $M(B)$ and thus $\widehat{a} \in \Aff(T(B))$,
\fi

\iffalse
$$\widetilde{\gamma}(\widehat{a})(\tau) = \tau(a) \label{equ:Feb22202110:30PM}$$
for all $\tau \in T(pBp)$.  In particular,
$$\widetilde{\gamma}(\widehat{p})(\tau) = 1 \label{equ:Feb22202111PM}$$
for all $\tau \in T(pBp)$.

Then $$\widetilde{\gamma} \circ L \in CA_+({\ppl{\Aff^b}}_+(T_f(A)),  \Aff_+(T(B))).$$
Since $p B p$ also satisfies the standing assumptions, and since $A \in \mathcal{A}$,
let $$h' : A \rightarrow p M(B)p$$  be a (necessarily unital) embedding  with $h'(A) \cap p B p = \{ 0 \}$ 
and
$$h'_{*0} = \widetilde{\gamma} \circ L \circ \rho_{A, f}.$$

Let $\iota : p M(B) p \hookrightarrow M(B)$ be the inclusion map, and let
$$h =_{df} \iota \circ h' : A \rightarrow M(B).$$

We have that for all $x \in K_0(A)$, for all $\tau \in T(pBp)$,
$$\tau(h'_{*0}(x)) =  \tau(\widetilde{\gamma} \circ L \circ \rho_{A,f}(x)).$$

Hence, for all $x \in K_0(A)$, for all $\tau \in T(B)$,
$$\tau(h_{*0}(x)) = \tau(L \circ \rho_{A,f}(x)).$$

Hence, $$h_{*0} = L \circ \rho_{A,f}.$$
\fi
%%%%%%
\end{proof}

\iffalse
\begin{equation} \label{equ:Feb20Counterexample} \end{equation}
{\blue{The definition of $H_{A, \rho}$ needs to be revised.   
Also, the former statement (iii) in the next result is also incorrect.
We can find essential $\tau : A \rightarrow M(B)$ with $[\tau(1_A)] = 
[1_{C(B)}] \in \rho_B(K_0(B))$, $\tau$ non-unital, but $\tau$ liftable.  E.g.,
let $A = \mathbb{C}$ and $B$ be a non-unital continuous scale
hereditary C*-subalgebra of
a stabilized UHF algebra with dimension group $\mathbb{Q}$ such that 
$\tau(1_{M(B)}) \in \tau(K_0(B))$.  Let $P \in M(B) - B$ be
a projection such that $1 - P \notin B$ and $\tau(P), 
\tau(1 - P) \in \tau(K_0(B))$.
Let $\psi : \mathbb{C} \rightarrow M(B) : 1 \mapsto P$ and
$\tau =_{df} \pi \circ \psi$.   Then $\tau$ is non-unital and $[\tau(1)] = 
[1_{C(B)}]$.  Moreover, $\tau$ is liftable.}}\\ 
\fi

Recall that for C*-algebras $A$ and $B$,
an extension $\tau : A \rightarrow C(B)$ is \emph{liftable} or 
\emph{trivial} if there exists a homomorphism $\phi : A \rightarrow M(B)$
such that $\tau = \pi \circ \phi$.
{{One notices that are
some  complicated restrictions on $\lambda.$  
It  is  not sufficient to assume 
the existence of a strictly positive 
\hm\, $\lambda\in {\rm Hom}(K_0(A), 
K_0(M(B)))$  such that $KK(\tau)=KK(\pi)\circ \lambda$
(see Example \ref{ExmX}).}}

\vspace{0.2in}

%{\ppl{\bf{It looks that there was some crossed files in some earlier revision.}}}

\begin{thm}\label{Tlifable}
Let $A\in {\cal A}_2$ satisfy the UCT, $B$ be a C*-algebra satisfying the
standing assumptions of this section,  and let $\tau: A\to C(B)$ be an 
essential extension.
Then $Ext(A, B)=KK(A, C(B)).$  

(i) Suppose that $A$ is not unital.  Then 
$\tau$ is liftable  if and only if
 $KK(\tau)=KK(\pi)\circ \lambda$ for some $\lambda\in H_{A, B,\rho}.$

(ii) Suppose that $A$ is unital and $[\tau(1_A)] \neq [1_{C(B)}]$.
 Then 
$\tau$ is liftable  if and only if
 $KK(\tau)=KK(\pi)\circ \lambda$ for some $\lambda\in H_{A, B, \rho}.$

(iii)  Suppose that $A$ is unital and $[\tau(1_A)] = [1_{C(B)}]$ and $\tau$ is non-unital.
Then 
$\tau$ is liftable  if and only if
 $KK(\tau)=KK(\pi)\circ \lambda$ for some $\lambda\in H_{A, B, \rho}^0.$ 

(iv) There is an example of $A$, $B$, $\tau$,
with $[\tau(1_A)] = [1_{C(B)}]$, $\tau$ non-unital, such that 
$KK(\tau)=KK(\pi)\circ \lambda$ for some $\lambda \in H_{A, B, \rho}$, 
but $\tau$
is not liftable. 

(v) Suppose that $A$ is unital,  $\tau$ is unital,
$B \otimes \cal K$ has strict comparison of positive elements, 
and $H_1(K_0(A), K_1(C(B)))=K_1(C(B))$. 
Then $\tau$ is liftable if and only if $KK(\tau)=KK(\pi)\circ \lambda$ for some 
$\lambda\in H_{A,B, \rho}.$\\
(See (\ref{equdf:H1}) for the definition of $H_1(K_0(A), K_1(C(B)))$.
See also the statement of Theorem \ref{TH1} and note that 
by our standing assumptions on $B$, $K_1(M(B)) = 0$.)

(vi) For any $\lambda\in H_{A, B, \rho},$ there is an essential extension
$\tau_1: A\to M(B)/B$ such that $KK(\tau_1)= KK(\pi) \circ \lambda.$

%(vi) If $KK(\tau)=0,$ then $\tau$ is not liftable.

\end{thm}

\begin{proof}

Statement (vi) follows immediately from the the definition of  
$\mathcal{A}_2$.

The ``only if" directions of (i),  (ii), (iii) and (v)  
are straightforward, i.e.,  
if $\tau$ is liftable then $KK(\tau)=KK(\pi)\circ \lambda$
for some $\lambda\in H_{A,B,\rho}.$ 
We sketch the argument for the convenience of the reader: 
Let 
$\phi: A\to M(B)$ be a (necessarily injective)
 \hm\,  such that $\pi\circ \phi=\tau.$
Note that since $B$  {{is simple and}} has continuous scale,
% (by our standing 
%assumptions on $B$), 
every {{quasitrace}} in ${{QT}}(B)$
extends uniquely to a {{quasitrace}} in ${{QT}}(M(B))$.
{{It follows that $\tau\circ \phi$ is a faithful trace of $A.$ 
Thus we obtain a continuous positive linear map
${\bar L_\phi}: \Aff^\lambda(T_f(A))\to \Aff(QT(B))$  (see \ref{Deflambda}).
By Proposition \ref{P230604}, 
if $A$ is non-unital, then ${\bar L}_{\phi}\in CA_+^{\bf 1}(\Aff^\lambda_+(T_f(A)), \Aff_+(QT(B))).$
If $A$ is unital, $p:=\phi(1_A)\in M(B)\setminus B$ is a projection such that
$p\le 1_{M(B)}.$  If $p\not=1_{M(B)},$ then $(1-p)B(1-p)\not=0.$
Since $B$ is simple, this implies that $\tau(1-p)>0$ for all $\tau\in QT(B).$ 
Therefore, in this case, we still have 
${\bar L}_\phi\in  CA_+^{\bf 1}(\Aff^\lambda_+(T_f(A)), \Aff_+(QT(B))).$}}

%Then $\phi$ induces an affine continuous map $T(B) \rightarrow T_f(A)$,
%which in turn induces $L \in {\ppl{CA_+({\ppl{\Aff^b}}_+(T_f(A)), \Aff_+(QT(B)))}}$.
By our standing hypotheses {{of this section}} on $B$ and since $\phi$ lifts $\tau$, we have
that $$K_0(\phi) = \lambda =_{df} {\bar L_\phi} \circ \rho_{A}^{w,f}\in Hom(K_0(A), 
{{\Aff(QT(B)))_+}}.$$   
Since $A$ satisfies the UCT and by our standing hypotheses on $B$,
{{$$KK(A, M(B)) = Hom(K_0(A), \Aff(QT(B)).$$  }} 
So $KK(\phi) = \lambda$, and 
thus, $$KK(\tau)=KK(\pi)\circ\lambda.$$
(Note that, in the computation for (iii), since $\tau$ is non-unital, $\phi$ is 
non-unital and so the corresponding $L$ satisfies that
$L(1_{T_f(A)})(t) < 1$ for all $t \in QT(B)$.  Thus the corresponding
$\lambda \in H^0_{A, \rho}$.)

{{Next  we assume}} that 
$KK(\tau)=KK(\pi)\circ \lambda$ for some $\lambda\in H_{A, B,\rho}.$ 
Since $A\in {\cal A}_2,$ there is an injective \hm\,  $\phi: A\to M(B)$
such that $\phi(A)\cap B=\{0\}$ and $KK(\phi)=\lambda.$ 
It follows that 
\beq
KK(\pi\circ\phi)=KK(\tau).
\label{equ:Feb20202111PM}
\eneq

We now prove the ``if" direction of (i) and (ii).
In case (ii), $[\tau(1_A)]\not=[1_{M(B)/B}]$ and hence, $\pi\circ \phi$ is also not unital.  Thus for both cases (i) and (ii), $\pi \circ \phi$ and $\tau$
are non-unital.
Thus, for both {\ppl{cases}} (i) and (ii), by (\ref{equ:Feb20202111PM}) and
 Theorem \ref{TH2}, 
 there is a unitary $U\in M(B)$ such that
\beq
\pi(U)^*\pi\circ \phi(a)\pi(U)=\tau(a)\rforal a\in A.
\eneq
Define $\psi: A\to M(B)$ by $\psi(a)=U^*\phi(a)U$ for all $a\in A.$ 
It follows that $\tau = \pi \circ \psi$ is liftable.

We now prove the ``if" {{part}} of (iii).
Say that $\lambda =  L \circ \rho_{A,f},$ where $L(1_{T_f(A)})$ is strictly 
{{less than}} $1_{T(B)}$.
Let $p \in M(B){\ppl{\setminus}} B$ be a projection such that
$[p] = L(1_{T_f(A)})$.   So  $p \neq 1_{M(B)}$.  Moreover, by 
our standing assumptions {{on $B,$}} we can choose $p$ so that 
$1 - p \notin B$ also.   By Lemma \ref{lem:Feb22202111:45PM}, we can find a unital embedding
$h : A \rightarrow M(pBp)$ such that $h(A) \cap B = \{ 0 \}$ and (viewing $h$ as a map into $M(B)$),
$$h_{*0} = \lambda.$$
Then, by the UCT, $$\pi \circ h : A \rightarrow C(B)$$
is a non-unital essential extension for which
$$KK(\pi \circ h) = KK(\pi) \circ \lambda = KK(\tau).$$
By Theorem \ref{TH2}, there exists a unitary $U \in M(B)$ 
 such that
\beq
\pi(U)^*\pi\circ h(a)\pi(U)=\tau(a)\rforal a\in A.
\eneq
Define $\psi: A\to M(B)$ by $\psi(a)=U^* h(a)U$ for all $a\in A.$ 
It follows that $\tau = \pi \circ \psi$ is liftable.

For statement (iv), let ${{{\cal Z}}}$ be the Jiang--Su algebra.
By {{Proposition 4.4 of \cite{LinNgZ},}} we can find a projection
$p \in M({{{\cal Z} \otimes {\cal K}}})\setminus \{0\}$ such that
$$\mu(p) = 1 = \mu(1_{\cal Z} \otimes e_{1,1}),$$
where $\mu$ is the unique tracial state of ${{{\cal Z}}}$.
Then, by Propositions 4.2 and 4.4 of \cite{LinNgZ},
$$B =_{df} p({{{\cal Z}\otimes {\cal K}}})p$$
is a simple projectionless continuous scale C*-algebra 
which satisfies the
standing assumptions.   Moreover, $T(B) = T({\cal Z}) =\{ \mu \}$ and
$[1_{C(B)}] = 0 \makebox{ in  } K_0(C(B)){{=\R/\Z}}$.

Since $C(B)$ is simple purely infinite, let $q \in C(B)$
be a {{nonzero}} projection such that $q \neq \pi(p)$, but
$$[q] = 0 = [ \pi(p) ].$$
{{We claim that there is no projection $r\in M(B)$ such that 
$\pi(r)=q.$ In fact, say $r\in M(B)$ such that $\pi(r)=q.$ Then $p-r\not=0.$}}
{{Moreover,}}
$$\mu(p) - \mu(r) \in \mu(K_0(B)) = \mathbb{Z},$$
which is impossible, since $\mu(p) = 1$ and $0 < \mu(r) < 1$. This is a contradiction
which proves the claim.

Now let $\tau : \mathbb{C} \rightarrow C(B)$
be the non-unital essential extension given by
$$\tau(1) =_{df} q.$$
%By the UCT,
{{One computes that}}
$$KK(\tau) = KK(\pi) \circ \lambda$$ 
where $\lambda \in H_{\mathbb{C}, B,\rho}$ is the 
the KK class of the map
$$\mathbb{C} \rightarrow M(B) : 1 \mapsto p.$$
However, since $q$ cannot be lift to a projection in $M(B),$ $\tau$ is not liftable.

\iffalse
Suppose, to the contrary, that $\tau$ is liftable.  Then $q$ is liftable
to a projection, say $r \in M(B) = M(p B p)$ for which  {{$\pi(r)=q.$}} {{Then}} $r \neq p$.
{{Moreover,}}
$$\mu(p) - \mu(r) \in \mu(K_0(B)) = \mathbb{Z},$$
which is impossible, since $\mu(p) = 1$ and $0 < \mu(r) < 1$.
This is a contradiction.  Hence, $\tau$ is not liftable.
\fi

(Before continuing, we note that, in the above example, we can avoid
$[\tau](K_0(\mathbb{C})) \cap \rho_B(K_0(B)) \neq \{ 0 \}$ by doing the
following:  Choose the projection $p \in M(B) \setminus B$ so that
$\mu(p) = \frac{1}{3}$.  In this case, 
for all $\nu \in T(p ({\cal Z} \otimes K)p)$,   $\nu(p) = 1$ and
$\nu(K_0(B)) \in 3 \mathbb{Z}$.  Proceed as in the above argument.)

We now prove the ``if" {{part}}  of (v). 
Let $\tau: A\to M(B)/B$ be a unital monomorphism
such that $KK(\tau)=KK(\pi)\circ \lambda$ for some $\lambda\in H_{A, f}.$
So $$1_{T(B)} - \lambda([1_A]) \in \rho_B(K_0(B)).$$
Since $B \otimes {{\cal K}}$ has strict comparison of positive elements, 
there exists a projection $r \in B$ for which
$$1_{T(B)} - \lambda([1_A]) = \rho_B([r]).$$

Let $p =_{df} 1_{M(B)} - r \in M(B) \setminus B$.
By Lemma \ref{lem:Feb22202111:45PM}, since $A \in \mathcal{A}$, 
let $\phi: A\to pM(B)p$ be a unital embedding
such that (when $\phi$ is viewed as a map into $M(B)$)
$$KK(\phi)=\lambda.$$
Note that $\pi\circ \phi$ is unital (as a map into $C(B)$), 
 and   
$$KK(\pi\circ \phi)=KK(\tau).$$ 
By Theorems \ref{TH2} and \ref{TH1} , there is a unitary $U\in M(B)$ such 
that 
\beq
\pi(U)^*\pi\circ \phi(a)\pi(U)=\tau(a)\rforal a\in A.
\eneq
(Recall that one of our standing assumptions is that $K_1(M(B)) = 0$, and
hence, $\pi_{*1}(K_1(M(B))) = 0$ in Theorem \ref{TH1}.)
Define $\psi: A\to M(B)$ by  $\psi(a):=U^*\phi(a)U$ for all $a\in A.$  
Therefore, $\psi$ is a lift of $\tau.$ \\  
\end{proof}

%{\Green Perhaps, we can omit the rest of section 6?}\\
%{\textbf{We can decide on the above question.  For now, I will 
%revise the rest of the section as if we are keeping it.}}\\  

\begin{NN}\label{DCC0}

{{To illustrate the  special case that $A=C(X),$ 
where $X$ is  a compact metric space, let us consider 
the rank function $d_X: K_0(C(X))\to C(X, \Z).$}}   It follows from \cite{Nob}
% (see the last paragraph of the proof of Theorem 5.7 of \textbf{REFERENCE
%NEEDED!} \cite{FJLX}) 
that $C(X, \Z)$ is a free abelian group. 
Thus one 
has the following splitting short exact sequence:  
\beq
0\to {\rm ker}d_X \to K_0(A)\,{\stackrel{s_0}{\leftrightarrows}}_{d_X}\,C(X, \Z)\to 0.
\eneq

{{Let $\Delta$ be a Choquet simplex.
Denote by ${\rm Hom}(C(X, \Z), \Aff(\Delta))_{++}^\R$ the set of those 
strictly positive \hm s $\lambda: C(X, \Z)\to \Aff(\Delta)$ 
(i.e., if $g\in C(X, \Z)_+\setminus \{0\},$ 
then $\lambda(g)\in \Aff_+(\Delta)\setminus \{0\}$)
such that 
there is a strictly positive affine map $\wtd \lambda: C(X, \R)\to \Aff(\Delta)$
such that $\wtd \lambda|_{C(X, \Z)}=\lambda.$  If we identify $C(X, \R)$ with $\Aff(T(C(X))),$ then $\wtd\lambda: \Aff(T(C(X))) \to \Aff(\Delta)$ 
is a strictly positive affine map}} and $\wtd \lambda \circ d_X=\lambda \circ d_X.$ 
%Therefore $${\rm Hom}(C(X, \Z), 
%\Aff(\Delta))_{++}^\R={\rm Hom}^T(K_0(X), \Aff(\Delta))_{++}.$$

\begin{NN}\label{DCWC}
Let $X$ be a compact metric space with finitely many connected components so that
$X=\sqcup_{i=1}^n Y_i,$ where each $Y_i$ is a connected.
Choose a base point in $y_i\in Y_i,$ $i=1,2,...,n.$ Let $C:=C(X)$ and $C_{0,X}:=\bigoplus_{i=1}^n C_0(Y_i\setminus y_i).$ 
Note that $K_0(C_{0, X})={\rm ker}\rho_C.$ 
Let $e_i\in C$ be the function which has the constant value $1$ 
on $Y_i$ and zero elsewhere, $i=1,2,...,n.$ 
Note that  $1_C=\sum_{i=1}^n e_i.$

%{\ppl{Note that, if $X$ has finitely many, say $n,$ connected components, 
It follows that 
$C(X, \Z)=\Z^n.$ 
% Put $1_i\in \Z$ be the element which is $1$ at $i$-th coordinate and 
%zero elsewhere. 
 If $\lambda: \Z^n\to \Aff(\Delta)$ is a strictly positive \hm, 
then  define $ \lambda_{\R}: \R^n \to \Aff(\Delta)$ by
$\lambda_\R (r_1, r_2,...,r_n)=\sum_{i=1}^n r_i \lambda(e_i).$
It is clear that $\lambda_\R$ is strictly positive and extends $\lambda.$
Write $X=\sqcup_{i=1}^n X_i,$ where $X_i$ is a connected component of $X.$ 
Choose a strictly positive Borel probability measure $\mu_i$ on $X_i.$ Define 
$s:C(X, \R)\to \R^n$ by $s(f)=\sum_{i=1}^n  \lambda(e_i)\int_{X_i}f d\mu_i.$ 
Then $s$ is a strictly positive affine \hm\, from $C(X, \R)$ to $\R^n.$ 
Define $\wtd \lambda: C(X, \R)\to \Aff(T(\Delta))$ by 
$\wtd \lambda(f)=\lambda_{\R} \circ s.$ 
Moreover,  $\wtd \lambda|_{C(X, \Z)}=\lambda.$

{{
Thus, if $X$ has finitely many  connected components, then 
$${\rm Hom}(C(X, \Z), \Aff(\Delta))_{++}={\rm Hom}(C(X, \Z), \Aff(\Delta))_{++}^\R.$$}}

\end{NN}

\end{NN}

%{\blue{We use $Hom$ instead of $Hom_{\pi}$.  
%See (\ref{equ:Feb20Counterexample}).}}\\ 

\begin{cor}\label{LAbelian}
Let $X$ be a compact metric space,  
%, $X_0\subset X$ be the compact subset, 
$C=C(X)$, and $B$ be a {{\CA}} 
%C*-algebra 
satisfying the standing assumptions
of this section.   Then $Ext(C,B)=KK(C, C(B)).$

%and $C_0=C_0(X\setminus X_0)$ as described  in the last part of  \ref{DCC0}.
Suppose that $\tau: C:=C(X)\to C(B)$ is an essential extension. 

(i)  Suppose that $\tau$ is not unital. Then $\tau$ is liftable
 if and only if 
 %$KK(\tau|_{C_0})=0$ and, 
 there is a\\
{{$\lambda\in {\rm Hom}(C(X, \Z), \Aff(QT(B)))_{++}^\R$}} such that 
\beq
KK(\tau)=\lambda\times KK(\pi)\tand 1-\lambda(1_{X_0})\in \Aff_+(QT(B))\setminus \{0\}.
\eneq

%$r_i\in \Aff(T(B))$ ($i=1,2,...,n$)
%such that $1-\sum_{i=1}^n r_i\in \Aff_+(T(B))\setminus \rho_B(B)$ and 
%$[\tau(e_i)]=\bar r_i\in \Aff_+(T(B))_+/\rho_B(K_0(B)),$ $i=1,2,...,n.$ 

(ii) Suppose that $\tau$ is unital.  Then $\tau$ is liftable if and only if 
 %$KK(\tau|_{C_0})=0$ and, 
 there is a \\
{{$\lambda\in {\rm Hom}(C(X_0, \Z), \Aff(QT(B)))^\R_{++}$}} such that 
{{ \beq
KK(\tau)=\lambda\times KK(\pi)\tand \lambda(1_{X_0})=1_{QT(B)}.
\eneq }}  
%$KK(\tau|_{C_0})=0$ and, there are $r_i\in \Aff_+(T(B))$ ($i=1,2,...,n$)
%such that $\sum_{i=1}^n r_i=1$ and 
%$[\tau(e_i)]=\bar r_i\in \Aff_+(T(B))/\rho_B(K_0(B)),$ $i=1,2,...,n.$ 

\iffalse
(i)  Suppose that $\tau$ is not unital. Then $\tau$ is liftable if and only if $KK(\tau|_{C_0})=0$ and, there are $r_i\in \Aff(T(B))$ ($i=1,2,...,n$)
such that $1-\sum_{i=1}^n r_i\in \Aff_+(T(B))\setminus \rho_B(B)$ and 
$[\tau(e_i)]=\bar r_i\in \Aff_+(T(B))_+/\rho_B(K_0(B)),$ $i=1,2,...,n.$ 

(ii) Suppose that $\tau$ is unital.  Then $\tau$ is liftable if and only if $KK(\tau|_{C_0})=0$ and, there are $r_i\in \Aff_+(T(B))$ ($i=1,2,...,n$)
such that $\sum_{i=1}^n r_i=1$ and 
$[\tau(e_i)]=\bar r_i\in \Aff_+(T(B))/\rho_B(K_0(B)),$ $i=1,2,...,n.$ 
\fi

\end{cor}

{{\begin{exm}\label{ExmX}
Example 5.4 in \cite{Linann}   shows that there is a compact subset $X$ of the real line and  a strictly positive \hm\,  
$\lambda: C(X, \Z)\to \Q\subset \R$ (as additive groups) may not be extended to a strictly positive \hm\,
$\wtd \lambda: C(X, \R)\to \R.$  Note in this case $K_0(C(X))=C(X, \Z)$ and $K_1(C(X))=\{0\}.$ 

Let $B={\cal W}$ be the  Razak algebra.
%a separable ${\cal Z}$-simple (stably projectionless) \CA\, with a unique tracial state
%and with $K_1(B)=\{0\}.$ 
Recall that $K_i(B)=\{0\},$ $i=0,1,$
$K_0(M(B))=K_0(M(B)/B)=\R,$ $K_1(M(B))=\{0\}$ and $K_1(M(B)/B)={\rm ker}\rho_B=\{0\}.$
Since $B$ has a unique tracial state, $B$ has continuous scale.
% and $K_0(M(B))=\R.$
Let $\pi: M(B)\to M(B)/B$ be the quotient map.
Let $A=C(X).$
Consider $\pi_{*0}\circ \lambda: K_0(C(X))\to K_0(M(B)/B).$  Since $A$ has the UCT,
we compute that $KK(C(X), M(B)/B)={\rm Hom}(K_0(C(X)), \R).$ 
By Theorem \ref{TH2}, there is $\psi: C(X)\to M(B)/B$ such that $[\psi]=(\pi_{*0}\circ \lambda)\in KK(C(X), M(B)/B).$
Since $\lambda\in {\rm Hom}(K_0(C(X)), K_0(M(B))),$ by UCT, 
as $K_1(M(B))=\{0\}$ and $K_1(C(X))=\{0\},$ 
$KK(C(X),M(B))={\rm Hom}(K_0(C(X)), K_0(M(B)))={\rm Hom}(K_0(C(X)), \R),$  
we have, viewing $\lambda$ as an 
element in $KK(C(X), M(B)),$  $\lambda$ lifts
$[\pi]\circ \lambda.$   However, if there is a \hm\, $\phi: C(X)\to M(B)$ such
that $\pi\circ \phi=\psi.$ Then $\phi_{*0}=\lambda: K_0(C(X))\to \Q\subset \R.$
On the other hand $\phi$ induces an order preserving continuous affine 
map $\phi_T: \Aff(T(C(X)))\to \R$ such that $\phi_T\circ \rho_A=\lambda.$
Recall that $\Aff(T(C(X)))\cong C(X, \R)$ and $\rho_A$ gives the embedding
$C(X, \Z)\subset C(X, \R).$
Thus $\phi_T: C(X, \R)\to \R$ 
extends $\lambda.$ Moreover, $\phi_T$ would be strictly positive as $\pi\circ \phi$ is injective.
However, as mentioned above, by Example 5.4 in \cite{Linann}, 
$\lambda$ cannot be extended to a strictly positive \hm. In other words, 
$\phi$ does not exist.

This example illustrates the  need of additional conditions on $\lambda$ 
in the statement of Theorem \ref{Tlifable}.
% one has additional condition 
%on $\lambda.$

\end{exm}}}

\iffalse
\begin{NN}\label{DCWC}
Let $X$ be a compact metric space with finitely many connected components so that
$X=\sqcup_{i=1}^n Y_i,$ where each $Y_i$ is a connected.
Choose a base point in $y_i\in Y_i,$ $i=1,2,...,n.$ Let $C:=C(X)$ and $C_{0,X}:=\bigoplus_{i=1}^n C_0(Y_i\setminus y_i).$ 
Note that $K_0(C_{0, X})={\rm ker}\rho_C.$ 
Let $e_i\in C$ be the function which has the constant value $1$ 
on $Y_i$ and zero elsewhere, $i=1,2,...,n.$ 
Note that  $1_C=\sum_{i=1}^n e_i.$ 
\end{NN}
\fi

\begin{cor}\label{CCWCC}
Let $X$ be a compact metric space as in \ref{DCWC}, $B$ be a {{\CA\,}}
%C*-algebra
satisfying the standing assumptions of this section,  and 
$\tau: C(X)\to M(B)/B$ 
be an essential extension.

(i)  Suppose that $\tau$ is not unital. Then $\tau$ is liftable if and only if $KK(\tau|_{C_{0,X}})=0$ and there are $r_i\in \Aff_+(T(B))$ ($i=1,2,...,n$)
such that $1-\sum_{i=1}^n r_i\in \Aff_+(QT(B))\setminus \{ 0 \}$ and 
$[\tau(e_i)]=\bar r_i\in \Aff_+(QT(B))/\rho_B(K_0(B)),$ $i=1,2,...,n.$ 

(ii) Suppose that $\tau$ is unital.  Then $\tau$ is liftable if and only if $KK(\tau|_{C_{0,X}})=0$ and there are $r_i\in \Aff_+(QT(B))$ ($i=1,2,...,n$)
such that $\sum_{i=1}^n r_i=1$ and 
$[\tau(e_i)]=\bar r_i\in \Aff_+(T(B))/\rho_B(K_0(B)),$ $i=1,2,...,n.$

\end{cor}

\begin{df}\label{DKKordereq}
Let $A\in {{{\cal A}_2}}$ be a \CA\, with compact $T(A).$ (Compactness
of $T(A)$ may be replaced by the weaker condition that
 $0\not\in {\overline{T(A)}^w}.$) 
One has the following short exact sequence: 
\beq
0\to {\rm ker}\rho_A\to K_0(A)\to \rho_A(K_0(A))\to 0. 
\eneq

We say that $A$ is \emph{ordered $KK$-equivalent} to $C(X)$ 
if there an order isomorphism from $\rho_A(K_0(A))$ onto $C(X, \Z)$ and an isomorphism from ${\rm ker}\rho_A$ onto  
${\rm ker}\rho_{d_X}$ and an isomorphism from $K_1(A)$ onto $K_1(C(X)).$ 
Denote by $j_A\in KK(A, C(X))$ a $KK$-equivalence given by the above mentioned isomorphism (from the UCT). 

Note that in this case, $\rho_A(K_0(A))$ is a free abelian group. 
\end{df}

\begin{thm}
Let $A\in {{{\cal A}_2}}$ be a separable amenable \CA\, which is ordered $KK$-equivalent to $C(X)$
for some compact metric space $X$, and let $B$ be a C*-algebra
satisfying the standing assumptions of this section.   Let $\tau: A\to M(B)/B$ be an essential extension.

(i)  If $\tau$ is non-unital, then $\tau$ is liftable if and only if  there is 
a \\
{{$\lambda\in {\rm Hom}(C(X, \Z), \Aff(QT(B)))^\R_{++}$ }} such that 
$KK(\tau)=j_A\times \lambda\times KK(\pi)$ and $1-\lambda(1_{X_0})\in \Aff_+(T(B))\setminus \{0\}.$

(ii) If $\tau$ is unital, then $\tau$ is liftable if and only {{if there is
a $\lambda\in {\rm Hom}(C(X, \Z), \Aff(QT(B)))^\R_{++}$ }} such that 
$KK(\tau)=j_A\times \lambda\times KK(\pi)$ and $\lambda(1_{X_0})=1.$
%\in \Aff_+(T(B))\setminus \{0\}.$

\end{thm}

\begin{rem}
If $A$ is properly infinite, then no essential extension
$\tau: A\to M(B)/B$ is liftable.
\end{rem}

\section{Null extensions, approximately trivial extensions and quasidiagonal extensions}

 {{Most}} results in this section are proven elsewhere.  We provide a
summary here for completeness and for the convenience of the reader
since we will be using some of these results in the present paper. We firstly
provide a little history, with some terminology which will be defined
precisely later.
 For more
complete details and history, we refer the reader to the paper 
\cite{LinExtQuasidiagonal}.

Let $T \in B(l_2)$ be an essentially normal operator, and 
let 
$\phi_T : C(X) \rightarrow B(l_2)/K$ be the corresponding essential extension,
where $X$ is the essential spectrum of $T$. 
A result of Brown--Douglas--Fillmore (\cite{BDF1})
% (BDF, \cite{BDF1}) 
states
that 
$T$ is quasidiagonal (i.e., $\phi_T$ is a quasidiagonal extension)
 if and only if $T$ is 
normal plus compact (i.e., $\phi_T$ is a trivial or liftable 
extension) if and only if $KK(\phi_T) = K_1(\phi_T) = 0$.
Much work has been done by multiple authors to generalize this result, {{let us}}
just mention the following:
Let $A$ be a separable nuclear C*-algebra that satisfies the
UCT, let $B$ be
 a separable stable {{\CA}}\,
% C*-algebra, 
and suppose that $A$ is quasidiagonal relative to $B$.  
Generalizing results of BDF, Brown, Salinas and others, Schochet proved that
an absorbing extension $\phi : A \rightarrow C(B)$ is quasidiagonal if and only
if $\phi$ is approximately trivial if and only if
$[\phi]=0$ in $KL(A, C(B))$, and such extensions are classified by
$Pext(K_*(A), K_*(B))$ (see \cite{SchochetKK1} and \cite{SchochetKK2}; see
also \cite{LinExtQuasidiagonal} Theorem 6.5). 
The restriction, that $\phi$ be absorbing, can be removed when the
canonical ideal $B \cong K$ or $B$ is simple purely 
infinite (e.g., see \cite{LinExtQuasidiagonal} Corollary 6.7); among other
things, these are exactly the separable stable {{\CA s}}
%C*-algebras 
which have
a Voiculescu theorem:  roughly speaking, every essential extension is
absorbing (\cite{VoiculescuWvN}, \cite{ElliottKucerovsky}).

For more general canonical ideals, 
%and for more
%general extensions (
{{when essential extensions are not necessarily absorbing,}} 
the classes of quasidiagonal extensions, approximately trivial extensions
and null extensions can be different from each other.  Examples and more details
 can be 
found in \cite{LinExtQuasidiagonal}.  Below, for the convenience of
the reader, 
 we briefly summarize some results
characterizing quasidiagonality and  approximately triviality,  
and mention some examples from \cite{LinExtQuasidiagonal}:

We begin by recalling some definitions.\\

\begin{df}
Let $A$ be a separable C*-algebra and $B$ a non-unital {{but}} $\sigma$-unital  C*-algebra.  
Let  $\tau : A \rightarrow C(B)$ be an essential extension.
\begin{enumerate}
\item The extension $\tau$ is said to be \emph{quasidiagonal}     
if 
%$\pi^{-1}(\phi(A))$ is a collection of simultaneously
%quasidiagonal operators in 
%$M(B)$, i.e., 
there exists an approximate unit $\{ e_n \}$ of $B$, consisting 
of an increasing sequence of projections, such that 
$$\| e_n x - x e_n \| \rightarrow 0 \makebox{ for all  } 
x \in \pi^{-1}(\phi(A)).$$
(Recall that $\pi : M(B) \rightarrow C(B)$ is the quotient map.)   
\item The extension $\tau : A \rightarrow C(B)$ is said to be
\emph{approximately trivial} if there exists a sequence 
$\{ \tau_n : A \rightarrow C(B) \}$ of trivial (or liftable)
 extensions such that for all $a \in A$,
$\tau(a) = \lim_{n \rightarrow \infty} \tau_n(a)$. 
\item The extension $\tau$ is said to be a \emph{null extension} if $[\tau] = 0$ in
$KK(A, C(B))$.
\end{enumerate}
\end{df}

{{In the characterization of quasidiagonality {{below,}} the ``UCT strong NF"   
assumption on the C*-algebra $A$ can be replaced by the assumption
that $A$ can be realized as an inductive limit $A = 
\overline{\bigcup_{n=1}^{\infty} A_n}$, where for all $n$, $A_n$ is a separable
amenable residually finite dimensional C*-algebra which satisfies the UCT. 
(See \cite{LinExtQuasidiagonal} Theorem 7.11 and the comment
before it; see also 6.16 in \cite{BlackadarKirchberg}.)}}

\begin{thm}[\cite{LinExtQuasidiagonal}  Theorem 7.11]\label{Tqdextensions}
Let $A$ be a separable strong NF- \CA\, which satisfies the UCT.
Let $B$ be a non-unital but $\sigma$-unital simple \CA\, 
with real rank zero, stable rank one, weakly unperforated 
$K_0(B)$ and has continuous scale.
Suppose that $\tau: A\to M(B)/B$ is an essential extension.
Then $\tau$ is quasidiagonal if and only if $\tau_{*1}=0,$
${\rm im}\tau_{*0}\subset \Aff(QT(B))/\rho_B(K_0(B))$ and 
$KL(\tau)|_{K_i(A, \Z/k\Z)}=0,$ $i=0,1$ and $k=2,3,....$ 
\end{thm}

\begin{df}
Let $A$ and $B$ be {{\CA s.}}
%C*-algebras.

Let $Pl(K_0(A), \Aff(QT(B))/\rho_B(K_0(B)))$ denote the set of 
$\alpha \in Hom(K_0(A), \Aff(QT(B))/\rho_B(K_0(B)))$ for which there
exists a strictly positive homomorphism $\beta \in Hom(K_0(A), \Aff(QT(B)))$  with $\beta([1_{\wtd A}])=1$ 
or $\bt([1_{\wtd A})<1$ such that 
$$\Phi \circ \beta = \alpha.$$ 
Here, $\Phi :  \Aff(T(B)) \rightarrow \Aff(QT(B))/\rho_B(K_0(B))$ is the usual quotient map.

Let $Apl(K_0(A), \Aff(QT(B))/\rho_B(K_0(B)))$ denote the set of 
$\alpha \in Hom(K_0(A), \Aff(QT(B))/\rho_B(K_0(B)))$ for which there exist 
an increasing sequence $\{ G_n \}$ of finitely generated subgroups of $K_0(A)$
and a sequence $\{ \alpha_n \}$ in $Pl(K_0(A), \Aff(QT(B))/\rho_B(K_0(B)))$ such that
$$\bigcup_{n=1}^{\infty} G_n = K_0(A)$$
and
$$\alpha_n |_{G_n} = \alpha|_{G_n}  \makebox{  for all  }n.$$ 
\end{df}
 
Note that if $\alpha \in Apl(K_0(A), \Aff(QT(B))/\rho_B(K_0(B)))$ then 
$\alpha |_{ker \rho_A} = 0$.  Here is the short argument:  Suppose that 
 $x \in ker \rho_A$.  Then there
exists an $m$ for which $x \in G_m$.  But $\alpha_m |_{ker \rho_A} = 0$.
Hence, $\alpha(x) = \alpha_m(x) = 0$.

%\vspace*{2ex}
%\textbf{I reversed the order of the first sentence in the next theorem, and
%then chopped it into two sentences.}\\

\begin{thm}({{\cite{LinExtQuasidiagonal}}}  Theorem 8.9)\label{TApproxtrivial}
Let $A$ be a separable amenable \CA\, satisfying the UCT which can be embedded
in a unital AF-algebra {{$C$ }} such that
{{$K_0(A)/{\rm ker}\rho_A=K_0(C)/{\rm ker}\rho_C$  }}
{{(as ordered groups). }}
Let  $B$ be a non-unital but $\sigma$-unital   simple \CA\, with {{real rank zero,
stable rank one, weakly unperforated $K_0(B)$ and has continuous scale.}} 
Suppose that $\tau: A\to M(B)/B$ {{  is  }} an essential extension.
Then $\tau$ is approximately trivial if and only if $\tau_{*1}=0,$
$KL(\tau)|_{K_i(A, \Z/k\Z)}=0,$ $i=0,1,$ $k=1, 2,...,$ and   
$\tau_{*0}\in Apl(K_0(A), \Aff(QT(B))/\rho_B(K_0(B))).$
\end{thm}

%{\ppl{The same proof works---need to write the modification.}}\\  
%\textbf{You erased the above comment, and I put it back in purple. I rewrote the above
%theorem so that it is exactly like the one from your paper \cite{LinExtQuasidiagonal}
%Theorem 8.9.  Do we really need a separate result?}\\

For more examples like the one below, please see \cite{LinExtQuasidiagonal}.

%%%%%%%%%%%%%%
\iffalse
\textbf{The original first part in the next Proposition is NOT true. E.g.,
if $A = C_0(0, 1]$, then every extension $\tau : A \rightarrow C(B)$ is trivial.  So
all null extensions would be trivial.}
\textbf{We even have a counterexample in the unital case.  Let $C$ be any UHF algebra,
and let $B \subset C$ be any nonunital hereditary C*-subalgebra.
Then every extension $\tau : \mathbb{C} \rightarrow C(B)$ lifts, since 
every projection in $C(B)$ lifts (as consequence of $M(B)$ having real rank zero).
Thus, all null extensions are trivial.}\\

\textbf{I also was not able to prove the 2nd part without additionally assuming that $A$
is unital, simple and Jiang--Su stable.}\\
%
%%%%%%%%%%%%%%%%%%%%%%%
\textbf{
I have thus majorly modified the first and second parts of the next Proposition. }\\ 
%%%%%
\fi
%%%%%%%%%

\begin{prop}\label{Pnullext}

Let $A$ be a separable amenable \CA\, which satisfies the UCT and let $B$ be a 
%$\sigma$-unital 
simple 
\CA\, with continuous scale {{and which satisfies the standing assumptions  }} 
in \ref{58}.

If $\tau_0 : A \rightarrow C(B)$ is a 
null essential extension then 
for any non-unital  essential extension $\tau: A\to M(B)/B,$
\beq
[\tau+\tau_0]=[\tau] \makebox{  in}\,\,\,   \Ext(A,B),  
\eneq
where the above $+$ is the BDF sum.

Suppose, in addition,   {{that $A\in {\cal A}_1$
%unital
%, simple, $\mathcal{Z}$-stable,    
with finitely generated $K_i(A)$}}   and that 
there is a unital simple AF-algebra $C$ such that $K_0(C)=\rho_B(K_0(B))\subset K_0(M(B))$ as ordered
groups and {{$[1_C]\le [1_{T(B)}].$}} 
 Denote by 
\beq
{{{\rm Hom}^T(K_0(A), \rho_B(K_0(B)))_{++}}}
\eneq
{{the set of those $\gamma\in {\rm Hom}(K_0(A), \rho_B(K_0(B)))_+$ such that
there exists 
$$L\in CA_+^{\bf 1}({{\Aff^b}}(T_f(A)), \Aff(QT(B)))_{++}$$ for which 
$\gamma=L\circ \rho_{A, f}.$}}

{{(1)  Then,  for any $\gamma
\in {\rm Hom}^T(K_0(A), \rho_B(K_0(B)))_{++},$}} 
%CA_+(\rho_{A,f}(K_0(A)), \rho_B(K_0(B)))_{++}$
%H_{A, C, \rho},$
 %{\rm Hom}(K_0(A), \rho_B(K_0(B)))_{++},$}} 
 %such that 
% $\gamma=L\circ \rho_{A,f},$ where\\
  %$L\in CA_+^{\bf 1}(\Aff(T_f(A)), \Aff(T(B)))_{++},$}}
there exists a null essential extension $\tau_0 : A \rightarrow C(B)$ which is trivial and 
$(\tau_0)_{*0}=\pi_{*0}\circ \gamma,$
where  $\pi: M(B)\to M(B)/B$ is the quotient map.  Consequently, 
(a)  if ${\rm Hom}^T(K_0(A), \rho_B(K_0(B)))_{++}\not=\emptyset$  and
 $A$ is not unital, then every null extension is liftable, 
 (b) if $A$ is unital and 
${\rm Hom}^T(K_0(A), \rho_B(K_0(B)))_{++}$ 
contains $\gamma$  such that $\gamma ([1_{A}]) \neq [1_{M(B)}]$, then 
every non-unital null extension is liftable, 
%{\ppl{(b) such that $\gamma([1_{A}])=[1_{M(A)}],$
%then every unital null extension is liftable.}}

(2) 
If {{${\rm Hom}^T(K_0(A), \rho_B(K_0(B)))_{++}=\emptyset,$}} then no null extension
$A \rightarrow M(B)/B$ is trivial. 

{{(3) If $A$ is unital and $T_f(A)\not=\emptyset,$ 
and $QT(B)$ is separable, then there are always some liftable (i.e., trivial) 
$\tau\in  Ext(A, B)$ such that $\tau$ is not null.}}

\end{prop}

\begin{proof}
%[Sketch of the proof:]  

  For the first part of the proposition,  since $\tau_0$ is a null extension, 
we have that in $KK(A, C(B))$,
$[\tau + \tau_0]_{KK} = [\tau]_{KK} + [\tau_0]_{KK} = [\tau]_{KK}.$
Hence, by Theorem \ref{TH2},
{{$[\tau + \tau_0]= [\tau]$ in ${\rm Ext}(A,B).$}}

 Now we prove the second part.  
By Theorem \ref{TembeddingAH}, $C \in {\cal A}_1 \subset {\cal A}_2$. 
 Hence, 
  there is an embedding $\psi_C: C\to M(B)$
such that $(\psi_C)_{*0}={\id_{\rho_B(K_0(B))}},$  $\psi_C(C)\cap B=\{0\}.$
Moreover, if $[1_C]=[1_{M(B)}],$ we may assume that $\psi_C(1_C)=1_{M(B)}$ and 
if $[1_C]<[1_{M(B)}],$ we may assume that $1_{M(B)}-\psi_C(1_C)\not\in B.$
%for some non-zero projection $e\in C.$  
Since $A\in {\cal A}_1,$ 
%s unital, simple, amenable,
%$\mathcal{Z}$-stable and satisfies the UCT, by Lemma \ref{Hom1},   given 
%{\blue{$\gamma\in {\rm Hom}(K_0(A), \rho_B(K_0(B)))_{++},$ }}
 there is an embedding  $\psi_A:A\to C$
such that ${\psi_A}_{*0}=\gamma.$ Consider the extension $\tau_0:=\pi\circ \psi_C\circ \psi_A: A\to M(B)/B.$
Then $\tau_0$ is trivial.  To see that 
 $\tau_0$ is null, one computes that  $[\tau_0]|_{K_0(A)}=0,$ 
 $[\tau_0]|_{K_1(A)}=0,$   and $[\tau_0]|_{K_i(A, \Z/k\Z)}=0,$  $i=0,1$ and $k=2,3,....$  (See standing assumptions on $B$ in \ref{58}.)  
 {{Since we assume that  $K_i(A)$ is finitely generated, $[\tau_0]=0$ in $KK(A, C(B)).$}} 
So $\tau_0$ is null. 

For the converse, for case (a), choose $\gamma\in {\rm Hom}^T(K_0(A), \rho_B(K_0(B)))_{++}.$
Then $\pi\circ \psi_C\circ \psi_A$ is liftable and $[\pi\circ\psi_C\circ \psi_A]=0.$
If $\tau$ is null, then $[\tau]=[\pi\circ\psi_C\circ \psi_A]$ in $KK(A, C(B)).$
Since $A$ is not unital, $[\tau]=[\pi\circ\psi_C\circ \psi_A]$ in ${\rm Ext}(A, B).$
Therefore $\tau$ is liftable.

For case (b),        
suppose that $\gamma$ (as before) satisfies that
$\gamma([1_A]) \neq [1_{M(B)}].$ 
{{Then $\psi_C\circ \psi_A$ is non-unital embedding of $A$ into $M(B)$ and 
since $\psi_C\circ \psi_A(A)\subset \psi_C(C),$ if $\psi_A(1_A)\not=1_C,$
then $\psi_C(1_C)-\psi_C\circ \Psi_A(1_A)\not\in B.$ Consequently, 
$1_{M(B)}-\psi_C\circ \psi_A(1_A)\not\in B.$ It follows that $\pi\circ \psi_C\circ \psi_A$ is not unital.
If $\psi_A(1_A)=1_C,$ then, since $\gamma([1_A])\not=[1_{M(B)}]$ and $\psi_{A_{*0}}=\gamma,$
we conclude that $[1_C]\not=[1_{M(B)}].$ As we constructed above, this implies that $1_{M(B)}-\psi_C(1_C)\not\in B.$ It follows that $1_{M(B)}-\psi_C\circ \psi_A(1_A)\not\in B.$ Hence $\pi\circ \psi_C\circ \psi_A$ 
is not unital. 

Let $\tau$ be a non-unital null extension. Then 
$[\pi\circ \psi_C\circ \psi_A]=[\tau].$   It follows that $[\tau]=[\pi\circ \psi_C\circ \psi_A]$ 
in ${\rm Ext}(A,B).$ So $\tau$ is liftable.}}
%%%%%%%%%%%%%%%%%%%%%%
\iffalse
%
and thus $\tau_0$ can be chosen
to be  non-unital (by the standing assumptions on $B$ in \ref{58}, 
there will be a projection
$P \in M(B)$ with $[P] = \gamma([1_A])$, and $P, 1- P \notin B$; take the
codomain of $\psi_C$ to be $P M(B) P = M(P B P)$, noting that $PBP$ also 
satisfies the standing assumptions \ref{58}), and
suppose that $\tau$ is a null extension, i.e., $[\tau] = 0$.  
 Hence $[\tau]=[\tau_0].$ Hence, by Theorem \ref{TH2}, $\tau$ is
unitarily equivalent to $\tau_0$ and thus liftable.  
Since $\tau$ was arbitrary, every non-unital null extension is liftable.
\fi
%%%%%%%%%%%%%% 

{{Now suppose that ${\rm Hom}^T(K_0(A), \rho_B(B))_{++}=\emptyset$ and
$\tau: A\to C(B)$ is a null extension.}} If there exists $\phi: A\to M(B)$ such that
$\tau=\pi\circ \phi,$ then since $[\tau]=0,$ by our standing assumptions on 
$B$ (see IV. in  \ref{58}), 
$\phi_{*0}\in {\rm Hom}^T(K_0(A), \rho_B(K_0(B)))_{++}.$
But ${\rm Hom}^T(K_0(A), \rho_B(K_0(B)))_{++}=\emptyset.$ Hence no null extension is liftable.

{{For (3), fix $g\in \Aff_+(QT(B))\setminus \rho_B(K_0(B))$ with $g<1.$ 
%Since $B$ is assumed to be separable, $QT(B)$ is a separable Choquet simplex. 
Consider a unital simple AF-algebra 
$C$ such that $T(C)=QT(B),$ $\rho_C$ is injective and $g\in \rho_C(K_0(C)_+).$  
(We can find such a $C$ by taking a countable dense ordered subgroup 
$H \subset \Aff(T(B))$ with $1, g \in H$, 
and choose a simple unital AF-algebra $C$ so that
$(K_0(C), K_0(C)_+, [1_C]) = (H, H_+, 1)$.) 
Denote by $\iota: K_0(C)\to \Aff_+(T(B))$ the  embedding. 
By  (the proof of) Theorem \ref{PA1A2},
there is a unital  embedding $h: C\to M(B)$ such that $h_{*0}=\iota.$}}

{{Fix $t_0\in T_f(A).$ 
%Define $\rho_0: K_0(A)\to  \R$ by 
%$\rho_0(x)=\rho_A(x)(t_0).$ 
%and $\lambda_0':\R\to \Aff_+(T(B))$ by 
%$\lambda_0'(r)=rg.$ 
Define $\lambda_0: {{\Aff^b}}_+(T_f(A))\to \Aff_+(T(C))$ by
$\lambda_0(f)(s)=f(t_0)g(s)$ for all 
%$x\in K_0(A),$ 
$f\in {{\Aff^b}}_+(T_f(A))$ and $s\in T(C).$ 
Put $\af: K_0(A)\to \Aff_+(T(C))$ by $\af(x)(s)=\lambda_0\circ \rho_{A,f}(x)(s)$
for all $x\in K_0(A)$ and $s\in T(C).$}}
Since $A \in {\cal A}_1$, there is a monomorphism 
$h_A: A\to C$ such that $(h_A)_{*0}=\af.$ In particular $(h_A)_{*0}([1_A])=g.$
Put $\tau=\pi\circ \iota \circ h_A.$ Then $\tau$ is liftable but 
$[\tau(1_A)]=\pi_{*0}(g)\not=\{0\}.$ So $\tau$ is not null.

\end{proof}

\begin{exm}\label{Exnull}
%\paragraphbreak
%{}\linebreak
Some comments are in order:

{{
(1) We first note that if $K_0(A)=\{0\}$ (such as $A=C_0((0,1])$), then $\rho_{A,f}=0.$
Then $0\in {\rm Hom}^T(K_0(A), \rho_B(K_0(B)))_{++}.$

(2) If $\rho_{A,f}\not=0,$ then $0\not\in {\rm Hom}^T(K_0(A), \rho_B(B))_{++}.$}}

(3) Let $Q$ be the UHF-algebra with $(K_0(Q), K_0(Q)_+, [1_Q])=(\Q, \Q_+,1)$ and
let $\tau$ be the unique tracial state of $Q$.

Let    
$a\in Q$ be a non-zero positive element with 
$d_\tau(a)=3/4$ which is not a projection. Let $B=\overline{aQa}$ and let $A$ be an irrational rotation algebra.  Note that $B$ satisfies the standing assumptions
of \ref{58} (see Section \ref{KComputations}, especially Theorem \ref{thm:K1M(B)}
and the statements that follow it). 
It follows from  Theorem \ref{TembeddingAH}   that there is a unital embedding 
$\phi_0: A\to M(B)$.  Necessarily, $(\phi_0)_{*0}([1_A])=1$ in $K_0(M(B)) = \Aff(T(B)) =
\mathbb{R}$.  Note that
$(\phi_0)_*([1_A])$ corresponds to the number $3/4$ in $K_0(Q)$, and
$K_0(B) = K_0(Q) = \Q$.   Hence, since $A$ is an irrational rotation algebra, 
$(\phi_0)_{*0}(x)\not\in \rho_B(K_0(B)) = \Q$ if $\rho_A(x)\not\in \Q$ for all
$x \in K_0(A)$. 
{{Thus,}}  since $A$ is an irrational rotation algebra, we can find $y \in K_0(A)_+$
such that $(\phi_0)_{*0}(y)$ is incommensurable with $(\phi_0)_*([1_A])$ (i.e., their ratio 
is irrational), and thus
must be an irrational number.  {{Therefore,}}
 $(\phi_0)_{*0} \notin {\rm Hom}^T(K_0(A), \rho_B(K_0(B)))_{++}.$  Hence,   
by Proposition \ref{Pnullext},
$\phi_0$ is an example of a trivial (or liftable) extension which is not a null extension. 

{{(4) In fact, ${\rm Hom}^T(K_0(A), \rho_B(K_0(B)))_{++} = \emptyset$.}}  {{To see this,  
assume,}}  to the contrary, that there exists 
 $\lambda \in {\rm Hom}^T(K_0(A), \rho_B(K_0(B)))_{++}.$ 
Since $A \in {\cal A}_1$, we can find an embedding $\phi : A \rightarrow M(B)$ so that
$(\phi)_{*0} = \lambda$. 
Since $\lambda \in {\rm Hom}^T(K_0(A), \rho_B(K_0(B)))_{++}$, we must have that
$0 \neq (\phi)_{*0}([1_A]) \in \mathbb{Q} = \rho_B(K_0(B))$.  
Since $A$ is an irrational rotation algebra, we can find $x \in K_0(A)_+$ such that
$(\phi)_{*0}(x)$ is incommensurable with $(\phi)_{*0}([1_A])$ (i.e., their ratio is an 
irrational number), and hence, $(\phi)_{*0}(x)$ is an irrational number.  
This contradicts that 
$(\phi)_{*0} = \lambda \in {\rm Hom}^T(K_0(A), \rho_B(K_0(B)))_{++}$. 
Thus, by Proposition \ref{Pnullext}, no null extension $\tau: A\to C(B)$ is trivial (i.e., liftable).

{{(5)  Suppose that $X$ is a compact metric  space with finitely many connected components
and $B$ be a $\sigma$-unital simple \CA\, with continuous scale and 
which satisfies the assumption in \ref{58}.
Recall that $C(X, \Z)\cong \Z^k$ for some $k\in \N.$ It is easy to find 
strictly positive \hm\,}} $\lambda\in {\rm Hom}(C(X, \Z), \rho_B(K_0(B)))_{++}$ with $\lambda([1_{C(X)}]) \neq [1_{M(B)}]$. 
Hence by (1) of Theorem \ref{Pnullext}, in this case, every non-unital null extension is liftable. 
\end{exm}

%\vspace*{3ex}
%%\textbf{I have removed the last sentence in this Section.}\\
%Consider first the case that $A$ is a unital separable simple amenable ${\cal Z}$-stable 
%\CA\, which satisfies the UCT. 

%Recall that $B$ has continuous scale. 
%Let $\{e_n\}$ be an approximate identity of $B$ consisting of projections.
%Let $p_1=e_1$ and $p_{n+1}=e_{n+1}-e_n,$ $n=1,2,....$ 

\section{Ordered groups}

Throughout  this section, 
 in $\Z^l,$ we write $e_j=(\overbrace{0,0,...0}^{j-1}, 1, 0,...0)\in \Z^l,$ $j=1,2,...,l.$

\begin{lem}\label{700}
Let $G\subset \Z^l$ (for some integer $l\ge 1$) be an ordered subgroup
with $g_1, g_2,...,g_m\in G_+\setminus \{0\}$ generating $G_+$ 
%$\lambda: G\to \R$ 
and   $1>\sigma>0.$  Let $\Delta$ be a compact metrizable Choquet simplex, let   
%
%Let $G\subset Z^l$ be an ordered subgroup, $\Delta$ be  as in Lemma \ref{Gext-5} with
$u_g=\sum_{j=1}^lk_je_j\in G$ for some $k_j\ge 1$ ($1\le j\le l$)
and let $\ep>0,$  $1>r>7/8.$ Let $K \geq 1$ be an arbitrary integer. 
%and integer $K>0.$ 
% ($g_1,g_2,...,g_m\in G_+\setminus \{0\}$
%generate $G_+.$)

Suppose that  $\D\subset \Aff(\Delta)$ is a dense ordered subgroup (with {{the}} strict order) and 
$\lambda_G: G\to \Aff(\Delta)$ is a strictly positive \hm. 
%with $\lambda_G(u_g)=1_\Delta.$
Let $\iota: \Z^l\to \Z^{l_1}$ (for some $l_1\ge l$) be a strictly positive \hm\, with $\iota(u_g)=
\sum_{k=1}^{l_1}m_ke_k',$ where $e_k'=(\overbrace{0,0,...,0,}^{k-1}1,0,...0)\in \Z^{l_1},$ 
$m_j\ge 1$ $(1 \leq j\leq l_1)$, and 
let $\bt: \Z^{l_1}\to \Aff(\Delta)$ be a strictly positive \hm\, with $\|\bt(\iota(u_g))\|\le 1.$
%\, with $\bt(\sum_{i=1}^le_j')=\bt\circ\iota(u_g)=1_\Delta.$
Suppose that 
\beq\label{700-1}
\Pi\circ \bt\circ \iota(g)=\Pi\circ \lambda_G(g)\tforal g\in G,
\eneq
where $\Pi: \Aff(\Delta)\to \Aff(\Delta)/\D$ is the quotient map. 
Then there are strictly positive \hm s $\td\lambda_\D: \Z^l\to K\D,$
$\bar\bt: \Z^{l_1}\to \Aff(\Delta),$ and  
% is a strictly positive \hm\, $\td \lambda: \Z^l\to \Aff(\Delta)$ 
  a positive \hm\, $\af: G\to \D$ 
such that 
\beq\label{92-nn-1}
&&\td\lambda_\D|_G+\bar \bt\circ \iota|_G+\af=\lambda_G,\\\label{92-nn-2}
&&\Pi\circ \bar \bt(x)=\Pi\circ \bt(x)\tforal x\in \Z^{l_1},\\\label{92-nn-3}
&&\|\bar\bt(\iota(u_g))\|<\ep,\,\,\,
\|\bar\bt(\iota(g_j))\|<\ep, \,\,\, \|\af(u_g)\|<\ep,\andeqn \|\af(g_j)\|<\ep, \\\label{700-T-10}
&&{\td\lambda_\D(g_i)(\tau)\over{\td\lambda_\D(u_g)(\tau)}}\ge r{\lambda_G(g_i)(\tau)\over{\lambda_G(u_g)(\tau)}}\tforal \tau\in \Delta,
\,\,1\le i\le m,\tand\\\label{700-T-11}
&&  {\bar\bt(e_k')(\tau)\over{\bar\bt(\iota(u_g))(\tau)}}\ge r{\bt(e_k')(\tau)\over{\bt(\iota(u_g))(\tau)}}\tforal \tau\in \Delta,\,\,\,1\le k\le l_1.
\eneq
%$\td\lambda|_G+\af=\lambda_G$ and $\Pi\circ \td\lambda(x)=\Pi\circ \bt.$

\end{lem}

{{To prove the lemma, let us state the following fact as a lemma using the notation 
stated above. Note in ${\rm Hom}(\Z^l, \Aff(\Delta)),$ 
$\dt_n\to \dt$ in the pointwise-norm topology means $\lim_{n\to\infty}\|\dt_n(g)-\dt(g)\|=0$
for all $g\in \Z^l,$ where the norm is the norm on the real Banach space $\Aff(\Delta).$ }}

{{A \hm\, $\dt: \Z^l\to \Aff(\Delta)$ is strictly positive, if
$\dt(z)\in \Aff_+(\Delta)\setminus \{0\}$ (i.e., $\dt(z)(t)>0$ for all $t\in \Delta$)  for all 
$z\in \Z^l_+\setminus \{0\}.$  Denote by
${\rm Hom}(\Z^l, \Aff(\Delta))_{++},$  the  set of all strictly positive \hm s from 
$\Z^l$ to $\Aff(\Delta).$}} {{Note that, here $0$ is not a point of $\Delta$.}}

\begin{lem} \label{lem:First}  
Let ${{\rm Hom}}(\mathbb{Z}^l, {{\Aff}}(\Delta))$  {{(for some $l\ge 1$)}} be the group of homomorphisms from 
$\mathbb{Z}^l$ to ${{\Aff}}(\Delta)$, given the pointwise-norm topology.  
{{Fix $x, y\in\Z^l_+$
% {\rm Hom}(\Z^l, \Aff(\Delta))_{++}$ 
with $y\not=0.$}} 
%(i.e., $y(t)>0$ for all $t\in \Delta$),}}

For each $1 \leq j \leq n$, the map
$${{\rm Hom}}(\mathbb{Z}^l, {{\Aff}}(\Delta))_{++} \rightarrow C(\Delta) : 
{{\dt}} \mapsto \frac{\dt(x)}{\dt(y)}$$
is continuous. 

In the above, $\frac{\delta(x)}{\delta(y)}$ is the function on $\Delta$
defined by $\left(\frac{\dt(x)}{\delta(y)}\right)(\tau)
:= \frac{\dt(x)(\tau)}{\dt(y)(\tau)}$ for all $\tau \in \Delta$.

{{Moreover, let $g_1, g_2,...,g_m, u_g\in \Z^l_{++}\setminus \{0\}$ 
be such that $u_g$ is in the subgroup generated by $g_1,g_2,...,g_m.$ 
%$G,$ $g_1,g_2,...,g_m, u_g\in G$ be as described in Lemma \ref{700}, 
Then for any $\dt\in {\rm Hom}(\Z^l, \Aff(\Delta))_{++}\setminus \{0\},$ 
for any $\eta>0,$ there exists $\ep_0>0$ satisfying the following:
if $\dt'\in {\rm Hom}(\Z^l, \Aff(\Delta))_{++}\setminus \{0\}$ and $\|\dt(g_j)-\dt'(g_j)\|<\ep_0,
$ $j=1,2,...,m,$ then 
\beq
\|{\dt(g_j)\over{\dt(u_g)}}-{\dt'(g_j)\over{\dt'(u_g)}}\|<\eta, \,\,\, 1\le j\le l.
\eneq}}

%
%In the above, the conclusion still holds when $\mathbb{Z}^l$ is replaced with
%$\mathbb{Z}^{l_1}$,  $\{ g_j \}$ is replaced with $\{ e'_k \}$,
%and $u_g$ is replaced with $\iota(u_g)$.\\\\ 
\end{lem}
 
 \begin{proof}[Proof of Lemma \ref{700}] 
 By Lemma 2.11 of [34], let $\tilde{\lambda}_G : \mathbb{Z}^l \rightarrow
{{\Aff}}(\Delta)$ be a strictly positive homomorphism (i.e., $\tilde{\lambda}_G(e_j) > 0$ for all $1 \leq j \leq l$) such that
$$\tilde{\lambda}_G |_G = \lambda_G.$$

By Lemma \ref{lem:First}, for each $1 \leq j \leq m$, the map
$$Hom(\mathbb{Z}^l, {{\Aff}}(\Delta))_{++} \rightarrow C(\Delta) : 
\dt \mapsto \frac{\dt(g_j)}{\dt(u_g)}$$ 
is continuous at the point $\dt = \tilde{\lambda}_G$.
Put
\beq\label{92-nn-10}
\eta_1=\inf\{\td \lambda_G(e_j)(s): s\in \Delta  \makebox{ and  } 1 \leq j \leq m   
 \}>0.
\eneq
{{Choose}} $0<\epsilon_1<\eta_1/10$ {{(recall that $7/8<r<1$)}}  so that for all 
$\dt \in {{\rm Hom}}(\mathbb{Z}^l, {{\Aff}}(\Delta))_{++}$, 
if $\| \delta(g_j) - \tilde{\lambda}_G(g_j) \| < \epsilon_1$ for all
$1 \leq j \leq m$, then 
\begin{equation}  \label{equ:Jan2820231AM}
\frac{{\dt}(g_j)(\tau)}{\dt(u_g)(\tau)} 
> r \frac{\tilde{\lambda}_G(g_j)(\tau)}{\tilde{\lambda}_G(u_g)(\tau)}
\end{equation}  
for all $1 \leq j \leq m$ and all $\tau \in \Delta$.
%%%%%%%%%%%%%%%%%%%%%%%%%%%
\iffalse
Contracting $\epsilon_1$ if necessary, we may further assume
that 
\begin{equation} \label{equ:Jan2820230AM}
 \tilde{\lambda}_G(e_k)(\tau) 
> 10 \epsilon_1 \end{equation}
for all $1 \leq k \leq l$  and 
$\tau \in \Delta$.  
\fi
%%%%%%%%%%%%%%%%%%%%%%
Choose $M \geq 1$ such that 
\begin{equation} \label{equ:Jan2820232AM}
\frac{1}{M}\max\left\{ \| \beta(\iota(e_k)) \|, 
\| \beta( \iota(g_j) ) \|, \| \beta(\iota(u_g)) \| \right\} 
< \min\left\{ \frac{\epsilon_1}{10}, 
\frac{\epsilon}{10} \right\} 
\end{equation}
for all $1 \leq k \leq l$ and $1 \leq j \leq m$. 

%By an argument similar to that in the first paragraph of this proof, 
{{By Lemma \ref{lem:First},}} for each $1 \leq k \leq l_1$,
since the map 
$${{\rm Hom}}(\mathbb{Z}^{l_1}, {{\Aff}}(\Delta))_{++} \rightarrow C(\Delta) 
: {{\dt}} \mapsto \frac{\dt(e'_k)}{\dt (\iota(u_g))}$$
is continuous at the point ${{\dt=}}\frac{1}{M}\beta$, 
we can choose $\epsilon_2 > 0$ so that for all $\dt
\in {\rm Hom}(\mathbb{Z}^{l_1}, {{\Aff}}(\Delta))_{++}$, 
if $\| {{\dt}}(e'_k) - \frac{1}{M}\beta(e'_k) \| < \epsilon_2$ for all
$1 \leq k \leq l_1$, then 
\begin{equation} \label{equ:Jan2820233AM}
\frac{\dt(e'_k)(\tau)}{\dt(\iota(u_g))(\tau)} > 
r \frac{\beta(e'_k)(\tau)}{\beta(\iota(u_g))(\tau)}
\end{equation}
for all $1 \leq k \leq l_1$ and 
\begin{equation}
\label{equ:Jan2820234AM}
\max\left\{ \| \dt(\iota(e_k)) \|, \| \dt(\iota(g_j)) \|,
\| \dt(\iota(u_g)) \| \right\}
 < \min \left\{\frac{\epsilon_1}{10}, 
\frac{\epsilon}{10} \right\}
\end{equation}
for all $1 \leq k \leq l$  and  
$1 \leq j \leq m$. 
%,  and $\tau \in \Delta$.   

Since $\mathbb{D}$ is an ordered dense subgroup of ${{\Aff}}(\Delta)$, 
{{$K\D$ is also dense in $\Aff(\Delta).$ So we may}}
choose a strictly positive homomorphism  
$\widetilde{\beta} : \mathbb{Z}^{l_1} \rightarrow K \mathbb{D}$
so that  
\beq\label{92-n-1}
{{0<}}\widetilde{\beta}(e'_k)(\tau) < \beta(e'_k)(\tau){{\rforal \tau\in \Delta}}
\eneq
and
$$\| (\beta(e'_k) - \widetilde{\beta}(e'_k)) - \frac{1}{M} \beta(e'_k)
\| < \epsilon_2$$
for all  
$1 \leq k \leq l_1.$  
%and  $\tau \in \Delta$.
Hence, if we define $$\overline{\beta} := \beta - \widetilde{\beta},$$
then 
$\overline{\beta} : \mathbb{Z}^{l_1} \rightarrow {{\Aff}}(\Delta)$ is a 
strictly positive homomorphism {{(see \eqref{92-n-1}),}} and by  
(\ref{equ:Jan2820233AM}) and (\ref{equ:Jan2820234AM}),
 \begin{equation}
\label{equ:Jan2820235AM}
\frac{\overline{\beta}(e'_k)(\tau)}{\overline{\beta}(\iota(u_g))(\tau)}
>  r \frac{\beta(e'_k)(\tau)}{\beta(\iota(u_g))(\tau)}
\end{equation}  
{{for all $\tau\in \Delta$ (hence \eqref{700-T-11} holds), and}} $1 \leq k \leq l_1,$
and
\begin{equation}  \label{equ:Jan282023-1AM}
\max\{ \| \overline{\beta}(\iota(e_k)) \|,  
\| \overline{\beta}(\iota(g_j)) \|, \| \overline{\beta}(\iota(u_g)) \|
\} < \min\left\{ \frac{\epsilon_1}{10},
\frac{\epsilon}{10} \right\}
\end{equation}  
for all $1 \leq k \leq l$
%$\tau \in \Delta$, 
and $1 \leq j \leq m$.
In particular,
$$
\| \overline{\beta}(\iota(u_g)) \| < \epsilon\andeqn \|\overline\bt(\iota(e_j))\|<\ep,\,\,\, 1\le j\le l.
$$
{{(So the first two  inequalities of \eqref{92-nn-3} hold.)}}
Moreover, by the definitions of $\overline{\beta}$,  
$\widetilde{\beta}$ and $\beta$, we immediately have that
\begin{equation}
\label{equ:Jan282023-2AM}
\Pi \circ \overline{\beta} = \Pi \circ \beta
\makebox{  and  } \Pi \circ \overline{\beta} \circ \iota |_G
= \Pi \circ \lambda_G.
\end{equation}
{{(So \eqref{92-nn-2} holds.)}}
In addition, by   {{\eqref{92-nn-10}}}  
%(\ref{equ:Jan2820230AM})   
{{and \eqref{equ:Jan282023-1AM},}}
$\gamma: = \widetilde{\lambda}_G - \overline{\beta} \circ \iota : \mathbb{Z}^l
\rightarrow {{\Aff(\Delta)}}$
is a strictly positive homomorphism, and by (\ref{equ:Jan2820231AM})
and (\ref{equ:Jan282023-1AM}),
we must have that 
$$\frac{\gamma(g_j)(\tau)}{\gamma(u_g)(\tau)} > r 
\frac{\widetilde{\lambda}_G(g_j)(\tau)}{\widetilde{\lambda}_G(u_g)(\tau)}$$
for all $1 \leq j \leq m$ and $\tau \in \Delta$.

Since $K \mathbb{D}$ is (norm) dense in ${{\Aff}}(\Delta)$, and by Lemma
\ref{lem:First},  
we can find a strictly positive homomorphism
${{\wtd \lambda_{\mathbb{D}}}} : \mathbb{Z}^l \rightarrow K \mathbb{D}$ such that
\beq
0<\wtd \lambda_{\mathbb{D}}(e_j)(\tau) < \gamma(e_j)(\tau) {{\rforal \tau\in \Delta, \,\,1\le j\le l,\andeqn}}
\eneq
%and for all $1 \leq j \leq l,$   
%and 
%and $\tau \in \Delta$,
$$\frac{{{\wtd \lambda_{\mathbb{D}}}}(g_j)(\tau)}{\wtd \lambda_{\mathbb{D}}(u_g)(\tau)} > r
\frac{\widetilde{\lambda}_G(g_j)(\tau)}{\widetilde{\lambda}_G(u_g)(\tau)}$$ 
for all $1 \leq j \leq m$ and $\tau \in \Delta$ {{(so \eqref{700-T-10} holds),}} and
if we define ${{\alpha_1}} := \gamma - {{\td\lambda_{\mathbb{D}}}}: 
\mathbb{Z}^l \rightarrow {{\Aff}}(\Delta)$, then ${{\alpha_1}}$ is a strictly
positive homomorphism such that 
\beq
{{\|\af_1(g_j)\|<\ep\andeqn}}\,\,\,\| {{\alpha_1}}(u_g) \| < \epsilon.
\eneq
%{\ppl{(hence the second inequality in \eqref{92-nn-3} holds).}}
Define $\alpha := \alpha_1|_G.$  {{Then the last two  inequalities in
 \eqref{92-nn-3} holds.}}
{{Moreover,}}
\beq
\widetilde{\lambda}_{\mathbb{D}} |_G + \alpha + \overline{\beta} \circ
\iota|_{G} = \lambda_G.
\eneq
{{It follows that \eqref{92-nn-1} holds.}}
Moreover, by (\ref{equ:Jan282023-2AM}), the range of $\alpha$ must be in 
$\mathbb{D}$. 
\end{proof}

\begin{lem} \label{lem:Feb2020231AM} 
Let $e_1, ..., e_l \in \Z^l$ be the standard basis, which also 
generates $\Z^l_+$ (as a semigroup).
Let $e'_1, ..., e'_{l_1} \in \Z^{l_1}$ be the standard basis, which again
also generates $\Z^{l_1}_+$ (as a semigroup).
Let $\iota: \Z^l \rightarrow \Z^{l_1}$ be a strictly positive homomorphism.
Let {{$u_g \in \Z^l_{++}$}} be a nonzero positive element.

Let $\Delta$ be a metrizable Choquet simplex, and    
suppose that $\D \subset \Aff(\Delta)$ is a dense ordered subgroup 
(with the strict order).
Let $\lambda : \Z^l \rightarrow \Aff(\Delta)$ be a strictly positive homomorphism
%$\iota : \Z^l \rightarrow \Z^{l_1}$ a strictly positive homomorphism,
and $\beta : \Z^{l_1} \rightarrow \Aff(\Delta)$ a strictly positive 
homomorphism for which 
$$\Pi \circ \beta \circ \iota(g) = \Pi \circ \lambda (g) \makebox{  for all  }
g \in \Z^l,$$
where $\Pi : \Aff(\Delta) \rightarrow \Aff(\Delta)/\D$ is the quotient map.
Let $\epsilon > 0$ and $0 < r < 1$ be given.

Then  
there are strictly positive homomorphisms
$\lambda_{\D} : \Z^l \rightarrow \D$ and
$\overline{\beta} : \Z^{l_1} \rightarrow \Aff(\Delta)$ such that 
\begin{equation}\label{add225-1}
\lambda_{\D} + \overline{\beta} \circ \iota = \lambda,
\end{equation}
\begin{equation}\label{add225-2}
\Pi \circ \overline{\beta} = \Pi \circ \beta,
\end{equation}
\begin{equation}\label{add225-3}
\| \overline{\beta} \circ \iota (u_g) \| < \epsilon,
\end{equation}
\begin{equation}
\label{equ:1}
\frac{{\lambda}_{\D}(e_j)(\tau)}{\lambda_{\D}(u_g)(\tau)}  
> r \frac{\lambda(e_j)(\tau)}{\lambda(u_g)} 
\makebox{  for all } \tau \in \Delta, 1 \leq j \leq l, \makebox{  and  }
\end{equation} 
\begin{equation}\label{add225-5}
\frac{\overline{\beta}(e'_k)(\tau)}{\overline{\beta}(\iota(u_g))(\tau)}
> r \frac{\beta(e'_k)(\tau)}{\beta(\iota(u_g))(\tau)}
\makebox{  for all } \tau \in \Delta, 1 \leq k \leq l_1.  
\end{equation}
\end{lem}
  
\begin{proof}
{{We apply Lemma \ref{700}.
Fix $l, l_1, \iota, u_g, \Delta, \D, \lambda, \bt, \ep, r$ as given.
Let $G=\Z^l.$ We also let $g_j=e_j,$ $1\le j\le l=m.$
Note that the map $\dt\mapsto {\dt(e_j)\over{\dt(u_g)}}$ is continuous at $\lambda.$
Choose $1>r_0>r.$  There is $0<\ep_0<\ep/2$ such that, if 
$\dt\in  {\rm Hom}(\Z^l, \Aff(\Delta))_{++}\setminus \{0\}$ and 
$\|\dt(e_j)-\lambda(e_j)\|<2\ep_0,$ $j=1,2,...,l,$  then 
\beq
{\dt(e_j)(\tau)\over{\dt(u_g)(\tau)}}>r_0 {\lambda(e_j)(\tau)\over{\lambda(u_g)(\tau)}}
\rforal \tau\in \Delta,\,\,\, 1\le j\le l.
\eneq
By applying {{Lemma \ref{700}}}
%{lem:First}, 
we obtain $\tilde\lambda_\D,$  ${\bar \bt}$   {{and $\af$}} as in 
{{Lemma \ref{700}}}
%{lem:First}  
such that  \eqref{92-nn-1} to \eqref{700-T-11} hold 
but for $\ep_0$ (in place of $\ep$) and $r_0$ (in place of $r$), and $g_j=e_j$ ($1\le j\le l$).
It follows that \eqref{add225-2}, \eqref{add225-3}  and \eqref{add225-5} hold. 

Next define $\lambda_\D=\tilde \lambda_\D+\af: \Z^l \to \D.$  Then
\eqref{add225-1} holds, and  \eqref{92-nn-3}
(with  $\ep_0$ in place of $\ep$) implies 
that
\beq
\|\lambda_\D(e_j)-\lambda(e_j)\|<2\ep_0,\,\,\, 1\le j\le l.
\eneq
By the choice of $\ep_0$ (and $r_0$), we also have 
\beq
\frac{{\lambda}_{\D}(e_j)(\tau)}{\lambda_{\D}(u_g)(\tau)}  
> r \frac{\lambda(e_j)(\tau)}{\lambda(u_g)} 
\makebox{  for all } \tau \in \Delta, 1 \leq j \leq l.
%, \makebox{  and  }
\eneq}} 
\end{proof}

\begin{prop}[Lemma 2.11 of \cite{Linsubh01}]\label{Pgexten}
Let $G\subset \Z^m\subset \R^m$ be a finitely generated  ordered subgroup
with order unit $1_g=(1,1,...,1),$ $T$ be a metrizable 
Choquet simplex
and $\phi: G\to \Aff(T)$   an order preserving \hm\, such that 
$\phi(g)(t)>0$ for all $g\in G_+ \setminus \{ 0 \}$ and $t\in T.$
Then there exists an order preserving 
\hm\, $\td \phi: \R^m\to \Aff(T)$ 
such that $\td \phi|_G=\phi$ and 
$\td \phi(x)(t)>0$  for all $x\in \R^m_+\setminus \{0\}.$ 
\end{prop}

\begin{proof}
By Lemma 2.11 of \cite{Linsubh01},
there exists $\psi: \Z^m\to \Aff(T)$ such 
that $\psi|_G=\phi$ and $\psi(e_j)(t)>0$ for all $t\in T$ 
($1\le j\le m$). 
For $(r_1,r_2,...,r_n)\in \R^m,$ we define $\td \phi(r_j)=r_j\psi(e_j).$
Then $\td \phi$ meets the requirement.
\end{proof}

\begin{lem}
\label{lem:RieszEmbed}
Let $G$ be a simple Riesz group with scale $e\in G_+$ (see definitions and
results in
[DavidsonBook] chapter IV, especially Section IV.6).
Suppose that $G$ can be expressed as an inductive limit 
of ordered groups 
$G = \lim_{n \rightarrow \infty} G_n$,
where for all $n \geq 1$, $(G_n, (G_n)_+) = (\Z^{l(n)}, \Z^{l(n)}_+)$,
and 
the connecting map $\iota_{n,n+1} : G_n \rightarrow G_{n+1}$ is 
strictly positive.
Suppose also that for all $n \geq 1$, ${{g_n}}
%e_n 
\in (G_n)_+$ is an order unit for
$G_n$ such that $\iota_{n,n+1}({{g_n}}) =
%e_n
 {{g_{n+1}}}$ and 
$\iota_{n, \infty}({{g_n}}) 
%(e_n)
= e$, where $\iota_{n, \infty} : G_n \rightarrow G$
is the induced morphism.

Let $\Delta$ be any metrizable Choquet simplex, and let
 $\D \subset \Aff(\Delta)$ be any dense ordered (with the strict order) subgroup
such that $1_{\Delta} \in \D$.

Let $0 < r_0 < 1$ be given.

Suppose that 
$\Lambda : G \rightarrow \Aff(\Delta)$ is a unital strictly positive homomorphism.
(Recall that here, unital means that $\Lambda(e) = 1_{\Delta}$.)

Then there exist a  sequence ${{\{ \widehat{g}_n \}}}$
%\{ \widehat{e}_n \}$ 
(with $\widehat{g}_0 = 0$)
in $\D_+ \setminus \{ 0 \}$ 
and a sequence $\lambda_n : G_{n} \rightarrow \D$ of strictly positive
homomorphisms with $\lambda_n ({{g_n}}
%e_n
) = {{\widehat{g}_n}}$
%\widehat{e}_n$ 
($n \geq 1$) such that
\begin{equation}\label{701-0}
\sum_{n=0}^{\infty} {{\widehat{g}_n}} 
%e_n
= 1_{\Delta} \makebox{  where the sum converges
uniformly over } \Delta, \makebox{  i.e., in uniform norm on  }\Aff(\Delta),
\end{equation}
and for all $m \geq 1$,
\begin{equation}\label{701-1}
\Lambda(\iota_{m, \infty}(x)) - \sum_{n=m}^{\infty} \lambda_n(\iota_{m,n}(x)) \in \D \makebox{ for all }
x \in {G_m}   
\end{equation}
(in the above, the sum $\sum_{n=m}^{\infty} \lambda_n(\iota_{m,n}(x))$ converges uniformly
on $\Delta$, i.e., in the Banach space $(\Aff(\Delta), \|. \|)$ where
the norm is uniform norm; also, for all $n \geq m$, $\iota_{m,n} : G_m
\rightarrow G_n$ is the induced connecting morphism),  
\begin{equation}\label{701-2}
\Lambda(\iota_{1,\infty}(x)) = \sum_{n=1}^{\infty} \lambda_n(\iota_{1,n}(x))
\makebox{  for all  } x \in G_{1},
\end{equation}
\begin{equation} \label{equ:Feb20202310AM}
\frac{\lambda_n(x)(\tau)}{\lambda_n({{g_n}}
%e_n
)(\tau)} > r_0 \Lambda({{\iota_{n, \infty}}}(x))(\tau)
\makebox{  for all  } \tau \in \Delta, \makebox{ }x \in (G_{n})_+ 
\makebox{  and  } n \geq 1.
\end{equation}
\end{lem}

\begin{proof}

Let $\{ r_k \}_{k=1}^{\infty}$ be a sequence in $(0, 1)$ such that
$\prod_{k=1}^{\infty} r_k > r_0$. Let $\{ \epsilon_k \}_{k=1}^{\infty}$
be a sequence in $(0,1)$ such that
$\sum_{k=1}^{\infty} \epsilon_k < 1$.  

For each $n \geq 1$, let $e_{n, 1}, ..., e_{n, l(n)}$ be the standard basis
for $G_n = \Z^{l(n)}$, i.e., $e_{n,k}     :=
(0,0,...,0,1, 0, ..., 0)$, with $1$ in the $k$th entry and $0$ in the other
entries, for all $1 \leq k \leq l(n)$.

For each $n \geq 1$, the map $\Lambda$ induces a map 
$\beta_n : G_n \rightarrow \Aff(\Delta)$, namely,
\begin{equation}\label{3123-1}
\beta_n := \Lambda \circ \iota_{n, \infty}.
\end{equation}  
Note that since $\Lambda$ and $\iota_{n, \infty}$ are unital (maps $e$ to $1_{\Delta}$
and ${{g_n}}$ to $e$ respectively)
strictly positive homomorphisms, $\beta_n$ is a unital (maps $g_n$ to $1_{\Delta}$)
 strictly positive homomorphism,
for all $n$.

We will now 
inductively construct sequences of strictly positive homomorphisms
$\lambda_n : G_n \rightarrow \D$ and  $\overline{\beta}_{n+1} : G_n 
\rightarrow \Aff(\Delta)$, for $n \geq 1$.

\vspace{0.1in}
%*{3ex}

\noindent Basis step: $n = 1$.  

By definition, $\beta_2 \circ \iota_{1,2} = \beta_1$. Hence, we can   
apply Lemma \ref{lem:Feb2020231AM}
to $\beta_1 : G_1 \rightarrow \Aff(\Delta)$, $\iota_{1,2} : G_1 \rightarrow G_2$,     
$\beta_2 : G_2 \rightarrow \Aff(\Delta)$,  $
%e_1 
{{g_1}}\in G_1$, $\epsilon_1$ and
$r_1$, to get
strictly positive homomorphisms  
$\lambda_1 : G_1 \rightarrow \D$ and $\overline{\beta}_2 : G_2 \rightarrow \Aff(\Delta)$
such that  {{(with $\bar \bt_1=\bt_1$)}}
\begin{equation} \label{equ:Feb2020232AM}
 \lambda_1 + \overline{\beta_2} \circ \iota_{1,2} = 
\overline{\beta}_1,
\end{equation} 
%(here, we define $\overline{\beta}_1 := \beta_1$), 
\begin{equation} 
\Pi \circ \overline{\beta_2} = \Pi \circ \beta_2,
\end{equation}
\begin{equation}
\| \overline{\beta_2} \circ \iota_{1,2}(
%e_1
{{g_1}}) \| < \epsilon_1,
\end{equation}
\begin{equation} \label{equ:Feb20202311AM} 
\frac{\lambda_1(e_{1,k})(\tau)}{\lambda_1(
%e_1
{{g_1}})(\tau)} > r_1 
\frac{\overline{\beta_1}(e_{1,k})(\tau)}{\overline{\beta}_1
%(e_1)
{{(g_1)}}(\tau)}
= r_1 \Lambda\circ \iota_{1,\infty}(e_{1,k})(\tau) \makebox{  for all }
\tau \in \Delta \makebox{  and  } 1 \leq k \leq l(1)
\end{equation}
(recall that $\overline{\beta}_1 := \beta_1$; also, $\beta_1({{g_1}}) = \Lambda(\iota_{1,\infty}({{g_1}})) = \Lambda(e) = 1_{\Delta}$), 
and
\begin{equation} \label{equ:Feb20202312PM}
\frac{\overline{\beta_2}(e_{2,l})(\tau)}{\overline{\beta_2}(\iota_{1,2}({{g_1}}))(\tau)}
> r_1 \frac{\beta_2(e_{2,l})(\tau)}{\beta_2(\iota_{1,2}({{g_1}}))} 
= r_1  \Lambda \circ \iota_{2, \infty}(e_{2,l})(\tau)
\makebox{ for all } \tau \in \Delta \makebox{  and  } 1 \leq l \leq l(2).
\end{equation}
(In the above, 
recall that $\beta_2(\iota_{1,2}({{g_1}})) = \beta_2({{g_2}}) 
= \Lambda(e) = 1_{\Delta}$.)

\vspace*{3ex}
\noindent Induction step: Suppose that we have constructed
$\lambda_m : G_m \rightarrow \D$ and $\overline{\beta}_{m+1}: G_{m+1}
\rightarrow \Aff(\Delta)$ for all $1 \leq m \leq n$  {{which satisfy the following:}}

{{\begin{equation}  \label{630-0}
%\label{equ:Feb2020233AM}  
\lambda_{m}  + \overline{\beta}_{m+1}  \circ \iota_{m, m+1} 
= \overline{\beta}_{m} 
\end{equation}}}
{{
\begin{equation} \label{equ:630-1}
\Pi \circ \overline{\beta}_{m} = \Pi \circ  \beta_{m}
\end{equation}

\begin{equation} \label{equ:630-2}
\| \overline{\beta}_{m+1}(\iota_{m,m+1}(g_{m})) \| < \epsilon_{m},
\end{equation}

\begin{equation} \label{equ:630-3}
\frac{\lambda_{m}(e_{m,k})(\tau)}{\lambda_{m}(g_{m})(\tau)}
> r_{m} \frac{\overline{\beta}_{m}(e_{m,k})(\tau)}{\overline{\beta}_{m} 
(g_{m})(\tau)}  \makebox{  for all  } \tau \in \Delta \makebox{  and  }
1 \leq k \leq l(m)  
\end{equation}
and
\beq \label{equ:630-4}
\frac{\overline{\beta}_{m+1}(e_{m+1,l})(\tau)}{\overline{\beta}_{m+1}(
\iota_{m,m+1}(g_{m}))(\tau)}
> r_{m} \frac{\beta_{m+1}(e_{m+1,l})(\tau)}{\beta_{m+1}(\iota_{m,m+1}
(g_{m}))(\tau)}
= r_{m} \Lambda(\iota_{m+1, \infty}(e_{m+1,l}))(\tau)  
\eneq
${\makebox{  for all } \tau \in \Delta
\makebox{ and } 1 \leq l \leq l(m+1).} $}}

We now construct $\lambda_{n+1} : G_{n+1} \rightarrow \D$ 
and $\overline{\beta}_{n+2} : G_{n+2} \rightarrow \Aff(\Delta)$.

By the definition of the $\beta_m$ {{(see \eqref{3123-1})}},
$$\beta_{n+2} \circ \iota_{n+1, n+2} = \beta_{n+1}.$$
Also, by the induction hypothesis,
$$\Pi \circ \overline{\beta}_{n+1} = \Pi \circ \beta_{n+1}.$$  
Hence,
$$
\Pi \circ \beta_{n+2} \circ \iota_{n+1, n+2} = \Pi \circ \overline{\beta}_{n+1}.
$$
{{Therefore,}} we {{may}} apply Lemma \ref{lem:Feb2020231AM} to  
$\overline{\beta}_{n+1} : G_{n+1} \rightarrow \Aff(\Delta)$, 
$\iota_{n+1, n+2} : G_{n+1} \rightarrow G_{n+2}$, 
$\beta_{n+2} : G_{n+2} \rightarrow \Aff(\Delta)$, 
${{g_{n+1}}} \in G_{n+1}$, $\epsilon_{n+1}$ and $r_{n+1}$
 to get strictly positive
homomorphisms $\lambda_{n+1} : G_{n+1} \rightarrow \D$ 
and $\overline{\beta}_{n+2} : G_{n+2} \rightarrow \Aff(\Delta)$ such 
that 
\begin{equation}  \label{equ:Feb2020233AM}  
\lambda_{n+1}  + \overline{\beta}_{n+2}  \circ \iota_{n+1, n+2} 
= \overline{\beta}_{n+1} 
\end{equation}

\begin{equation} \label{equ:Feb2020239AM}
\Pi \circ \overline{\beta}_{n+2} = \Pi \circ  \beta_{n+2}
\end{equation}

\begin{equation} \label{equ:Feb2020235AM}
\| \overline{\beta}_{n+2}(\iota_{n+1,n+2}({{g_{n+1}}})) \| < \epsilon_{n+1},
\end{equation}

\begin{equation} \label{equ:Feb20202313PM}
\frac{\lambda_{n+1}(e_{n+1,k})(\tau)}{\lambda_{n+1}({{g_{n+1}}})(\tau)}
> r_{n+1} \frac{\overline{\beta}_{n+1}(e_{n+1,k})(\tau)}{\overline{\beta}_{n+1} 
({{g_{n+1}}})(\tau)}  \makebox{  for all  } \tau \in \Delta \makebox{  and  }
1 \leq k \leq l(n+1)  
\end{equation}
and
\beq \label{equ:Feb2020232PM}
\frac{\overline{\beta}_{n+2}(e_{n+2,l})(\tau)}{\overline{\beta}_{n+2}(
\iota_{n+1,n+2}({{g_{n+1}}}))(\tau)}
> r_{n+1} \frac{\beta_{n+2}(e_{n+2,l})(\tau)}{\beta_{n+2}(\iota_{n+1,n+2}
({{g_{n+1}}}))(\tau)}
= r_{n+1} \Lambda(\iota_{n+2, \infty}(e_{n+2,l}))(\tau) \,\, 
\eneq
${{\makebox{  for all } \tau \in \Delta
\makebox{ and } 1 \leq l \leq l(n+2).}} $
%\end{equation}
(Recall that, by definition, 
 $\beta_{n+2}(\iota_{n+1, n+2}({{g_{n+1}}})) = \beta_{n+2}({{g_{n+2}}}) = \Lambda(e) = 1_{\Delta}$.)  
 
\vspace*{2ex}
\noindent This completes the inductive construction.\\

{{Moreover, by the induction process, we also have, for all $n\in \N,$
\beq
\frac{\lambda_{n+1}(e_{n+1,k})(\tau)}{\lambda_{n+1}({{g_{n+1}}})(\tau)}
> r_{n+1} r_n {\beta_{n+1}(e_{n+1,k})(\tau)\over{\beta_{n+1}
({{g_{n+1}}})(\tau)}}=r_{n+1}r_n \Lambda(\iota_{n+1,\infty}(e_{n+1, k}))(\tau)
\eneq
 $\makebox{  for all  } \tau \in \Delta \makebox{  and  }
1 \leq k \leq l(n+1).$ }}
{{Hence
\beq
\frac{\lambda_{n+1}(e_{n+1,k})(\tau)}{\lambda_{n+1}(g_{n+1})(\tau)}
> r_0 \Lambda(\iota_{n+1,\infty}(e_{n+1, k}))(\tau)
\eneq
 $\makebox{  for all  } \tau \in \Delta \makebox{  and  }
1 \leq k \leq l(n+1).$ }}
{{It follows that, for all $x\in G_{n+1},$ 
\beq
\frac{\lambda_{n+1}(x)(\tau)}{\lambda_{n+1}(g_{n+1})(\tau)}
> r_0 \Lambda(\iota_{n+1,\infty}(x))(\tau)\rforal \tau\in \Delta.
\eneq
Thus, together with \eqref{equ:Feb20202311AM}, \eqref{equ:Feb20202310AM}   holds. 
}}

{{Next, 
by}}  (\ref{equ:Feb2020232AM}) and repeated applications of (\ref{equ:Feb2020233AM}),
we have that for all $n \geq m \geq 1$, 
\begin{equation} \label{equ:Feb2020236AM}
\lambda_m + \lambda_{m+1} \circ \iota_{m, m+1} + \lambda_{m+2} \circ \iota_{m, m+2} + 
... + \lambda_{n} \circ \iota_{m, n} + \overline{\beta}_{n+1} \circ \iota_{m, n+1}
= \overline{\beta}_m.
\end{equation}  

In particular, for $m = 1$  and for $n \geq 1$, since 
$\overline{\beta}_1 := \beta_1$,
 we have that
\begin{equation}  \label{equ:Feb2020234AM} 
\lambda_1 + \lambda_2 \circ \iota_{1,2} + ... + \lambda_n \circ \iota_{1,n}
+ \overline{\beta}_{n+1} \circ \iota_{1, n+1}
= \beta_1.   
\end{equation}

For each $n \geq 1$, since $ran(\lambda_n) \subseteq \mathbb{D}$,
 let $\widehat{g}_n \in \D$ be given by
\begin{equation}
\widehat{g}_n := \lambda_n(\iota_{1,n}(g_1)) = \lambda_n({{g_n}}).   
\end{equation}

Now by (\ref{equ:Feb2020235AM}) (and since ${{g_k}}$ is an order unit
for $G_k$,  
for all $k$; and since the relevant maps are strictly positive) 
we know that for all $k$, for all $x \in G_k$,
\begin{equation}
\label{equ:Feb2020237AM} 
\overline{\beta}_{n+1} \circ \iota_{k, n+1}(x) \rightarrow 0
\end{equation}  
uniformly on $\Delta$ as $n \rightarrow \infty$.
From this and    
(\ref{equ:Feb2020233AM}), for all $k$, for all $x \in G_k$, 
\begin{equation} \label{equ:Feb2020238AM}
\lambda_{n+1} \circ \iota_{k, n+1}(x) \rightarrow 0
\end{equation}
uniformly on $\Delta$ as $n \rightarrow \infty$.
Thus, by (\ref{equ:Feb2020234AM}),
since $\beta_1({{g_1}}) = \Lambda(e) = 1_{\Delta}$, we must have that 
$$\sum_{n=1}^{\infty} \widehat{g}_n = 1_{\Delta}$$
where the sum converges
uniformly on $\Delta$.  

Let $m \geq 1$ and $x \in G_m$.
By (\ref{equ:Feb2020236AM}),  {{and}} (\ref{equ:Feb2020237AM}),
% and (\ref{equ:Feb2020238AM}),
{{\beq
\overline{\bt}_m(x)-\sum_{n=m}^{m+k}\lambda_n\circ \iota_{m,n}(x)\to 0
\eneq
uniformly on $\Delta$ (as $k\to\infty$).}} 
{{We then}} have that 
$$\overline{\beta}_m(x) = \sum_{n = m}^{\infty} \lambda_n \circ \iota_{m,n}(x),$$
where the sum converges uniformly on $\Delta$.
Hence, by (\ref{equ:Feb2020239AM}) and since 
$\Lambda(\iota_{m, \infty}(x)) = \beta_m(x)$,  
$$\Lambda(\iota_{m,\infty}(x)) - \sum_{n=m}^{\infty} \lambda_n \circ \iota_{m,n}(x) \in \D.$$

By the same argument, using (\ref{equ:Feb2020234AM}),  
for all $x \in G_1$,
since $\Lambda(\iota_{1,\infty}(x)) = \beta_1(x)$,
$$\Lambda(\iota_{1,\infty}(x)) = \sum_{n=1}^{\infty} \lambda_n(\iota_{1,n}(x))$$
where the sum converges uniformly on $\Delta$.
%
%Finally, (\ref{equ:Feb20202310AM}) follows from 
%(\ref{equ:Feb20202311AM}), (\ref{equ:Feb20202312PM}), (\ref{equ:Feb20202313PM})
%and (\ref{equ:Feb2020232PM}), and by the definition of the numbers $\{ r_k \}$.  
\end{proof}

\iffalse
\begin{lem}
Let $C$ be a separable unital \CA\, such that $K_0(C)$ and $K_0(C)_+$ are finitely generated,
$A$ be a unital separable simple \CA\, with tracial rank zero,
 and $\lambda: \Aff(T(C))\to \Aff(T(A))$ be an order preserving linear map
 such that $\lambda(\rho_C(C))\subset \rho_A(K_0(A))$ and $\lambda([1_A])\in \rho_C(K_0(A))$ 
 with $\lambda(\rho_C([1_C])(\tau)<1$ for all $\tau\in T(A).$ 
 
Then, for any $\ep>0$ and any finite subset ${\cal H}\subset C_{s.a.},$ there 
exists a continuous affine map $\gamma: T(pAp)\to T(C)$
such that
\beq
\lambda\circ \rho_C(x)(\tau)=
\eneq

\end{lem}

\fi

%%%%%%%324----

\section{Maps to the multiplier algebras}

Throughout this section, $B$ is a non-unital  separable 
%($\sigma$-unital?) 
simple \CA\, with tracial rank zero
and with continuous scale. 
\begin{lem}\label{AFdiag}
Let $C$ be a unital simple AF-algebra with ${\rm ker}(\rho_C) = \{ 0 \}$,
 and let $J: C\to M(B)$ be 
a unital \hm.
% and let $\{M_n\}$ be a sequence of positive integers. 
Suppose that $C=\overline{\cup_{n=1}^{\infty}F_n},$ where 
$\{F_n\}$ is an increasing sequence of finite dimensional \SCA s of $C$ with $1_{F_n}=1_C$ for all $n.$
Fix an integer $n_0\ge 1$  
%$0<\ep<1$ 
and $1>r_0>3/4.$

Then there exist an approximate identity $\{e_k\}$ (with $e_0=0$) of $B$  consisting of projections, 
a subsequence $\{n_k\}_{k=1}^{\infty}$ 
of the positive integers (with $n_1\ge n_0$),
and a sequence of unital injective 
 \hm s $\phi_k: F_{n_k}\to (e_k-e_{k-1})B(e_k-e_{k-1})$
such that
\beq
%&&\phi_k(K_0(F_{n_k}))\subset  K_0(B),\\
\label{82-T-1}
&& {{\Pi\circ J_{*0}(x)=(\Pi(\sum_{k=1}^\infty\phi_{k}\circ E_k))_{*0}(x)}}
%\in {{\rho_B(K_0(B))}}
\rforal x\in K_0(C),\\\label{82-T-2}
&&\tau(J(c))=\tau\circ \sum_{k=1}^\infty \phi_k\circ E_k(c)\tforal \tau\in T(M(B)) \tand c\in F_{n_1},\\
%\label{82-T-3}
%&&\hspace{-0.5in}{\rho_B({\phi_1}_{*0})(g))(\tau)\over{\rho_B({\psi_1}_{*0}([1_{F_1}]))(\tau)}}>r_0\rho_B(J_{*0}({\iota_{1, \infty}}_{*0})(g))(\tau)
%\rforal g\in K_0(F_1)\setminus \{0\},\,\,\, \tau\in T(B),\\
\label{82-T-4}
&&\hspace{-0.5in}\tand {\rho_B((\phi_{k})_{*0}(g))(\tau)\over{\rho_B((\phi_k\circ E_k)_{*0}([1_C]))(\tau)}}> r_0\rho_B((J\circ \iota_{k,\infty})_{*0}(g))(\tau)
\eneq
for all  $g\in K_0(F_{n_k})_+\setminus \{0\},$ and 
%\andeqn 
%\|J|_{F_{n_1}}-(\sum_{k=1}^\infty\pi_k\circ E_k)|_{F_{n_1}}\|<\ep,
for all  $\tau\in T(B),$
where {{$\Pi: \Aff(T(B))\to \Aff(T(B))/\rho_B(K_0(B))$ is the quotient map and}} $E_k: C\to F_{n_k}$ is an expectation 
and $\iota_{k, \infty} : F_{n_k}\to C$ is the unital embedding from
the above inductive limit decomposition of $C$, $k=1,2,....$ 
\end{lem}

\begin{proof}
Since $C$ is simple, passing to a subsequence (starting with $n_1 \geq n_0$)
if necessary, we may assume that for all $n \geq 1$, the
inclusion map ${{\iota_{n, n+1}:}} F_n \hookrightarrow F_{n+1}$ induces a strictly positive
homomorphism $K_0(F_n) \rightarrow K_0(F_{n+1})$.

Plug $G = K_0(C)$, $e = [1_C]$, $G_n = K_0(F_n)$, ${{g_n}}
%e_n
 = [1_{F_n}]
\in K_0(F_n) = G_n$ (for all $n$),
$\Delta = T(B)$, and $\D = \rho_B(K_0(B))$ (recall that since $B$ has real
rank zero, $\rho_B(K_0(B))$ is norm dense in $\Aff(T(B))$) {{together with $\Lambda=J_{*0}:
K_0(C)\to K_0(M(B))=\Aff(T(B))$}}
into Lemma \ref{lem:RieszEmbed}, to get    
$\{ \widehat{g}_n \}$ and $\{ \lambda_n \}$ {{which satisfy \eqref{701-0},
\eqref{701-1},  \eqref{701-2} and \eqref{equ:Feb20202310AM}.}}

For each $n \geq 1$, let $f_{n,1}, ..., f_{n, N(n)}$
be a set of minimal projections from different {simple} summands of $F_n$ (so
for all $j 
\neq k$, $f_{n,j}$ and $f_{n,k}$ are from different summands) and let
$m_{n,1}, ..., m_{n,N(n)} \geq 1$   
be integers
so that the ideal (in $F_n$) generated by $f_{n,1}, ..., f_{n,N(n)}$
is all of $F_n$, and  
$$\lambda_n (m_{n,1} [f_{n,1}] + ... + m_{n,N(n)} [f_{n, N(n)}])
= {{\widehat{g}_n}}.$$
For each $n \geq 1$, since ${\rm ran}(\lambda_n) \subseteq \rho_B(K_0(B))$,
we can find {{mutually orthogonal}} projections $e_{n,1}, ..., e_{n,N(n)} \in B$ so that
$$\lambda_n([f_{n,j}])(\tau) = \tau(e_{n,j})$$
for all $1 \leq j \leq N(n)$ and  $\tau \in T(B)$.

Denote by $\{v_{i,j}\}$ a system of matrix units for ${\cal K}.$
For the rest of the proof, we identify $B$ with $B \otimes v_{1,1}$
and $M(B)$ with $M(B) \otimes v_{1,1}$. (E.g., we use $1_{M(B)}$ and 
$1_{M(B)} \otimes v_{1,1}$ interchangeably, and also 
$T(B)$ and $T(B \otimes v_{1,1})$ interchangeably.)

{{Recall that ${{\widehat{g}_n}}\in \rho_B(K_0(B))$ are non-zero and $\sum_{n=1}^\infty 
{{\widehat{g}_n}}=1_{T(B)}.$
In particular, $0<{{\widehat{g}_n}}
%\widehat{e}_n
(\tau)<1$ for all $\tau\in T(B).$ 
We may choose, for each $n$ 
a projection}} $q_n\in B\otimes v_{n,n}$
such that  $q_n\sim \bigoplus_{j=1}^{N(n)} {\bar e}_{n,j},$ 
{{ where  ${\bar e}_{n,j}$ are projections and $[{\bar e}_{n,j}]=m_{n,j} [e_{n,j}]$ in $K_0(B).$}} 
%\tau(q_n)=\widehat{e}_n(\tau)$ for all $\tau\in T(B).$  
Put $P:= \bigoplus_{n=1}^{\infty} q_n,$
where the sum converges strictly in $M(B \otimes {{\cal K}}).$

{{We compute that $\tau(P)=\tau(1_{M(B)})$ for all $\tau\in T(B).$ 
Recall that $B$ is a $\sigma$-unital simple \CA\, of tracial rank zero. It follows (see Theorem \ref{125-2331}) that
$P\sim 1_{M(B)}$ in $M(B\otimes {\cal K}).$  In other words,  there exists a partial isometry 
$W\in M(B\otimes {\cal K})$ such that
\beq
W^*W=1_{M(B)}\andeqn WW^*=P.
\eneq
Let $p_n=W^*q_nW,$ $n\in \N.$
Then  $p_ip_j=p_jp_i=0,$ if $i\not=j,$ and 
$\sum_{n=1}^\infty p_n  =1_{M(B)}$ which converges in the strict topology. 
Put $e_0=p_0=0$}} and $e_k= \sum_{l=1}^k p_k$ for $n\in \N.$ 
{{Then $\{e_k\}$ forms an approximate identity (consisting of projections)  
for $B.$}} 
 
 %
 %
 \iffalse 
Then 
we can define a projection in $M(B \otimes K)$ with the form 
$$P := \bigoplus_{n=1}^{\infty} \bigoplus_{j=1}^{N(n)}  
\bigoplus^{m_{n,j}} e_{n,j}$$
where the sum converges strictly in $M(B \otimes K)$.

Since $\tau(P) = \tau(1_{M(B)})$ for all $\tau \in T(B)$ and
since $B \otimes K$ has projection injectivity (since $B$ is simple with
tracial rank zero),  
$P \sim 1_{M(B)} \otimes p$ in $M(B \otimes K)$ (where $p$ is a
rank one projection in $K$) .
Hence, we can find a sequence $\{ p_n \}$ of pairwise orthogonal
projections in $B$ such that 
$$1_{M(B)} = \sum_{n=1}^{\infty} p_n$$    
where the sum converges strictly in $M(B)$, and 
for all $n$,
$$p_n \otimes p \sim \bigoplus_{j=1}^{N(n)} \bigoplus^{m_{n,j}} e_{n,j}.$$
\fi
%%%%%%%%%%%%%%%%%%%%%%%%%
Note that,  for all $n$ for all $\tau \in T(B)$,
$$\tau(e_n - e_{n-1}) = \tau(p_n) = 
%\widehat{e}
{{\widehat{g_n}}}(\tau).$$
For each $n$, let $\phi_n : F_n \rightarrow p_n B p_n$ be a  
unital injective *-homomorphism for which
$$\rho_B \circ K_0(\phi_n) = \lambda_n.$$
{{Recall that $K_0(C)=\cup_{n=1}^\infty \iota_{n, \infty}(G_n)$  and
$E_n\circ \iota_{n, \infty}={\rm id}_{F_n}$ for all $n\in \N.$  
Therefore, for each $m\in \N,$
by \eqref{701-1},
\beq
J_{*0}(x) - \sum_{n=m}^{\infty} \lambda_n((\iota_{m,n}\circ E_m)_{*0}(x)) \in \D \makebox{ for all }
x \in \iota_{m, \infty}(G_m).   
\eneq
Hence \eqref{82-T-1} holds. Then  \eqref{82-T-2} follows from \eqref{701-2}, and \eqref{82-T-4} 
follows from \eqref{equ:Feb20202310AM}.}}
%{\ppl{Then the lemma  follows}} immediately from the conclusions of
%Lemma \ref{lem:RieszEmbed}. 
\end{proof}

\begin{cor}\label{C82+}
Let $C$ be a unital simple AF-algebra with ${{{\rm ker}}}(\rho_C) = \{ 0 \}$,
 and let  $J: C\to M(B)$ be 
a unital \hm.
% and let $\{M_n\}$ be a sequence of positive integers. 
Suppose that $C=\overline{\cup_{n=1}F_n},$ where 
$\{F_n\}$ is an increasing sequence of finite dimensional \SCA s of $C$ with $1_{F_n}=1_C$ for all $n$.  
Fix an integer $n_0\ge 1,$ $0<\ep<1$  and $3/4<r<1.$

Then there exist an approximate identity $\{e_k\}$ (with $e_0=0$) of $B$  consisting of projections, 
a subsequence $\{n_k\}_{k=1}^{\infty}$ 
of the positive integers  (with $n_1\ge n_0$), 
  and a sequence of unital injective  \hm s $\phi_k: F_{n_k}\to 
(e_k-e_{k-1})B(e_k-e_{k-1})$
such that
\beq
%&&\phi_k(K_0(F_{n_k}))\subset  K_0(B),\\
\label{83-T-1}
&&J(x)-\sum_{k=1}^\infty\phi_k\circ E_k(x)\in B\rforal x\in C\tand\\\label{83-T-2}
&&\|J(c)-\sum_{k=1}^\infty\phi_k\circ E_k(c)\|<\ep \tforal c\in {{(F_{n_1})^{\bf 1}}},\\
%\label{83-T-3}
%&&\hspace{-0.6in}{\tau(\phi_1(c))\over{\tau(\psi_1(1_{F_1}))}}>r\tau(J(\iota_{1, \infty}(c))
%\rforal c\in (F_1)_+\setminus \{0\},\,\, \tau\in T(B),\\
\label{83-T-4}
&&\hspace{-0.6in}\tand {\tau(\phi_{k}(c))\over{\tau(\phi_{k}\circ E_k(1_C))}} \ge r\tau(J\circ \iota_{k,\infty}(c))
\tforal c\in (F_{n_k})_+\setminus\{0\},
\eneq
 and 
%\andeqn 
%\|J|_{F_{n_1}}-(\sum_{k=1}^\infty\pi_k\circ E_k)|_{F_{n_1}}\|<\ep,
for all  $\tau\in T(B),$
%
%\andeqn 
%\|J|_{F_{n_1}}-(\sum_{k=1}^\infty\pi_k\circ E_k)|_{F_{n_1}}\|<\ep,
%
where $E_k: C\to F_{n_k}$ is an expectation and $\iota_{k, \infty} :
F_{n_k} \rightarrow C$ is the unital embedding from the above inductive
limit decomposition of $C$, $k=1,2,....$ 
\end{cor}

\begin{proof}

Let $0<\ep<1.$  
Set  $r_0=r+{1-r\over{64}}.$ Then $r<r_0<1.$

As in the proof of Lemma \ref{AFdiag}, 
we simplify notation by
letting $n_0 = 1$.

   Let $\{e_k \}$ and unital \hm\, 
$\psi_k: F_k\to (e_k-e_{k-1})B(e_k-e_{k-1})$ be as in 
(the proof of) Lemma \ref{AFdiag}  so that \eqref{82-T-1}, \eqref{82-T-2}
%\eqref{82-T-3} 
and \eqref{82-T-4} all hold 
(with $r_0$ as above).
Note that, as in the proof of Lemma \ref{AFdiag}, we are simplifying
notation and letting $n_k = k$ for all $k$.  In  particular,
$n_1 = 1 \makebox{ } (= n_0)$.  

%As previously, we 
We will often  simplify notation 
% be simplifying notation and saving writing,
 by writing 
``$c$" for $\iota_{n,\infty}(c) \in C$ or $\iota_{n,m}(c) \in F_m$
for all $c \in F_n$ and all $n \leq m$. 
%The similar will hold 
%for group elements in the inductive limit decomposition of $K_0(C)$. 

Let ${\cal F}\subset F_{1}$ be a finite subset of 
 the unit ball of $F_{1}$.
(which is compact) and 
  % We may assume that all the elements of ${\cal F}$ are positive
%and  
$1_C\in {\cal F}.$

Put 
\beq
\sigma_0:=\inf\{\tau(J\circ \iota_{1,\infty}(e)): e\in (F_1)_+\setminus \{0\},\,\, e\,\,{\rm is\,\, a\,\,minimal\,\, projection}\}>0
\eneq  
 and     $\ep_0:=\min\{\sigma_0/64, \ep/4, {1-r\over{64}}\}.$ 
   
Choose  $0<\dt<\ep_0^3/64$ 
such that, for  any  unital ${\cal F}$-$\dt$-multiplicative \morp\, $L': F_1\to D$ (for any unital \CA\, 
$D$), there is a unital \hm\, $\phi: F_1\to D$ such that
\beq
\|L'-\phi\|<\ep_0^3/64.
\eneq

{{Denote by $\pi: M(B)\to M(B)/B$ as well as $\pi: M_2(M(B))\to M_2(M(B)/B)$ the quotient maps.}}
%Let $\{e_n\}$ and unital \hm\, $\psi_k: F_k\to (e_n-e_{n-1})B(e_n-e_{n-1})$ be as in 
%Lemma \ref{AFdiag}  so that 
Define $L: C\to M(B)$ by $L(c)=\sum_{n=1}^\infty\psi_n\circ E_n.$
Note $\pi\circ L: C\to M(B)/B$ is a unital essential extension of $C.$ 
%Denote by $\pi: M(B)\to M(B)/B$ as well as $M_2(M(B))\to M_2(M(B)/B)$ the quotient maps.
It follows from Lemma \ref{AFdiag} that $[\pi\circ J]=[\pi\circ L].$ 
Therefore, by Theorem \ref{TH2}, there exists a unitary $W\in M_2(M(B))$ such that
\beq
{\rm Ad}\, \pi(W)\circ \pi\circ L=\pi\circ J.
\eneq
It follows, for any $c\in C,$ 
\beq
J(x)-W^*(\sum_{n=1}^\infty \psi_n\circ E_n(c){{)}}W\in M_2(B).
\eneq

We will simplify notation by sometimes writing $1_2 = 1_{M_2(M(B))}$ and $1 = 1_{M(B)}$.

Let $p_n'=\diag(e_n, e_n)\in M_2(B).$ Then $\{p_n:=W^*p_n'W\}$ forms an approximate 
identity for $M_2(B).$  
Since ${\cal F}$ is compact, there is $N_1\ge 1$ such that
\beq\label{83-a-1-00}
&&\hspace{-0.5in}{{\|(1_2-p_{N_1})J(x)-W^*(\sum_{N_1+1}^\infty \psi_n\circ E_n(x))W\|<\dt/2^9}}\andeqn\\
&&\hspace{-0.5in}\|(1_2-p_{N_1})J(x)(1_2-p_{N_1})-W^*(\sum_{N_1+1}^\infty \psi_n\circ E_n(x))W\|<{{\dt/2^9}}\rforal x\in {\cal F},\andeqn\\
\label{83-a-1-01}
&&\hspace{-0.5in}(1_2-p_{N_1})J(c)(1_2-p_{N_1})-W^*(\sum_{N_1+1}^\infty \psi_n\circ E_n(c))W\in M_2(B)\rforal  c\in C.
\eneq
{{Since $B$ has continuous scale, we}} may choose a sufficiently large $N_1$ such that
\beq\label{84-TTT}
\sup\{\tau(1_2-p_{N_1}): \tau\in T(B)\}<\ep_0^3/64.
\eneq
Note $p_n$ commutes with $W^*(\sum_{n=1}^\infty \psi_n\circ E_n(c)){{W}}$ for all $c\in C.$	
Hence, by \eqref{83-a-1-00},
\beq\nonumber
\hspace{-0.3in}(1_2-p_{N_1})1_{M(B)}&=&(1_2-p_{N_1})W^*L(1_{M(B)})W+(1_2-p_{N_1})(1_{M(B)}-W^*L(1_{M(B)})W)\\\nonumber
&\approx_{\dt/{{2^9}}}&(1_2-p_{N_1})W^*L(1_{M(B)})W\\\nonumber
&=&W^*L(1_{M(B)})W(1-p_{N_1})\\\nonumber
&\approx_{\dt/{{2^9}}}&(1_{M(B)}-W^*L(1_{M(B)})W)(1_2-p_{N_1})+W^*L(1_{M(B)})W(1_2-p_{N_1})\\
&=&1_{M(B)}(1_2-p_{N_1}).
\eneq
Note that we identify $J(1_C) = 1_{M(B)}$ with 
 the identity of $\wtd B.$ It follows that 
$1_{M(B)}(1_2-p_{N_1})1_{M(B)}\in \wtd B.$
%Note also $1-p_{N_1}\in \td B+M_2(B).$
Thus there is a projection $d\in 1+B$ such that
\beq\label{84-phi0-e2}
\|d-(1_2-p_{N_1})J(1_C)(1_2-p_{N_1})\|<\dt/2^7.
\eneq
It follows that  {{(see also \eqref{83-a-1-00})}}
\beq\label{83-515+1}
&&\|dJ(x)-J(x)d\|<{{\dt/2^7+\dt/2^{9}}}\rforal x\in {\cal F}\andeqn\\\label{83-515+2}
&&\|d-W^*(\sum_{N_1+1}^\infty \psi_n\circ E_n(1_C))W\|<{{\dt/2^7+\dt/2^9.}}
%\andeqn
%dJ(x)d
\eneq
Note {{that}} both $d$ and $W^*(\sum_{N_1+1}^\infty \psi_n\circ E_n(1_C))W$ are projections 
in $\td B+M_2(B).$
There is {{(by \eqref{83-515+2})}} a unitary $U\in M_2(\td B)$ such that
\beq
\|U-1_{M_2(\td B)}\|<\dt/{{32}}\andeqn
U^*W^*(\sum_{N_1+1}^\infty \psi_n\circ E_n(1_C))WU=d.
\eneq
It follows that
\beq
dJ(c)d-U^*W^*(\sum_{N_1+1}^\infty\psi_n\circ E_n(c)){{WU}}\in B\rforal c\in C\andeqn\\
\|dJ(x)d-U^*W^*(\sum_{N_1+1}^\infty\psi_n\circ E_n(x))W{{U}}\|<\dt/2\rforal x\in {\cal F}.
\eneq
Thus, by the choice of $\dt$ {{(see also \eqref{83-515+1}),}} there is a unital \hm\, $\phi_0: F_1\to (1-d)B(1-d)$ such that
\beq\label{84-phi0-e1}
\|(1-d)J(c)(1-d)-\phi_0(c)\|<\ep_0^3/64\rforal c\in  F_1 \andeqn \|c\|\le 1.
\eneq
One then checks that, for all $x\in F_1$ with $\|x\|\le 1,$ 
\beq
\|J(x)-(U^*W^*(\sum_{n=N_1+1}^\infty\psi_n\circ E_n(x))WU-\phi_0(x))\|<\ep_0^3/32.
\eneq
We may write $\phi_1:=\phi_0$  and replace $e_1$ by $e_1':=(1-d),$  $e_n$ by $e_n':=U^*W^*e_{n+N_1}WU,$
if $n>1,$ 
and  define $\phi_n:F_{n+N_1}\to (e_n'-e_{n-1}')B(e_n'-e_{n-1}')$ by 
$\phi_n(c)=U^*W(\psi_{n+N_1}(c))WU$ for all $c\in F_n$ and $n>1.$
Then $\{e_n'\}$ is an approximate identity for $B$ and 
\beq
&&\|J(x)-\sum_{n=1}^\infty\phi_n\circ E_n(x)\|<\ep_0^3/16<\ep\rforal x\in {\cal F}\andeqn\\
&&J(c)-\sum_{n=1}^\infty\phi_n\circ E_n(c)\in B \makebox{  for all  }
c \in C. 
\eneq
We also have,  {{by our choice of $\psi_k$ (from Lemma \ref{AFdiag}) which satisfies \eqref{82-T-2},}}
for all $k\ge 1$ and  any $g\in K_0(F_k)_+\setminus \{0\},$  and, for all $\tau\in T(B),$
\beq\label{84-E-10}
{\rho_B((\phi_{k+1}\circ\iota_{k,k+1})_{*0}(g))
(\tau)\over{\rho_B\circ \phi_{k+1}([1_{F_k}])(\tau)}}> r_0\rho_B(J\circ \iota_{k,\infty})_{*0}(g))(\tau).
\eneq
Note, since $F_k$ is a finite dimensional \CA, for any $c\in (F_k)_+\setminus \{0\},$ 
 there are $\af_1, \af_2,...,\af_m\in \R_+$ (not all zero) and minimal projections 
$e_1, e_2,...,e_m\in (F_k)_+\setminus\{0\}$
such that $c=\sum_{l=1}^m\af_l e_l.$ 

Then we have, by 
\eqref{84-E-10},   for all $c\in (F_k)_+\setminus\{0\},$
that 
\beq\label{84-phike-1}
{\tau(\phi_{k+1}\circ \iota_{k, k+1}(c))\over{\tau(\phi_{k+1}(1_{F_{k+1}}))}}={\sum_{l=1}^m\af_l \tau(\phi_{k+1}\circ \iota_{k,k+1}(e_l))\over{\tau(\phi_{k+1}(1_{F_{k+1}}))}}=\sum_{l=1}^m\af_l{\tau(\phi_{k+1}\circ \iota_{k, k+1}(e_l))\over{\tau(\phi_{k+1}(1_{F_{k+1}}))}}\\
>r_0(\sum_{l=1}^m\af_l\tau(J\circ \iota_{k, \infty}(e_l)))>r\tau(J\circ \iota_{k, \infty}(c)) \rforal \tau\in T(B).
\eneq
Then  \eqref{83-T-4} holds.  {{The lemma follows.}}
\end{proof}

For each $i=0,1,$ and $k\ge 0$ define $P_{i,k}: \underline{K}(C)\to K_i(C, \Z/k\Z)$ to be the projection map.

%Let $C\in {\cal A}$ be a unital \CA \ (see Definition \ref{df:ClassCalA}), 
%and let $G\subset K_0(A)$ be a subgroup. 

Also, for any group $H$ and any subset $E \subseteq H$, $E^g$ will denote the
subgroup of $H$ which is generated by $E$.
We will also sometimes use the notation $G^E$ to denote the same object $E^g$ (e.g., see item (3)(b) in
Definition \ref{DefA1}).
%{\ppl{\bf Better to have one of them}}

\begin{df}\label{DefKL-triple}
Recall that for a %C*-algebra 
{{\CA\,}} $C$, a finite subset $G \subset
\underline{K}(C)$, a finite subset ${\cal G} \subset C$ and 
$\delta > 0$, the triple $(\delta, {\cal G}, G)$ is called a 
\emph{KL-triple} if for every 
%C*-algebra 
{{\CA}}\, $D$,  every completely positive
contractive ${\cal G}$-$\delta$-multiplicative linear map  
$\phi : C \rightarrow D$ induces a well-defined group homomorphism
$G^g \rightarrow \underline{K}(D)$. %{\ppl{\bf References?}}
\end{df}

\begin{lem}\label{L84}
Let $C$ be a unital \CA\,
% in the class
 %${\cal A}_0$
 % (see Definition \ref{df:ClassCalA})
 {{and}} let $\Gamma: K_0(C)\to \Aff(T(B))$ be a strictly positive \hm\,
such that $\Gamma([1_C])=1_{T(B)}.$ 
%{\blue{(i) If $C\in {\cal A}_{00},$  then}}  there is a unital injective \hm\, $\Phi: C\to M(B)$
%such that $\Phi_{*0}=\Gamma.$  
Suppose that $C\in {\cal A}_1,$ and there is an affine continuous map $\gamma: T(B)\to T(C)$ such that
$\rho_{M(B)}\circ \Gamma(g)(\tau)=\rho_C(g)(\gamma(\tau))$
for all $g\in K_0(C)$ and $\tau\in T(B).$ 

Then
there is a unital injective \hm\, $\Phi: C\to M(B)$ such that
$\Phi_{*0}=\Gamma$ and $\tau(\Phi(c))=\gamma(\tau)(c)$
for all $c\in C$ and $\tau\in T(B).$

Let $\{G_n\}$ be an increasing sequence of finite subsets of $\underline{K}(C)$ such that
$\cup_{n=1}^\infty G_n=\underline{K}(C)$ and $P_{i,k}(G_n)\subset G_n$ 
for all $n$, 
 $i= 0,1$ and $k \geq 0$. 
Let $\{{\cal G}_n\}$ be an increasing sequence of finite subsets 
of $C$ such that $\cup_{n=1}^\infty{\cal G}_n$ is dense in $C,$  $\{ \dt_n \}$
  a decreasing  sequence of positive numbers
with $\sum_{n=1}^\infty\dt_n<1,$  $\ep>0$ and  let ${\cal F}\subset C$ be a finite subset.
We also assume that $(\dt_n, {\cal G}_n, G_n)$ is a $KL$-triple for every $n$.
Let $\{ {\cal H}_n \}$
be an increasing sequence 
of finite subsets of  
 $C_+^{\bf 1}\setminus \{0\}$ 
such 
that $\cup_{n=1}^\infty {\cal H}_n$ is dense in  $C_+^{\bf 1}$
and let $1>r>3/4.$
%any ${\cal G}_n$-$\dt_n$-multiplicative \morp\, $L'$ from $C$ well defines $[L']|_{G_n}.$

Then, %{\blue{in both cases (i) and (ii) above,}}  
we can choose $\Phi$ as above so that 
there exist an approximate identity $\{e_n\}$ of $B$
(with $e_0=0$ and $e_n-e_{n-1}\not=0$ for all $n\geq 1$) 
consisting of projections, 
and a sequence of maps $\{ \phi_n \}$ where 
$\phi_n: C\to (e_n-e_{n-1})B(e_n-e_{n-1})$ 
is a ${\cal G}_n$-$\dt_n$-multiplicative  unital \morp \ 
such that $[\phi_n]|_{G_n}$ is well defined, 
\beq
&&[\phi_n]|_{G_n^g\cap {\rm ker}\rho_{C,f}}=0,\,\,
[\phi_n]|_{ G_n\cap K_1(C)}=0,\\
&&{[\phi_n]}|_{G_n\cap K_i(C, \Z/k\Z)}=0\tforal k\ge 2\,\,\, (i=0,1),
\eneq
$n=1,2,...,$ (where $G_n^g$ is the the subgroup of 
$\underline{K}(C)$ generated by $G_n$) and 
\beq\label{84-B-1}
&&\Phi(c)-\sum_{n=1}^\infty \phi_n(c)\in B\tforal c\in C,\\
&&\|\Phi(c)-\sum_{n=1}^\infty \phi_n(c)\|<\ep\tforal c\in {\cal F},\tand\\
&& {\tau(\phi_n(x))\over{\tau(e_n-e_{n-1})}}\ge r\tau(\Phi(x))\tforal x\in {\cal H}_n.
\eneq
\iffalse
%{\blue{In case (i), i.e.,  
%if $C \in {\cal A}_{00}$ and }}    
%there is an affine continuous map $\gamma: T(B)\to T(C)$ such that
%$\rho_{M(B)}\circ \Gamma(g)(\tau)=\rho_C(g)(\gamma(\tau))$
%for all $g\in K_0(C)$ and $\tau\in T(B),$ we may further require that
%\beq
%\tau(\Phi(c))=\gamma(\tau)(c)\tforal c\in C\tand \tau\in T(B).
%\eneq
\fi
\end{lem}

%(Recall that $[\psi]|_{G^g}=0$ implies also that 
%$[\psi]$ is well defined on a generating set $G'$ of $G^g.$)
%in addition to the condition that $[\phi]|_{G_n}$ is well defined.)

\begin{proof}
Choose a dense ordered countable subgroup $H\subset \Aff(T(B))$ such that   $1_{T(B)}\in H$ and 
\beq\label{monoid-1}
\Gamma(K_0(C))_+\subseteq H_+:=H\cap \Aff_+(T(B)).
%\andeqn 
%\rho_B(K_0(B))\cap H_+=\{0\}.
\eneq
%where 
Let $H_1:=\{th: t\in \Q, h\in H\} \subseteq \Aff(T(B))$. 
There exists a separable simple unital AF-algebra $F$ such that $T(F)=T(B),$ 
${\rm ker}\rho_F=\{0\},$ 
$K_0(F)=H_1$  
and $[1_F]=1_{T(B)}.$ 
(Note that since $H_1$ is a dense ordered subgroup of $\Aff(T(B))$, {{we have that}}
the state space of $H_1$ is affine homeomorphic to $T(B)$. Hence, since $(K_0(F), K_0(F)_+, [1_F])
 = (H_1, {H_1}_+, 1_{T(B)})$,
$T(F) = T(B)$.)
Now
$H_1 \otimes \mathbb{Q} \cong H_1$, and so, by the Elliott classification of
AF algebras, $F \otimes Q \cong F$, where $\mathbb{Q}$ is the group of
rational numbers and $Q$ is the UHF algebra with dimension group
$(\mathbb{Q}, \mathbb{Q}_+, 1)$.   
%Let $D:=F\otimes Q$ where $Q$ is the universal
%UHF algebra (with dimension group $(\mathbb{Q}, \mathbb{Q}_+, 1)$).   

Since 
 $C$ is in the class 
${\cal A}_1$,  there is a unital injective \hm\, $\Psi: C\to F$ such that $\Psi_{*0}=\iota\circ \Gamma,$
where $\iota: \Gamma(K_0(C))\hookrightarrow K_0(F)$ is the inclusion  as just 
described;  moreoever,
\begin{equation} \label{equ:Mar920231AM} \tau(\Psi(c)) = \gamma(\tau)(c) \makebox{  for all } c \in C \makebox{  and  }
\tau \in T(B). 
\end{equation}   
%Applying  Lemma 4.2 of \cite{PR}, as in  the proof of  Theorem \ref{TembeddingAH},
Let $D := F \otimes Q$ and  $j_F: F\to D$ be the unital embedding 
defined by $j_F(a)=a\otimes 1_Q.$  
Since $D \cong F$,
and since $D$ is in the class ${\cal A}$, there is a unital embedding
$j_D : D \rightarrow M(B)$ such that ${j_D}_{*0}$ gives the identification
of $K_0(F) \makebox{ } (= K_0(D))$ with 
$H_1 \subseteq \Aff(T(B)) = K_0(M(B))$.   
Hence, since ${j_F}_{*0} :  K_0(F) \rightarrow K_0(D)$ is the
identity map on the scaled ordered group $(K_0(F), K_0(F)_+, [1_F])$,
$$(j_D \circ j_F \circ \Psi)_{*0} = \Gamma.$$ 
Note that since $D$ is simple, $j_D(D)\cap B=\{0\}.$
Note also that $D \cong F$ and we can identify $T(D) = T(F) = T(B) = T(M(B))$. (The last equality
{{holds}}  because $M(B)/B$ is simple purely infinite.) In fact, $j_F$ and $j_D$ induce
affine homeomorphisms (which can be viewed as  identity maps) $T(D) \rightarrow T(F)$ and
$T(B) \rightarrow T(D)$ respectively.

\iffalse
%{\blue{Suppose that we are in Case (i), i.e., suppose that 
%$C \in {\cal A}_{00}$ and suppose that  }} 
% there is an affine continuous map $\gamma: T(B)\to T(C)$ such that
%$\rho_{M(B)}\circ \Gamma(g)(\tau)=\rho_C(g)(\gamma(\tau))$
%for all $g\in K_0(C)$ and $\tau\in T(B).$
%{\blue{Then, we have $\Psi$, $D$, $j_F$, $j_D$ as in the previous
%paragraph (i.e., as in Case (ii)).}} 
% By identifying $T(B)$ with $T(F)
%\makebox{ }(=T(D))$
%(and identifying $\Gamma(K_0(C))$ with a subgroup of $K_0(F)$),
%we have that  
%\beq
%\rho_D \circ \Gamma(g)(\tau)=\rho_C(g)(\gamma(\tau))\rforal g\in K_0(C)\andeqn \rforal \tau\in T(F).
%\eneq
\fi

So if we define   $$\Phi:= j_D \circ j_F\circ \Psi:C\to M(B)$$
then   ${\Phi}_{*0} = \Gamma$ and by (\ref{equ:Mar920231AM}),   
\beq
\tau(\Phi(c))=\gamma(\tau)(c) \rforal c\in  C\andeqn \tau\in T(B).
\eneq

We may also assume that for any $n,$   since $G_n^g$ is finitely generated, 
$G_n$ contains a generating 
set of $G_n^g\cap {\rm ker}\rho_{C,f}$.   
%and $(\dt_n, {\cal G}_n, G_n)$ is a $KL$-triple.
For all $n$, set 
\beq\label{85-e-4}
\sigma_n=\inf\{\tau(\Phi(x)):x\in {\cal H}_n\andeqn \tau\in T(B)\}>0.
\eneq
Choose $1>r_0>r$ such that 
\beq\label{85-e-5}
r_0(1-{1-r_0\over{16}})>r.
\eneq

Let $F=\overline{\cup_{n=1}^\infty F'_n},$ where for all $n$,
$F'_n$ is a finite dimensional C*-algebra, 
$F'_n\subset F'_{n+1},$ and 
$1_{F'_n}=1_F.$   Let $E'_n: F\to F'_n$ be an expectation, $n=1,2,....$
\Wlog, passing to a subsequence of $\{F'_n\}$ if necessary,
 we may assume that for all $n$, the map  
 $E'_n\circ \Psi : C \rightarrow F'_n$ is ${\cal G}_n$-$\dt_n$-multiplicative. 
%and $(\dt_k, {\cal G}_k, G_k)$ is a $KL$-triple.
Moreover, we may also assume that 
\beq\label{85-e-10}
\|E'_n\circ \Psi(x)-\Psi(x)\|<((1-r_0)\dt_n\sigma_n/64)\|x\|\rforal x\in {\cal H}_n, \,\,n=1,2,....
\eneq
{{Choose $m_k\in \N$ such that $m_kg=0$ for any $g\in G_k\cap Tor(K_0(C))$ 
and any $g\in G_k\cap K_i(C, \Z/m\Z)$, for all $m\in \N$ and $i=0,1.$}}
{{Note that we may write $Q=\overline{\cup_{k=1}^\infty Q_k},$ where 
$Q_k\subset Q_{k+1},$   $1_{Q_k}=1_Q,$
and $Q_{k}=M_{r(k)}$ such that ${m_k} |r(k),$ $k\in \N.$}}
%Let $Q_k=M_{m_k}\subset Q$ with $1_{Q_k}=1_Q.$ 
Put $F_n=F_n'\otimes Q_n,$ $n=1,2,....$ Note $j_D(x)=x\otimes 1_{Q_k}.$
%
\iffalse
Since $Q$ is the UHF algebra with dimension group $(\mathbb{Q}, \mathbb{Q}_+,
1)$, we  {{may write}}
%have a C*-inductive limit decomposition
$Q = \overline{\bigcup_{k=1}^{\infty} Q_k}$, where 
for all $k$, $Q_k$ is a full matrix algebra, $Q_k \subset Q_{k+1}$
and $1_{Q_k} = 1_Q$. 
We may assume that for all $k$, 
there exists $m_k\in \N$ such that 
$$Q_k = {{M}}_{m_k}
\makebox{ and }  m_kg=0,$$ 
\noindent for any $g\in G_k\cap Tor(K_0(C))$ 
and any $g\in G_k\cap K_i(C, \Z/m\Z)$, for all $m\in \N$ and $i=0,1.$ 
\fi
%%%%%%%%%%%%%%%
Clearly, $F \cong F \otimes Q = \overline{\bigcup_{n=1}^{\infty} F'_n \otimes
Q_n}.$
%
%%%
\iffalse
For all $n$, 
define  $F_n:= F_n'\otimes Q_n$.  Note that $j_D(x)=x\otimes 1_{Q_k}$
for all $x \in F_n$.
\fi
%%%%%%%%%%%%% 
 
To simplify the notation, we may assume that $\Psi(x)\in F_1'$ (and
hence, $j_F \circ \Psi(x) \in F_1$) for all $x\in {\cal F}.$ 
%Moreover $E_n'\circ \Psi$ is ${\cal G}_n$-$\dt_n$-multiplicative. 

Since $\Gamma(x)=0$ for all $x\in {\rm ker}\rho_{C,f},$ and 
 since   $K_1(F)=\{0\},$
 we have that 
 \beq
 [E_n'\circ \Psi]|_{G_n^g\cap {\rm ker}\rho_{C,f}}=0\andeqn
 [E_n'\circ \Psi]|_{G_n\cap K_1(C)}=\{0\},
 \eneq
$n=1,2,....$
Let $E_k'': Q\to Q_k$ be an expectation and $E_k$ the expectation given
by $E_k=E_k'\otimes E_k'': D\to F_k\otimes Q_k$.  

It follows 
that,  {{for $n\in \N,$}} 
\beq
&&[E_n\circ j_F\circ \Psi]|_{G_n\cap K_i(C, \Z/m\Z)}=0\,\,\,\rforal m\in {{\N}}\andeqn i=0,1, {{\andeqn}}\\
%\eneq
%as well as 
%\beq
 &&[E_n\circ j_F\circ \Psi]|_{G_n^g\cap {\rm ker}\rho_C}=0\andeqn
 [E_n\circ j_F\circ \Psi]|_{G_n\cap K_1(C)}=\{0\}.
 \eneq
%$n=1,2,....$

By  Corollary \ref{C82+}, \wilog, we may assume that
there is an approximate identity $\{e_n\}$ of $B$ and 
for each $n$, a unital \hm\, $\psi_n: {{F_n'}} \otimes Q_n 
\to (e_n-e_{n-1})B(e_n-e_{n-1})$
such that 
\beq
&&\hspace*{-0.8in} j_D(d)-\sum_{n=1}^\infty \psi_n\circ E_n(d) \in B \rforal d\in D,\\
&&\hspace*{-0.8in} \|j_D(j_F\circ \Psi(a))-\sum_{n=1}^\infty \psi_n\circ E_n(j_F\circ \Psi(a))\|<\ep\rforal a\in {\cal F},\andeqn\\\label{85-e-20}
&&\hspace*{-0.8in} {\tau(\psi_k(d))\over{\tau(e_k-e_{k-1})}}>r_0\tau(j_D(d))\rforal d\in ({{F_k'}}\otimes Q_k)_+\setminus \{0\}, \tau \in T(B) \makebox{ and } k \geq 1.
\eneq
Let $\phi_n:=\psi_n\circ E_n\circ j_F\circ \Psi,$ $n=1,2,....$ 
Thus, 
\beq
&&j_D\circ j_F\circ \Psi(c)-\sum_{n=1}^\infty \phi_n(c)\in B\rforal c\in C,\\
&&\|j_D\circ j_F\circ \Psi(a)-\sum_{n=1}^\infty \phi_n(a)\|<\ep\rforal a\in {\cal F}.
\eneq
By  \eqref{85-e-20}, \eqref {85-e-10}, \eqref{85-e-4} and \eqref{85-e-5}, we have that for all $n$, for all $x\in {\cal H}_n,$ 
\beq
{\tau(\phi_n(x))\over{\tau(e_n-e_{n-1})}}&>&r_0\tau(j_D(E_n\circ j_F\circ \Psi(x)))\\
&>&r_0\tau(j_D\circ j_F\circ \Psi(x))-r_0{(1-r_0)\dt_n\sigma_n\over{16}}\\
&>&r_0(1-{(1-r_0)\dt_n\over{16}})\tau(\Phi(x))>r\tau(\Phi(x)). \hspace{0.1in}
%\rforal x\in {\cal H}_n.
\eneq

\end{proof}

\begin{lem}\label{L86} 
Let $C$ be a unital \CA\, in ${{{\cal A}_1}}$  and $\Psi: C\to M(B)$ be a unital 
\hm\, such that $\pi\circ \Psi$ is injective.
%and let $\Gamma: K_0(C)\to \Aff(T(B))$ be a strictly positive \hm\,
%such that $\Gamma([1_C])=1_{T(B)}.$ 

Let $\{G_n\}$ be an increasing sequence of finite subsets of $\underline{K}(C)$ such that
$\cup_{n=1}^\infty G_n=\underline{K}(C),$ let $\{{\cal G}_n\}$ be an increasing sequence of finite subsets 
of $C$ such that $\cup_{n=1}^\infty{\cal G}_n$ is dense in $C,$  and  a decreasing  sequence of positive numbers
$\{\dt_n\}$ with $\sum_{n=1}^\infty\dt_n<1,$  $\ep>0$ and  ${\cal F}\subset C$ be a finite subset.
We also assume  that $(\dt_n, {\cal G}_n, G_n)$ is a $KL$-triple and $1>r>0.$
Let ${\cal H}_n$ be an increasing sequence of finite subsets 
of $C_+^{\bf 1}\setminus \{0\}$ such that $\cup_{n=1}^\infty{\cal H}_n$ is dense in $C_+^{\bf 1}.$

Let $\ep>0$ and ${\cal F}\subset C$ be a finite subset.
Let  $\Phi: C\to M(B)$ be a unital injective homomorphism
 with $\Phi_{*0}=\Psi_{*0}$  
for which
there exists an approximate identity $\{e_n\}$ of $B$ consisting of projections 
%Then 
%there exists an approximate identity $\{e_n\}$ 
 (with $e_0=0$ and $e_n-e_{n-1}\not=0$)
and a sequence of ${\cal G}_n$-$\dt_n$-multiplicative \morp s $\phi_n: C\to (e_n-e_{n-1})B(e_n-e_{n-1})$
such that
\beq
&& \label{equ:Mar920232AM} [\phi_n]|_{G_n^g\cap {\rm ker}\rho_{C,f}}=0,\,\,
[\phi_n]|_{ G_n\cap K_1(C)}=0,\\
&&{[\phi_n]}|_{G_n\cap K_i(C, \Z/k\Z)}=0\tforal k\ge 2\,\,\, (i=0,1),\\
&&{\tau(\phi_n(x))\over{\tau(e_n-e_{n-1})}}>r\tau(\Phi(x))\tforal x\in {\cal H}_n\tand \tau\in T(B),
\eneq
$n=1,2,...,$    and 
\beq
\Phi(c)-\sum_{n=1}^\infty \phi_n(c)\in B\tforal c\in C\tand\\\label{86-T-01}
\|\Phi(a)-\sum_{n=1}^\infty \phi_n(a)\|<\ep\tforal a\in {\cal F}.
\eneq
Then there exist {{a unitary $V\in M_2(M(B)),$}}
an integer $N_1,$ and a ${\cal G}_1$-$\dt_1$-multiplicative \morp\,
$L: C\to  e_{N_1}'Be_{N_1}'$ such that $[L]|_{G_1}$ is well defined, 
{{$V^*(1_{M(B)}-e_{N_1})V\in 1_{M(B)}+B,$
$e_{N_1}'=1_{M(B)} -V^*(1_{M(B)} -e_{N_1})V\in B,$}} 
\beq
&&\Psi(c)-\sum_{n=N_1+1}^\infty  V^*\phi_n(c)V\in B\tforal c\in C,\tand\\\label{86-T-02}
&&\|\Psi(c)-(\sum_{n=N_1+1}^\infty V^*\phi_n(c)V+L(c))\|<\ep\tforal c\in {\cal F}.
\eneq

If, in addition, $\tau\circ \Psi(c)=\tau\circ \Phi(c)$ for all $c\in C,$ then we may further assume 
that, 
\beq\label{86t-1}
%\|\tau(L(x))-\tau(\sum_{i=1}^{N_1}\phi_i(x)\|< \ep,\,\,\,
%\tforal x\in {\cal H}_1\tand \tau\in T(B)\tand\\
\tau(L(x))
%u(e_{N_1})}}
>r\tau(\Phi(x))\tand\\\label{86t-2}
|\tau(L(x))-\tau(\sum_{j=1}^{N_1}\phi_j(x))|<\ep\hspace{0.2in}
\eneq
for all $ x\in {\cal H}_1\tand\tau\in T(B)$, and 
{{if}} 
%for any given finite subset of projections $p_1,p_2,...,p_s, q_1,q_2,...,q_s\in M_l(C)$ (for some integer $l\ge 1$)
%{\blue{we have that 
%$\tau(p_i)=\tau(q_i)$ for all $\tau\in T(C),$ }} 
{{a finite subset of projections $p_1,p_2,...,p_s, q_1,q_2,...,q_s\in M_l(C)$ (for some integer $l\ge 1$)
with the property that $\tau(p_i)=\tau(q_i)$ for all $\tau\in T(C),$
$i=1,2,...,s,$  is given,}}
we may also require that 
\beq\label{86t-3}
\rho_B([L(p_i)])(\tau)=\sum_{j=1}^{N_1}\rho_B([\phi_j(p_i)])(\tau)\tand
\rho_B([L(q_i)])(\tau)=\rho_B([L(p_i)](\tau))
\eneq
for all $\tau\in T(B),\,\,1\le i\le s.$

We may further assume that, for each $n,$ the map
$c\mapsto (1-e_{N_1+n})\Psi(c) (1-e_{N_1+n})$ is ${\cal G}_{n+1}$-$\dt_{n+1}$-multiplicative, and 
\beq
\lim_{N\to\infty} \|\Psi(c)-(\sum_{n=N}^\infty V^*\phi_n(c)V+e_N\Psi(c)e_N)\|=0\tforal c\in C.
\eneq
\end{lem}

\begin{proof}
{{We will use the fact that ${\rm ker}\rho_{C,f}={\rm ker} \rho_C$ (see \ref{Rrho-1}).}}
Let ${\cal F}\subset  C$ be a finite subset. 
%If $p_1, p_2,...,p_s, q_1,q_2,...,q_s\in  M_l(C)$ are given projections,
%we may assume that $[p_i]\in G_1,$ $i=1,2,....$  Write $p_i=(a_{p,i,j})_{l\times l}$ for some $a_{p,i,j}\in C.$ 
%Let ${\cal F}_p\subset C$ be a finite subset 
%such that  $a_{p,i,j}\in {\cal F}_p.$ In what follows we may assume 
%that ${\cal F}'\supset {\cal F}_p.$
%
Put ${\cal F}'={\cal F}\cup {\cal G}_1\cup {\cal H}_1.$ 
If $p_1, p_2,...,p_s, q_1,q_2,...,q_s\in  M_l(C)$ are given projections,
we may assume that $[p_{{k}}],\, {{[q_{{k}}]}} \in G_1,$ $i=1,2,..., s.$ 
Write $p_k=(a_{p_k, i,j})_{l\times l}$  {{and $q_k=(a_{q_k, i,j})_{l\times l}$}} for some $a_{p_k,i,j}, {{a_{q_k, i,j}}}\in C$
($1\le k\le s$), 
{{respectively.}}
Let ${\cal F}_p\subset C$ be a finite subset 
such that  $a_{p_k,i,j}, a_{q_k, i,j}\in {\cal F}_p.$ In what follows we may assume 
that ${\cal F}'\supset {\cal F}_p.$

   We assume also that $1_C\in {\cal F}.$
     Let $0<\ep<{1\over{16(l+1)^2}}.$ 
   Choose $1>r_0>r$ such that $r_0(1-{1-r_0\over{4}})>r.$
   Let $\Gamma:=\Psi_{*0}: K_0(C)\to \Aff(T(B)).$ 
   
   Put
\beq\label{86-sigma}
\sigma:=\inf\{\tau(\Psi(c)):c\in  {\cal H}_1\andeqn \tau\in T(B)\}>0.
\eneq

   Let $\ep_0:=\min\{\ep\sigma/2, (1-r)\sigma/2\}.$

%Choose  $0<\dt<\ep/4$ 
%such that, for  any  unital ${\cal G_1}$-$\dt$-multiplicative \morp\, $L': F_1\to D$ (for any unital \CA\, 
%$D$), there is a unital \hm\, $\phi: F_1\to D$ such that
%\beq
%\%|L'-\phi\|<\ep/2.
%\eneq
Let $\Phi$  with $\Phi_{*0}=\Gamma$ and $\{\phi_n\}$ be given by Lemma \ref{L84}
(with ${\cal F}'$ in place of ${\cal F},$ $r_0$ in place of $r,$ and $\ep_0/4$ in place of $\ep$).

We will repeat the argument used in the proof of Lemma \ref{C82+}.
%Let $\{e_n\}$ and unital \hm\, $\psi_k: F_k\to (e_n-e_{n-1})B(e_n-e_{n-1})$ be as in 
%Lemma \ref{L84}.  
Define $L_M: C\to M(B)$ by $L_M(c)=\sum_{n=1}^\infty\phi_n(c)$ for $c\in C.$
Then
\beq\label{86-Phi}
\Phi(c)-L_M(c)\in B\rforal c\in C\andeqn
\|\Phi(x)-L_M(x)\|<\ep_0/4\rforal x\in {{{\cal F}'}}.
\eneq

Note $\pi\circ L_M: C\to M(B)/B$ is a unital essential extension of $C.$ 
Denote by $\pi: M(B)\to M(B)/B$ as well as $M_2(M(B))\to M_2(M(B)/B)$ the quotient maps.
{{Recall that $K_0(M(B))=\Aff(T(B))$ (which is torsion free) and $K_1(M(B))=0.$
It follows that \\$KK(C, M(B))={\rm Hom}(K_0(C), \Aff(T(B)).$ In particular, 
$[\Psi]=[\Phi]$ in $KK(C, M(B)).$}} 
It follows from  {{\eqref{86-Phi}}} that $[\pi\circ \Psi]=[\pi\circ L_M].$ 
Therefore, by {{Theorem \ref{TH2},}}  there exists a unitary $W\in M_2(M(B))$ such that
\beq
{\rm Ad}\, \pi(W)\circ \pi\circ L_M=\pi\circ \Psi.
\eneq
It follows, for any $c\in C,$ 
 \beq
\Psi(c)-W^*(\sum_{n=1}^\infty \phi_n(c))W\in M_2(B).
\eneq  
Let $p_n'=\diag(e_n, e_n)\in M_2(B),$ {{$n\in \N.$}} Then $\{p_n:=W^*p_n'W\}$ forms an approximate 
identity for $M_2(B).$  

To simplify notation, we henceforth sometimes write $1_2 = 1_{M_2(M(B))}$ and
$1 = 1_{M(B)}$.

Choose $0<\dt=\min\{\ep_0/4, \dt_1/4\}.$
Since ${\cal F}$ is finite, there is $N_1\ge 1$ such that
\beq\label{83-a-1}
&&\hspace{-0.6in}\|(1_2-p_{N_1})\Psi(x) %(1_2-p_{N_1}) 
-W^* (\sum_{N_1+1}^\infty \phi_n(x))W\|
<\dt/{{2^9}} \andeqn\\
&&\hspace{-0.6in}\|({{1_2}}-p_{N_1})\Psi(x)(1_2-p_{N_1})
-W^*(\sum_{N_1+1}^\infty \phi_n(x))W\|<\dt/{{2^9}}\rforal x\in {\cal F'} \cup 
{\cal F'}^*, \andeqn\\
\label{86-a-1++}
&&\hspace{-0.6in}(1_2-p_{N_1})\Psi(c)(1_2-p_{N_1})-W^*(\sum_{N_1+1}^\infty \phi_n(c))W\in M_2(B)\rforal  c\in C.
\eneq
Recall that $B$ has continuous scale, so, by choosing sufficiently large $N_1,$
we may also assume 
that
\beq\label{85-trace-1}
\sup\{\tau(1_2-p_{N_1}):\tau\in T(B)\}<\ep_0/4\rforal \tau\in T(B).
\eneq
Note {{also that, for all $k,m\in \N,$}} $p_k$ commutes with $W^*(\sum_{n=m}^\infty \psi_n(c))W$ for all $c\in C.$	
Hence, by \eqref{83-a-1},
\beq\nonumber
\hspace{-0.1in}({{1_2}}-p_{N_1})1_{M(B)}&=&(1_2-p_{N_1})W^*L_M(1_{M(B)})W+(1_2-p_{N_1})(1_{M(B)}-W^*L_M(1_{M(B)})W)\\\nonumber
&\approx_{\dt/{{2^9}}}&({{1_2}}-p_{N_1})W^*L_M(1_{M(B)})W=
%\\\nonumber
%&=&
W^*L_M(1_{M(B)})W(1_2-p_{N_1})\\\nonumber
&\approx_{\dt/2^9}&(1_{M(B)}-W^*L_M(1_{M(B)})W)(1_2-p_{N_1})+W^*L_M(1_{M(B)})W(1_2-p_{N_1})\\
&=&1_{M(B)}(1_2-p_{N_1}).
\eneq
{{Note that we identify $1_{M(A)}$ with the identity of $\wtd B.$ It follows that 
$1_{M(A)}(1_2-p_{N_1})1_{M(A)}\in \wtd B.$}} 
Thus there is a projection $d\in 1+B$ such that 
%Thus there is a projection $d\in \td B$ such that
\beq
\|d-(1_2-p_{N_1})\Psi(1_C)(1_2-p_{N_1})\|<\dt/{{2^7}}.
\eneq
It follows that  {{(see also \eqref{83-a-1})}}
\beq
&&\|d\Psi(x)-\Psi(x)d\|<\dt/2^7+\dt/2^9\rforal x\in {{{\cal F}'}}\andeqn\\
&&\|d-W^*(\sum_{N_1+1}^\infty \phi_n(1_C))W\|<{{\dt/2^7+\dt/2^9.}}
%dJ(x)d
\eneq
Note both $d$ and $W^*(\sum_{N_1+1}^\infty \phi_n(1_C))W{{=W^*(1-e_{N_1})W}}$ are projections 
in $\td B+M_2(B).$
There is a unitary $U\in M_2(\td B)$ such that
\beq\label{86-d-1}
\|U-1_{M_2(\td B)}\|<\dt/{{2^5}}\andeqn
U^*W^*(\sum_{N_1+1}^\infty \phi_n(1_C))WU=d.
\eneq
It follows that
\beq
d\Psi(c)d-U^*W^*(\sum_{N_1+1}^\infty\phi_n(c))WU\in B\rforal c\in C\andeqn\\
\|d{{\Psi}}(x)d-U^*W^*(\sum_{N_1+1}^\infty\phi_n(x))WU\|<\dt/2\rforal x\in {\cal F'}.
\eneq
Put $e_{N_1}':=1-d\in B.$ Define $L: C\to e_{N_1}'Be_{N_1}'$ by 
$L(c)=(1-d)\Psi(c)(1-d)$ for all   $c \in C$. 
Then $L$ is ${\cal F'}$-$\dt$-multiplicative.
%Thus, by the choice of $\dt,$ there is a unital \hm\, $\phi_0: F_1\to (1-d)B(1-d)$ such that
%\beq
%\|(1-d)J(c)(1-d)-\phi_0(c)\|<\ep/2\rforal c\in  F_1 \andeqn \|c\|\le 1.
%\eneq
One then checks that, for all $x\in {\cal F'},$
\beq
\|\Psi(x)-(U^*W^*(\sum_{n=N_1+1}^\infty\phi_n(x))WU+L(x))\|<\ep_0.
\eneq
%We may write $\phi_1:=\phi_0$  and replace $e_1$ by $e_1':=(1-d),$  $e_n$ by $e_n':=U^*W^*(e_{n+N_1}WU,$
%if $n>1,$
Let $e_n':= U^*W^*e_{n}WU$ for all $n>1,$ 
and  define $\psi_n:C\to (e_n'-e_{n-1}')B(e_n'-e_{n-1}')$ by  
$\psi_n(c)=U^*W(\phi_{n}(c))WU$ for all $c\in C$   and $n>1.$
Then $\{e_n'\}$ is an approximate identity for $B$ and 
\beq\label{86-nn519}
&&\|\Psi(x)-(\sum_{n=N_1+1}^\infty\psi_n(x)+L(x))\|<\ep\rforal x\in {\cal F'}\andeqn\\
&&\Psi(c)-(\sum_{n=N_1+1}^\infty\psi_n(c)+L(c))\in B \rforal
%{\blue{ \makebox{  for all  } 
c \in C.
\eneq
{{Choose $V=WU,$ then, by \eqref{86-d-1}, $d=V^*(1-e_{N_1})V$ which implies that 
$V^*(1-e_{N_1})V\in 1+B.$}} Hence  the first part of the lemma follows. 
We also have, by  the second part of \eqref{86-d-1}, \eqref{85-trace-1}, \eqref{86-sigma},
\beq
%&&\hspace{-0.6in}
%{\tau(L(x))\over{\tau(e_{N_1})}}\ge
 \tau(L(x))&=&\tau((1-d)\Psi(x)(1-d))\\
&=&\tau(\Psi(x))-\tau(d\Psi(x)d)\ge \tau(\Psi(x))-\tau(d)=\tau(\Psi(x))-\tau(1-e_{N_1})\\
&\ge& \tau(\Psi(x))-\tau(1_2-p_{N_1})>\tau(\Psi(x))-\ep_0/4\\
&\ge& (1-{1-r\over{4}})\tau(\Psi(x))>r\tau(\Psi(x))\rforal x\in {\cal H}_1\andeqn \tau\in T(B).
\eneq

Suppose now that $\tau\circ \Phi=\tau\circ \Psi.$ Then estimates above {{imply}} that 
\beq
{{\tau(L(x))
%\over{\tau(e_{N_1})}}
>r\tau(\Phi(x))}}\rforal x\in {\cal H}_1\andeqn \tau\in T(B).
\eneq
Moreover,  by \eqref{86-Phi} and \eqref{85-trace-1}, for all $x\in {\cal H}_1,$ 
\beq\nonumber 
&&\hspace{-0.5in}|\tau(L(x))-\tau(\sum_{j=1}^{N_1}\phi_j(x))|\le |\tau(L(x))-\tau(\Psi(x))|+|\tau(\Phi(x))-\sum_{j=1}^{N_1}\tau(\phi_(x))|\\
\nonumber
&&\hspace{0.2in}\le \tau(1_2-p_{N_1})+|\tau(\Phi(x))-\tau(L_M(x))|+|\tau(L_M(x))-\sum_{j=1}^{N_1}\tau(\phi_(x))|\\\nonumber
&&\hspace{0.2in}<\ep_0/4+\ep_0/4+\ep_0/4<\ep\rforal \tau\in T(B).
\eneq

%If $p_1, p_2,...,p_s\in M_l(C)$ are given projections, we may assume that 
Recall $[p_k], [q_k] \in G_1$, $\tau(p_k) = \tau(q_k)$ 
for all $\tau \in T(C)$,  ${{a_{p_k,i,j}, a_{q_k, i,j}\in {\cal F}',}}$
for $k = 1, ..., s$, $1 \leq i,j \leq l$, 
and 
$\ep<1/16(l+1)^2.$ Then,  since $\tau\circ \Phi=\tau\circ \Psi$ for all $\tau\in T(B),$
{{and $\Phi_{*0}=\Psi_{*0},$}} 
%by \eqref{86-T-01} and 
{{by  \eqref{86-nn519},}}
%\eqref{86-T-02},
\beq
[\Phi(p_i)]=[\Psi(p_i)]=[L'(p_i)]\andeqn [\Phi(q_i)]=[\Psi(q_i)]=[L'(q_i)]\,\,\,i=1,2,...,s,
\eneq
where $L'(x)=\sum_{i=N_1+1}^\infty \psi_n(x)+L(x)$  for all $x\in C.$
Then (recall that $B$ has continuous scale)
\beq
&&\rho_B({{[L(p_i)]}})(\tau)=\rho_B([\sum_{j=1}^{N_1}\phi_j(p_i)])(\tau)\andeqn\\
&&\rho_B({{[L(q_i)]}})(\tau)=\rho_B([\sum_{j=1}^{N_1}\phi_j(q_i)])(\tau) \rforal \tau
\in T(B),\,\, i=1,2,...,s.
\eneq
Also, for all $1 \leq i \leq s$,  
since $[p_i] - [q_i] \in G_n \cap {{\rm ker}} \rho_{C,f}$ for all $n$, by (\ref{equ:Mar920232AM}), we must
have that  
$$\rho_B([\sum_{j=1}^{N_1}\phi_j(p_i)])(\tau)= \rho_B([\sum_{j=1}^{N_1}\phi_j(q_i)])(\tau)
\makebox{  for all   }  \tau \in T(B).$$
Hence, 
$$\rho_B({{[L(p_i)]}})(\tau) = \rho_B({{[L(q_i)]}})(\tau) \makebox{  for all  }
\tau \in T(B).$$

{{Hence}} the  {{second}} last part of the lemma also follows {{(so \eqref{86t-1}, \eqref{86t-2} and \eqref{86t-3} hold).}}

For the last part of the lemma, we note that $\lim_{n\to\infty}\|(1-e_{N_1+n})\Phi(c)-\sum_{k=N_1+n}^\infty \phi_n(c)\|=0.$
So the last part of the lemma follows from passing to a subsequence of $\{e_n\}.$

\end{proof}

\section{{A Voiculescu--Weyl--von Neumann theorem}}

%{\blue{The following may be found in ??}}

\begin{prop}\label{Ptrace}
Let $C$ be a unital separable  amenable \CA\, with $T(C)\not=\emptyset.$ 
For any $\ep>0$ and finite subset ${\cal F}\subset C,$ there exist 
$\dt>0$ and a finite subset ${\cal G}\subset C$ satisfying the following:
For any unital \CA\, $A$ with $T(A)\not=\emptyset,$ and any unital ${\cal G}$-$\dt$-multiplicative \morp\, 
$L: C\to A,$ there exist a continuous affine map $L_\sharp: T(A)\to T(C)$ 
such that
\beq
|\tau(L(c))-L_{\sharp}(\tau)(c)|<\ep\tforal c\in {\cal F}\tand \tau\in T(A).
\eneq

\end{prop}

\begin{proof}
Fix $\ep>0$ and a finite subset ${\cal F}\subset C.$
By Lemma 9.4 of  {{\cite{Linclr1},}}  there exist a $\dt>0$ and a finite subset ${\cal G} \subset C$
such that for any unital \CA\, $A$, 
for any ${\cal G}$-$\dt$-multiplicative \morp\,
$L: C\to A$, 
for each $\tau\in T(A),$ there is a tracial state $\gamma'(\tau)\in T(C)$ 
such that
\beq\label{Ptrace-e-1}
|\tau \circ L(c)-\gamma'(\tau)(c)|<\ep/2\rforal c\in {\cal F}.
\eneq 
Fix any such $A$ and $L.$
By Lemma 8.10 of {{\cite{GLIII},}} 
there exist a finite subset ${\cal T}\subset \partial_e(T(A))$ 
and a continuous affine map
$\lambda: T(A)\to \triangle,$ where $\triangle$ is 
the convex hull of ${\cal T},$ 
such that
\beq\label{Ptrace-e-2}
|\lambda(\tau)(a)-\tau(a)|<\ep/2\rforal a\in L({\cal F})\makebox{ and  } \tau \in T(A).
\eneq
Let  ${\cal T}=\{\tau_1, \tau_2,...,\tau_m\}.$ 
Define a continuous affine map $\gamma_1: \triangle\to T(C)$
by $\gamma_1(\tau_i)=\gamma'(\tau_i)$ (for a choice of $\gamma'(\tau_i)$), $i=1,2,...,m.$
Define $L_\sharp: T(A)\to T(C)$ by 
$L_\sharp=\gamma_1\circ \lambda.$ Then $L_\sharp$ is a continuous affine map.
Moreover, for each $\tau\in T(A),$ let $\lambda(\tau)=\sum_{i=1}^m \af_i(\tau) \tau_i$,
where $\af_i(\tau) \in [0,1]$ for all $i$, and $\sum_{i=1}^m \af_i(\tau) = 1$.  
Then  by \eqref{Ptrace-e-1} and \eqref{Ptrace-e-2},
for all $f \in {\cal F}$ and $\tau \in T(A)$, 
\beq
L_\sharp(\tau)(f)&=&\gamma_1(\sum_{i=1}^m \af_i(\tau) \tau_i)(f)=\sum_{i=1}^m\af_i(\tau)\gamma'(\tau_i)(f)\\
&\approx_{\ep/2}&\sum_{i=1}^m\af_i(\tau)\tau_i(L(f))=\lambda(\tau)(L(f))\approx_{\ep/2} \tau(L(f)).
\eneq
%
%%%%%%%%%%%%%
\iffalse
%
Assume that the proposition fails. 
Then there exists $\ep_0>0$ and a finite subset ${\cal F}\subset C$   satisfying the following:
There exists a sequence of unital \morp s $L_n: C\to A,$  a sequence of unital \CA s $A_n,$ and 
a sequence  $\tau_n\in T(A_n)$ such that 
\beq\label{Ptrace-p-1}
\inf \{\max\{|\tau_n(L_n(c))-\gamma_n(\tau_n)(c)|: c\in {\cal H}\}\}\ge \ep_0,
\eneq
where infimum is taken among all possible continuous affine 
maps $\gamma: T(A_n)\to T(C_n).$
%
\fi
%%%%%%%%%%%%%%%%%%%%%%%%%% 
\end{proof}

%\begin{df}
\begin{df}\label{Ad89}
Let $C$ be a unital C*-algebra, $\delta > 0$, ${\cal G} \subset C$ be a
finite subset, and ${\cal P} \subset \underline{K}(C)$ be a finite subset. 
Let ${\cal H}\subset C_{s.a.}^{\bf 1}\setminus \{0\}$
be a finite subset  and $1>\eta>0.$
The triple $({\cal G}, \dt, {\cal P})$ is called a \emph{$KL$-triple  
associated with 
${\cal H}$ and $\eta$}  if it is a $KL$-triple and  for any unital separable \CA\, $D$ with $T(D)\not=\emptyset$ 
and any unital ${\cal G}$-$\dt$-multiplicative \morp\, 
$L: C\to D,$ there exists a continuous affine map $L_\sharp: T(D)\to T(C)$
such that
\beq
|\tau\circ L(c)-L_\sharp(\tau)(c)|<\eta\rforal c\in {\cal H}\andeqn \tau\in T(D).
\eneq
\end{df}

Note that, by Proposition  \ref{Ptrace},
 for any unital separable amenable \CA\, with $T(C)\not=\emptyset$ and 
any finite subset ${\cal H}\subset C_{s.a.}^{\bf 1}\setminus \{0\}$, 
such a triple $({\cal G}, \dt, {\cal P})$ always exists. 
%Here, we just define a name for 
%$KL$-triple associated with ${\cal H}$ and $\eta$ for convenience.
%\end{df}

%{\blue{??????????
Recall that if $H$ is a group and $E \subseteq H$, we let 
%$G^E$ {\blue{(or $E^g$)}} 
$E^g$
denote the subgroup of $H$ generated by $E$.  We also sometimes use the notation $G^E$ instead of $E^g$ as in
Definition \ref{DefA1} (3) (b).
%}} 
Also, recall that for a C*-algebra $C$,
$K_0(C)^\ro_+ := \{ x \in K_0(C) : \rho_C(x) > 0 \} \cup \{ 0 \}$, and
$C_+^{q, {\bf 1}}$ {{is the image of $C_+^{\bf 1}$ in $\Aff(T(C)).$
%is the set of strictly positive functions in the closed unit
%ball of $\Aff(T(C))_{++}$.   
Finally, recall that for any $h \in C_{s.a.}$, $\hat{h}$ is the induced element  of 
$\Aff(T(C))$, i.e., $\hat{h}(\tau) := \tau(h)$ for all $\tau \in T(C)$.}} 
%\end{df}

{{Next we introduce a class of \CA s ${\cal A}_0.$ 
It will be shown in section 11 that many \CA s  (such as ${\cal Z}$-stable simple \CA s satisfying 
the UCT)
that we are interested 
are in ${\cal A}_0.$    Some formulations  (such as  axiom (3)(c) (Existence 2))  are a bit technical.
However,  they would make  some proofs  in this section easier as we will use these properties. }}

\iffalse
Recall that if $H$ is a group and $E \subseteq H$, we let 
$G^E$ {\blue{(or $E^g$)}} denote the subgroup of $H$ generated by $E$. 
{\blue{Also, recall that for a C*-algebra $C$,
$K_0(C)^\ro := \{ x \in K_0(C) : \rho_C(x) > 0 \} \cup \{ 0 \}$, and
$C_+^{q, {\bf 1}}$ is the set of strictly positive functions in the closed unit
ball of $\Aff(T(C))_{++}$.   
Finally, recall that for any $h \in C_{s.a.}$, $\hat{h}$ is the induced element  of 
$\Aff(T(C))$, i.e., $\hat{h}(\tau) := \tau(h)$ for all $\tau \in T(C)$.}} 
%\end{df}

{\blue{We note that in the next definition (Definition \ref{DefA1}), the
formulation of axiom (3)(c) (Existence 2) is a bit complicated and 
technical.  However, this formulation makes many proofs easier.  Moreover,
in the next section (Section \ref{section:ClassA0}), we will 
 prove that
many C*-algebras satisfy Definition \ref{DefA1} (including (3)(c)).}} 
\fi

\begin{df}\label{DefA1}
Denote by ${\cal A}_0$ the class of unital separable amenable \CA s $C$ {{in ${\cal A}_1$}} which  satisfy the following conditions:

(1) (Inductive limit structure) There is an increasing  sequence of unital separable amenable {{\SCA\,s}}
%C*-subalgebras 
$C_n \subseteq C$ with $1_{C_n}=1_C$, which satisfy the UCT such that 
% {{$K_0(C)$  has property ({\Large{$\varrho$}}),}}
$K_0(C_n)={{G_{0,n}}}
%\Z^{r(n)'}
\oplus {\rm ker}\rho_{C_n}$
is finitely generated, {{where $G_{0,n}$ is a finitely generated free group, 
$G_{0,n}\cap {\rm ker}\rho_{C_n}=\{0\}$}} 
and $K_0(C_n)_+^\rho \cap G_{0,n}$ {{(see \ref{Drho} and the paragraphs before this definition)}}
is finitely generated as
a {{part of a}} positive cone, 
%{{$[1_{C_n}]=(z_n,\zeta)$ for some $z_n\in \Z^{(r(n)'}$ and $\zeta\in {\rm ker}\rho_{C_n},$}}
and $K_1(C_n)$ is also finitely generated, and 
$C=\overline{\cup_{n=1}^\infty C_n}.$  Denote by $\iota_n: C_n\to C_{n+1}$ and $\iota_{n,\infty}: C_n\to C$ the (injective) connecting maps and the induced embedding respectively 
(so that $C=\lim_{n\to\infty}(C_n, \iota_n)$).

%from the inductive limit 
%decomposition of $C$. 

(2) (Uniqueness)  Let $\Delta: C_+^{q, {\bf 1}}\setminus \{0\}\to (0,1)$ be a non-decreasing map.  For any $\ep>0$ and 
any finite subset ${\cal F}\subset C,$ there are a finite subset ${\cal G}\subset C,$ $\dt>0,$ a finite 
subset ${\cal P}\subset \underline{K}(C),$ a finite subset ${\cal H}_q\subset C_+^{\bf 1}\setminus \{0\},$ 
$\eta>0,$ 
and a finite subset ${\cal H}\subset C_{s.a.}$ such that for any unital
separable simple  {{\CA\,}}
% C*-algebra 
A with tracial rank zero, for any 
two  ${\cal G}$-$\dt$-multiplicative \morp s 
 $L_1, L_2: C\to A$ which satisfy 
\beq
&&[L_1]|_{\cal P}=[L_2]|_{\cal P},\\
&&|\tau(L_1(x))-\tau(L_2(x))|<\eta\rforal x\in {\cal H}\andeqn\\
&&\tau(L_1(h))\ge \Delta(\hat h)\hspace{0.4in}\rforal h\in {\cal H}_q,
\makebox{  } \tau \in T(A),
\eneq
there is a unitary $u\in A$ such that
\beq
\|u^*L_1(a)u-L_2(a)\|<\ep\tforal a\in {\cal F}.
\eneq
Moreover, for each $n$,
 if ${\cal F}\subset \iota_{n, \infty}(C_n),$ then we may assume 
that ${\cal G}\subset \iota_{n, \infty}(C_n),$ ${\cal P}\subset [\iota_{n, \infty}](\underline{K}(C_n)),$
${\cal H}_q\subset \iota_{n, \infty}({C_n}_+^{\bf 1}\setminus \{0\})$ 
and ${\cal H}\subset \iota_{n, \infty}({C_n}_{s.a.}).$

(3) (Finite generation and existence)
\iffalse 
There is an increasing  sequence of unital separable amenable 
\CA s $C_n$ with $1_{C_n}=1_C$ which satisfy the UCT such that $K_0(C_n)=\Z^{r(n)}\oplus {\rm ker}\rho_{C_n}$
is finitely generated, and $K_0(C_n)_+\cap \Z^{r(n)'}$ is finitely generated as positive cone, 
and $K_1(C_n)$ is also finitely generated, and 
$A=\overline{\cup_{n=1}^\infty C_n}.$
Denote by $\iota_n: C_n\to C_{n+1}$ and $\iota_{n,\infty}: C_n\to C$ be the embeddings. 
\fi
For each $n$, 
there is a finite subset ${\cal Q}_n\subset \underline{K}(C_n)$ 
such that for any 
%C*-algebra 
{{\CA\,}} $A$,  any  $\kappa\in {\rm Hom}_{\Lambda}(\underline{K}(C_n), \underline{K}(A)),$
$\kappa|_{{\cal Q}_n}$ determines $\kappa$ {{(this always holds as $K_i(C_n)$ 
is finitely generated---see  Cor. 2.11 \cite{DL}).}}
There is a finite subset $K'_{n,p}\subset K_0(C_n)_+\setminus \{0\}$ such that
%$K_0(C_n){{_+}}$ {\red{and 
{{$K_0(C_n)$ is generated by $K'_{n,p},$ 
and, for any $x\in  K_0(C_n)_+,$ there  exist  
%$z\in \Z^{r(n)},$  $\zeta\in {\rm ker}\rho_{C_n},$ 
$r_1, r_2, ..,r_{m(x)}\in  \N,$ 
$s_1,s_2,...,s_{n(x)}\in \Z$ and $x_1, x_2,..., x_{m(x)},
 y_1,y_2,...,y_{n(x)}
\in K_{n,p}'$ 
such  that 
\beq
%x=(z, \zeta), z=
x=\sum_{i=1}^{m(x)}r_i x_i+\sum_{j=1}^{n(x)} s_j  y_j
%{\ppl{(x_j-y_j),}}
\eneq
where 
%$x_j-y_j
$\xi=\sum_{j=1}^{n(x)} s_j  y_j\in {\rm ker}\rho_{C_n},$ }}
%
%{\Green{ and $K_{n,p}'$   contains a generating set of $K_0(C_n)_+\cap \Z^{r(n)'}$ (as a part of positive cone),}}
% 
{{and  $[1_{C_n}]\in K_{n,p}'.$}} We may assume that $K'_{n,p} 
\subseteq {\cal Q}_n$ and  ${{[\iota_n]}}(K'_{n,p} )\subseteq
K'_{n+1,p}$ for all $n$.   
{{(Warning: the statement above does not mean that 
$K_0(C_n)_+$ is finitely generated as positive cone.)}}
%We also assume that $[1_C]\in K_{n,p}.$

For each $n$, 
let ${\cal P}_n=[\iota_{n, \infty}]({\cal Q}_n),$  $K_{n,p} = [\iota_{n, \infty}](K'_{n,p}),$ 
and $G_n$ be the subgroup generated by ${\cal P}_n\cap K_0(C).$ Note that $G_n$ is generated by
$K_{n,p}$; and  $G_n = \iota_{n, \infty}(K_0(C_n))$.    We have that 
 $[1_C]\in K_{n,p}$   %\subset {\cal P}_n$
 and ${\cal P}_n\subset {\cal P}_{n+1}$ for all $n$,
and
% {\blue{ we have that }} 
$\cup_{n=1}^\infty{\cal P}_n=\underline{K}(C).$ 
Note that by definition, $K_{n,p} \subseteq K_{n+1,p}$ for all $n$. 
We may assume that $l(n)\in \N$ such that there is a finite subset $\mathtt{P}_n\subset M_{l(n)}(C)$ 
such that $K_{n,p}=\{[p]: p\in \mathtt{P}_n\}.$
{{It should be noted this first part of (3) always holds.}}

%There are an increasing sequence of finite subsets ${\cal P}_n\subset \underline{K}(C)$ and 
%an increasing  sequence of finite subsets $K_{n,p}\subset K_0(C)_+\setminus \{0\}$  with $[1_C]\in K_{n,p}$
%satisfying the following conditions:

 (a)  We have that $K_0(C)=\cup_{n=1}^\infty G_n$,  and for all $n$, 
 %where $[1_C]\in G_n$ is the subgroup generated by ${\cal P}_n\cap K_0(C),$
we assume that
 $G_n=\Z^{r(n)}\oplus G_{o,n}$ where $G_{o,n}=G_n\cap {\rm ker}(\rho_{C}),$   and 
 {{$\Z^{r(n)}\cap {\rm ker} \rho_C=\{0\}.$}}
 %and  $K_{n,p}$ contains a generating set of $\Z^{r(n)}\cap K_0(C)_+^\ro.$}}
 %and $K_{n,p}$ 
 %generates  $G_n.$
 %$\Z^{r(n)}.$ 

\textbf{Note:} Recall that by Remark \ref{Rrho-1},      
${\rm ker}\rho_C = {\rm ker} \rho_{C,f}$.  Hence, in the descriptions 
of $G_n$ and $G_{o,n}$ in the previous paragraph, every occurrence of 
${\rm ker}\rho_C$ can be replaced with ${\rm ker}\rho_{C,f}$.

(b) (Existence 1: K theory)
Fix $n \geq 1$.
Let ${\cal G} \subset \iota_{n,\infty}(C_n)$ be a finite subset
and $\delta > 0$ such that  
$({\cal G}, \delta, {\cal P}_n)$ is a KL triple.  
Let $A$ be a unital separable simple 
%C*-algebra 
{{\CA\,}} with tracial rank zero, and
let $\kappa_n\in KL_{loc} (G^{{\cal P}_n\cup K_{n,p}}, \underline{K}(A))$ 
with 
$\kappa_n([1_C])=[1_A]$,  and $\kappa_n(g)\in K_0(A)_+\setminus \{0\}$ for all $g\in K_{n,p},$
and $\kappa_n(G_{o,n})\subset {\rm ker}\rho_A.$ 
Then there exists a unital ${\cal G}$-$\dt$-multiplicative \morp\,
$L : C \rightarrow A$ such that 
\beq
&&\hspace{-0.4in}
[L]|_{{\cal P}_n} ={\kappa_n}|_{{\cal P}_n}.
\eneq

(Note that since $K_{n,p}  \subseteq {\cal P}_n$, $G^{{\cal P}_n \cup K_{n,p}}
= G^{{\cal P}_n}$.) 

To simplify notation,
we will fix an increasing sequence $\{ \G_n \}$ of finite subsets of
$C$ and a strictly decreasing sequence $\{ \delta_n \}$ in $(0,1)$ 
such that $\G_n \subset \iota_{n, \infty}(C_n)$ for all $n$,
$\bigcup_{n=1}^{\infty} \G_n$ is dense in $C$, and
$\sum_{n=1}^{\infty} \delta_n < 1$.  Note that for each $n$,
{{we assume that 
% the above
there exists a unital ${\cal G}_n$-$\dt_n$-multiplicative \morp\,
$L_n : C \rightarrow A$ such that 
%\beq
%&&&\hspace{-0.4in}
$[L_n]|_{{\cal P}_n} ={\kappa_n}|_{{\cal P}_n}.$}}
%\eneq

We further impose the requirement that $(\G_n, {{2}}\delta_n, \Pp_n)$
is a KL triple, for all $n$.

(c) (Existence 2: K theory and traces) Fix $n \geq 1$.  Let ${\cal H}\subset \iota_{n, \infty}({C_n}_{s.a.}^{\bf 1}\setminus \{0\})$ be a finite subset  
and $\eta>0.$  Suppose that that $({\cal G}_n, 2\dt_n, {\cal P}_n)$ is a $KL$-triple 
{{(see \ref{DefKL-triple} and \ref{Ad89})}} associated with ${\cal H}$ and $\eta/2.$ 

 For  any finite set ${\cal G}_0\supset {\cal G}_n \cup {\cal H},$
 %\iota_{n, \infty}({\cal H})$ 
 and $0<\dt_0<\dt_n,$
for any unital separable simple
% C*-algebra 
\CA\, $A$ with tracial rank zero,  
% there is $1>\sigma_0>0$ satisfying the following: 
if $\phi_1: C\to A$ is a unital ${\cal G}_0$-$\dt_0$-multiplicative \morp\,
such that $[\phi_1]$ induces an element in $KL_{loc}(G^{{\cal P}_n\cup K_{n,p}}, \underline{K}(A))$,  
and  
%$[\phi_1(x)]\in K_0(B)_+\setminus \{0\}$ 
there exists $\sigma_0 > 0$
% {\ppl{and $m_0\in \N$}} 
such that for all $x\in K_{n,p},$ 
\beq\label{608d-0}
\rho_A([\phi_1](x))(\tau)>\sigma_0\rforal \tau\in T(A),\,\,%{\rm for\,\, some}\,\, \sigma_0>0,
\eneq
%there exists a continuous affine map $(\phi_1')_T: T(B)\to T(C_n)$ such that
and $[\phi_1](G^{{\cal P}_n\cup K_{n,p}}\cap {\rm ker}\rho_{C,f})\subset {\rm ker}\rho_A,$ 
%and 
%\beq
%\rho_B([\phi_1](x))(\tau)=\rho_{C_n}(x)((\phi_1')_T(\tau)\rforal x\in K_{n,p}\andeqn\\
%|\tau(L(h))-L_\sharp (\tau)(h)|<\dt_n\rforal h\in {\cal H},
%\eneq
%
%for any finite subset ${\cal H}\subset (C_n)_{s.a.}^{\bf1},$ 
%A_{s.a}^{\bf 1},$
%there exists $0<\sigma<1,$ 
and 
if $\kappa_0\in KL_{loc}(G^{{\cal P}_n\cup K_{n,p}}, \underline{K}(A))$ such that 
$\kappa_0(K_{n,p})\subset K_0(A)_+\setminus \{0\},$  ${\kappa_0}(G^{{\cal P}_n\cup K_{n,p}}\cap {\rm ker}\rho_{C,f})
\subset {\rm ker}\rho_A$, 
%and there is an $\af_0 \in (0, 1/2)$  {{such that}}
%where 
%for all $x \in K_{n,p}$, %$\kappa_0 < [\phi_1]$ on $K_{n,p}$, and 
\beq\label{608d}
%\|\rho_A(\kappa_0(x))\|<\af_0\min\{1/2, \sigma_0\},  
\|\rho_A(\kappa_0([1_C]))\|<\af_0\min\{1/2,\sigma_0\} 
\eneq
{{for some $\af_0\in (0, 1/2l(n))$}} 
and $p_0\in A$  is a projection such 
that $[p_0]=[\phi_1(1_C)]-\kappa_0([1_C]),$ 
then  there exists a unital ${\cal G}_n$-$\dt_n$-multiplicative \morp\, 
$\psi: C\to p_0 A p_0$ such that
\beq
&&[\psi]|_{{\cal P}_n}=([\phi_1]-\kappa_0)|_{{\cal P}_n}\andeqn\\
&&|\tau(\psi(h))-\tau(\phi_1(h))|<2\af_0+\eta\rforal h\in {\cal H}\andeqn \tau\in T(A).
\eneq
\end{df}

\begin{rem}\label{R810}
{{Note that if  
\eqref{608d}  holds, then, for all $x\in K_{n,p},$
\beq\label{608d-1}
\|\rho_A(\kappa_0(x))\|\le l(n) \|\rho_A(\kappa_0([1_C])\|<\sigma_0/2,
\eneq
and  thus if  \eqref{608d-0} also holds, then 
\beq 
[\phi_1](x)-\kappa_0(x)>0\rforal x\in K_{n,p}
\eneq}}

{{To see \eqref{608d-1} holds, let $x\in K_{n,p},$ 
%where $z\in \Z^{r(n)}\setminus \{0\}$ and $\zeta\in {\rm ker}\rho_{C}.$
We assume that $x\not=l(n)[1_C].$ }}
% Write $[1_C]=(z_0, \zeta_0),$ where $z_0\in \Z^{r(n)}\cap K_0(C)_+^\ro\setminus \{0\}$ and 
%$\zeta_0\in {\rm ker}\rho_C.$}} 
{{Thus, 
by the choice of $l(n),$ $x$ is represented by a projection in $M_{l(n)}(C). $ 
Hence $l(n)[1_C]-x\ge 0.$
%$l(n)[1_C]-x\in K_0(C)_+\setminus \{0\}.$   In other words, 
%$(l(n)[1_C]-z,\zeta_0-\zeta)\in K_0(C)_+.$ 
By the assumption  of $K_{n,p}',$ 
%and \ref{Rrho-1},  
there are $x_1, x_2,...,x_l, y_1, y_2,...,y_m\in K_{n,p}$ and 
integers $r_1, r_2,...,r_l\in \N,  s_1, s_2,...,s_m\in \Z$ such that  
\beq
l(n)[1_C]-x=\sum_{i=1}^l r_i x_i+\sum_{j=1}^m s_j y_j\andeqn
\zeta := \sum_{j=1}^m s_j y_j\in {\rm ker}\rho_C.
\eneq  
  Since $\rho_A(\kappa_0(\zeta))=0,$ it follows that 
  \beq
  \rho_A\circ \kappa_0((l(n)[1_C]-x)=\sum_{i=1}^l r_i \rho_A\circ  \kappa_0(x_i)>0.
  \eneq}}  
%Since $K_0(A)$ has property ({\Large{$\varrho$}}), 
Hence 
$0<\rho_A(\kappa_0(x))<l(n)\rho_A(\kappa_0([1_C])).$ Thus \eqref{608d-1} holds.
%So $\rho_A(\kappa_0(\zeta))=0.$ 

%(d) For each $n,$ if $x=f\oplus z\in \Z^{r(n)}\oplus G_{o,n}$ such that $x\in K_0(C)_+\setminus \{0\}.$
%Then, there is an integer $M\ge 1$ such that, for integer $L\ge M,$ $Lf\oplus (\pm z)\in K_0(C)_+\setminus \{0\}.$

%For any $\kappa\in {\rm Hom}_{\Lambda}(\underline{K}(C), \underline{K}(B))^{++},$ with $\kappa([1_C])=[1_B],$ 
%for any unital simple \CA\, $B$ with real rank zero, stable rank one and with weakly unperforated 
%$K_0(B),$  and, for any continuous affine map $\kappa_T: T(B)\to T_f(C)$ which is compatible with $\kappa,$ 
%there is a unital \hm\, $\phi: A\to B$ such that $[\phi]=\kappa$ and $\phi$ induces $\phi_T=\kappa_T: T(B)\to T_f(C).$ 

\end{rem}

{{
\begin{prop}\label{10-3directsum}
Finite direct sums of \CA s in ${\cal A}_0$ are in ${\cal A}_0.$
\end{prop}

\begin{proof}
Let  $D_1, D_2,\cdots, D_m\in {\cal A}_0.$
Put $C=\bigoplus_{j=1}^m D_j.$ 
Since each $D_j$ satisfies (1), $C$ also satisfies (1).
Hence we may write $C=\overline{\cup_{n=1}^\infty C_n}$ as in (1) of \ref{DefA1}.
Moreover, each $C_n$ may be obtained as $C_n=\bigoplus_{j=1}^m C_{j,n},$
where $C_{j,n}$ are given as \SCA s of $D_j$  by the assumption that each 
$D_j$ is in ${\cal A}_0.$ 
It is also clear that $C$ satisfies (2). 

For condition (3), we first note that ${\cal Q}_n$  and ${\cal P}_n$ can be easily obtained 
from each $D_j.$   Moreover, we may assume that $[1_{D_j}]\in K_{n, p}'$
with $[1_{C}]=[1_{D_1}]+[1_{D_2}]+\cdots [1_{D_m}].$ 
We  also note that $C$ 
satisfies  condition (a). 
We may assume that 
${\cal P}_n=\cup_{j=1}^m {\cal P}_{j,n},$
where ${\cal P}_{j,n}=\iota_{n, \infty}({\cal Q}_{j,n})$ and ${\cal Q}_n$ is a finite subset of $\underline{K}(C_{j,n}),$ $1\le j\le m.$

Now let $A$ be a unital separable simple \CA\, of tracial rank zero and let 
$\kappa_n\in \in KL_{loc} (G^{{\cal P}_n\cup K_{n,p}}, \underline{K}(A))$ 
with 
$\kappa_n([1_C])=[1_A]$,  and $\kappa_n(g)\in K_0(A)_+\setminus \{0\}$ for all $g\in K_{n,p},$
and $\kappa_n(G_{o,n})\subset {\rm ker}\rho_A.$ 

Choose mutually orthogonal projections $e_1, e_2,...,e_m$  in $A$ such that
$[e_i]=\kappa_n([1_{D_{i,n}}]),$ $i=1,2,...,m.$  This is possible since $A$ has tracial rank zero., 
Note that $\kappa_n([1_C])=[1_A].$ Hence we may assume that 
$\sum_{i=1}^m e_i=1_A.$   Since each $D_i$ is in ${\cal A}_0,$ we conclude that $C$ also satisfies 
condition (b) of (3) in \ref{DefA1}. 

To see (c) also holds, let  ${\cal H}\subset \iota_{n, \infty}({C_n}_{s.a.}^{\bf 1}\setminus \{0\})$ be a finite subset  
and $\eta>0.$  Suppose that that $({\cal G}_n, 2\dt_n, {\cal P}_n)$ is a $KL$-triple associated with ${\cal H}$ and $\eta/2.$   We may write 
${\cal G}_n= \cup_{j=1}^m\iota_{n,\infty}(G_{j,n})$
and ${\cal H}=\cup_{i=1}^m {\cal H}_i,$ where  $G_{j,n}\subset  \iota_{n, \infty}(C_{j,n})$
is a finite subset and where 
${\cal H}_i\subset \iota_{n, \infty}({C_{j,n}}_{s.a.}^{\bf 1}\setminus \{0\}),$ $i=1,2,...,m.$ 

Let ${\cal G}_0\subset C$ be a finite subset which contains ${\cal G}_n\cap H.$ 
We may write ${\cal G}_0=\cup_{i=1}^m {\cal G}_{0,i},$ where 
${\cal G}_{0,i}\subset D_i$ be a finite subset which contains ${\cal G}_{0,i}\cap H_i,$
$i=1,2,...,m.$  Let $\sigma_{0,j}$ and $m_{0,j}$ be given by
(c) of (3) in \ref{DefA1} as each $D_j$ is in ${\cal A}_0.$
Choose $\sigma_0=\min\{\sigma_{0,j}: 1\le j\le m\}.$ 
%and $m_0=m \max\{m_{0,j}: 1\le j\le m\}.$  
Let $0<\dt_0<\dt_n$ 
and  $A$ be any 
unital separable simple
% C*-algebra 
 with tracial rank zero. 
% there is $1>\sigma_0>0$ satisfying the following: 
Let  $\phi_1: C\to A$ is a unital ${\cal G}_0$-$\dt_0$-multiplicative \morp\,
such that $[\phi_1]$ induces an element in $KL_{loc}(G^{{\cal P}_n\cup K_{n,p}}, \underline{K}(A))$,  
%and  
%$[\phi_1(x)]\in K_0(B)_+\setminus \{0\}$ 
%there exists $\sigma_0 > 0$ {\ppl{and $m_0\in \N$}} such that 
for all $x\in K_{n,p},$ 
\beq\label{n608d-0}
\rho_A([\phi_1](x))(\tau)>\sigma_0\rforal \tau\in T(A),\,\,%{\rm for\,\, some}\,\, \sigma_0>0,
\eneq
and $[\phi_1](G^{{\cal P}_n\cup K_{n,p}}\cap {\rm ker}\rho_{C,f})\subset {\rm ker}\rho_A,$ 
and 
if $\kappa_0\in KL_{loc}(G^{{\cal P}_n\cup K_{n,p}}, \underline{K}(A))$ such that 
$\kappa_0(K_{n,p})\subset K_0(A)_+\setminus \{0\},$  ${\kappa_0}(G^{{\cal P}_n\cup K_{n,p}}\cap {\rm ker}\rho_{C,f})
\subset {\rm ker}\rho_A$, 
%and there is an $\af_0 \in (0, 1/2)$  {{such that}}
%where 
%for all $x \in K_{n,p}$, %$\kappa_0 < [\phi_1]$ on $K_{n,p}$, and 
\beq\label{n608d}
%\|\rho_A(\kappa_0(x))\|<\af_0\min\{1/2, \sigma_0\},  
\|\rho_A(\kappa_0([1_C]))\|<\af_0\min\{1/2,\sigma_0\} 
\eneq
{{for some $\af_0\in (0, 1/2l(n))$}} 
and $p_0\in A$  is a projection such 
that $[p_0]=[\phi_1(1_C)]-\kappa_0([1_C]).$ 

We may assume that $\phi_1(1_{D_1}), \phi_1(1_{D_2}),...,\phi_1(1_{D_m})$ 
are mutually orthogonal projections.  Put $q_j=\phi_1(1_{D_j}),$ $1\le j\le m.$ 
Denote by $A_j=q_jAq_j,$ $1\le j\le m.$ 
%Choose a projection $p_j\le A_j$ such that 
Consider $[q_j]-\kappa_0([1_{D_j}]).$
Note 
\beq
\rho_A([q_j])-\kappa_0([1_{D_j}])\ge \sigma_0-\af_0\min\{1/2, \sigma_0\}>0.
\eneq
One then chooses projections $p_j\in A_j$ 
such that $[p_i]=[\phi_1(1_{D_j})]-\kappa_0([1_{D_j}]),$ $1\le j\le m.$  
Since each $D_j$ is in ${\cal A}_0,$ there exists a unital ${\cal G}_{j,n}$-$\dt_n$-multiplicative \morp\, 
$\psi_j: D_j\to A_j$ such that
\beq\label{105-a}
&&[\psi_j]|_{{\cal P}_{j,n}}=([\phi_1]-\kappa_0)|_{{\cal P}_{j,n}}\andeqn\\\label{105-b}
&&|\tau(\psi_j(h))-\tau(\phi_1(h))|<2\af_0+\eta\rforal h\in {\cal H}_j\andeqn \tau\in T(A_j).
\eneq
Define $\psi: C\to A$ by 
$$
\psi(c_1\oplus c_2\oplus \cdots \oplus c_m)=\oplus_{i=1}^m \psi_i(c_i)
$$
for $c_j\in D_j,$ $1\le j\le m.$ 
Then, by  \eqref{105-a},  
$$
[\psi]|_{{\cal P}_n}=([\phi_1]-\kappa_0)|_{{\cal P}_{n}}
$$
Let  $\tau\in T(A)$ and $\bt_j=\|\tau|_{A_j}\|,$ $1\le j\le m.$
Let $h=\oplus_{j=1}^m h_j\in {\cal H},$ where $h_j\in {\cal H}_j,$
$1\le j\le m.$ Then, by \eqref{105-b},
\beq
|\tau(\psi(h))-\tau(\phi_1(h))|&\le& \sum_{j=1}^m|\tau(\psi_j(h_j))-\tau(\phi_1(h_j))|\\
&\le & \sum_{j=1}^m \bt_j(2\af_0+\eta)=2\af_0+\eta.
\eneq
This completes the proof.

\end{proof}
}}

\begin{prop}\label{Plowerbd}
Let $C$ be a 
%{\blue{unital(?)}} 
\CA, $\Delta: C_+^{q, {\bf 1}}\to (0,1)$ be a nondecreasing map,
and ${\cal Q}\subset K_0(C)_+\setminus \{0\}$ be a finite subset.
%
%%%%%%%%%%%%%%%%%%%%
\iffalse
(1) Then there is $H_q\subset A_+\setminus \{0\}$ satisfies the following:
If $\phi: C\to B$ (for any \CA\, $B$ with $T(B)\not=\emptyset$) is a \morp\, 
such that
\beq
\tau(\phi(h))\ge \Delta(\hat{h})\tforal h\in H_q\tand \tau\in T(B),
\eneq
then, for some projections $p_1,p_2,...,p_n\in M_R(A),$  one has  that
${\cal Q}=\{[p_1], [p_2],...,[p_n]\}$ and 
\beq
\inf\{\tau(\phi(p_i)): \tau\in T(B)\}\ge \inf\{\Delta(\hat{h}): h\in H_q\}>0, 1\le i\le n.
\eneq
\fi
%%%%%%%%%%%%%%%
%
Then there exist a finite subset 
${\bar  H}_q\subset C_+^{\bf 1}\setminus \{0\}$, a  $\dt>0,$  and a finite subset ${\cal G}\subset C$ satisfying the
following: 
For any C*-algebra $A$ with $T(A) \not= \emptyset$,
if $\phi: C\to A$ is a 
${\cal G}$-$\dt$-multiplicative \morp\, 
such that
\beq
\tau(\phi(h))\ge \Delta(\hat{h})\tforal h\in {\bar H}_q\tand \tau\in T(A),
\eneq
then 
\beq
\inf\{\rho_A(\phi_{*0}(x))(\tau) : \tau\in T(A)\}\ge \min\{\Delta(\hat{h}): h\in {\bar H}_q\}>0.
\eneq
for all $x\in {\cal Q}.$
\end{prop}

\begin{proof}
This is known and has been used many times. 
For the convenience of the reader, let us briefly sketch the proof.   
Suppose that  $p_1, p_2,...,p_n\in M_R(C)$  are projections, for some 
$R\in \N$, such that ${\cal Q}=\{[p_1], [p_2],...,[p_n]\}.$   
For each $1 \leq k \leq n$, 
we may write $p_k=(a_{i,j}^{(k)})_{1\le i, j\le R},$ where 
$a_{i,j}^{(k)}\in C$ for all $i,j$.   Note that $a_{i,i}^{(k)}\ge 0$ for all
$1 \leq i \leq R$.  
Assume that $a_{i,i}^{(k)}\in {\bar H}_q$ for all $1 \leq i \leq R$. 
{{Then, for each $1 \leq k \leq n$ and}}  for all $\tau \in T(A)$,  
\beq
\tau(\phi(p_k))=\sum_{i=1}^R \tau(\phi(a_{i,i}^{(k)}))\ge \inf\{\Delta(\hat{h}): h\in {\bar H}_q\}>0.
\eneq
The proposition follows from the above observation. 
\end{proof}

\begin{lem}\label{Lpartuniq}
Let $C$ be a {{\CA\,}}  in the class ${\cal A}_0$,
 and let {{$C_k$, ${\cal P}_k,$ $K_{k,p}$  and $G_{o,k}$, for $k \geq 1$,}}
satisfy the conditions of Definition 
\ref{DefA1}.

Fix $n \geq 1$ and   a nondecreasing map $\Delta: C_+^{q,{\bf 1}}\to (0,1).$ 
Let $\eta > 0$, $1>\dt>0$ and ${\cal G}\subset \iota_{n, \infty}(C_n)$ 
%C_n$ 
be a finite subset such that {{$({\cal G}, \dt, {\cal P}_n)$}}
%$(\iota_{n, \infty}({\cal G}), \dt, {\cal P}_n)$ 
is a KL-triple.
% for some integer $n\ge 1$ and let $1>\eta,\,\sigma>0.$
%, \sigma>0.$
Then %{{(for each fixed $n$),}} 
there exist $0<\dt_1<\dt,$  a finite subset ${\cal G}_1\subset C$ (which contains  {{${\cal G}$}})
%$\iota_{n, \infty}({\cal G})$) 
and a finite subset {{${\cal H}_q\subset{ C}_+^{\bf 1}\setminus \{0\}$}} 
%${\cal H}_q'\subset{ C_n}_+^{\bf 1}\setminus \{0\}$  
%and $\sigma\in (0,1)$ 
satisfying the following  conditions:

For every unital separable simple {{\CA\,}} $A$ with tracial rank zero,  if $\phi: C\to A$ is a unital ${\cal G}_1$-$\dt_1$-multiplicative \morp\,
such that  $[\phi](G_{o,n})\subset {\rm ker}\rho_A$ and 
\beq
\tau(\phi(h))\ge \Delta(\hat{h})\tforal h\in {\cal H}_q
%:=\iota_{n, \infty}(H_q'),
\eneq
and $\kappa_0\in KL_{loc}(G^{{\cal P}_n\cup K_{n,p}}, \underline{K}(A))$ 
such that $\kappa_0(K_{n,p})\subset K_0(A)_+\setminus \{0\}$ 
and 
$\kappa_0(G_{o,n})\subset {\rm ker}\rho_A,$
then there are non-zero mutually orthogonal projections $e_0, e_1\in A$ and  unital
${\cal G}$-$\dt$-multiplicative \morp s $\psi_0: C\to e_0Ae_0$ and $\psi_1:C\to e_1Ae_1,$
and a unitary $u\in A$
such that
\beq
&&e_0+e_1=1_A, \,\,\, \tau(e_0)<\sigma\tforal \tau\in T(A),\\
&&{[\psi_0]}|_{{\cal P}_n\cap {\rm ker}\rho_{C,f}}={\kappa_0}|_{{\cal P}_n\cap {\rm ker}\rho_{C,f}},\,\,\,\\
&&{[}\psi_0{]}|_{{\cal P}_n\cap K_1(C)}={\kappa_0}|_{{\cal P}_n\cap K_1(C)},\\
&&{[\psi_0]}|_{{\cal P}_n\cap K_i(C, \Z/k\Z)}={\kappa_0}|_{{\cal P}_n\cap K_i(C, \Z/k\Z)},\,\,\, i = 0,1 \makebox{  and  }k=2,3,...,\\
&&{{[\psi_0]|_{{\cal P}_n}+[\psi_1]|_{{\cal P}_n}=[\phi]|_{{\cal P}_n}\tand}}\\
&&\|u^*\phi(c)u-(\psi_0(c)+\psi_1(c))\|<\eta\rforal c\in \iota_{n, \infty}({\cal G}).
\eneq

\end{lem}

\begin{proof}

%Fix $N\geq 1$. 
{{Recall that}} $C=\overline{\cup_{k=1}^\infty C_k}$   
Fix $\dt>0$ and a finite subset ${\cal G}\subset \iota_{n,\infty}(C_n)$
and
a finite subset ${\cal P}_n\subset [\iota_{n, \infty}](\underline{K}(C_n))$
and finite subset $K_{n,p}\subset K_0(C)_+\setminus \{0\}$
as in the statement of the lemma.  
{{Recall that  we}}  assume that  {{$[1_C]\in K_{n,p}=[\iota_{n, \infty}](K_{n,p}')$
(which also satisfies the condition described in Definition 
\ref{DefA1}),
}} ${\cal G}=\iota_{n, \infty}({\cal G}'),$ 
where ${\cal G}'\subset C_n$ is a finite subset,
and ${\cal H}_q\subset \iota_{n, \infty}({\cal H}_q'),$ where ${\cal H}_q'\subset {C_n}_+^{\bf 1}\setminus \{0\}.$
We may further assume that 
${\cal G}'\subset {C_n}_{s.a.}^{\bf 1}.$

%Recall that for all $n \geq 1$, objects like $C_n,
%{\cal Q}_n, {\cal P}_n,
%G_n, G_{0,n}$ etc. are those fixed objects from Definition \ref{DefA1}.  
%\Wlog, we may assume that ${\cal G}\subset (C_n)_{s.a.}^{\bf 1}.$
Let $\eta, \sigma>0$ be given.
Let $\Delta_1:=(3/4)\Delta.$
Let ${\bar H}_q\subset C_+\setminus \{0\}$ be a finite subset, $\delta'_1 > 0$
and ${\cal G}'_1 \subset C$ be a finite subset  
given by Proposition \ref{Plowerbd}, for $K_{n,p}$ in place of ${\cal Q}$. 

Set  ${\bar \sigma}=\min\{\Delta(\hat{h}): h\in {\bar H}_q\}>0.$ 

Let ${\cal G}_1'$ (in place of ${\cal G}$), {{$\dt_1''$}} (in place of $\dt$), ${\cal P}\subset [\iota_{n, \infty}](\underline{K}(C_n)),$ 
${\cal H}_q\subset \iota_{n, \infty}({C_n}_+^{\bf 1}\setminus \{0\}),$  ${\cal H}\subset \iota_{n, \infty}({C_n}_{s.a.})$
and $\eta_0>0$ (in place of $\eta$)  be as in the condition (2) of \ref{DefA1} associated with
$\min\{\dt/2, \eta/2\}$ (in place of $\ep$), ${\cal G}$ 
(in place of ${\cal F}$)  and $\Delta_1$ (in place of $\Delta$). 
%Let ${\cal H}_q'\subset {C_n}_+^{\bf 1}\setminus \{0\}$ be a finite subset such that
%$\iota_{n, \infty}({\cal H}_q').$
We may choose $\eta_0<\eta/4.$ 
Put ${\cal H}_1={\cal H}_q\cup {\cal H}.$  By choosing smaller  {{$\dt_1''$}} and large ${\cal G}_1,$ we may further assume 
that $({\cal G}_1, {{\dt_1',}} {\cal P}_n)$ is a $KL$-triple associated with ${\cal G}$ and $\eta/4.$
{{We may also assume that 
%$\dt_1''\le \min\{\dt, \dt_1'\}$ and 
${\cal G}_1$ is larger than ${\cal G}\cup {\cal G}_1'.$ Let $\dt_1=\min\{\dt, \dt_1', \dt_1''\}.$}}

\iffalse
Let ${\cal G}$, $\dt$, ${\cal P}$, ${\cal H}_q$, ${\cal H}$,  
%${\cal P}\subset [\iota_{N, \infty}](\underline{K}(C_N)),$ 
%${\cal H}_q\subset \iota_{N, \infty}({C_N}_+^{\bf 1}\setminus \{0\}),$  ${\cal H}\subset \iota_{N, \infty}({C_N}_{s.a.})$
and $\eta_0>0$ (in place of $\eta$)  be as in the condition (2) of Definition
 \ref{DefA1} associated with
$\min\{\dt_N/100, \eta/2\}$ (in place of $\ep$), $\G_N$ (in place of ${\cal F}$)  and $\Delta_1$ (in place of $\Delta$). 
%Let ${\cal H}_q'\subset {C_n}_+^{\bf 1}\setminus \{0\}$ be a finite subset such that
%$\iota_{n, \infty}({\cal H}_q') = {\cal H}_q.$
We may assume that $\eta_0<\eta/4.$ 
Let ${\cal H}_1={\cal H}_q\cup {\cal H}.$  By choosing smaller $\dt$,  larger ${\cal G}$ and by small perturbations of the elements of $\G$ if necessary, we may further assume 
that $\delta < \delta'_1$, $\G'_1 \subset \G$,
and there exists $n \geq N$ such that $\G_n = \G$,
$\delta_n = \delta$, ${\cal P}_n = {\cal P}$, and $({\cal G}, \dt, {\cal P}_n)$ is a $KL$-triple associated with 
${\cal H}_q$ and $\eta/4$.  Of course, the rest of the argument will also
work for bigger $n$ (and hence, for sufficiently large $n$).  
\fi 

We may also assume that ${\bar H}_q\subset {\cal H}_q,$ 
%$$\dt_1 < \min\{ \dt,
%\dt'_1 \}$  and 
%and $K_{n,p} \cup {\cal G}'_1 \subseteq {\cal G}_1$.
% is larger than ${\cal Q}$ in Proposition \ref{Plowerbd}
%corresponding to ${\cal Q}=K_{n,p},$  respectively. 
%%%%%%%%%%%%%%%
%
\iffalse
Let $1>\sigma_{000}>0$ be the number $\sigma$ in part (d) of \ref{DefA1} associated with 
$\iota_{n, \infty}({\cal G})$ (in place of ${\cal G}_n$), $\dt_n,$ ${\cal P}_n,$ 
%$K_{n,p}$ 
${\cal H}_1$ (in place of ${\cal H}$) and $\eta/2.$
\fi
%%%%%%%%%%%%%%%%%%%%
%
Let 
\beq
\delta_0:=(1/16)\inf\{\Delta(\hat{h}): h\in {\cal H}_q\}>0.
\eneq

Let $p_1, p_2,..,p_l\in M_{{l(n)}}(\iota_{n, \infty}(C_n))$ be projections 
{{(with ${{l(n)}}\ge 1$) as in Definition \ref{DefA1} (3)}} 
%{\Green{Ping:  Why $N'?$}}) 
such that $K_{n,p}=\{[p_1], [p_2],...,[p_l]\}.$ 
Let us assume that 
$[p_1] = [1_C]$. (Recall that in Definition
\ref{DefA1}, in condition (1), it is  required that $1_{C_n} = 1_C$,  and
in condition (3), it is required that $[1_C] \in K_{n,p}$.) 

Suppose that $A$, $\phi$ and $\kappa_0$ satisfy the assumptions  of the lemma 
for $\dt,$ ${\cal G},$ ${\cal H}_q$ (and the already
fixed $G_{o,n},$ $K_{n,p},$  ${\cal P}_n,$ from Definition \ref{DefA1}).  
%We remind the reader that $\dt = \dt_n$ and $\G = \G_n$. 

Set 
\beq
&&\sigma_0:=\inf\{\rho_A([\phi(p_i)])(\tau): \tau\in T(A),\,\, 1\le i\le l\}
>0 \makebox{ and  }\\
&&\sigma_{00}:=\inf\{\rho_A(\kappa_0([p_i]))(\tau): \tau\in T(A),\,\, 1\le i\le l\}>0.  % \makebox{  and  }\\
%&&\sigma_{000} := \inf\{\rho_A \circ ([\phi]- \kappa_0)([p_i])(\tau) :
%\tau \in T(A),\,\, 1 \leq i \leq l \} > 0. 
\eneq
Note that by Proposition \ref{Plowerbd}, $\sigma_0\ge {\bar \sigma}$.

Set $\sigma_1:={\min\{1/2, \sigma_0, \sigma_{00}, % \sigma_{000}, 
\eta_0, \eta/4\}\cdot \dt_0\cdot {\bar \sigma}\cdot \sigma\over{2l{{l(n)}}(M+1)}}.$
%{\min\{1/2, \sigma_0, \sigma_{00}, \sigma_{000},\eta_0, \eta/4\}\cdot \dt_0\cdot \sigma\over{2l(M+1)}}.$
From Definition \ref{DefA1},
 $G_n = K_0(C)\cap G^{{\cal P}_n\cup K_{n,p}}=\Z^{r(n)}\oplus G_{o,n},$
and $G_{o.n}\subset {\rm ker}\rho_{C,f}$ {{and $\Z^{r(n)}\cap {\rm ker}\rho_{C,f}=\{0\}.$}}   Let $e_i=(\overbrace{0,...,0}
^{i-1}, 1,0,...,0)\in  \Z^{r(n)},$ $i=1,2,...,r(n),$ be the standard generators
of $\mathbb{Z}^{r(n)}$.
Let $\Pi_{g,n}: G_n\to \Z^{r(n)}$ be the projection map. Set
 $x_j=\Pi_{g,n}([p_j]),$
$j=1,2,...,l.$
  We may write  $x_j=\sum_{i=1}^{r(n)} m_{i,j}e_i$ for some 
  integers $m_{i,j}\in \Z,$ $i=1,2, ..., r(n)$, $j=1,2,...,l.$ 
  Set $M=r(n) \max\{|m_{i,j}|: 1\le i\le r(n), 1 \le j \leq l \}.$ 
  Let $K' \geq 1$ be an integer such that $K'x=0$ for all $x\in {\rm Tor}(K_i(C))\cap G^{{\cal P}_n\cup K_{n,p}}$
  ($i=0,1$).
  Suppose that $K_i(C, \Z/k\Z)\cap G^{{\cal P}_n\cup K_{n,p}}=\{0\}$ ($i=0,1$)
  for all $k\ge K''.$   Choose an integer $K'''\ge l+1$ such that 
  \beq\label{88-1C-e-1}
M  \max\{\|\rho_A([\phi](e_j))\|, \,\|\rho_A(\kappa_0(e_j))\|: 1\le j\le r(n)\}
  %\{\|\rho_A([\phi(1_C)])\|,\, \rho_A(\kappa_0([1_C]))\}
  /K'''<{\sigma_1^2\over{32(l+1)}}.
  \eneq
  Let $K=K'''(K'K'')!.$

Define $\sigma_2:={\sigma_1^2\over{16(K+1)(M+1)}}.$

Since $\rho_A(K_0(A))$ is dense in $\Aff(T(A)),$  
there are $f_i\in \rho_A(K_0(A))$
such that
\beq
\|f_i-{\rho_A(\kappa_0(e_i))\over{K+1}}\| & < &{\sigma_2\over{4(K+1)^2}}\,\,\,
\makebox{ for all } i=1,2,...,r(n). 
\eneq
It follows that
\beq
\|(K+1)f_i-\rho_A(\kappa_0(e_i))\|<{\sigma_2\over{4(K+1)}},\,\,\, i=1,2,...,r(n).
\eneq
Choose $f_{0,j}\in K_0(A)_+$ such that $\rho_A(f_{0,j})=f_j,$ $j=1,2,...,r(n).$
Let $a_j:=\kappa_0(e_j)-Kf_{0,j}\in K_0(A),$ $j=1,2,...,r(n).$
Then, for $j = 1, 2, ..., r(n)$, 
\beq
\|\rho_A(a_j)-{\rho_A(\kappa_0(e_j))\over{K+1}}\| &\le& \|\rho_A(a_j)-f_j\|
+\|f_j-{\rho_A(\kappa_0(e_j))\over{K+1}}\|\\
&{{<}}&\|\rho_A(\kappa_0(e_j))-(K+1)f_j\|
+{\sigma_1\over{4(K+1)^2}}\\\label{88-a-e-2}
&<&{\sigma_2\over{4(K+1)}}+{\sigma_2\over{4(K+1)^2}}<{\sigma_2\over{(K+1)}}.
\eneq

Define  a \hm\, $\kappa_1: G^{{\cal P}_n\cup K_{n,p}}\to 
%\in KL_{loc}(G^{{\cal P}_n\cup K_{n,p}}, 
\underline{K}(A)$ by taking  
\beq\label{87-k-1}
\kappa_1(e_j)=a_j,\,\,\,j=1,2,...,r(n),\,\,\,
{\kappa_1}|_{G_{o,n}}=\kappa_0|_{G_{o,n}},\\
{\kappa_1}|_{K_1(C)\cap G^{{\cal P}_n\cup K_{n,p}}}={\kappa_0}|_{K_1(C)\cap G^{{\cal P}_n\cup K_{n,p}}}\andeqn\\
{\kappa_1}|_{K_i(C, \Z/k\Z)\cap G^{{\cal P}_n\cup K_{n,p}}}={\kappa_0}|_{K_i(C, \Z/k\Z)\cap G^{{\cal P}_n\cup K_{n,p}}}, 
\eneq
\noindent for $i =0,1$ and $k \geq 2$.   {{In particular, $\kappa_1(G_{o,n})\subset {\rm ker}\rho_A.$}}

\iffalse
%Define 
%$\kappa_2: G^{{\cal P}_n\cup K_{n,p}}\to 
%\in KL_{loc}(G^{{\cal P}_n\cup K_{n,p}}, 
%\underline{K}(A))$ by taking 
%\beq
%&&\kappa_2(e_j)=Kf_{0,j}, \,\,\,j=1,2,...,r(n), \makebox{  and}\\
%&&\kappa_2(x)=-\kappa_1(x)\rforal x\in G_{o,n} \cup G' \makebox{ where } 
%\eneq
%\noindent $G' = G^{{\cal P}_n \cup K_{n,p}} \cap (K_1(C) \cup 
%\bigcup\{ K_i(C, \mathbb{Z}/k \mathbb{Z}) : i =0, 1, k \geq 2 \})$. 
\fi

Note that, ${\bar a_j}=\overline{\kappa_0(e_j)}$ in $K_0(A)/kK_0(A)$ for all $1<k\le K'',$  $j=1,2,...,r(n).$
Moreover, ${\kappa_1}|_{{\rm Tor}(K_i(C)) \cap G^{{\cal P}_n \cup K_{n,p}}}={\kappa_0}|_{{\rm Tor}(K_i(C)) \cap G^{{\cal P}_n \cup K_{n,p}}}$, for $i = 0,1$.
One then checks that $\kappa_1\in KL_{loc}(G^{{\cal P}_n\cup K_{n,p}}, \underline{K}(A))$ (satisfying the ``local Bockstein operations"). 

\iffalse
%Now consider $\bt:=[\phi]-\kappa_0\in KL_{loc}(G^{{\cal P}_n\cup K_{n,p}},\underline{K}(A))_+.$ 
%Since $\rho_A(K_0(A))$ is dense in $\Aff(T(A)),$  
%there are $f_i'\in \rho_A(K_0(A))$ (not necessarily positive) 
%such that
%\beq
%\|f_i'-{\rho_A(\bt(e_i))\over{K+1}}\| & < & {\sigma_2\over{4(K+1)^2}} 
%\makebox{ and }\\
%f'_i(\tau) & < & \frac{\rho_A(\bt(e_i))(\tau)}{K+1} \makebox{ for all }
%\tau \in T(A), \makebox{ } i=1,2,...,r(n).
%\eneq
%It follows that
%\beq
%\|(K+1)f_i'-\rho_A(\bt(e_i))\|<{\sigma_2\over{4(K+1)}},\,\,\, i=1,2,...,r(n).
%\eneq
%Choose $f_{0,j}'\in K_0(A)$ such that $\rho_A(f_{0,j}')=f_j',$ $j=1,2,...,r(n).$

%Let $b_j:=\bt(e_j)-Kf_{0,j}'\in K_0(A)_+,$ $j=1,2,...,r(n).$
%Then for $j = 1, 2, ..., r(n)$,  
%\beq
%\|\rho_A(b_j)-{\rho_A(\bt(e_j))\over{K+1}}\| &\le& \|\rho_A(b_j)-f_j'\|
%+\|f_j'-{\rho_A(\bt(e_j))\over{K+1}}\|\\
%&=&\|\rho_A(\bt(e_j))-(K+1)f_j'\|+{\sigma_2\over{4(K+1)^2}}\\
%&<&{\sigma_2\over{4(K+1)}}+{\sigma_2\over{4(K+1)^2}}<{\sigma_2\over{(K+1)}}.
%\eneq

%Define  a \hm\, $\kappa_2: G^{{\cal P}_n\cup K_{n,p}}\to 
%\in KL_{loc}(G^{{\cal P}_n\cup K_{n,p}}, 
%\underline{K}(A)$ by taking  
%\beq\label{87-k-2}
%\kappa_2(e_j)=b_j,\,\,\,j=1,2,...,r(n),\,\,\,
%{\kappa_2}|_{G_{o,n}}=\bt|_{G_{o,n}},\\
%{\kappa_2}|_{K_1(C)\cap G^{{\cal P}_n\cup K_{n,p}}}={\bt}|_{K_1(C)\cap G^{{\cal P}_n\cup K_{n,p}}}\andeqn\\
%{\kappa_2}|_{K_i(C, \Z/k\Z)\cap G^{{\cal P}_n\cup K_{n,p}}}={\bt}|_{K_i(C, \Z/k\Z)\cap G^{{\cal P}_n\cup K_{n,p}}}, 
%\eneq
%\noindent for $i =0,1$ and $k \geq 2$.  
\fi

We estimate that, for $j=1, 2, ..., l$,  
\beq
&&\hspace{-1in}\rho_A(\kappa_1([p_j]))(\tau)=\rho_A(\kappa_1(x_j))(\tau)=
\sum_{i=1}^{r(n)}m_{i,j}\rho_A(\kappa_1(e_i))(\tau)
\\
&&=\sum_{i=1}^{r(n)}m_{i,j} \rho_A(a_i)(\tau)>\sum_{i=1}^{r(n)}m_{i,j}({\rho_A(\kappa_0(e_i))(\tau)\over{K+1}}-{\sigma_2\over{K+1}})\\
&&={\rho_A(\kappa_0([p_j]))(\tau)-M\sigma_2\over{K+1}}>0\hspace{0.3in}\rforal \tau\in T(A).
\eneq
%Since $M(l+1)/K'''<\sigma_1,$ 
By \eqref{88-1C-e-1}, 
we have that for all $i = 1, 2, ..., l$,  
\beq
\rho_A(\kappa_1([p_i]))(\tau) 
& = & \sum_{j=1}^{r(n)} m_{j,i} \rho_A(\kappa_1(e_j))(\tau) \\
&\le&  M\max\{|\rho_A(\kappa_1(e_j))(\tau)|: 1\le j\le r(n)\}\\
&=&M\max\{|\rho_A(a_j)(\tau)|: 1\le j\le r(n)\}\\
&<&{M\max\{|\rho_A(\kappa_0(e_j))(\tau)| : 1 \leq j \leq r(n) \}\over{K+1}}+{\sigma_2\over{(K+1)}}\\\label{Lpartuniq-21}
&\le& {\sigma_1^2\over{32(l+1)}}+{\sigma_2\over{(K+1)}}<{\sigma_1^2\over{16(l+1)}} \,\rforal 
% {\sigma_1\over{16(K+1)}}\,\,\rforal 
\tau\in T(A).
\eneq

\iffalse
%Similarly,
%by \eqref{88-1C-e-1}, 
%we have that for all $i = 1, 2, ..., l$,  
%\beq
%\hspace{-0.3in}\rho_A(\kappa_2([p_i]))(\tau) 
%& = & \sum_{j=1}^{r(n)} m_{j,i} \rho_A(\kappa_2(e_j))(\tau)\\ 
%&\le&  M\max\{|\rho_A(\kappa_2(e_j))(\tau)|: 1\le j\le r(n)\}\\
%&=&M\max\{|\rho_A(b_j)(\tau)|: 1\le j\le r(n)\}\\
%&<&{M\max\{|\rho_A(([\phi] - \kappa_0)(e_j))(\tau)| : 1 \leq j \leq 
%r(n) \}\over{K+1}}+{\sigma_2\over{(K+1)}}\\\label{Lpartuniq-27}
%&\le& {\sigma_1^2\over{32(l+1)}}+{\sigma_2\over{(K+1)}}<{\sigma_1^2
%\over{16(l+1)}} \,\rforal 
% {\sigma_1\over{16(K+1)}}\,\,\rforal 
%\tau\in T(A).
%\eneq

%We also estimate that for $j = 1, 2, ..., l$, 

%\beq
%&&\hspace{-1in}\rho_A(\kappa_2([p_j]))(\tau)=\rho_A(\kappa_2(x_j))=
%\sum_{i=1}^{r(n)}m_{i,j}\rho_A(\kappa_2(e_i))(\tau)
%\\
%&&=\sum_{i=1}^{r(n)}m_{i,j} \rho_A(b_i)(\tau) \geq 
%\sum_{i=1}^{r(n)}m_{i,j}({\rho_A(\bt(e_i))(\tau)\over{K+1}}-{\sigma_2\over{K+1}})\\
%&&\geq {\rho_A(([\phi] - \kappa_0)([p_j]))(\tau)-
%M\sigma_2\over{K+1}} > {\sigma_1 - M \sigma_2\over{K+1}} > 0 \hspace{0.1in}\rforal \tau\in T(A).
%\eneq

%Note that, by \eqref{87-k-1}, $\kappa_1(G_{o,n})\subset {\rm ker}\rho_A.$ 
%Since $\bt:=[\phi]-\kappa_0,$  by \eqref{87-k-2}, 
%$\kappa_2(G_{o,n})\subset {\rm ker}\rho_A.$
\fi

 Choose a %two mutually orthogonal  
non-zero projection $e_0 \in A$  such that $[e_0]=\kappa_1([1_C])$.
% and  $[P_{00}]=\kappa_2([1_C])$.
Then $\tau(e_0)<\sigma_1^2/2$   
%and $\tau(P_{00})<\sigma_1^2/2$
for all $\tau\in T(A).$
Since $C$ is in the class ${\cal A}_0,$ by (b) of part  (3) of Definition 
\ref{DefA1}, there is a unital {{${\cal G}$-$\dt$}}-multiplicative \morp\  
$\psi_0: C\to e_0Ae_0$ %and $\psi_{00}: C\to P_{00} A P_{00}$
such that
\beq
[\psi_0]|_{{\cal P}_n}={\kappa_1}|_{{\cal P}_n}.
%\andeqn [\psi_{00}]|_{{\cal P}_n}={\kappa_2}|_{{\cal P}_n}.
\eneq

%Recall that $\G =\G_n$ and $\delta = \delta_n$.

%Put $\phi_0:=\psi_{00}.$
Now consider $\af:=[\phi]-[\psi_0].$ 
%and 
%recall the choice of $\sigma_1.$ 
{{Since $[\psi](G_{o,n})\subset {\rm ker}\rho_A,$ we have $\af(G_{o,n})\subset {\rm ker}\rho_A.$}}
By  {{the the choice of $\sigma_1,$}}  \eqref{Lpartuniq-21} and %and \eqref{Lpartuniq-27}, and 
by applying (c) of part (3) of Definition \ref{DefA1}, we obtain a  unital  ${\cal G}$-$\dt$-multiplicative \morp\,
$\psi_1: C\to (1-e_{0})A(1-e_0)$ such that
\beq
&&[\psi_1]|_{{\cal P}_n}=\af|_{{\cal P}_n}\andeqn\\
&&|\tau(\psi_1(x))-\tau(\phi(x))|<{\ppl{3}}\sigma_1+\eta/2\rforal x\in {\cal H}\andeqn \rforal \tau\in T(A).
\eneq
Define $e_1=1_A -e_0.$
We note that 
\beq
&&e_0+e_1=1_A\andeqn \tau(e_0)<\sigma \rforal \tau\in T(A),\\
&&[\psi_0]|_{{\cal P}_n}+[\psi_1]|_{{\cal P}_n}=[\phi]|_{{\cal P}_n}\andeqn\\
&&|\tau(\psi_0(x)+\psi_1(x))-\tau(\phi(x))|<3\sigma_1+\sigma_1/2+\eta/2<\eta\rforal x\in {\cal H}_1.
\eneq
From the above, we also estimate that
\beq
\tau((\psi_0+\psi_1)(h))\ge \Delta_1(\hat{h})\rforal h\in {\cal H}_q\andeqn \tau\in T(A).
\eneq
We then apply (2) of Definition \ref{DefA1} to obtain the required unitary, 
which works for the {{finite subset $\G.$}}
%$\G_N$.

\end{proof}

Note that, in the next statement, the {{\CA\,}} $C$ has at least one faithful tracial state.

\begin{thm}\label{T8unitaryequiv}
Let $C$ be a unital separable \CA\, in the class ${\cal A}_0$, and let $B$ be {{a non-unital separable simple \CA\, 
with tracial rank zero  and continuous scale.}}
%in the standing assumptions for this section. 
Let  ${{\Psi, \Phi}}: C\to M(B)$ be  two unital \hm s such that $\pi\circ \Psi$ and $\pi\circ \Phi$ are   both 
injective.

Suppose that $\tau\circ \Psi=\tau\circ \Phi$ for all $\tau\in T(B).$ 
Then there exists a sequence of unitaries $\{ u_n \}$ in $M(B)$ such that
\beq
&&\lim_{n\to\infty}\|u_n^*\Psi(c)u_n-\Phi(c)\|=0\rforal c\in C,\andeqn\\
&&u_n^*\Psi(c)u_n-\Phi(c)\in B\rforal n \geq 1 \makebox{ and } c \in C.
\eneq
The converse also holds.
\end{thm}

%{\red{\bf Proof version A --- a modified old version}}

\begin {proof}
Fix $\ep>0$ and a finite subset ${\cal F}\subset C.$
For each 
%$s\in C_+^{q, {\bf 1}}\setminus \{0\},$ there exists 
$h\in C_+^{\bf 1}\setminus \{0\}$ define 
\beq\label{813-n1}
\Delta_0(\hat{h}):= \inf \{\tau(\Psi(h)): \tau\in T(B)\}. 
\eneq
{{Recall that, since $B$ has continuous scale, $T(B)$ is compact.}}
Since $\Psi(h)\not=0$  {{and}} $B$ is simple,
$\Delta_0(\hat{h})>0.$ It follows that $\Delta_0: (C_+)^{q, {\bf 1}}\setminus \{0\}\to (0,1]$ is an order preserving 
map. 

Let ${\cal F}_1:={\cal F}$ and let ${\cal F}_1\subset {\cal F}_2\subset \cdots {\cal F}_n\subset \cdots C$ be an 
increasing sequence of finite subsets of $C^{\bf 1}$ which is dense in the unit ball of $C.$
Put $\Delta_1:=(1/4)\Delta_0.$
{{Recall that  we may assume that}} $C=\overline{\cup_{n=1}^\infty C_n}$ as in Definition \ref{DefA1} for ${\cal A}_0.$  
Let $2\dt_n>0$ (in place of $\dt$),  $\eta_n'>0$ (in place of $\eta$), ${\cal G}_n\subset C,$   ${\cal P}_n\subset \underline{K}(C),$  ${\cal H}_{q,n}\subset C_+^{\bf 1}\setminus \{0\},$ 
and a finite subset ${\cal H}_n\subset C_{s.a.}$   be finite subsets given by  part  (2)  of Definition \ref{DefA1}  associated 
with $\ep/2^{n+5}$ (in place of $\ep$), ${\cal F}_{n}$ (in place of ${\cal F}$) and $\Delta_1$ (in place of $\Delta$).
%By applying \ref{Ptrace} and by choosing smaller $\dt_n$ and larger ${\cal G}_n,$ if required, 
%We may assume that $({\cal G}_n, 2\dt_n, {\cal P}_n)$  is a KL-triple associated 
%with ${\cal H}_n$ and $\eta_n'/4.$ 
We may also assume that $\dt_n<\ep/2^{n+5},$ $\sum_{n=1}^\infty \dt_n<1,$  ${\cal F}_n\subset {\cal G}_n\subset {\cal G}_{n+1}$
and ${\cal H}_{q,n}\subset {\cal H}_n$ for all 
 $n.$  
 Moreover, following Defintion \ref{DefA1},  we may assume that $({\cal G}_n, 2\dt_n, {\cal P}_n)$ is a KL-triple and  
 $\dt_n<\dt_{n+1}$ for all $n$.
 Continuing to follow Definition \ref{DefA1}, we have the following: Let $G_n$ be the subgroup generated by 
${\cal P}_n\cap K_0(C)$ and there is a finite subset $K_{n,p}\subset K_0(C)_+\setminus \{0\}$ which also generates 
$G_{n}.$ Thus we may also assume that $K_{n,p}\subset {\cal P}_n.$
{{Recall that we also assume that $K_{n,p}=[\iota_{n, \infty}](K_{n,p}'),$ where 
$K_{n,p}'\subset K_0(C_n)_+\setminus \{0\}$ is a finite subset such that $K_{n,p}'$ generates 
$K_0(C_n)$ and satisfying other conditions in as described in Definition \ref{DefA1} including (1) and (3). }}

Again continuing to follow Definition \ref{DefA1}, we may write $G_n=F_{0,n}\oplus G_{o,n},$ where $F_{0,n}$ is a finitely generated free group
and $G_{o,n}\subset {\rm ker}\rho_{C, f}$ {{and $F_{0,n}\cap {\rm ker}\rho_{C,f}=\{0\},$ and $K_{n,p}$ 
contains a generating set of $F_{0,n}\cap K_0(C)_+.$}}
We may further assume that 
$K_{n,p}=\{[p_{n,1}], [p_{n,2}],..., [p_{n,s(n)}]\}$  {{with $[1_C]=[p_{n,1}]$}}
and $G_{o,n}$ is generated by $\{[q_{+,n,1}]-[q_{-,n,1}], [q_{+, n,2}]-[q_{-,n,2}],..., [q_{+, n,m(n)}]-[q_{-,n,m(n)}]\},$
and $p_{n,i}, q_{+, n,j}, q_{-n,j}\in M_{l(n)}(C)$ are projections. 
We also assume that $[q_{+,n,j}], [q_{-,n,j}]\in K_{n,p}.$
Put $K_{n,p,p}=\{p_{n,j}: 1\le j\le s(n)\}.$

Put
 \beq\label{87-sigma}
 \sigma_n':=\inf\{\Delta_1(\hat{h}): h\in {\cal H}_{q,n}\}/2^{n+2}{{l(n)}} \andeqn \sigma_n:=\min\{\sigma_n'/8, \eta_n'/8, \dt_n/8\},
 %\,\,\, n=1,2,....
 \eneq
 {\ppl{$n\in \N.$}}
  \Wlog, we may further assume that ${\bar H}_{q,n}\subset {\cal H}_{q,n},$ 
  where ${\bar H}_{q,n}\subset A_+^{\bf 1}\setminus \{0\}$ corresponds $K_{n,p}$ (in place of 
  ${\cal Q}$ in Proposition \ref{Plowerbd}{{)}}. Moreover, we may also assume that 
  each $\dt_n$ is smaller than half of $\dt$ and ${\cal G}_n$ is larger then 
  ${\cal G}$ in Proposition \ref{Plowerbd} associated with $K_{n,p}$ (in place of ${\cal Q}$).
%%%%%%%%%%%%%%%%%%%%%%%%%%%%%
%
\iffalse 
 \Wlog, we may further assume 
 that $\sigma_n$ is smaller than $\sigma$ in part (d) of (3) of \ref{DefA1} associated 
 to  ${\cal H}_n,$ ${\cal P}_n,$  $\dt_n$ and $\eta_n'/2.$ 
 \fi
 %%%%%%%%%%%%%%%%%%%%%%%%%%%
 %
 
 Choose a finite subset ${\cal G}_{n,T}\subset \iota_{n, \infty}(C_n)$  larger than ${\cal G}_ n$ and 
 $0<\dt_{n, T}<\sigma_n$  such that $({\cal G}_{n,T}, 2\dt_{n,T}, {\cal P}_n)$ is a KL-triple 
 associated with ${\cal H}_n$ and $\sigma_n$. 
%(in place of $\eta$C)
 
 For each $n,$  we will apply Lemma \ref{Lpartuniq}. Let $\tilde \dt_n>0,$ ${\tilde{\cal G}_n}\subset C$ and 
 ${\tilde {\cal H}}_{q,n}\subset C_+^{\bf 1}\setminus \{0\}$ be finite subsets given by \ref{Lpartuniq} associated 
 with $\dt_{n, T}$ (in place of $\dt$), $\sigma_n^2$ (in place of $\sigma$),  $\sigma_n$ (in place 
 of $\eta$), ${\cal P}_n,$ $K_{n,p}$ and $G_{o,n},$ $n=1,2,....$ 
We assume  ${\cal G}_{n,T}\subset {\tilde{\cal G}_n}$ and ${\cal H}_{q, n}\subset 
{\tilde {\cal H}}_{q,n}.$ 
 
 Choose    ${\td {\cal H}}_n={\tilde {\cal H}}_{q,n}\cup {\cal H}_n$ and 
 %$\eta_1:=\min\{{\tilde \dt_1}/2, \sigma_1'/2, \eta_1'/4\}\cdot \inf\{\tau(e_{N_1}):\tau\in T(B)\}$ ands 
 \beq
 \eta_n:=\min\{{\tilde \dt_n}/4, \sigma_n'/4, \eta_n'/4\},
 %\cdot \inf\{\tau(e_{N+n}-e_{N_1+n-1}): \tau\in T(B)\}
 \eneq 
 $n=2,3,....$

To prove Theorem \ref{T8unitaryequiv}, 
it suffices to prove that the theorem holds for one fixed $\Phi$ 
and arbitrary $\Psi$ satisfying the hypotheses of Theorem \ref{T8unitaryequiv}.  
Therefore,
by applying Lemma \ref{L84}, we may assume that 
there are
% {\blue{a unital injective *-homomorphism
%$\Phi : C \rightarrow M(B)$ with $\Phi_{*0} = \Psi_{*0} = \Psi'_{*0}$, }}
an approximate identity consisting of projections $\{e_n\}$ (with $e_0=0$ and $e_n-e_{n-1}\not=0$) and 
${\td{\cal G}}_n$-$\eta_n$-multiplicative \morp s 
$\phi_n: C\to (e_n-e_{n-1})B(e_n-e_{n-1})$ such that $[\phi_n]|_{G_n}$ is well defined, 
\beq\label{89-k-1}
&&[\phi_n]|_{G_n\cap {\rm ker}\rho_{C,f}}=0,\,\,
[\phi_n]|_{ G_n\cap K_1(C)}=0,\\
&&{[\phi_n]}|_{G_n\cap K_i(\Z/k\Z)}=0\tforal k\ge 2\,\,\, (i=0,1),
\eneq
$n=1,2,...,$
% (were $G_n^g$ is the the subgroup generated by $G_n$) 
and
\beq\label{84-B-1}
&&\Phi(c)-\sum_{n=1}^\infty \phi_n(c)\in B\tforal c\in C,\\
&&\|\Phi(c)-\sum_{n=1}^\infty \phi_n(c)\|<\eta_1\tforal c\in {\cal G}_1\cup {\td{\cal H}}_1,\tand\\\label{TT-t-01}
&& {\tau(\phi_n(x))\over{\tau(e_n-e_{n-1})}}\ge (3/4)\tau(\Phi(x))\tforal x\in {\td{\cal H}}_{q,n},\\\label{813-n2}
%\eneq
%If, in addition,  there is an affine continuous map $\gamma: T(B)\to T(C)$ such that
%$\rho_{M(B)}\circ \Gamma(g)(\tau)=\rho_C(g)(\gamma(\tau))$
%for all $g\in K_0(C)$ and $\tau\in T(B),$ we may further require that
%\beq
&&\tau(\Phi(c))=\gamma(\tau)(c)\tforal c\in C\tand \tau\in T(B).
\eneq

By applying Lemma \ref{L86}, we obtain {{a unitary  $V\in M_2(M(B)),$}} 
an integer $N_1,$ and a ${\td{\cal G}_1}$-$\eta_1$-multiplicative \morp\,
$L: C\to e_{N_1'}Be_{N_1'}$ such that $[L]|_{\td G_1}$ is well defined, {{$V^*(1_{M(B)}-e_{N_1})V\in 1_{M(B)}+B,$
$e_{N_1}'=1_{M(B)} -V^*(1_{M(B)} -e_{N_1})V\in B,$}}

\beq\label{TT-T-001}
&&\Psi(c)-\sum_{n=N_1+1}^\infty V^*\phi_n(c)V\in B\tforal c\in C,\tand\\\label{86-T-02}
&&\|\Psi(c)-(\sum_{n=N_1+1}^\infty V^*\phi_n(c)V+L(c))\|<\eta_1\tforal c\in {\cal G}_1\cup {\td{\cal H}}_1.
\eneq
Moreover, 
\beq\label{TT-T-1}
\hspace{-0.2in} |\tau(L(x))-\tau(\sum_{i=1}^{N_1}\phi_i(x))|<  \eta_1,\,\,(x\in \td {\cal H}_1),\,\,
%\tforal x\in {\cal H}_1\tand \tau\in T(B)\tand\\
{\tau(L(h))\over{\tau(e_{N_1})}}>(3/4)\tau(\Phi(h)),\,\,(h\in {\wtd{{\cal H}}}_{q,1}),\\\label{TT-T-2}
\eneq
%&&|\tau(L(x))-\tau(\sum_{j=1}^{N_1}\phi_j(x))|<\eta_1\hspace{0.2in}
%or all  $x\in {\td {\cal H}}_1
and, for all $\tau\in T(B),$ and 
%\eneq
%for any given finite subset of projections $p_1,p_2,...,p_s, q_1,q_2,...,q_s\in M_l(C)$ (for some integer $l\ge 1$)
%such that $\tau(p_)=\tau(q_i)$ for all $\tau\in T(B),$ $i=1,2,...,s,$ we may also require that 
\beq\label{TT-e-1}
&&\rho_B([L(p])(\tau)=\sum_{i=1}^{N_1}\rho_B([\phi_i(p)])(\tau)\rforal p\in K_{1,p,p}
\tand\\\label{TT-e-2}
&&\rho_B([L(q_{+,1,j})])(\tau)=\rho_B([L(q_{-, 1, j})])(\tau)
\eneq
for all {{$\tau\in T(B),\,\,1\le j\le m(1).$}} 
{{Put $e_{N_1+n}'=V^*e_{N_1+n}V,$ $n=1, 2,....$}} 
{{Since $V^*(1_{M(B)}-e_{N_1})V\in 1_{M(B)}+B,$ $e_{N_1+n}'\in B,$ $n=1,2,....$
($e_{N_1}'$ has been defined.)
Note that 
\beq\label{813-c-n1}
V^*\sum_{i=N_1+1}^\infty (e_n-e_{n-1})V\perp e_{N_1}'\andeqn \tau(e_{N_1}')=\tau(e_{N_1})\rforal \tau\in T(B).
\eneq}}

Define  $\Psi_n: C\to e_{N_1+n}'Be_{N_1+n}'$ by $\Psi_n(c)=e_{N_1+n}'\Psi(c)e_{N_1+n}'$
for all $c\in C.$ 
We also assume that $\Psi_n$ are ${\td{\cal G}_{n+1}}$-$\eta_{n+1}$-multiplicative. 
By choosing large $N_1,$  and passing to a subsequence, we may assume 
that 
\beq
\tau(\Psi_n(p_{n+1,j}))>0\rforal \tau\in T(B),\,\,\, {{j=1,2,...,s(n+1).  }}  
\eneq
In other words, we may assume that $[\Psi_n](K_{n+1,p})\subset K_0(B)_+\setminus \{0\}.$ 

Define  $\Phi_n: C\to e_{N_1+n}Be_{N_1+n}$ by $\Phi_n(c)=e_{N_1+n}\Phi(c)e_{N_1+n}$
for all $c\in C.$ 
We also assume that $\Phi_n$ are ${\td{\cal G}_{n+1}}$-$\eta_{n+1}$-multiplicative
and $[\Phi_n](K_{n+1,p})\subset K_0(B)_+\setminus \{0\}.$ 

Since 
\beq
&&\lim_{n\to\infty}\|\Psi(c)-(\sum_{k=N_1+n+1}^\infty V^*\phi_k(c)V+\Psi_n(c))\|=0,\\
&&\lim_{n\to\infty}\|\Phi(c)-(\sum_{k=N_1+n+1}^\infty\phi_k(c)+\Phi_n(c))\|=0\rforal c\in C,
\eneq
and since close projections are equivalent, 
we may assume 
that,  for all $n\in\N,$
\beq
&&[\Psi(p)]=[ (\sum_{k=N_1+n}^\infty\phi_k(p))+\Psi_n(p)]\andeqn\\
%[\Psi(p_{j, n+1})]=[ (\sum_{k=N_1+n}^\infty\phi_k(p_{j,n+1}))+\Psi_n(p_{j, n+1})],\\
&&{[}\Phi(p){]}=[(\sum_{k=N_1+n}^\infty\phi_k(p))+\Phi_n(p)]
%{[}\Phi(p_{j, n+1}){]}=[(\sum_{k=N_1+n}^\infty\phi_k(p_{j,n+1}))+\Phi_n(p_{j, n+1}))]
\eneq
{for $p\in \{p_{n+1,j}: 1\le j\le s(n+1)\}\cup \{q_{\pm,n+1,i}: 1\le i\le m(n+1)\}.$}
Also, since $\tau\circ \Phi=\tau\circ \Psi$ for all $\tau\in T(B),$  from the above,  we have 
\beq\label{TT-e-10}
&&\tau(\Phi_n(p_{n+1,j}))=\tau(\Psi_n(p_{{n+1},j})),\\
&&\tau(\Psi_n(q_{+,n+1,i}-q_{-,n+1,i}))=\tau(\Phi_n(q_{+,n+1,i}-q_{-,n+1,i}))=0
\eneq 
for all  $\tau\in T(B),\,\,\, j=1,2,...,s(n+1), \makebox{ } i=1, ..., m(n+1), \,\,n=1,2,....$ 
%Since $\{[p_{1,n+1}],[p_{2, n+1}],...,[p_{r(n+1), n+1}]\}$ generates $G_n,$  \eqref{TT-e-10}
%implies that 
It follows that 
\beq\label{TT-e-kerrho}
[\Phi_n]({G_{o,n+1}}),\,\,[\Psi_n]({G_{o,n+1}})\subset {\rm ker}\rho_B.
\eneq
%and $[\Psi_n]|_{K_{n,p}}\subset K_0(B)_+\setminus \{0\}.$
We also have (with $\Phi_0=\phi_{N_1}$),  for $n=1,2,...,$ 
\beq\label{87-k-10}
&&[\Phi_{n}]|_{{\cal P}_{n}}=([\phi_{N_1+n}]+[\Phi_{n-1}])|_{{\cal P}_{n}},\\\label{87-K-11}
&&{[}\Psi_{n}{]}|_{{\cal P}_{n}}=([\phi_{N_1+n}]+[\Psi_{n-1}])|_{{\cal P}_{n}}\andeqn\\\label{87-K-12}
&&{[}\Psi_1{]}|_{{\cal P}_1}=([\phi_{N_1+1}]+[L])|_{{\cal P}_1}.
\eneq

%Put $L_1: C\to e_{N_1+1}Be_{N_1+1}$ be defined by $L_1(c)=(1-e_{N_1+1})\Psi(c)(1-e_{N_1+1})$
%for all $c\in C.$ Then, by ?, $L_1$ is ${\td {\cal G}}_2$-$\eta_2$-multiplicative.  Similarly,
%the map $L_1': C\to e_{N_1+1}Be_{N_1+1}$ define by $L_1'(c)=(1-e_{N_1+1})\Phi(c)(1-e_{N_1+1})$
%is also ${\td{\cal G}_2}$-$\dt_2$-multiplicative. 

Recall that we assume that $K_{n,p}\subset {\cal P}_n.$  
Since $\Psi_1$ is ${\td{\cal G}_2}$-$ \eta_2$-multiplicative, $[\Psi_1]|_{{\cal P}_2}$
gives an element $\kappa_1'\in KL_{loc}(G^{{\cal P}_2}, \underline{K}(A))$  and $\kappa_1'({{K}}_{2,p})\subset K_0(B)_+\setminus \{0\}.$
By \eqref{TT-e-kerrho},
%\eqref{89-k-1}   and 
%and \eqref{TT-e-1}, $\kappa_0'({\cal K}_{2,p})\subset K_0(B)_+\setminus \{0\},$ and by
%\eqref{TT-e-2}, 
$\kappa_1'(G_{o,2})
\subset {\rm ker}\rho_B.$  
 
Write $[p_{2,j}]=f_{2,j}\oplus x_{2,j}\in F_{0,2}\oplus G_{o,2},$ 
{{$j=1,2,...,s(2).$}} 
Let $T_1\in \N$ such that $T_1x=0$ for all 
$x\in {\rm Tor}(G^{{\cal P}_2}).$  We may choose $T_1>2.$
%Choose an integer $M(1)\ge 1$ such that $Mf_{1,j}\oplus (\pm x_{1,j})\in K_0(C)_+\setminus \{0\}$
%for any $M>M(1),$ $j=1,2,...,r(1).$ Choose $M_1\ge \max\{M(1)+1, T_1+1\}.$
 Define $\kappa_1\in KL_{loc}(G^{{\cal P}_2}, \underline{K}(B))$
as follows: 

On $F_{0,2}\oplus G_{o,2},$ define $\kappa_1(f\oplus z)=(T_1!-1)\kappa_1'(f)-\kappa_1'(z)$
for all $f\in  F_{0,2}$ and $z\in G_{0,2};$
define $\kappa_1(x)=-\kappa_1'(x)$ for all $x\in K_1(C)\cap G^{{\cal P}_2},$
and for all $x\in G^{{\cal P}_2}\cap K_i(C, \Z/m\Z)$ ($i=0,1$),  $m=2,3,...,.$ 
Note, for $[p_{2,j}]\in K_{2,p}$ {{and $1\le j\le r(2),$}}
\beq
\kappa_1([p_{2,j}])=(T_1!-1)\kappa_1'(f_{2,j})-\kappa_1'(x_{2,j})
=(T_1!-1)\kappa_1'([p_{2,j}])-T_1! \kappa_1'(x_{2,j}).
\eneq
Note that
%by \eqref{TT-e-2},  
$\kappa_1'(x_{2,j})
\in {\rm ker}\rho_B.$ 
Since $B$ is a simple \CA\, with tracial rank zero, 
$\kappa_1([p_{2,j}])\in K_0(B)_+\setminus \{0\}.$ 

Put $B_0:=e_{N_1}Be_{N_1},$ $B_n:=(e_{N_1+n}-e_{N_1-(n-1)})B(e_{N_1+n}-e_{N_1-(n-1)}),$
$n=1,2,....$

{{Note that, by \eqref{TT-t-01}
%\eqref{813-n2}
%{Mar2120222AM} 
and by  \eqref{813-n1},
$\tau(\phi_{N_1 +1}(h)) \geq \Delta_1(\hat{h})$ for all 
$h \in \tilde{{\cal H}}_{q,N_1 +1}$ and $\tau \in T(B_1)$.}}
By applying  Lemma \ref{Lpartuniq} to the map 
$\phi_{N_1+1}$ (in place of $\phi$) and $\kappa_1,$   we obtain 
mutually {{non-zero}} orthogonal projections $e_{0,1}, e_{1,1}\in (e_{N_1+1}-e_{N_1})B(e_{N_1+1}-e_{N_1})$ and unital
${\cal G}_{2,T}$-$\dt_{2, T}$-multiplicative \morp s $\psi_{0,1}: C\to e_{0,1}Be_{0,1}$ and $\psi_{1,1}:C\to e_{1,1}Be_{1,1},$
and a unitary $u_{1,1}\in B_1$
%e_{N_1+1}-e_{N_1})B(e_{N_1+1}-e_{N_1})$
such that
\beq
&&\hspace{-0.4in}e_{0,1}+e_{1,1}=e_{N_1+1}-e_{N_1}, \,\,\, \tau(e_{0,1})<\sigma_2^2\tforal \tau\in T(B_1),\\
&&\hspace{-0.4in}{[\psi_{0,1}]}|_{{\cal P}_2\cap {\rm ker}\rho_{C,f}}={\kappa_1}|_{{\cal P}_2\cap {\rm ker}\rho_{C,f}},\,\,\,\\
&&\hspace{-0.4in}{[}\psi_{0,1}{]}|_{{\cal P}_2\cap K_1(C)}={\kappa_1}|_{{\cal P}_2\cap K_1(C)},\\
&&\hspace{-0.4in}{[\psi_{0,1}]}|_{{\cal P}_2\cap K_i(C, \Z/k\Z)}={\kappa_1}|_{{\cal P}_2\cap K_i(C, \Z/k\Z)},\,\,\, k=2,3,...,\\
&&\hspace{-0.4in}\|u_{1,1}^*\phi_{N_1+1}(c)u_{1,1}-(\psi_{0,1}(c)+\psi_{1,1}(c))\|<\sigma_2 \rforal c\in {\cal G}_2\andeqn \tau\in T(B_1).
\eneq
Note 
$\psi_{0,1}(1_C)=e_{0,1}$ and $\psi_{1,1}(1_C)=e_{1,1}.$
{{Put $e_{0,1}'=V^*e_{0,1}V.$}}
Note that  $K_0(C, \Z/k\Z)\cap G^{{\cal P}_2}=\{0\}$ for any $k>T_1.$ 
Therefore, for $x\in (K_0(C)/kK_0(C))\cap G^{{\cal P}_2}$ ($k\le T_1$)
\beq
\kappa_1'(x)+[\psi_{0,1}](x)=(T_1!-1)\kappa_1(x)+\kappa_1(x)=-\kappa_1(x)+\kappa_1(x)=0.
\eneq
Also, if $x\in {\rm Tor}(K_0(A))\cap G^{{\cal P}_2},$ 
\beq
\kappa_1'(x)+[\psi_{0,1}](x)=(T_1!-1)\kappa_1(x)+\kappa_1(x)=0.
\eneq
One then computes that
\beq
&&[\psi_{0,1}+L]|_{K_i(C, \Z/k\Z)\cap {\cal P}_1}=0,\,\,i=0,1,\\
&&{[}\psi_{0,1}+L{]}|_{{\rm ker}\rho_{C,f}\cap G^{{\cal P}_1}}=0,\andeqn\\
&&{[}\psi_{0,1}+L{]}|_{K_1(C)\cap {\cal P}_1}=0.
\eneq

Define $\lambda_1\in KL_{loc}(G^{{\cal P}_2}, \underline{K}(B))$  by 
${\lambda_1}|_{{\cal P}_2}=([\Psi_1]-[\Phi_1]+[\psi_{0,1}])|_{{\cal P}_2}.$
%[\sum_{n=1}^{N_1}\phi_n])|_{{\cal P}_1}}.$
Note that, by \eqref{TT-e-kerrho},
\beq
\lambda_1(x)=([\Psi_1]-[\Phi_1]+[\psi_{0,1}])(x)=\kappa_1'(x)-[\Phi_1](x)+[\psi_{0,1}](x)=-[\Phi_1](x)\in {\rm ker}\rho_B
\eneq
for $x\in G^{{\cal P}_2}\cap {\rm ker}\rho_{C,f},$ 
and, by \eqref{89-k-1} and \eqref{87-k-10},  if $x\in G^{{\cal P}_1}\cap {\rm ker}\rho_{C,f},$
\beq
\lambda_1(x)=
%([\Psi_1]-[\Phi_1]+[\psi_{0,1}])(x)=\kappa_1'(x)
-[\Phi_1](x)=0.
%+[\psi_{0,1}](x)
\eneq
Moreover,  by \eqref{87-k-10} and \eqref{87-K-12},
\beq\label{87-k-13}
&&\hspace{-0.7in}\lambda_1(x)=([\Psi_1]-[\Phi_1]+[\psi_{0,1}])(x)\\\label{87-k-14}
&&\hspace{-0.3in}=
%=\kappa_1'(x)-[\Phi_1](x)+[\psi_{0,1}](x)
([\phi_{N_1+1}]+[L]-[\phi_{N_1+1}]-[\sum_{i=1}^{N_1}\phi_i]+[\psi_{0,1}])(x)
=([L]-[\sum_{i=1}^{N_1}\phi_i]+[\psi_{0,1}])(x)
\eneq
for all $x\in G^{{\cal P}_1}.$
%\cap K_1(C).$
%and 
%$x\in G^{{\cal P}_1}\cap K_i(C, \Z/k\Z),$ $k=2,3,...,$ and $i=0,1.$
Furthermore, by \eqref{TT-e-10}, for  $j=1,2,...,s(2),$ 
\beq
&&\hspace{-0.7in}\rho_B(\lambda_1([p_{2,j}]))=\rho_B([\Psi_1]([p_{2,j}])
+[\psi_{0,1}]([p_{2,j}])-[\Phi_1]([p_{2,j}]))
=\rho_B([\psi_{0,1}]([p_{2,j}]))>0.
\eneq
In other words, $\lambda_1([p_{2,j}]) \in K_0(B)_+\setminus \{0\},$ $j=1,2,...,s(2).$
Note that {{$\rho_B(\lambda_1([1_C]))=\rho_B([e_{0,1}])$ (see \eqref{813-c-n1}).
Choose $e_{0,1}''\le  e_{N_1+1}-e_{N_1}$ such that 
$[e_{0,1}'']=\lambda_1([1_C]).$}}
By  (b) of  part (3) of  \ref{Lpartuniq}, there is a unital ${\cal G}_{2,T}$-$\dt_{2,T}$-multiplicative \morp\,
$\phi_{0,1}: C\to e_{0,1}''Be_{0,1}''$ such that 
\beq
[\phi_{0,1}]|_{{\cal P}_2}={\lambda_1}|_{{\cal P}_2}.
\eneq
Note that $(e_{N_1+1}-e_{N_1})-e_{0,1}=e_{1,1}.$
By \eqref{TT-T-1},
we have
\beq
 {\tau(\phi_{N_1}(x))\over{\tau(e_{N_1+1}-e_{N_1})}}\ge (3/4)\tau(\Phi(x))\ge \Delta_1(\widehat{x}) \tforal x\in {\td{\cal H}}_{q,N_1}.
\eneq

Recall that $\|\rho_{B_1}(\lambda_1([1_C]))\|=\sup\{\tau(e_{0,1}): \tau\in T(B_1)\}<\sigma_2^2$ and 
{{$\sigma_n'<1/2l(n).$  Set $e_{1,1}''=e_{N_1+1}-e_{N_1}-e_{0,1}''.$}}
Applying part (d) of (3) of  \ref{DefA1} (also recall  the choice of ${\bar H}_{q,N_1}$ and Proposition \ref{Plowerbd})
%\ref{Lpartuniq} 
to the map 
$\phi_{N_1+1}$ (in place of $\phi$) and $\lambda_1,$
we obtain  a unital ${\cal G}_{2,T}$-$\dt_{2,T}$-multiplicative \morp\, $\phi_{1,1}: C\to  {{e_{1,1}''}}B{{e_{1,1}''}}$
such that
\beq
&&[\phi_{1,1}]|_{{\cal P}_2}=([\phi_{N_1+1}]-\lambda_1)|_{{\cal P}_2},\andeqn\\
&&|\tau(\phi_{1,1}(h))-\tau(\phi_{N_1+1}(h))|<3\sigma_2\rforal h\in {\cal H}_2\andeqn \tau\in T(B_1).
\eneq

Put $B_1':=(e_{0,1}+e_{N_1})B(e_{0,1}+e_{N_1})$ {{and $B_1'':=(e_{0,1}'+e_{N_1}')B(e_{0,1}'+e'_{N_1}).$}}
Consider maps ${\bar \psi}_1:=L\oplus  {\rm Ad}\, V\circ \psi_{0,1}: C\to {{B_1''}}$ and ${\bar \phi}_1:=(\sum_{i=1}^{N_1}\phi_i)\oplus \phi_{0,1}:
C\to B_1'.$   
It follows from \eqref{87-k-14} that
\beq
[{\bar \psi}_1]|_{{\cal P}_1}=[{\bar \phi_1}]|_{{\cal P}_1}.
\eneq
{{In particular, $[{\bar \psi}_1(1_C)]=[{\bar \phi_1}(1_C)].$ It follows that there is a unitary $w_1\in {\tilde B}$ 
such  that $w_1^*{\bar \psi}(1_C)w_1={\bar \phi_1}(1_C).$ In other words, $w_1^*(e_{N_1}'+e_{0,1}')w_1=e_{N_1}+e_{0,1}.$}} 
%$w_1^* 
Note also, by \eqref{TT-T-1}, 
\beq
\tau\circ {\bar \psi_1}(h)\ge \tau\circ L(h)\ge (3/4) \Delta_0(\hat{h})-2\sigma_2\ge (1/2)\Delta_0(\hat{h})\rforal h\in {\cal H}_{q,1}.
\eneq
Moreover,   by \eqref{TT-T-1},  for all $c\in {\cal H}_1,$ 
\beq
|\tau\circ {\bar \psi}_1(c)-\tau\circ {\bar \phi}_1(c)|<2\sigma_1\le \eta_1'\rforal \tau\in T(B_1').
\eneq
By (2) of \ref{DefA1}, there is a unitary $u_1'\in U(B_1'')$ such that
\beq
(u_1')^*w_1^*{\bar \psi_1}(c)w_1u_1'\approx_{\ep/2^{1+5}} {\bar \phi}_1(c)\rforal c\in {\cal F}_1.
\eneq
{{Put $u_1=w_1u_1'.$}}

\vspace*{2ex}
{{To make the notation consistent:  
From now until the end of the proof, we replace ``k" with ``n" (except in mod k K groups
$K_i(C, Z/kZ)$).
We replace ``m" with ``M", and ``n" with ``N".}}\\

For each $n,$ since $[\Psi_n]$ is $\td {\cal G}_{n+1}$-$\eta_{n+1}$-multiplicative, $[\Psi_n]|_{{\cal P}_{n+1}}$
gives an element $\kappa_n'\in KL_{loc}(G^{{\cal P}_{n+1}}, \underline{K}(B))$  
and $\kappa_n'(K_{n+1,p})\subset K_0(B)_+\setminus \{0\}.$ 
By 
%\eqref{89-k-1}   and 
%and \eqref{TT-e-1}, $\kappa_0'({\cal K}_{2,p})\subset K_0(B)_+\setminus \{0\},$ and by
%\eqref{TT-e-2}, 
\eqref{TT-e-kerrho}, 
$\kappa_n'(G_{o,n+1})
\subset {\rm ker}\rho_B.$

Write $[p_{n+1,j}]=f_{n+1,j}\oplus x_{n+1,j}\in F_{0,n+1}\oplus G_{o,n+1},$ 
$j=1,2,...,s(n+1).$
Let $T_n\in \N$ such that $T_nx=0$ for all 
$x\in {\rm Tor}(G^{{\cal P}_{n+1}}).$  We may choose $T_n>2.$
%Choose an integer $M(1)\ge 1$ such that $Mf_{1,j}\oplus (\pm x_{1,j})\in K_0(C)_+\setminus \{0\}$
%for any $M>M(1),$ $j=1,2,...,r(1).$ Choose $M_1\ge \max\{M(1)+1, T_1+1\}.$
 Define $\kappa_n\in KL_{loc}(G^{{\cal P}_{n+1}}, \underline{K}(B))$
as follows:

On $F_{0,n+1}\oplus G_{o,n+1},$ define $\kappa_n(f\oplus z)=(T_n!-1)\kappa_n'(f)-\kappa_n'(z)$
for all $f\in  F_{0,n}$ and $z\in G_{0,n};$
define $\kappa_n(x)=-\kappa_n'(x)$ for all $x\in K_1(C)\cap G^{{\cal P}_{n+1}},$
and for all $x\in G^{{\cal P}_{n+1}}\cap K_i(C, \Z/k\Z)$ ($i=0,1$),  $k=2,3,...,.$ 
Note, for $[p_{n+1,j}]\in K_{n+1,p}$ and $1\le j\le s(n+1),$ 

\beq
\kappa_n([p_{n+1,j}])&=&(T_{n}!-1)\kappa_n'(f_{n+1,j})-\kappa_n'(x_{n+1,j})\\
&=&(T_n!-1)\kappa_1'([p_{n+1,j}])-T_n! \kappa_1'(x_{n+1,j}).
\eneq
Note that
%, by \eqref{TT-e-2},  
$\kappa_n'(x_{n+1,j})
\in {\rm ker}\rho_B.$ 
Since $B$ is a  simple \CA\, with tracial rank zero, 
$\kappa_n([p_{n+1,j}])\in K_0(B)_+\setminus \{0\},$   $1\le j\le s(n+1).$

{{Note that, by \eqref{TT-t-01}
%{Mar2120222AM} 
and by  \eqref{813-n1},
%
%\eqref{Mar2120222AM}, 
$\tau(\phi_{N_1 +n}(h)) \geq \Delta_1(\hat{h})$
for all $h \in \tilde{\cal H}_{q,N_1 +n}$ and $\tau \in T(B_n)$.}} 
By applying  Lemma \ref{Lpartuniq} to the map 
$\phi_{N_1+n}$ (in place of $\phi$) and $\kappa_n,$   we obtain 
mutually orthogonal projections $e_{0,n}, e_{1,n}\in (e_{N_1+n}-e_{N_1+n-1})B(e_{N_1+n}-e_{N_1+n-1})$ and unital
${\cal G}_{n+1}$-$\dt_{n+1}$-multiplicative \morp s $\psi_{0,n}: C\to e_{0,n}Be_{0,n}$ and $\psi_{1,n}:C\to e_{1,n}Be_{1,n},$
and a unitary $u_{1,n}\in B_n$
%e_{N_1+1}-e_{N_1})B(e_{N_1+1}-e_{N_1})$
such that
\beq\label{TT-t-01+}
&&e_{0,n}+e_{1,n}=e_{N_1+n}-e_{N_1+n-1}, \,\,\, \tau(e_{0,n})<\sigma_{n+1}^2\tforal \tau\in T(B_n),\\
&&{[\psi_{0,n}]}|_{{\cal P}_{n+1}\cap {\rm ker}\rho_{C,f}}={\kappa_n}|_{{\cal P}_{n+1}\cap {\rm ker}\rho_{C,f}},\,\,\,\\
&&{[}\psi_{0,n}{]}|_{{\cal P}_{n+1}\cap K_1(C)}={\kappa_n}|_{{\cal P}_{n+1}\cap K_1(C)},\\
&&{[\psi_{0,n}]}|_{{\cal P}_{n+1}\cap K_i(C, \Z/k\Z)}={\kappa_n}|_{{\cal P}_{n+1}\cap K_i(C, \Z/k\Z)},\,\,\, k=2,3,...,\\\label{TT-t-9n}
&&{{[\psi_{0,n}]|_{{\cal P}_n}+[\psi_{1,n}]|_{{\cal P}_n}=[\phi_{N_1+n}]|_{{\cal P}_n}\andeqn}}\\\label{TT-t9n++}
&&\|u_{1,n}^*\phi_{N_1+n}(c)u_{1,n}-(\psi_{0,n}(c)+\psi_{1,n}(c))\|<\sigma_{n+1}\rforal c\in {\cal G}_{n+1}.
\eneq
Note 
$\psi_{0,n}(1_C)=e_{0,n}$ and $\psi_{1,n}(1_C)=e_{1,n}.$
Note that  $K_0(C, \Z/k\Z)\cap G^{{\cal P}_{n+1}}=\{0\}$ for any $k>T_n.$ 
Therefore, for $x\in (K_0(C)/kK_0(C))\cap G^{{\cal P}_{n+1}}$ ($k\le T_n$)
\beq
\kappa_n'(x)+[\psi_{0,n}](x)=(T_n!-1)\kappa_n(x)+\kappa_n(x)=-\kappa_n(x)+\kappa_n(x)=0.
\eneq
Also, if $x\in {\rm Tor}(K_0(A))\cap G^{{\cal P}_{n+1}},$ 
\beq
\kappa_n'(x)+[\psi_{0,n}](x)=(T_n!-1)\kappa_n(x)+\kappa_n(x)=0.
\eneq
One then computes that, for $n\ge 2,$  if $x\in G^{{\cal P}_n}\cap {\rm ker}\rho_{C,f},$ 
$x\in K_1(C)\cap {\cal P}_n,$ and $x\in  K_i(C, \Z/k\Z)\cap {\cal P}_n,$
{{by \eqref{TT-t9n++},}}
\eqref{87-k-10} and \eqref{87-K-11},
%\eqref{TT-e-kerrho},
\beq
([\psi_{0,n}+ \psi_{1,n-1}])(x)=[\psi_{0,n}](x)+([{{\phi}}_{N_1+n-1}]-[\psi_{0,n-1}])(x)\\
=-[\Psi_n](x)+[{{\phi}}_{N_1+n-1}](x) +  [\Psi_{n-1}](x)=0. 
\eneq
%\beq
%&&[\psi_{0,n}+\psi_{1,n-1}]|_{K_i(C, \Z/k\Z)\cap {\cal P}_1}=0,\,\,i=0,1,\\
%&&{[}\psi_{0,1}+L{]}|_{{\rm ker}\rho_{C,f}\cap G^{{\cal P}_1}}=0,\andeqn\\
%&&{[}\psi_{0,1}+L{]}|_{K_1(C)\cap {\cal P}_1}=0.
%\eneq

Define {{(for $n=1,2,...$)}} $\lambda_n\in KL_{loc}(G^{{\cal P}_{n+1}}, \underline{K}(B))$  by 
${\lambda_n}|_{{\cal P}_{n+1}}=([\Psi_n]-[\Phi_n]+[\psi_{0,n}])|_{{\cal P}_{n+1}}.$
%[\sum_{n=1}^{N_1}\phi_n])|_{{\cal P}_1}}.$
Note that, by \eqref{TT-e-kerrho},
{{\beq
\hspace{-0.3in}\lambda_n(x)=([\Psi_n]-[\Phi_n]+[\psi_{0,n}])(x)=\kappa_n'(x)-[\Phi_n](x)+[\psi_{0,n}](x)=-[\Phi_n](x)\in {\rm ker}\rho_B
\eneq}}
for $x\in G^{{\cal P}_{n+1}}\cap {\rm ker}\rho_{C,f},$  and,  by \eqref{89-k-1} and \eqref{87-k-10},  if $x\in G^{{\cal P}_n}\cap {\rm ker}\rho_{C,f},$
\beq
\lambda_n(x)=
%([\Psi_1]-[\Phi_1]+[\psi_{0,1}])(x)=\kappa_1'(x)
-[\Phi_{n-1}](x)=0,
%+[\psi_{0,1}](x)
\eneq
 By \eqref{TT-e-10}, for  $j=1,2,...,s(n+1),$ 
\beq
\rho_B(\lambda_n([p_{{n+1},j}]))&=&\rho_B([\Psi_n]([p_{n+1,j}])+[\psi_{0,n}]([p_{n+1,j}])-[\Phi_n]([p_{n+1,j}]))\\
&=&\rho_B([\psi_{0,n}]([p_{n+1,j}]))>0.
\eneq
In other words, $\lambda_n([p_{n+1,j}]) \in K_0(B)_+\setminus \{0\},$ $j=1,2,...,s(n+1).$
Note that, for $n\ge 1,$ 
%{{(see \eqref{813-c-n1})}}
\beq
\lambda_n([1_C]))=[e_{N_1+n}']-[e_{N_1+n}]+[e_{0,n}])=
[e_{0,n}].
\eneq
By  (b) of  part (3) of   \ref{DefA1},
%\ref{Lpartuniq}, 
there is a unital ${\cal G}_{n+1, T}$-$\dt_{n+1, T}$-multiplicative \morp\,
$\phi_{0,n}: C\to e_{0,n}Be_{0,n}$ such that 
\beq
[\phi_{0,n}]|_{{\cal P}_{n+1}}={\lambda_n}|_{{\cal P}_{n+1}}.
\eneq
Note that $(e_{N_1+n}-e_{N_1+n-1})-e_{0,n}=e_{1,n}.$
By \eqref{TT-T-1},
we have
\beq
 {\tau(\phi_{N_1+n}(x))\over{\tau(e_{N_1+n}-e_{N_1+n-1})}}\ge (3/4)\tau(\Phi(x))\ge \Delta_1(\widehat{x}) \tforal x\in {\td{\cal H}}_{q,N_1+n}.
\eneq

Recall that $\|\rho_{B_n}(\lambda_n([1_C]))\|=\sup\{\tau(e_{0,n}): \tau\in T(B_n)\}<\sigma_{n+1}^2.$
Applying part (d) of (3) of  \ref{DefA1} (also recall  the choice of ${\bar H}_{q,N_1+n}$ and Proposition \ref{Plowerbd})
%\ref{Lpartuniq} 
to the map 
$\phi_{N_1+n}$ (in place of $\phi$) and $\lambda_n,$
we obtain  a unital ${\cal G}_{n+1}$-$\dt_{n+1}$-multiplicative \morp\, $\phi_{1,n}: C\to  e_{1,n}Be_{1,n}$
such that
\beq
&&[\phi_{1,n}]|_{{\cal P}_{n+1}}=([\phi_{N_1+n}]-\lambda_n)|_{{\cal P}_{n+1}},\andeqn\\\label{TT-t-10}
&&|\tau(\phi_{1,n}(h))-\tau(\phi_{N_1+n}(h))|<3\sigma_{n+1}\rforal h\in {\cal H}_{n+1}\andeqn 
\tau\in  T(B_n).
\eneq
Note that
\beq
&&[\phi_{N_1+n}]|_{{\cal P}_{n+1}}=([\phi_{1,n}]+[\phi_{0,n}])|_{{\cal P}_{n+1}},\\
&&|\tau(\phi_{N_1+n}(c))-\tau(\phi_{1,n}(c)+\phi_{0,n}(c))|<4\sigma_{n+1}\le \eta_{n+1}'\tforal c\in {\cal H}_{n+1}
\eneq
for all  $\tau\in T(B_n).$  Combining with \eqref{TT-t-01},  applying part (2) of \ref{DefA1}, we obtain a unitary
$v_{1,n}\in B_n$
such that
\beq\label{TT-e-18}
v_{1,n}^*\phi_{N_1+n}(c)v_{1,n}\approx_{\ep/2^{n+5}} \phi_{1,n}(c)+\phi_{0,n}(c)\rforal c\in {\cal F}_{n+1}.
\eneq

Put $B_n':=(e_{0,n}+e_{1,n-1})B(e_{0,n}+e_{1,n-1}),$ $n=2,3,...$
Consider maps ${\bar \psi}_n:=\psi_{0,n}\oplus \psi_{1, n-1}: C\to B_n'$ and ${\bar \phi}_n:=\phi_{0,n}\oplus \phi_{1,n-1}:
C\to B_n'.$     For $n\ge 2,$  {{by \eqref{TT-t9n++},}}
\eqref{87-k-10} and \eqref{87-K-11},
\beq
\hspace{-0.6in}{[{\bar \phi}_n]}|_{{\cal P}_n}&=&(([\phi_{N_1+n-1}]-[\phi_{0,n-1}])+[\phi_{0,n}])|_{{\cal P}_n}\\
&=&(({[\phi_{N_1+n-1}]}-[\phi_{0,n-1}])+([\Psi_n]-[\Phi_n]+[\psi_{0,n}]))|_{{\cal P}_n}\\\nonumber
\hspace{-0.2in}&=&(([\phi_{N_1+n-1}]-([\Psi_{n-1}]-[\Phi_{n-1}]+[\psi_{0,n-1}])+([\Psi_n]-[\Phi_n]+[\psi_{0,n}])))|_{{\cal P}_n}\\\nonumber
&=&(([\psi_{0,n}]+[\phi_{N_1+n-1}]-[\psi_{0,n-1}]))+(([\Psi_n]-[\Psi_{n-1}])-([\Phi_n]-[\Phi_{n-1}])))|_{{\cal P}_n}\\
&=&(([\psi_{0,n}]+[\psi_{1,n-1}])+([\phi_{N_1+n}]-[\phi_{N_1+n}]))|_{{\cal P}_n}\\
&=&(([\psi_{0,n}]+[\psi_{1, n-1}])|_{{\cal P}_n}=[{\bar \psi}_n]|_{{\cal P}_n}.
\eneq
Note also, by \eqref{TT-t-10},
%{TT-T-1}, 
\beq
\tau\circ {\bar \psi_n}(h) &\ge& \tau\circ \psi_{1, n-1}(h)\ge \tau\circ \phi_{N_1+n}(h)-3\sigma_{n+1}\\
&\ge& 
(3/4) \Delta_0(\hat{h})-3\sigma_{n+1}\ge (1/2)\Delta_0(\hat{h})\rforal h\in {\cal H}_{q,n}.
\eneq
Recall that ${\cal H}_n\subset {\cal G}_n.$
By \eqref{TT-t-9n}, \eqref{TT-t-01},  for all $c\in {\cal H}_n$   and, for all $\tau\in T(B_n'),$
\beq
\tau({\bar \psi}_n(c))\approx_{\dt_{n+1}+\sigma_{n+1}}\tau(\phi_{N_1+n}(c))\approx_{\sigma_{n+1}}
\tau(\phi_{1,n}(c))\approx_{\sigma_{n+1}}
\tau({\bar \phi}_{n}(c)).
%|\tau\circ {\bar \psi}_1(c)-\tau\circ {\bar \phi}_1(c)|<2\sigma_1\le \dt_1\rforal \tau\in T(B_1').
\eneq
It follows from part of (2) of \ref{DefA1} that there exists a unitary 
$u_n\in U(B_n')$   such that  ($n=2,3,...$)
\beq\label{TT-t-22}
u_n^*{\bar \psi}_n(c)u_n\approx_{\ep/2^{n+5}} {\bar \phi}_n(c)\rforal c\in {\cal F}_{n}.
\eneq
Define ${\bar e}_1:=e_{0,1}+e_{N_1},$ 
${\bar e}_{n+1}:=(e_{N_1+n}-e_{0,n})+e_{0, n+1},$ $n=1, 2,....$ 
%Define $E_N:=\sum_{n=1}^N{\bar e}_n,$  $N=1,2,....$
Define 
\beq
&&\hspace{-0.8in}U_1:=\sum_{n=1}^\infty V^*u_{1,n}V+e_{N_1}',\,\,U_2:=V^*\sum_{n=2}^\infty u_n+u_1{\bar e}_1
\andeqn U_3:=\sum_{n=1}^\infty v_{1,n}^*+{\bar e}_1.
\eneq
Then $U_1,\, U_2,\, U_3\in U(M(B)).$  Put $U:=U_1U_2U_3$ {{and
${\tilde \Psi}(c)=\sum_{n=1}^\infty V^*\phi_{N_1+n}(c)V+L(c)$ for all $c\in C.$}}
By  \eqref{TT-t-9n}, \eqref{TT-t-22}  and  \eqref{TT-e-18},  for all $c\in {\cal F}={\cal F}_1,$ 
\beq
\hspace{-0.2in}U^*{{{\tilde \Psi}(c)}}U=U_3^*U_2^*U_1^*{\tilde \Psi}(c)U_1U_2U_3=U_3^*U_2^*(\sum_{n=1}^\infty V^*u_{1,n}\phi_{N_1+n}(c)u_{1,n}V+L(c))U_2U_3\\
\approx_{\ep/2^5} U_3^*U_2^*(\sum_{n=1}^\infty {{V^*(\psi_{0,n}(c)+\psi_{1,n}(c))V}}+L(c))U_2U_3\\
=U_3^*U_2^*(\sum_{n=1}^\infty V^*(\psi_{0,n+1}(c)+\psi_{1,n}(c))V+(V^*\psi_{0,1}(c)V+L(c)))U_2U_3\\
=U_3^*(\sum_{n=1}^\infty u_{n+1}^*(\psi_{0,n+1}(c)+\psi_{1,n}(c))u_{n+1}+u_{1}^*(V^*\psi_{0,1}(c)V+L(c))u_{1})U_3\\
\approx_{\ep/2^5}
%U_3^*(\sum_{n=1}^\infty {\bar \phi_{n+1}}+\sum_{n=1}^{N_1}\phi_n+\psi_{0,1}(c))U_3\\
U_3^*(\sum_{n=1}^\infty (\phi_{0,n+1}(c)+\phi_{1,n}(c)) +\sum_{i=1}^{N_1}\phi_i(c)+\phi_{0,1}(c))U_3\\
=U_3^*(\sum_{n=1}^\infty (\phi_{1,n}(c)+\phi_{0,n}(c)) +\sum_{i=1}^{N_1}\phi_i(c))U_3\\
=\sum_{n=1}^\infty v_{1,n}(\phi_{1,n}(c)+\phi_{0,n}(c))v_{1,n}^*+\sum_{i=1}^{N_1}\phi_i(c)\\
\approx_{\ep/2^5}\sum_{n=1}^\infty \phi_{N_1+n}(c)+\sum_{i=1}^{N_1}\phi_i(c)\approx_{\eta_1} \Phi(c).
\eneq 
In other words, 
\beq\label{TT-e-19}
\|U^*\Psi(c)U-\Phi(c)\|<\ep\rforal c\in {\cal F}.
\eneq
Similarly, for each $M>1$ and any $c\in {\cal F}_M,$ if $N\ge M,$ 
\beq
{{U^*(\sum_{n=N}^\infty V^*\phi_{N_1+n}(c)V)U=U_3^*U_2^*(\sum_{n=N}^\infty (V^*u_{1,n}^*\phi_{N_1+n}(c)u_{1,n}V))U_2U_3}}\\
\approx_{\ep/2^{N+4}} U_3^*U_2^*(\sum_{n=N}^\infty V^*(\psi_{0,n}+\psi_{1,n})V)U_2U_3\\
=U_3^*U_2^*V^*(\sum_{n=N}^\infty (\psi_{0,n+1}(c)+\psi_{1,n}(c))+(\psi_{0,N}(c)))VU_2U_3\\
=U_3^*(\sum_{n=N}^\infty u_{n+1}^*(\psi_{0,n+1}(c)+\psi_{1,n}(c))u_{n+1}+u_{N}^*\psi_{0,N}(c)u_{N})U_3\\
=U_3^*(\sum_{n=N}^\infty (\phi_{0,n+1}(c)+\phi_{1,n}(c)) +u_{N}^*\psi_{0,N}(c)u_{N})U_3\\
=\sum_{n=N+1}^\infty v_{1,n}(\phi_{1,n}(c)+\phi_{0,n}(c))v_{1,n}^*+U_3^*\phi_{1, N}(c)+u_{N}^*\psi_{0,N}(c)u_{N})U_3\\
\approx_{\ep/2^{N+4}}\sum_{n=N+1}^\infty \phi_{N_1+n}(c)+U_3^*(\phi_{1,N}(c)+u_{N}^*\psi_{0,N}(c)u_{N})U_3.
\eneq
Put $\Lambda_N(c)=\sum_{n=N+1}^\infty \phi_{N_1+n}(c)+U_3^*(\phi_{1,N}(c)+u_{N}^*\psi_{0,N}(c)u_{N})U_3.$
{{Then}} the {{inequalities}} above may also be written as 
\beq\label{TT-e-20}
\|U^*(\sum_{n=N}^\infty  V^*\phi_{N_1+n}(c)V)U-\Lambda_N(c)\|<\ep/2^{N+3}\rforal c\in {\cal F}_M.
\eneq
For each $c\in {\cal F}_M,$   by  \eqref{TT-T-001} and  \eqref{84-B-1},
\beq
&&\hspace{-0.6in}c_{0,N}:=U^*\Psi(c)U-U^*(\sum_{n=N}^\infty V^*\phi_{N_1+n}(c)V)U=U^*(\Psi(c)-\sum_{n=N}^\infty V^*\phi_{N_1+n}(c)V)U\in B\\
&&\hspace{-0.4in}\andeqn
c_{1,N}:=\Lambda_N(c)-\Phi(c)\in B.
\eneq
In particular,
\beq
c_{0,N}+c_{1,N}\in B.
\eneq
We have, for $c\in {\cal F}_M,$
\beq
(U^*\Psi(c)U-\Phi(c))=(c_{0,N}+c_{1,N})+(U^*(\sum_{n=N}^\infty V^*\phi_{N_1+n}(c)V)U-\Lambda_N(c)).
\eneq
It follows from \eqref{TT-e-20}
that
\beq
\lim_{N\to\infty}\|(U^*\Psi(c)U-\Phi(c))-(c_{0,N}+c_{1,N})\|=0.
\eneq
Hence $U^*\Psi(c)U-\Phi(c)\in B$ for all $c\in {\cal F}_M.$ 
Since $\cup_{M=1}^\infty {\cal F}_M$ is dense in the unit ball of $C,$ we conclude that
\beq
U^*\Psi(c)U-\Phi(c)\in B\rforal c\in C.
\eneq
Combining with \eqref{TT-e-19}, the lemma follows.

\end{proof}

\section{Class ${\cal A}_0$}

\label{section:ClassA0}

Let us first recall some terminology.   

{{Recall that for a unital C*-algebra $A$,  and for every
unitary $u \in {{U_0(A)}}$, the \emph{exponential length of $u$} is 
$$cel(u) := \inf \left\{ \sum_{l=1}^m \| h_l \| : u = \prod_{l=1}^m exp(i h_l),
\makebox{ } h_l \in A_{SA} \makebox{ for all } 1 \leq l \leq m,
\makebox{ and } m \geq 1 \right\}.$$
Define the \emph{exponential length of $A$} to be   
$$cel(A) := \sup \{ cel(u) : u \in {{U_0(A)}} \}.$$
For each $n \geq 1$, 
let $exp(iA_{SA})^n$ denote all products of $n$ exponentials of the form
$\prod_{l=1}^n exp(i h_l)$, where $h_l \in A_{SA}$ for all $1 \leq l \leq n$.
Recall that for a unitary $u \in {{U_0(A)}}$, the \emph{exponential rank of $u$} 
is 
\[
cer(u) =
\begin{cases}
 n & \makebox{ if } u \in exp(i A_{SA})^n \setminus 
\overline{exp(i A_{SA})^{n-1}}\\
n + \epsilon & \makebox{ if  } 
u \in \overline{exp(i A_{SA})^n} \setminus exp(i A_{SA})^n.\\
\end{cases}
\]
Finally, recall that the \emph{exponential rank of $A$} is
$$cer(A) := \sup \{ cer(u) : u \in U(A)_0 \}.$$}}\\
(See \cite{PhillipsRingrose}.)

\begin{df}
Let $A$ be a unital C*-algebra.
\begin{enumerate}
\item For any map $T : \mathbb{N}^2 \rightarrow \mathbb{N}$, $A$ is said to have
\emph{$K_0$-divisible rank $T$} if for every $n, k \in \mathbb{N}$, for every
$x \in K_0(A)$, if $-n [1_A] \leq kx \leq n [1_A]$, then
$-T(n, k ) [1_A] \leq x \leq T(n,k) [1_A]$.
\item For any map $E : \mathbb{R}_+ \times \mathbb{N} \rightarrow 
\mathbb{R}_+$, $A$ is said to have
\emph{exponential length divisible rank $E$} if for every
$m, k \geq 1$ and $u \in U(M_m(A))_0$, if $cel(u^k) \leq L$ then 
that $cel(u) \leq E(L, k)$.     
\item Let $C$ be a C*-algebra and $T_1 = N \times K : C_+ \setminus \{ 0 \}
\rightarrow \mathbb{N} \times (\mathbb{R}_+ \setminus \{ 0 \})$ be a map.
Let $\mathcal{H} \subset C_+ \setminus \{ 0 \}$ be a finite subset. 
Suppose that $\theta : C \rightarrow A$ is a contractive completely positive 
linear map.   Then $\theta$ is said to be 
\emph{$T_1$-$\mathcal{H}$-full} if for all $a \in \mathcal{H}$,
there are $x_{a,j} \in A$ with $\| x_{a,j} \| \leq K(a)$, $j = 1, 2, ..., N(a)$,
such that 
$$\sum_{j=1}^{N(a)} x_{a,j}^* \theta(a) x_{a,j} = 1_A.$$

   The map $\theta$ is said to be \emph{$T_1$-full}, if in the above, $\mathcal{H}$ can be replaced with $A_+ \setminus \{ 0 \},$ {\red{i.e.,}}
 for all $a \in A_+ \setminus
\{ 0 \}$, there exist $x_{a,j} \in A$ with $\| x_{a,j} \| \leq K(a)$, 
$j=1, 2, ..., N(a)$, such that  
$$\sum_{j=1}^{N(a)} x_{a,j}^* \theta(a) x_{a,j} = 1_A.$$  
{{(See Definition 5.5 of \cite{eglnp}.)}} 

%Note that {\red{both}} conditions all imply that $\theta$ is injective.

\end{enumerate}
\end{df}

The following theorem is an immediate consequence of Lemma 4.15 of \cite{Lincbms} (see also Theorem 7.1 of \cite{Linhomtrk1}): 

\begin{thm}\cite[Theorem 7.1]{Linhomtrk1} \label{LH71}
Let $C$ be a unital separable amenable \CA\, satisfying the UCT.  Let $b\ge 1,$ $T: \N^2\to \N, 
 \makebox{ }{\bf L}: U(M_\infty(C))\to \R_+,$ $E: \R_+\times \N\to \R_+$
 and 
$T_1=N\times K: C_+\setminus \{0\}\to \N\times (\R_+\setminus\{0\})$ be four maps. For any $\ep>0$ and any 
finite subset ${\cal F} \subset C,$ there exist a $\dt>0,$ a finite subset ${\cal G}\subset C,$ a finite subset 
${\cal H}\subset C_+\setminus \{0\},$   a finite subset
${\cal P}\subset \underline{K}(C),$ a finite subset ${\cal U}\subset U(M_\infty(C)),$ and integers $l>0$ and $k>0$ satisfying 
the following:
For any unital \CA\, $A$ with stable rank one, $K_0$-divisible rank $T,$ exponential length divisible rank $E$ and ${\rm cer}(M_m(A))\le b$ for all $m$, if $\phi, \psi: C\to A$ are two ${\cal G}$-$\dt$-multiplicative contractive completely positive linear maps, 
with $p_0:=\phi(1_C)$ and $q_0:=\psi(1_C)$ being projections, such that   
\beq
[\phi]|_{\cal P}=[\psi]|_{\cal P}\andeqn {\rm cel}(\la\phi(u)\ra^*\la \psi (u)\ra)\le {\bf L}(u)
\eneq
for all $u\in {\cal U}$, 
then for any unital ${\cal G}$-$\dt$-multiplicative \morp\, 
$\theta: C\to M_l(A)$ which is $T_1$-${\cal H}$-full, there exists a unitary 
$u\in M_{1+kl}(A)$ such that 
\beq
u^*\diag(\phi(c), \overbrace{\theta(c), \theta(c),...,\theta(c)}^k)u\approx_{\ep} \diag(\psi(c), \overbrace{\theta(c), \theta(c),...,\theta(c)}^k)
\eneq
for all $c\in {\cal F}$.  Also, $u^*pu=q,$
where 
$p=\diag(p_0, \theta(1_C), ...,\theta(1_C))$ and $q=\diag(q_0, \theta(1_C),...,\theta(1_C)).$
\end{thm}

\begin{proof}
We may assume that $0<\epsilon < \frac{1}{10}$, and that all
the elements of ${\cal F}$ have norm less than or equal to one.

This is just a {{minor}} variation on Theorem 7.1 of \cite{Linhomtrk1}. 
If we assume that $\phi(1_C)=\psi(1_C)=1_A,$ then the current statement is the same as that of Theorem 7.1 of \cite{Linhomtrk1}.
One way to prove the case where $\phi$ or $\psi$ is non-unital  is by reducing
to the unital case via unitizing both $\phi$ and $\psi$. 
%Set $T_2:=2T_1.$ 
Let  $C':=\C\oplus C$ and $A':=M_2(A)$. Let $T_1=N\times K: C_+\setminus \{0\}\to  \N \times (\R_+ \setminus \{ 0 \})$
be as given. Define $T_2: N_1\times K_1: C'_+\setminus \{0\}\to \N \times 
(\R_+ \setminus \{ 0 \})$
as 
follows: $N_1((\lambda, c))=2+2N(c)$ and $K_1((\lambda, c))=2/\lambda + 
2/\sqrt{\lambda} +2K(c)$ if $\lambda\in \R_+\setminus\{0\}$ 
and $c\in C_+\setminus \{0\}$;  $N_1((0, c))=2N(c)$ and
 $K_1((0,c))=2K(c)$ if $c\in C_+\setminus \{0\}$;  
$N_1((\lambda, 0))=2$ and $K_1((\lambda, 0))=2/\lambda + 2/\sqrt{\lambda}$ if
$\lambda\in \R_+\setminus \{0\}.$ 
Define ${\bf L}_1: U(M_\infty(C'))\to \R_+$ by 
${\bf L}_1((\lambda, u))={\bf L}(u) + 2 \pi$ for $\lambda\in \T$ and $u\in U(M_\infty(C)).$

Set ${\cal F}_1=\{(\lambda, c): \lambda= 0, 1 \makebox{ and } c\in {\cal F}\}.$
Applying Theorem 7.1 of \cite{Linhomtrk1} to $C'$ with ${\ppl{\ep/4}}$ and ${\cal F}_1$ (in place of ${\cal F}$),
we obtain a $\dt_1 > 0$, finite subsets ${\cal G}_1\subset C',$ ${\cal H}_1\subset C'_+\setminus \{0\}$, 
${\cal P}_1\subset \underline{K}(C')$ and ${\cal U}_1\subset U(M_\infty(C')),$  and integers $l, k \geq 1.$
We may assume that ${\cal G}_1=\{(\lambda, g): \lambda\in \C, g\in {\cal G}\},$ where ${\cal G}\subset C$ is a finite subset,
${\cal H}_1=\{(\af, 0), \makebox{ }(0, h): \af >0, h \in {\cal H}\},$ where ${\cal H}\subset C_+\setminus \{0\}$ is 
a finite subset, ${\cal P}_1=\{[(1,0)]\}\cup {\cal P},$ where ${\cal P}\subset \underline{K}(C)$ is a finite subset,
and ${\cal U}_1=\{\diag(1, u): u\in {\cal U}\},$ where ${\cal U}\subset U(M_\infty(C))$ is a finite subset. 
We may assume that $[1_A] \in \Pp$.
Note that here, $\G_1$ and $\Hh_1$ are no longer finite sets, but this will not
be a problem for us.

Now let $\phi,  \psi : C \rightarrow A$ and $\theta: C\to M_{l}(A)$ be
given as in the statement of the theorem, with $\delta = \delta_1$. 
 Define $\phi_0, \psi_0,: C'\to A' = M_2(A)$ by 
$\phi_0(\lambda\oplus c)=\diag(\lambda (1_A - p_0)\oplus \phi(c), \lambda 1_A)$ 
and 
$\psi_0(\lambda\oplus c)=\diag(\lambda(1_A - q_0)\oplus \psi(c), \lambda 1_A)$
for all $\lambda \oplus c \in C' = \C \oplus C$. 
 Define 
$\theta_0:  C'\to  M_l(A') = M_{2l}(A)$ by $\theta_0(\lambda\oplus c)=
\diag(\lambda 1_{M_l(A)}, \theta(c))$ for all $\lambda \oplus c \in 
C'$.   
Then $\phi_0$ and $\psi_0$ are ${\cal G}_1$-$\dt_1$-multiplicative and 
\beq
[\phi_0]|_{{\cal P}_1}=[\psi_0]|_{{\cal P}_1}\andeqn {\rm cel}(\la \phi_0(v)^*\ra \la \psi_0(v)\ra)\le {\bf L}_1(v)
\eneq
for all $v\in {\cal U}_1.$ Moreover $\theta_0$ is ${\cal G}_1$-$\dt_1$-multiplicative and 
$T_2$-${\cal H}_1$-full.
By Theorem 7.1 of \cite{Linhomtrk1},  we obtain a unitary $W\in M_{kl+1}(A')$ such that
\beq
W^*\diag(\phi_0(c), \overbrace{\theta_0(c),..., \theta_0(c)}^k)W\approx_{\ep/4} \diag(\psi_0(c), \overbrace{\theta_0(c),..., \theta_0(c)}^k)
\eneq
for all $c\in {\cal F}_1.$
In particular, 
$\|W^*pW-q\|<\ep/2$, where $p=\diag(p_0, \theta(1_C), ...,\theta(1_C))$ 
and $q=\diag(q_0, \theta(1_C), ...,\theta(1_C))$. 
{{It follows that
%Hence, $\| q W^* p W q - q \| < \ep < \frac{1}{4}$.
%Hence, by \cite{ElliottMatroid2} Lemma 2.1 and since $A'$ has stable rank one
%(and hence is stably finite),  
%\if
there is a unitary $W_1\in  M_{kl+1}(A')$ with $\| W_1 - 1 \| <\ep/2$}} 
such that $W_1^*W^*pWW_1=q$.
%and 
%$\|W_1-1_{M_{kl+1}(A')}\|<\ep/2.$  
Let $V:=WW_1 \in M_{kl+1}(A')$ and set  $v=pVq \in M_{kl +1}(A)$. 
Then 
\beq
v^*\diag(\phi(c), \overbrace{\theta(c), ..., \theta(c)}^k)v
\approx_{\ep} \diag(\psi(c), \overbrace{\theta(c), ..., \theta(c)}^k)
\eneq
for all $c\in {\cal F}.$ Since $A$ has stable rank one, we find a unitary $u\in M_{kl+1}(A)$ such 
that $pu=v.$ This $u$ is the required unitary, and 
the theorem then follows. 
\end{proof}

%{\red{\bf We should move the following to a preliminary section as we might use this before.}}  \textbf{I am o.k. with this. Is this function used in
%Definition 2.8?}\\  

\iffalse
For the next lemma, we fix a notation.
For $0 < \delta < 1$, let
$f_{\delta} : [0, \infty) \rightarrow [0,1]$ be the unique continuous function
determined by
$$f_{\delta}(t) := 
\begin{cases}
1 & t \in [\delta, \infty)\\
0 & t \in [0, \frac{\delta}{2}]\\
\makebox{ linear on  } & [\frac{\delta}{2}, \delta].
\end{cases}
$$
\fi

For a unital \CA\,  $A$ (not necessarily simple), we say that $A$
has \emph{
strong 
strict comparison} if for all $a, b \in {{(A \otimes {\cal K})_+}}$,
if $d_{\tau}(a) < d_{\tau}(b)$ or $d_{\tau}(b) = \infty$ for all $\tau 
\in T(A)$, then $a {{\lesssim}} b$, i.e., there exists a sequence $\{ x_n \}$
in $A \otimes {\cal K}$ for which $x_n b x_n^* \rightarrow a$.

The next lemma is also known.

\begin{lem}\label{LDelta=full}
Let $C$ be a separable \CA\, and let $\Delta: C^{q, {\bf 1}}_+\setminus \{0\}\to (0,1)$ 
be an order preserving map.   Then there is  a map $T_{\Delta}: C_+\setminus \{0\}\to \N\times (\R_+ \setminus \{ 0 \})$ 
satisfying the following:

For any finite subset ${\cal H}\subset C_+\setminus \{0\},$ 
there exist $\dt>0$ and a  finite subset ${\cal G} \subset C$ (both depending on $C, \Delta$ and ${\cal H}$ only)
such that the following holds:
If $A$ is a 
unital \CA\,  with strong strict comparison and
$L: C\to A$ is a  ${\cal G}$-$\dt$-multiplicative contractive completely
positive linear map   
%and any finite subset ${\cal H}\subset C_+\setminus \{0\},$
such that
\beq
\tau(L(c))\ge \Delta(\hat{c})\rforal \tau\in  T(A)\andeqn c\in {\cal H}_{\Delta},
\eneq
where ${\cal H}_{\Delta}=\{ f_{\|a\|\over{2}}(a), f_{\|a\|\over{4}}(a): a\in {\cal H}\},$ 
then $L$ is $T_{\Delta}$-${\cal H}$-full.
%, where 
%$T(a)=([1/\Delta(\hat{a})]+1, \|a\|+1)$ for all $a\in C_+^{\bf 1}\setminus \{0\}.$ 

\end{lem}

\begin{proof}
For every $a\in C_+\setminus \{0\}$, 
there exists  $b(a)\in C^*(a)_+\setminus \{0\}$ such 
that $b(a)a=ab(a)=f_{\|a\|\over{4}}(a).$ 
%and 
%$\|b(a)\|=r(a).$ 
Define 
\beq
T_{\Delta}(a)=([2/\Delta(\widehat{f_{\|a\|\over{2}}(a)})]+1, 4(\|a\|+1)\cdot (\|b(a)\| + \| b(a) \|^{1/2} +1))
\eneq
for all $a\in C_+\setminus \{0\}$, where for every real number $r$, $[r]$ is
the least integer greater than or equal to $r$. 
Set $N(a)=[2/\Delta(\widehat{f_{\|a\|\over{2}}(a)})]+1$ and $R(a)= (\|a\|+1)\cdot (\| b(a) \| + \|b(a)^{1/2}\|+1)$, for all $a \in C_+ \setminus \{ 0 \}$.  

Fix a finite subset ${\cal H}\subset C_+\setminus \{0\}.$ Set ${\cal H}_{\Delta}=\{ f_{\|a\|\over{2}}(a), f_{\|a\|\over{4}}(a): a\in {\cal H}\}.$
%Put ${\cal H}':=\{a, f_{\|a\|\over{2}}(a), f_{\|a\|\over{4}}(a): a\in {\cal H}\}.$ 
%To simplify notation, \wilog, we may assume that ${\cal H}\subset C_+^{\bf 1}\setminus \{0\}.$ 
Let 
\beq
\eta:=\min\{\min\{1/4N(a)R(a)^2: a\in \Hh \},  %{\cal H}_{\Delta}\}, 
(1/16)\Delta(\hat{c}): c\in {\cal H}_{\Delta}\}.
\eneq
{{Choose}} $\dt>0$ and finite subset ${\cal G}\subset C$ such that for
any C*-algebra $D$ and for any ${\cal G}$-$\dt$-multiplicative 
contractive completely positive linear
 map $\Phi: C\to D$,  there are \hm s 
$\phi_a': C_0({\rm sp}(a))\to D$  (for all $a\in {\cal H}$)  such that
\beq
\|\Phi(c)-\phi_a'(c)\|<\eta\rforal c\in  %{\cal H}_{\Delta} \cup\{a, b(a)\}.
\Hh_a := \{a, b(a),  f_{\sigma}(a): \sigma=\|a\|/2, \|a\|/4\}.
\eneq

Now suppose that   $L: C\to A$ is a ${\cal G}$-$\dt$-multiplicative contractive
completely positive linear map
such that 
$\tau(L(c))\ge \Delta(\hat{c})$ for all $\tau\in T(A)$ and $c\in {\cal H}_{\Delta}.$
Then,  by the choice of $\dt,$ there are \hm s $\phi_a: C_0({\rm sp}(a))\to A$
(for all $a \in \Hh$) 
such that
\beq\label{D=f-Lphia}
\|L(c)-\phi_a(c)\|<\eta\rforal c\in \Hh_a.% \{a, b(a),  f_{\sigma}(a): 
%\sigma=\|a\|/2, 
%\|a\|/4\}. 
%{\cal H}_a\cup\{a, b(a)\}\andeqn a\in {\cal H}.
\eneq
It follows that for all $a\in {\cal H}$ and all $\tau \in T(A)$,   
\beq
\tau(\phi_a(c))\ge  {15\Delta(\hat{c})\over{16}}\rforal c\in \Hh_a. 
%{\cal H}_{\Delta}.
\eneq

%Put $d(c)= [2/\Delta(\widehat{f_{\|c\|\over{2}}(c)})]+1$ for all $c\in C_+\setminus \{0\}.$
Then, for all $a\in {\cal H}$ and all $\tau \in T(A)$,   
\beq
N(a)d_\tau(\phi_a(f_{\|a\|\over{2}}(a)))>1.  %\rforal \tau\in T(A).
\eneq

Fix an arbitrary $a \in \Hh$.

Hence, since $A$ has strong strict comparison, 
there exists a sequence $\{x_n\}\in M_{N(a)}(A)$ 
such that $x_n^*x_n\to \diag(1_A, 0,...,0)$ and for all $n$,
$x_n x_n^*\in  
{\rm Her}(\phi_a(f_{\|a\|\over{2}}(a))\otimes 1_{N(a)}).$ 
{{Since  $\diag(1_A, 0,...,0)$ is a projection,}} it follows that 
 there is an $x\in M_{N(a)}(A)$ such that
$x^*x=\diag(1_A,0,...,0)$ and $xx^*\in {\rm Her}(\phi_a(f_{\|a\|\over{2}}(a))\otimes 1_{N(a)}).$
Note that $\|x\|=1$ and $xx^*(\phi_a(f_{\|a\|\over{4}}(a))\otimes 1_{N(a)})=xx^*=(\phi_a(f_{\|a\|\over{4}}(a))\otimes 1_{N(a)})xx^*.$
By replacing $x$ with $x\diag(1_A,0,...,0),$ 
we may write 
\beq
x=\sum_{i=1}^{N(a)} x_i\otimes e_{i,1},
\eneq
where $x_i\in A,$ $i=1,2,...,N(a)$ and $\{e_{i,j}\}$ is a system of matrix units for $M_{N(a)}.$
Note that $\|x_i\|\le 1,$ $i=1,2,...,N(a).$
Thus, we have that $\sum_{i=1}^{N(a)}x_i^*\phi_a(f_{\|a\|\over{4}}(a))x_i=1_A.$ 
It follows that
\beq
\sum_{i=1}^{N(a)} x_i^*\phi_a(b(a))^{1/2}\phi_a(a)\phi_a(b(a))^{1/2}x_i=1_A.
\eneq
By \eqref{D=f-Lphia} and the choice of $\eta,$ 
\beq
\|\sum_{i=1}^{N(a)} x_i^*\phi_a(b(a))^{1/2}L(a)\phi_a(b(a))^{1/2}x_i-1_A\|
<1/4.   %\rforal a\in {\cal H}.
\eneq
There is an $e(a)\in A_+$ with $\|e(a)\|<\sqrt{4/3}$ such that
\beq
\sum_{i=1}^{N(a)} e(a)x_i^*\phi_a(b(a))^{1/2}L(a)\phi_a(b(a))^{1/2}x_ie(a)
=1_A.  %\rforal a\in {\cal H}.
\eneq
Note that $\|\phi_a(b(a))^{1/2}x_ie(a)\|\le \sqrt{4/3}R(a).$
Since $a \in \Hh$ was arbitrary,
it follows that $L$ is $T_{\Delta}$-${\cal H}$-full. 
\end{proof}

{{Let $k\in \N.$ Denote by ${\cal S}_k$ the family of all separable \CA s whose irreducible representations 
have rank no more than $k.$}}

\begin{lem}\label{LA1-Sr}
Every unital separable \CA\, in the class $S_r$ (for some $r\in\N$) satisfies  
condition (2) of Definition \ref{DefA1}.

\end{lem}

\begin{proof}
Fix an arbitrary $C\in S_r$. Let $\ep>0,$ ${\cal F}\subset C$ be a finite subset and $\Delta:  C_+\setminus \{0\}\to (0,1)$ be an order preserving map.    Set
 $\Delta_1=\Delta/7.$
Let $T_{\Delta_1}$ be given by Lemma \ref{LDelta=full}.

 To apply Theorem \ref{LH71}, 
let $b=2,$  {{$T(n,j)=[n/j]+1$ for all $(n,j) \in \N^2$,}} ${\bf L}(u)=2\pi+1$ for all $u\in U(M_\infty(C)),$  and 
{{$E(s,j)=2\pi+1$ for all $(s,j) \in \R \times \N$.}}  
(Recall that for every 
real number $t$, $[t]$ is the least integer greater than or equal to $t$.) 

Let $\dt_1>0$ (in place of $\dt$), and ${\cal G}_1\subset C$ (in place of ${\cal G}$), ${\cal H}_1\subset C_+\setminus \{0\}$ 
(in place of ${\cal H}$), 
${\cal P}_1\subset \underline{K}(C)$  (in place of ${\cal P}$)  and ${\cal U}\subset U(M_\infty(C))$ be finite subsets,
and $l$ and $k$ be integers given by Theorem \ref{LH71}  associated with ${\cal F},$ $\ep/4,$ and $b,$ $T,$ 
$T_{\Delta_1}$ (in place of $T_1$), ${\bf L}$ and 
$E$ as above. 
We may assume that $m \geq 1$ is an integer for which
 $u\in U(M_m(C))$ for all $u\in {\cal U}$. 
We may also assume that $[u]\in {\cal P}_1$ for all $u\in {\cal U}.$

Let $\dt_2>0$ (in place of $\dt$)  and finite subset ${\cal G}_2\subset C$ (in place of ${\cal G}$) be as in Lemma \ref{LDelta=full} for $C$, ${\cal H}_1$ and $\Delta_1.$ 

Let $\ep_0:=\min\{1/2, \ep/16, \dt_1/2, \dt_2/2\}$ and ${\cal F}_0:={\cal F}\cup {\cal G}_1\cup {\cal G}_2\cup {\cal H}_1\cup 
({\cal H}_1)_{\Delta}$,
where $(\Hh_1)_{\Delta} = \{ f_{\| a \|\over{2}}(a), f_{\| a \|\over{4}}(a) 
: a \in \Hh_1 \}.$  (See the statement of Lemma \ref{LDelta=full} for
the meaning of $\Hh'_{\Delta}$ for a general set $\Hh'$.) 

We will also apply Lemma 4.3.4 of \cite{Lincbms}. 
Let $0<\ep_1<1/4k(l+1).$ 
Let $\dt_3>0$ (in place  of $\dt$), $\sigma>0$, ${\cal G}_3\subset C$ (in place of ${\cal G}$) and ${\cal H}_2\subset C_+^{\bf 1}\setminus \{0\}$ be required by Lemma 4.3.4 of \cite{Lincbms} for $C$ (in place of $A$), 
$\ep_0$ (in place of $\ep$), ${\cal F}_0$ (in place of ${\cal F}$), 
$\ep_1$ (in place of $\ep_0$), ${\cal F}_0$ (in place of ${\cal G}_0$), 
%in the statement of Lemma 4.3.4 of \cite{Lincbms},  
$({\cal H}_1)_{\Delta}$ (in place of ${\cal H}_1$), as well as $\Delta_1.$ 

Let $\dt_4:=\min\{\ep_0, \dt_3\},$ ${\cal G}:={\cal G}_3\cup {\cal F}_3\cup {\cal H}_3,$ 
${\cal H}_q:=({\cal H}_2\cup {\cal H}_1)_{\Delta},$  and ${\cal H}:={\cal H}_q\cup {\cal H}_2\cup {\cal H}_1.$ 
Define $\eta:=\sigma/4k(l+1).$  Set $\dt=\dt_4/4.$ 

Now suppose that $A$ is a unital separable simple C*-algebra
with tracial rank zero, and
$L_1, L_2: C\to A$ are ${\cal G}$-$\dt$-multiplicative 
contractive completely positive {{linear maps such that }}  
\beq\label{LA1-Sr-1}
&&[L_1]|_{\cal P}=[L_2]|_{\cal P},\\\label{LA1-Sr-2}
&&\tau\circ L_1(c), \tau\circ L_2(c)\ge \Delta(\hat{c})\rforal c\in {\cal H}_q\andeqn\\\label{LA1-Sr-3}
&&|\tau\circ L_1(c)-\tau\circ L_2(c)|<\eta\rforal c\in {\cal H} 
{{\makebox{  and  } \tau \in T(A).}}  
\eneq
We may assume that $L_1(1_C)$ and $L_2(1_C)$ are projections.

Applying \cite{CP},   we obtain  that $x_j^{(i)}(c)\in A$ ($1\le j\le D(c)$ and $i=1,2$) such that, for $i=1,2,$
\beq
\sum_{j=1}^{D(c)}x_j^{(i)^*}(c)x_j^{(i)}(c)=\Delta(\hat{c})\cdot 1_A\andeqn
\sum_{j=1}^{D(c)}x_j^{(i)}(c)x_j^{(i)*}(c)\le L_i(c)\rforal c\in {\cal H}_q;
\eneq
and  $y_j(c)\in A$ ($1\le j\le \rho(c)$) such that
\beq
\sum_{j=1}^{\rho(c)}y_j^*(c)y_j(c)=L_1(c)\andeqn {{\|\sum_{j=1}^{\rho(c)}y_j(c)y_j^*(c)-L_2(c)\|<\eta}}  
\eneq
for all $c\in {\cal H}.$ 
Put 
\beq\nonumber
&&\hspace{-0.2in}M_1:=\max\{\|x_j^{(i)}(c)\|: 1\le j\le D(c), i=1,2,\, c\in {\cal H}_q\}+{{\max\{\|y_j(c)\|: 1\le j\le \rho(c), c\in  {\cal H}\}}}\\\nonumber
&&\andeqn
N_1:=\max\{D(c): c\in {\cal H}_q\}+\max\{\rho(c): c\in {\cal H}\}.
\eneq

Recall that $A$ has tracial rank {{zero}}. Therefore $A$ has stable rank one, real rank zero and 
has strict comparison (Theorem 3.4 and 3.6 of \cite{LinTAF}).  In particular, ${\rm cer}(A)\le 2,$ ${\rm cel}(A)\le 2\pi+1,$
$K_0(A)$ is weakly unperforated 
%({\blue{\bf References}}
{{(see Theorem 5 of \cite{LinCERRR0} and Theorem 3.6 of \cite{LinTAF}).}}  It follows that $K_0(A)$ has  {{divisible rank $T$}} and 
$A$ has exponential divisible rank $E$ defined above and ${\rm cer}(M_m(A))\le 2.$ 

Put $\eta_1:=\min\{\eta/2, \min\{\Delta_1(c)/4k(l+1): c\in {\cal H}_q\}\}$ and $\eta_2:=\eta_1/2(M_1+1)N_1.$

By 
{{\eqref{LA1-Sr-1},}} there are $\dt_A>0$ and a finite subset ${\cal G}_A\subset A$
such that, if 
$e\in A$ is a non-zero projection  such that $\|ea-ae\|<\dt_A$
{{for all  $a \in {\G}_A$,}}  
then $L_i': C\to A$  defined by $L_i'(c)=eL_i(c)e$ is ${\cal G}$-$\dt$-multiplicative and 
\beq
[L_1']|_{\cal P}=[L_2']|_{\cal P}.
\eneq
Choose 
\beq
&&\hspace{-0.4in}\dt_A':=\min\{\dt_A/2, \dt,\eta_2\}]\andeqn\\
&&\hspace{-0.4in}{\cal G}_A':={\cal G}_A\cup \{x_j^{(i)}(c): 1\le j\le D(c), i=1,2,\, c\in {\cal H}_q \}\cup\{y_j(c): 1\le j\le \rho(c),\, c\in {\cal H}\}.
\eneq

For any finite subset ${{{\cal G}_A''}}\subset A$ with ${{{\cal G}_A''\supset {\cal G}_A'}},$ 
since $A$ has tracial rank zero,  there are non-zero projection $e\in A$ and 
a finite dimensional \CA\, $F\subset A$ with $1_F=(1-e)$ such that
\beq\label{LA1-Sr-5}
&&\|eg-ge\|<\dt_A'\rforal g\in {{{\cal G}_A''}},\\
&&\tau(e)<\eta_2/4k(l+1)\rforal \tau\in T(A)
\eneq
and there is a unital {{contractive completely positive linear}}
 map $\Psi: A\to F$ 
such that 
\beq\label{LA1-Sr-7}
\|\Psi(a)-(1-e)a(1-e)\|<\dt_A' \rforal a\in {{{\cal G}_A''.}}
\eneq
By choosing even  larger ${{{\cal G}_A'',}}$ we may 
assume that
 $L_i'$ ( as defined above) is ${\cal G}$-$\dt$-multiplicative.
 We also have  
\beq\label{LA1-Sr-10}
&&[L_1']|_{\cal P}=[L_2']|_{\cal P},\\\label{LA1-Sr-11}
&&t(\Psi\circ L_i(c))\ge \Delta(c)/2\rforal c\in {\cal H}_q,\,\,i=1,2, \andeqn\\\label{LA1-Sr-12}
&&|t(\Psi\circ L_1(c))-t(\Psi\circ L_2(c))|<\sigma\rforal c\in {\cal H}\,\,
\makebox{ and for all } t\in T(F).
\eneq
By Lemma 4.3.4 of \cite{Lincbms} (applying to each simple summand of $F$), 
there are mutually orthogonal non-zero projections $e_0, e_1,e_2,...,e_{2k(l+1)}\in F$ such 
that $e_1,e_2,...,e_{2k(l+1)}$ are mutually equivalent, $e_0\lesssim e_1,$ $t(e_0)<\ep_1$ for all 
$t\in T(F)$ and $\sum_{i=0}^{2k(l+1)}e_i=1_F,$ and there are unital  ${\cal F}_0$-$\ep_0$-multiplicative c.p.c. maps
$\psi_1, \psi_2: C\to e_0Fe_0$ and a unital \hm\, $\psi: C\to e_1Fe_1$ and a unitary $v\in F$ such that, for all $c\in {\cal F}_0,$
\beq\label{LA1-Sr-13}
&&\|\Psi\circ L_1(c)-\diag(\psi_1(c),\overbrace{\psi(c),\psi(c),...,\psi(c)}^{2k(l+1)})\|<\ep_0\\\label{LA1-Sr-14}
&&\|v^*\Psi\circ L_2(c)v-\diag(\psi_2(c),\overbrace{\psi(c),\psi(c),...,\psi(c)}^{2k(l+1)})\|<\ep_0\andeqn\\\label{LA1-Sr-15}
&& t(\psi(c))\ge \Delta_1(c)/(2k(l+1))\rforal c\in {\cal H}_q.
\eneq
Put $\Psi_i:=L_i'\oplus \psi_i: C\to eAe\oplus e_0Fe_0\subset (e+e_0)A(e+e_0).$
Note that $e+e_0\lesssim e_1.$ 
We also have 
that
\beq
[\Psi_1]|_{\cal P}=[\Psi_2]|_{\cal P}.
\eneq
By \eqref{LA1-Sr-3}, {{for all $\tau \in T(e_1Ae_1),$}} 
\beq
\tau(\psi(c))\ge \Delta_1(c)\rforal c\in {\cal H}_q.
\eneq
It follows from the choice of ${\cal H}_q$ that $\psi$ is $T$-${\cal H}_1$-full.
Since $[u]\in {\cal P}$ for all $u\in {\cal U}$ and $e_1Ae_1$ has real rank zero, 
\beq
{\rm cel}({{\la \Psi_1(u)\ra^*\la \Psi_2(u)\ra}})\le 2\pi+1={\bf L}(u)\rforal u\in {\cal U}.
\eneq
By Theorem \ref{LH71}, there is a unitary $w\in U(A)$ such that, for all $c\in {\cal F},$ 
\beq\label{LA1-Sr-16}
w^*(\Psi_2(c)\oplus \diag(\overbrace{\psi(c),..., \psi(c)}^{2k(l+1)}))w \approx_{\ep/4}
\Psi_1(c)\oplus \diag(\overbrace{\psi(c),..., \psi(c)}^{2k(l+1)}).
\eneq
Let $v_1=e\oplus v.$ 
Put $W:=v_1w$ and $\Phi_i=L_i'\oplus \Psi\circ L_i,$ $i=1,2.$   Then, by \eqref{LA1-Sr-14}, \eqref{LA1-Sr-16}
and \eqref{LA1-Sr-13}
\beq
W^*\Phi_2(c)W\approx_{\ep_0+\ep/4+\ep_0} \Phi_1(c)\rforal c\in {\cal F}.
\eneq
Combining this with \eqref{LA1-Sr-5} and \eqref{LA1-Sr-7}, we finally obtain that
\beq
W^*L_2(c)W\approx_{\ep} L_1(c)\rforal c\in {\cal F}.
\eneq
\end{proof}

\begin{cor}\label{LA1-IndSr}
Let $A$ be a unital separable \CA\, which is an inductive limit of \CA s in the
class $S_r.$
Then $A$ satisfies condition (2) of Definition \ref{DefA1}.
\end{cor}

\iffalse
\begin{lem}
Let $X$ be a compact metric space, $r\ge 1$ be an integer, 
$P\in M_r(C(X))$ be a non-zero projection and $C=PM_r(C(X))P.$
Suppose that $A$ be a unital separable simple \CA\, with tracial rank at most one. 

\end{lem}
\fi

%{\bf{{Ping: Please proof-read the following again---note the previous 
%revision did not reflect the previous change Definition of ${\cal A}_1$}}}

\begin{thm}
Every {{unital}} AH-algebra $C$ with at least one faithful tracial state is in
the class ${\cal A}_0.$ 
\label{thm:AHInA0}
\end{thm}

\begin{proof}
 By Theorem \ref{TembeddingAH}, $A\in {\cal A}_1$.
We can realize $C$ as $C=\lim_{n\to\infty}(C_n, \phi_n),$
where for each $n \geq 1$,
 $C_n=P_nM_{j(n)}(C(X_n))P_n,$  $j(n)\ge 1$ is an integer,
$X_n$ is a finite CW complex,  $P_n\in M_{j(n)}(C(X_n))$ is a projection, and 
$\phi_n : C_n \rightarrow C_{n+1}$ is a unital injective \hm\, (Theorem 2.1 of \cite{EGL-inj}). 
%\textbf{Reference----EGL-CJM?}
%(references:   $C(X)=\lim_{n\to\infty} C(X_n)$).

Then 
% {{$K_0(A)$ has property ({\Large{$\varrho$}}) and}}
$K_0(C_n)=\Z^{R(n)}\oplus {\rm ker}\rho_{C_n},$  where $R(n) \geq 0$
is an integer and ${\rm ker}\rho_{C_n}$ is a finitely generated 
abelian group, for all $n$. 
Moreover, for all $n$,
\beq
{{\{(z, 0): z\in \Z^{R(n)}_+\}\subset}} K_0(C_n)_+\subset \{(z, x): z\in \Z^{R(n)}_+\setminus \{0\}, x\in {\rm ker}\rho_{C_n}\}\cup \{(0,0)\}
\eneq
{{(as $K_0(C_n)$ may have perforation).}}
{{Let $z_1, z_2,...,z_{R(n)}$ form a generator set for $\Z^{R(n)}_+.$ 
Suppose that $\zeta_1, \zeta_2, ...,\zeta_{m(n)}\in {\rm ker}\ro_{C_n}$ form a generator set of ${\rm ker}\ro_{C_n}.$
There are $z_1', z_2',...,z_{m(n)}'\in \Z^{R(n)}_+\setminus \{0\}$ such that $z_i'+\zeta_i\in K_0(C_n)_+.$
Put $K_{n,p}'=\{(z_i, 0):1\le i\le R(n)\}\cup \{z_j'+\zeta_j: 1\le j\le m(n)\}\subset K_0(C_n)_+.$ 
Then $K_{n,p}'$ generates $K_0(C_n)$ and, for any $x\in K_0(C_n)_+,$ 
there are $x_1, x_2,...,x_{m(x)}, y_1, y_2, ....,y_{J(x)}\in K_{n,p}',$ $r_1,r_2,...,r_{m(x)}\in \N,$
$s_1, s_2,....,s_{J(x)}\in \Z$ such that 
\beq
x=\sum_{i=1}^{m(x)} r_i x_i+\sum_{j=1}^{J(x)} s_j y_j\andeqn \sum_{j=1}^{J(x)} s_j y_j\in {\rm ker}\ro_{C_n}.
\eneq}}

Note that this, in particular, implies that $C$ satisfies condition (1) in 
Definition \ref{DefA1}.
By Lemma \ref{LA1-IndSr}, $C$ also satisfies condition (2) of Definition
\ref{DefA1}.

%%%%%%%%%%%%%%%%%%%%%
\iffalse
Furthermore, by the first two paragraphs of this proof,  there is a finite subset of non-zero projections $K_{n,p}''\subset M_{l(n)}(C_n)$
(for some ${{l(n)}}\in  \N$) 
such that $K_0(C_n)$ is generated by 
$K_{n,p}' := \{ [p] : p \in K_{n,p}'' \}.$   We may assume that $[1_{C_n}]\in K_{n,p}'.$
 Note also that $K_0(C_n)_+\cap \Z^{R(n)}_+$ is finitely generated, and 
we may also assume that 
$$
K_0(C_n)\cap \Z^{R(n)}_+\subset \{\sum_{i=1}^m k_{i} d_{n,i}: k_i\in  \Z_+, d_{n,i}\in K_{n,p}'\}.
$$
\fi
%%%%%%%%%%%%%%%%%%%%

For each $n \geq 1$, we may also write $C_n=\bigoplus_{i=1}^{R(n)} C_{n,i},$
where for each $1 \leq i \leq R(n)$, 
$C_{n,i}=P_{n,i}M_{j(n,i)}(C(X_{n,i}))P_{n,i},$
$j(n,i) \geq 1$ is an integer,  $X_{n,i}$ is a connected finite  complex
and $P_{n,i} \in M_{j(n,i)}(C(X_{n,i}))$ is a projection. 
%If $\phi_{n, \infty}(C_{(n,j})=0,$ we may delete the summand $C_{n,i}$ without changing $C.$
%So we may assume that $\phi_{n, \infty}(C_{n,i})\not=0$ for all $1\le i\le R(n).$
Note that
 ${\phi_{n, \infty}}_{*0}({\rm ker}\rho_{C_n})\subset {\rm ker}\rho_{C,f}.$ 
Moreover, since ${\phi_{n, \infty}}_{*0}$ is order preserving, 
we  may write $G_n:={\phi_{n, \infty}}_{*0}(K_0(C_n))=\Z^{r(n)}\oplus G_{o,n},$ where 
$G_{o,n}=G_n\cap {\rm ker}\rho_{C, f}$, for all $n$.
  Let $K_{n,p}$ be the image of $K_{n,p}'$ in 
$K_0(C)$, for all $n$. Then $G_n$ is generated by $K_{n,p}$, for all $n$.

Let $n \geq 1$ be arbitrary. 
Let $A$ be a unital separable simple {{\CA\,}} with tracial rank zero.
Consider $\kappa_n\in KL_{loc} (G^{{\cal P}_n\cup K_{n,p}}, \underline{K}(A))$  satisfying the assumptions 
in (b) of part (3) of Definition \ref{DefA1}. 
Set $\kappa_n':=\kappa_n \circ [\phi_{n, \infty}]\in {\rm Hom}_{\Lambda}(\underline{K}(C_n), \underline{K}(A)).$
Since $\kappa'_n ({{\rm ker}} \rho_{C_n}) \subset 
\kappa_n(G_{o, n})\subset {\rm ker}\rho_A$ and $\kappa_n'(K'_{n,p})
= \kappa_n(K_{n,p}) \subset K_0(A)_+\setminus \{0\},$ and since 
$A$ is simple with tracial rank zero,
we conclude $\kappa_n'(K_0(C_n)_+\setminus \{0\})\subset K_0(A)_+\setminus \{0\}.$
Recall that $C_n:=\bigoplus_{i=1}C_{n,j}.$ 
Applying  Theorem 4.7 of \cite{LinKT}, we obtain a unital \hm\, $\psi_n: C_n\to A$ such 
that $[\psi_n]=\kappa_n'.$ 
Since $C_n$ is amenable, for any $\dt>0$   and any finite subset ${\cal  G}\subset C_n,$ there exists 
a unital $\G$-$\delta$-multiplicative \morp\, $\Psi: C\to C_n$ (recall that $\phi_{n, \infty}$ is injective) such that
\beq
\|\Psi(\phi_{n, \infty}(c))-c\|<\dt \rforal c\in {\cal G}. 
\eneq
Set $L_n:=\psi_n\circ \Psi_n: C\to A.$ 
Since $n$ is arbitrary, we have shown that 
$C$ satisfies the conditions (a) and (b) of part (3) of Definition \ref{DefA1}.

To prove part  {{(c) of (3)}} of Definition \ref{DefA1}, we note that it suffices to prove the case where $C$ has the form of the one of the building blocks in
the C*-inductive limit decomposition of $C$. 
By considering each summand, we may assume that $C$ has the form $PM_r(C(X))P,$
where $X$ is a connected finite  complex. We may further reduce this case to the case 
that $C=C(X).$ 
Since we now assume that $K_i(C)$, for $i = 0,1$, is finitely generated,
there is an integer $M\ge 1$ such that $Mx=0$  for any 
$x\in {\rm Tor}(K_i(C))$, 
$i=0,1.$ Set $K=M!.$

For convenience,
let us keep the notation described in parts (1), (2) and (3) (a) and (b)
of Definition \ref{DefA1}
and take $C = C_n$ for all $n$.  
So $C = C_n = C(X)$ for all $n$. 
Moreover, following the notation of Definition \ref{DefA1},
 we may assume that $K_{n,p}=\{[p]: p\in \mathtt{P}_n\}$ and $\mathtt{P}_n$ is a finite 
subset of projections in $M_{l(n)}(C(X))$ for some integer $l(n)\in 
\mathbb{N}.$

%Let ${\cal H}\subset \iota_{n, \infty}({C_n}_{s.a.}^{\bf 1}\setminus \{0\})$
Fix $n \geq 1$.  Let $\Hh \subset \iota_{n, \infty}({C_n}_{s.a.}^{\bf 1} \setminus \{ 0 \})$ 
be a finite subset
and $\eta>0.$   Assume that $({\cal G}_n, \dt_n, {\cal P}_n)$ is a KL triple associated with ${\cal H}$ and $\eta/2.$ 
Note that by Proposition \ref{Ptrace}, for sufficiently large $n$, 
such a KL triple will indeed be
associated with $\Hh$ and $\eta/2$. 

Fix  any ${\cal G}_0\supset {\cal G}_n\cup {\cal H}$ and $0<\dt_0<\dt_n.$
Let $A$ be a unital separable simple  {{\CA\,}}
% C*-algebra 
with tracial rank zero.
%choose  $0<\sigma_0<1/4.$
%\beq
%\sigma={1\over{4l_n}}.
%\eneq
Let  $\phi_1: C\to A$ be a unital ${\cal G}_0$-$\dt_0$-multiplicative \morp\,
such that $[\phi_1]$ induces an element in $KL_{loc}(G^{{\cal P}_n\cup K_{n,p}}, \underline{K}(A)),$ 
$[\phi_1(x)]\in K_0(A)_+\setminus \{0\}$ for all $x\in K_{n,p},$ 
and $[\phi_1](G^{{\cal P}_n\cup K_{n,p}}\cap {\rm ker}\rho_{C,f})\subset {\rm ker}\rho_A$. 
%and  there is  a continuous  affine map $(\phi_1)'_T: T(B)\to T_{\rm f}(C)$ such that
%$\rho_B([\phi_1](x))(\tau)=\rho_C(x)((\phi_1')_T(\tau)$ for all $x\in K_{n,p}$ 
%and 
%\beq
%|\tau(\phi_1(h))-(\phi_1')_T(\tau)(h)|<\dt_0\rforal  \iota_{n, \infty}({\cal H}),
%\eneq
%for any finite subset ${\cal H}\subset (C_n)_{s.a.}^{\bf1},$ 
%A_{s.a}^{\bf 1},$
%there exists $0<\sigma<1,$ 
Let  $\kappa_0\in KL_{loc}(G^{{\cal P}_n\cup K_{n,p}}, \underline{K}(A))$  be such that 
$\kappa_0(K_{n,p})\subset K_0(A)_+\setminus \{0\},$ ${\kappa_0}(G^{{\cal P}_n\cup K_{n,p}}\cap {\rm ker}\rho_{C,f})
\subset {\rm ker}\rho_A,$ and 
\beq
\|\rho_A(\kappa_0([1_C]))\|<\af_0\min\{\sigma_0, 1/2\},
\eneq
where $\af_0\in (0,1/2l(n))$ and   
 \beq\label{93-n-10}
 \min\{\rho_A([\phi_1](y))(\tau) : y\in K_{n,p}, 
\tau \in T(A) \}> \sigma_0 > 0.
 \eneq
 %We may assume that $\|\rho_B(\kappa_0([1_C])\|<1/2.$ 
Let  $\chi_1\in KL_{loc}(G^{{\cal P}_n\cup K_{n,p}}, \underline{K}(A))$  be defined by
$\chi_1(x)=[\phi_1](x)-\kappa_0(x)$ for all $x\in {\cal P}_n \cup K_{n,p}.$ 
Recall that since we now assume that $C=C(X)$ for some connected finite 
 complex
$X$,
$K_0(C)=\Z\oplus {\rm ker}\rho_C$ {{and $(1, 0)\in K_{n,p}.$}}
Since $[\phi_1](x)-\kappa_0(x)\in {\rm ker}\rho_A$
for all $x\in {\rm ker}\rho_{f,C}$ and $A$ is simple with tracial
rank zero, and  by \eqref{93-n-10}, 
one computes that
that $\chi_1(K_0(C)_+\setminus \{0\})\subset K_0(A)_+\setminus \{0\}.$ 

Since $({\cal G}_n, \dt_n, {\cal P}_n)$ is a KL-triple associated with ${\cal H}$
  and $\eta/2$, 
there exists a continuous affine map
$\gamma_1': T(A)\to T(C)$
such that
\beq
|\tau\circ \phi_1(c)-\gamma_1'(\tau)(c)|<\eta/2\rforal c\in {\cal H}\andeqn \tau\in T(A).
\eneq
%Choose a non-zero projection $e_0'\in B$ such that
%$\tau(e_0')<\eta/64$ for all $\tau\in T(B).$    Let $\kappa'\in KK(C, B)$ 
%be such that   $\kappa'([1_C])=[e_0'],$ $\kappa'(x)=0$ for all $x\in {\rm ker}\rho_B.$
%$\kappa'(K_1(C))=0.$ 
 Let $t_f\in T_f(C).$ 
 Define $\gamma_1'': T(A)\to T_f(C)$ by $\gamma_1''(\tau)=t_f$ for all
$\tau \in T(A)$. 
 Define $\gamma_1:=(1-\eta/128)\gamma_1'+(\eta/128)\gamma_1''.$
 Then $\gamma_1: T(A)\to T_f(C)$ is a continuous affine map and 
 \beq\label{95-T-1}
 |\tau\circ \phi_1(c)-\gamma_1(\tau)(c)|<\eta/2+\eta/64\rforal c\in {\cal H}\andeqn \tau\in T(A).
 \eneq
 %Then 
% $\kappa'$ and $\gamma_1''$ are compatible. It follows from Theorem 4.5 of \cite{Linann} (see (e.4.67) in the last paragraph of the proof of Theorem 
%4.5 of \cite{Linann}) that there is a unital \hm

Let $\gamma_1^{\sharp}: {{\Aff^b}}(T_f(C))\to \Aff(T(A))$ be the continuous
linear map 
induced by $\gamma_1.$ 
Let  ${\cal H}_{A,q}\subset A_{s.a}$ be a finite subset such that
$\{\hat{b}: b\in {\cal H}_{A,q}\}=\gamma_1^{\sharp}(\widehat{{\cal H}}),$ 
where $\widehat{{\cal H}}=\{\hat{h}|_{T_f(C)}: h\in {\cal H}\}.$

Let $e_0\in A$ be a projection  such that $[e_0]=\kappa_0([1_C])$ and let 
$A_1:=(1_A-e_0)A(1_A-e_0).$   \Wlog, we may identify $A$ with $pM_2(A_1)p$
for some projection $p \in M_2(A_1)$ 
with $[p]=[1_A]=[e_0]+[1_{A_1}]=[e_0]+[1_A-e_0].$
We may further assume that $p=1_{A_1}\oplus e_0',$ where $e_0'\in e_{2,2}M_2(A_1)e_{2,2}$
is a projection such that $[e_0']=[e_0].$  
Consider $H_1:=\{1_{A_1},1_{A_1} \phi_1(c) 1_{A_1}, 1_{A_1}x1_{A_1}, e_{1,2}e_0'xe_0'e_{2,1}: c \in \Hh, \makebox{ } x\in {\cal H}_{A,q}\}\subset A_1.$

By  {{Lemma 8.10 of \cite{GLIII},}}
%\textbf{Reference?}, 
there exist a finite subset $T\subset \partial_e(T(A_1))$ and a continuous affine
map $\gamma_T:  T(A_1)\to {\rm conv}(T)$ such that
\beq\label{95-T-20}
|\gamma_T(\tau)(x)-\tau(x)|<(\eta/16)\|x\|\rforal x\in  \widehat{H_1} \makebox{ and } \tau \in T(A_1).
\eneq
Write $T=\{s_1, s_2,..,s_m\}.$ 
Define $t_i(b)= \frac{s_i(b)}{s_i(p)}$ for all $b \in A$ ($=pM_2(A_1)p$), 
for $1 \leq i \leq m$. 
Then $t_i\in T(A),$ for $1\le i\le m.$  Set $\Delta_T={\rm conv}(\{t_1,t_2,...,t_m\}) \subset T(A)$.   
Define $\gamma_{T,p}: {\rm conv}(T) \to \Delta_T$  by $\gamma_{T,p}(s_i)=t_i$, 
for $1\le i\le m$.  

Note that for all $t \in T(A)$, if $s = \frac{t}{t(1_{A_1})} \in T(A_1)$
(or, equivalently, if $t = \frac{s}{s(1_A)}$), then for all $a \in A_+$,
$$t(a) \leq s(a) = \frac{t(a)}{1 - t(e_0)} < t(a)(1 + 2 t(e_0)) < t(a)(1 
+ 2 \alpha_0 \sigma_0).$$
Hence, $$\| s - t \| < 2 \alpha_0 \sigma_0$$
where we are viewing $s$ and $t$ as positive linear functionals on $A$
{{(and taking the norm as elements in $A^*$).}}
%(and taking the operator norm on $B(A)$).

By a similar argument, if $s \in T(A_1)$ and $\gamma_T(s) = \sum_{j=1}^m 
\lambda_j s_j$ (where $\lambda_j \geq 0$ for all $j$ and $\sum_{j=1}^m 
\lambda_j  = 1$), then 
for all $a \in A_+$,
$$\sum_{j=1}^m \lambda_j t_j(a) \leq \sum_{j=1}^m \lambda_j s_j(a) 
< \sum_{j=1}^m \lambda_j t_j(a)(1 + 2 t_j(e_0)) < \sum_{j=1}^m \lambda_j t_j(a)
(1 + 2 \alpha_0 \sigma_0).$$ 
Hence,
$$\| \gamma_T(s) - \gamma_{T,p} \circ \gamma_T(s) \| < 2 \alpha_0 \sigma_0$$
where again, we are viewing all maps as positive linear functionals on $A$
{{(and taking the norm as elements in $A^*$).}}
%operator norm on $B(A)$).  

Define $\gamma_2: T(A_1)\to T_f(C)$
by $\gamma_2=\gamma_1\circ j\circ \gamma_{T,p}\circ \gamma_T,$ where 
$j: \Delta_T\to T(A)$ is the embedding.

Note that $\chi_1([1_C])=[\phi_1(1_C)]-\kappa_0([1_C])=[1_{A_1}].$ 
Then since $K_0(C) = \Z \oplus ker \rho_C$,
 $\chi_1$ and $\gamma_2$ are compatible. Recall that
$\chi_1(K_0(C)_+\setminus \{0\})\subset K_0(A)_+\setminus \{0\}.$ 
Thus, by Theorem 4.5 of \cite{Linann} (see (e.4.67) in the last paragraph of the proof of Theorem 
4.5 of \cite{Linann}), there is a unital \hm\, $\Phi: C\to A_1$ such that
\beq
[\Phi]=\chi_1\andeqn s(\Phi(c))=\gamma_2(s)(c)\rforal c\in C_{s.a.}\andeqn s\in T(A_1).
\eneq
If $t\in T(A),$ then $s=\frac{t}{t(1_{A_1})}\in T(A_1).$ With this convention, 
we estimate, by \eqref{95-T-1}, by \eqref{95-T-20}, and by the above
norm estimate for $\gamma_T(s)$ and $\gamma_{T,p} \circ \gamma_T(s)$, that
for all $x\in \Hh,$ 
\beq
t(\Phi(x))=t(1_{A_1})(s(\Phi(x)))=t(1_{A_1})(\gamma_2(s)(x))\\
=t(1_{A_1})(\gamma_1\circ j\circ \gamma_{T,p}\circ \gamma_T)(s)(x)\\
=t(1_{A_1})\gamma_1((\gamma_{T,p}\circ \gamma_T)(s))(x)
\approx_{33\eta/64}t(1_{A_1}) (\gamma_{T,p}\circ \gamma_T)(s)(\phi_1(x))\\
\approx_{2 \alpha_0 \sigma_0} t(1_{A_1})\gamma_T(s)(\phi_1(x))\approx_{\eta/16}t(1_{A_1})s(\phi_1(x))\\
= t(\phi_1(x)).
\eneq
This implies that (c) of part (3) of Definition
\ref{DefA1} holds for this case, and the general case follows.
\end{proof}
\begin{df}\label{DC0An}
Let $E$ and $F$ be finite dimensional \CA s, $\phi_0, \phi_1: F \to E$
be unital \hm s.  Define
$$
C:=C(E, F, \phi_0, \phi_1)=\{ (f,a)\in C([0,1], E) \oplus F: f(0)=\phi_0(a),\, \, f(1)=\phi_1(a)\}.
$$
$C$ is sometimes called one-dimensional non-commutative CW complex. 
As in Proposition 3.5 of \cite{GLNI}, ${\rm ker}\rho_C=\{0\}.$
Denote by  ${\cal C}_0$ the class of those one-dimensional non-commutative CW complexes $C$ such that $K_1(C)=\{0\}.$

Denote by $X'=S^1\vee S^1\vee \cdots S^1\vee T_{2, k_1}\vee T_{2, k_i}\vee T_{3, m_1}\vee\cdots T_{3, m_j}$ exactly as in 13.27 of \cite{GLNI}.  Denote by $x^1$ the base point which is the common 
point of all spaces $S^1,$ $T_{2, k}$ and $T_{3, m}$ {{in $X'.$}}
%{\blue{which appear above. }} 
 Set $X=[0,1]\vee X'.$ We write $X=[x^0, x^1]\vee X',$ where $x^0=0,\,\, x^1=1$ (in $[0,1]$). 

{{Write $F=F^{(1)}\oplus F',$ where $F_1=M_{r_1}$ for some $r_1\ge 1.$}}
%Fix $F^{(1)}$ as one of the (non-zero) simple summand of $F.$ 
%Denote by $F'$  the direct sum of the remaining simple summands of $F$ (i.e., all summands except for $F^{(1)}$).
%In other words, $F'\cong F/F^{(1)}.$

Now we consider $PM_r(C(X))P,$ where $r\ge 1$ is an integer and $P\in M_r(C(X))$ is a 
projection with the {{rank $r_1,$}} the 
same rank as $F^{(1)}.$ 

We define 
\beq\nonumber
%\label{defB}
A=\{ (f,g,a)\in C([0,1], E)\oplus PM_r(C(X))P \oplus F: \\\label{defB}
 f(0)=\phi_0(a),\,\, f(1)=\phi_1(a),\,\, g(x^0)=\phi_X(a)\},
\eneq
where  $\phi_X: F\to PM_r(X)P|_{x^0}$ is a \hm\, which factors through $F^{(1)}.$ 
(This is the same as $A_n$ in  Definition 13.28 of \cite{GLNI}). 
Let $J=\{(f, g,a)\in A: f=0, a=0\}$ and $I=\{(f, g,a)\in A: g=0,a=0\}.$
Put $A_C=A/J$ and $A_X=A/I.$
Note that
\beq
&&A_C=\{(f,a)\in C([0,1], E)\oplus F: f(0)=\phi_0(a),\,\, f(1)=\phi_1(a)\}\andeqn\\
&&A_X\cong \{ (g, a) \in  PM_r(C(X))P\oplus F :  g(x^0) = \phi_X(a) \}.
\eneq
We require that $A_C\in {\cal C}_0,$ i.e., $K_1(A_C) = 0$.

{{Since $\phi_X$ factors through $F^{(1)}=M_{r_1},$ 
the map ${\phi_X}|_{F^{(1)}} M_{r_1}\to PM_r(C(X))P|_{x^0}\cong M_{r_1}$ is invertible.
Denote by $\psi$ the inverse.
Let
$$
A_X'=\{(f,a)\in PM_r(C(X))P\oplus M_{r_1}: f(x^0)=\phi(a){\red{\}.}}
$$
Define $\Phi: PM_r(C(X))P\to A_X'$ by
$\Phi(f)=(f,\psi(f(x^0)))$ for $f\in PM_r(C(X))P.$ 
Then  the map $\Phi$ gives the isomorphism 
\beq\label{21712-1}
A_X'\cong PM_r(C(X))P.
\eneq
Therefore we have
\beq\label{21712-2}
A_X=PM_r(C(X))P\oplus F'.
\eneq
We will use this simple fact.}}

Denote by ${\cal B}$ the class of \CA s with the form in \eqref{defB} above. 
Note that $K_0(A)_+$ is finitely generated as positive cone.   % $K_0(A)_+.$

It should be noted that
\beq\label{DC0An-01}
K_0(A)=K_0(A_C)\oplus K_0(C_0(X\setminus x^0))\andeqn
K_1(A)=K_1(C_0(X\setminus x^0)),
\eneq
with 
\beq
K_0(A)_+\subset \{(m, z): m\in K_0(A_C)_+\setminus \{0\}, z\in K_0(C_0(X\setminus x^0))\}\cup\{(0,0)\}.
\eneq
%K_0(A_C)_+.$ 
Note that $K_0(A_C)_+$ is finitely generated and $K_0(A)$ is also finitely generated.
Thus there is a finite subset $K_p\in K_0(A)_+\setminus \{0\}$ such 
that, for any $x\in K_0(A)_+,$ there are $x_1, x_2,...,x_l, y_1,y_2,...,y_m\in K_p,$
$r_1, r_2,...,r_l\in \N$ and $s_1, s_2,...,s_m\in \Z$ such that
\beq
x=\sum_{i=1}^l r_i x_i+\sum_{j=1}^m s_j y_j
\eneq
and $\xi:=\sum_{j=1}^m s_j y_j\in {\rm ker}\rho_A.$
By \eqref{DC0An-01}, for any \CA\, $B,$ 
one may write 
\beq\label{AnKLdec}
KL(A,B)=KL(A_C, B)\oplus KL(C_0(X\setminus x^0), B).
\eneq

\end{df}

\begin{prop}\label{P2304}
Let $A$ be as in \eqref{defB} and $B$ be a unital separable simple \CA\, 
with tracial rank zero and $p\in B$ be a non-zero projection. Suppose that $\af\in 
KL(A, B)^{++}$ such that $\af([1_A])\le [p].$ 
Then there exists  a sequence of \cpc s
$L_n: A\to pBp$ such that
\beq
\lim_{n\to\infty} \|L_n(ab)-L_n(a)L_n(b)\|=0\tforal a, b\in A\andeqn [\{L_n\}]=\af.
\eneq

\end{prop}

\begin{proof}
Write 
\beq
A=\{(f,g,a)\in C([0,1], E)\oplus PM_r(C(X))P \oplus F: \\
 f(0)=\phi_0(a),\,\, f(1)=\phi_1(a),\,\, g(x^0)=\phi_X(a)\},
\eneq
where  $\phi_X: F\to PM_r(X)P|_{x^0}$ is the \hm\, that factors through $F^{(1)}.$

We may write (see  \eqref{AnKLdec})
\beq\label{P23-04-1}
\af=\af_c\oplus \af_x,
\eneq where 
$\af_c\in KL(A_C, B)^{++}$ and $\af_x\in KL(C_0(X\setminus x^0), B).$
Since $\af\in KL(A, B)^{++},$
$\rho_B\circ (\af_x(z))=0$ for all $z\in {\rm ker}\rho_A.$
Let $K_p\in K_0(A)_+\setminus \{0\}$ be a finite subset 
such that, for any $x\in K_0(A)_+,$ there  are $x_1, x_2,...,x_l, y_1,y_2,...,y_m\in K_{p}$,
$r_1, r_2,...,r_l\in \N$ and $s_1, s_2,...,s_m\in \Z$
such that
\beq
x=\sum_{i=1}^l r_i x_i+\sum_{j=1}^m s_j y_j,
\eneq
where $\xi:=\sum_{j=1}^m s_j y_j\in {\rm ker}\rho_A.$
%$K_p$ generates $K_0(A_C)_+\subset K_0(A)_+$ (as in \eqref{DC0An-01}).
Since $\af\in KL(A, B)^{++},$ there is $\sigma>0$ such that
\beq
\rho_B(\af(x))(\tau)>\sigma\rforal x\in K_p \andeqn \tau\in T(B).
\eneq

We may assume that there are projections $e_1, e_2,...,e_m\in M_{l(n)}(A)$
such that $K_p=\{[e_1], [e_2],...,[e_m]\}.$ 

Choose a non-zero projection $q\le p$ such that 
\beq
\sup\{\tau(p)-\tau(q):\tau\in T(B)\}>(1-1/4l(n))\sigma
\eneq
Since $B$ has tracial rank zero, there is an embedding 
$\psi_{F,0}: F\to qBq.$  Let $\pi: A_X\to F$ be the quotient map. 
Define $\af_X'\in KL(A_X, B)$ by 
\beq\label{P2304-2}
\af_X'=[\psi_{F,0} \circ \pi]\oplus \af_x.
\eneq
(Note that for all $g \in PM_r(C(X))P$, $g = (g - g(x^0)) + g(x^0)$ and recall that if $(g, a) \in A_X$, then
$g(x^0) = \phi_X(a)$ where $\phi_X$ factors through $F^{(1)}$.)
Since $[\psi_{F,0}]$  is strictly positive, we conclude that 
$\af_X'\in KL(A_X, B)^{++}.$
 
 {{By Theorem 2.3
of \cite{Li},  there exists a \hm\, $h_x'': C(X)\oplus F'\to E(B\otimes {\cal K})E,$
 where $E\in B\otimes {\cal K}$ is a projection such 
 that $[E]=\af_x'([1_C(X)]\oplus [1_{F'}])$ {{and}}  $[h_x'']=\af_X'.$ 
 Then $h_x''$ induces a \hm\, $h_x''': PM_r(C(X))P\oplus F'\to B\otimes {\cal K}$ such that
 $$
 [h_x'''(P)\oplus 1_{F'}]=[\af_X'([1_{A_X})]\le [q].
 $$ }}
 %{\red{\bf{Ping: Do you like the above paragraph?}}}
Since $B$ has tracial rank zero, there is a unitary $u\in \widetilde{B\otimes {\cal K}}$ 
such that
$$
u^*h_x'''(1_{A_X})u\le q.
$$
Put $h_x'={\rm Ad} u\circ h_x''': A_X\to qAq.$ Then
% {\textbf{(Reference needed)}} By ?, there is a \hm\, 
%$h_x': A_X\to qAq$
%such that 
$[h_x']=\af_X'$ (as an element in $KL(A_X, B)$).
Let $\pi_X: A\to A/I=A_X$  be the quotient and $h_x= h_x' \circ \pi_X: A\to qBq.$ 
Put 
\beq	
\bt:=\af-[h_x].
\eneq
For any $\xi\in K_p$ and any $\tau \in T(B)$, 
we have that
\beq
\rho_B(\af(\xi)-[h_x](\xi))(\tau)>\sigma -l(n)[h_x](1_A)(\tau)>(3/4)\sigma>0
\eneq
It follows that $\bt\in KL(A, B)^{++}.$
Note that, in view of \eqref{AnKLdec}, by \eqref{P23-04-1} and \eqref{P2304-2},
\beq
\bt=\af-[h_x]\in KL(A_C, B)^{++}.
\eneq
Note that, since $B$ has tracial rank zero, $B\in {\cal B}_1$ (as defined in Definition 9.1 of \cite{GLNI}). 
By 18.7 of \cite{GLNI},  there is a sequence of \cpc\, maps $\Phi_n: A_C\to (p-q)B(p-q)$
such that 
\beq
\lim_{n\to\infty}\|\Phi_n(ab)-\Phi_n(a)\Phi_n(b)\|=0\rforal a,b\in A_C\andeqn [\{\Phi_n\}]=\bt
\eneq
(where we view $\bt$ as an element in $KL(A_C, B)$).
Denote by $\pi_C: A\to A/J=A_C$ the quotient map.

Define $L_n=\pi_C\circ L_n\oplus h_x: A\to pBp,$ $n\in \N.$
Then 
\beq
\lim_{n\to\infty}\|L_n(ab)-L_n(a)L_n(b)\|=0\andeqn
[\{L_n\}]=\bt+[h_x]=\af
\eneq
\end{proof}

\begin{thm}\label{Ain A1}
Let $C$ be a unital separable simple amenable ${\cal Z}$-stable \CA\, 
satisfying the UCT.
Then $C$ is in the class ${\cal A}_0.$ 
\end{thm}

\begin{proof} {{Firstly, by Theorem \ref{TembeddingAH}, $C$ is in
the class ${\cal A}_1$.   
We also}} note that by Theorem 13.50  of \cite{GLNI},
%\ref{Pfinitestage}, 
$C$ satisfies
part (1) of Definition \ref{DefA1}.
Next, {{ by Corollary 4.12 and Theorem 4.11 of \cite{EGLN} as well as Theorem 13.50 of \cite{GLNI}),
$C=\lim_{n\to \infty}(C_n, \iota_n),$  where each $C_n\in {\cal S}_{K(n)}$ (for some $K(n)\in \N$), and 
$\iota_n$ is injective.
Then,}} by Corollary
 \ref{LA1-IndSr},  
 %{{and Theorem 13.50 and Theorem ? of \cite{GLNII} and Theorem  of \cite{EGLN},}}  
 $C$ satisfies part (2) of Definition \ref{DefA1}. 
It remains to show that $C$ satisfies part (3) of Definition \ref{DefA1}.

In fact, by Theorem 13.50 and 13.28 of \cite{GLNI}, we may assume that
each $C_n\in {\cal B}$ has the form in \eqref{defB}.

%{{We then rewrite}} $C=\lim_{n\to\infty} (C_n, \Phi_{n,n+1})$ as in Proposition \ref{Pfinitestage}.
For each $n$, 
since $K_i(C_n)$ is finitely generated for $i =0,1$, one can choose a finite subset  ${\cal Q}_n$ of $\underline{K}(C_n)$ 
which satisfies the requirement for  ${\cal Q}_n$   in the first part 
(3) of Definition \ref{DefA1}  {{(before (a)).}}
%{\red{(Reference needed:)(see \cite{DL}).}} 
For each $n$,
we thus get $\Pp_n$ as in Definition \ref{DefA1}. 
As in \ref{DC0An}, there exists a finite subset $K_{n,p}'\subset K_0(C_n)_+\setminus \{0\}$
such that, for any $x\in K_0(C_n)_+,$ there are $x_1, x_2,...,x_{l_n}, y_1,y_2,...,y_{m_n} \in K'_{n,p},$
$r_1, r_2,...,r_{l_n}\in \N$ and $s_1, s_2,...,s_{m_n}\in \Z$ such that
\beq
x=\sum_{i=1}^{l_n}r_i x_i+\sum_{j=1}^{m_n} s_j y_j
\eneq
and $\xi:=\sum_{j=1}^{m_n} s_j y_j\in {\rm ker}\rho_{C_n}.$
We may assume that $[1_{C_n}]\in K_{n,p}'.$ 

Then  condition  (a)  of part (3)  of Definition \ref{DefA1}  follows.
{{Let ${\mathtt{P}}_n\subset M_{l(n)}(C_n)$ (for some $l(n)\in \N$) be a finite subset of projections 
such that $K_{n,p}'=\{[p]:p\in {\mathtt{P}}_n\}.$}}
% by choosing ${\cal Q}_n$ as generating sets 
%of $K_0(C_n)$ and $K_1(C_n).$ 
Then (b) of part (3) of Definition \ref{DefA1} follows from Proposition \ref{P2304}.
It remains to show that $C$ also satisfies (c) of part (3) of Definition \ref{DefA1}.

Define $F=\rho_C(K_0(C))$ with positive cone $F_+$ given by the strict order in $\Aff(T(C)).$
Let $D$ be a unital separable simple ${\cal Z}$-stable \CA\, which is an inductive limit
of sub-homogeneous \CA s such that  
\beq
{\rm Ell}(D)=(F, F_+, [1_C], K_1(C), T(C), \rho_C)
\eneq
(given by Theorem 13.50  of \cite{GLNI}).
We may write $F=\cup_{n=1}^{\infty} F_n,$
where each $F_n$ is a finitely generated abelian subgroup of $F\subset \Aff(T(C))$ with 
the strict order inherited from $\Aff(T(C))_{++}$ and $[1_C]\subset F_n$ and 
$F_n\subset F_{n+1},$ $n\in \N.$ 
It follows from Theorem 13.50 of \cite{GLNI} again that there exists a sequence 
of unital simple ${\cal Z}$-stable \CA s $D_n$ which are inductive limits of sub-homogeneous \CA s
such that
\beq
{\rm Ell}(D_n)=(F_n, (F_n)_+, [1_C], K_1(C), T(C), \rho_C).
\eneq 
As in Theorem 5.4 of \cite{GLNhom2} (see Theorem 5.2 of \cite{GLNhom2}), we obtain embeddings $j_n: D_n\to D_{n+1}$ such that
\beq
D=\lim_{n\to\infty} (D_n, j_n),
\eneq
where $(j_n)_{*0}$ gives an embedding from $F_n$ to $F_{n+1}$ 
and $(j_n)_{*1}={\rm id}_{K_1(C)}$ and $j_n$ induces the identity 
map on $T(D_n)=T(C).$
It follows from Theorem 5.2 of \cite{GLNhom2} that there is a unital \hm\,
$H_{C,D}: C\to D$ such that
${H_{C,D}}_{*0}=\rho_C: K_0(C)\to K_0(D),$
${H_{C,D}}_{*1}={\rm id}_{K_1(C)}$ and $H_{C,D}$ induces the identity map
on $T(C)=T(D).$ 

For each $n,$ there exists $s_n: K_0(D_n)\to K_0(C)$ 
such that $\rho_C\circ s_n=\id_{K_0(D_n)}.$ 
Thus, by Theorem 5.2 of \cite{GLNhom2} again, there exists a unital \hm\, $H_{D_n,C}:
D_n\to C$ such that ${H_{D_n,C}}_{*0}=s_n,$
${H_{D_n, C}}_{*1}=\id_{K_1(C)}$ and $H_{D_n, C}$ induces the identify map on $T(D_n) = T(C)$. 
There is also a unital \cpc s $\Psi_n': D\to D_n$ such that
\beq
\lim_{n\to\infty} \|d-j_{n, \infty}\circ \Psi_n'(d)\|=0\rforal d\in D.
\eneq
Put $\Psi_n: C\to C$ defined by
$\Psi_n(c)=H_{D_n, C}\circ \Psi_n'\circ  H_{C, D}(c)$ for all $c \in C$.
Then $\Psi_n$ is a \cpc. 
\beq
\lim_{n\to\infty}\|\Psi_n(ab)-\Psi_n(a)\Psi_n(b)\|=0\rforal a, b\in C.
\eneq
We note that, for any finitely generated subgroup $G_{o}'\subset {\rm ker}\rho_C,$ 
\beq
[\Psi_n]|_{G_o'}=0
\eneq
for all sufficiently large $n\in \N.$

We may assume that $\rho_C\circ {\iota_{n, \infty}}_{*0}(K_0(C_n))\subset F_n,$ $n\in \N.$ 

To prove condition (c) of part (3) of Definition \ref{DefA1}, we note that it suffices to prove the case 
%where $C=C_n$ for  each $n.$
$C_n$ satisfies the condition (c) of part (3) of Definition \ref{DefA1} for each $n.$ 
%
%%%%%%%
\iffalse
Fix $n \geq 1$.
In what follows we write $C=C_n.$
So we may write $K_0(C_n)=F_n\oplus {\rm ker}\rho_{C_n},$ where $F_n$ is a finitely generated 
free abelian group and ${\rm ker} \rho_C$ is a finitely generated abelian 
group. 
\fi
%%%%%%%%%%

Again, we will also {{use}} the notation of (a) and (b) in part (3) of \ref{DefA1}, as well as the notation from the present argument so far.
%
%%%%%%%
\iffalse
Let $C_0$ be a unital separable simple amenable {{${\cal Z}$-stable}} \CA\, in the UCT class such that 
$K_1(C_0)=K_1(C)$,  $(K_0(C_0), K_0(C_0)_+)=(F_n, (F_n)_+)$ and
$T(C_0) = T(C)$.  
It follows from {\red{(Exact reference needed)}} Lemma  of \cite{GLNhom2} that there is a unital \hm\, 
$\psi_{C,C_0}: C\to C_0$ such that $(\psi_{C,C_0})_{*0}|_{F_n}=\id_{F_n},$ 
$(\psi_{C, C_0})_{*0}|_{{\rm ker}{\rho_{C}}}=0$,  $(\psi_{C, C_0})_{*1}=\id_{K_1(C)},$
and $\psi_{C,C_0}$ induces the identity map on $T(C_0).$ 
Also, by  {{Lemma  5.2 of \cite{GLNhom2},}}  there is a unital \hm\, 
$\psi_{C_0, C}: C_0\to C$ such that $(\psi_{C_0,C})_{*0}=\id_{F_n},$ 
$(\psi_{C_0, C})_{*1}=\id_{K_1(C)},$
and $\psi_{C_0,C}$ induces the identity map on $T(C).$ 
\fi
%%%%%%%%%%%

Let us first fix an arbitrary $n \geq 1$.
Let ${\cal H}\subset \iota_{n, \infty}({C_n}_{s.a.}^{\bf 1}\setminus \{0\})$
be a finite set 
and $\eta>0.$   Assume that $({\cal G}_n, 2\dt_n, {\cal P}_n)$ is a $KL$-triple associated with ${\cal H}$ and $\eta/2.$ 
Let ${\cal G}_0\supset {\cal G}_n\cup {\cal H}$ be a finite 
set and $0<\dt_0<\dt_n.$  Let $A$ be an arbitrary unital separable simple
C*-algebra with tracial rank zero. 

%Recall that we {{have assumed}} that $C = C_n$.  
Let $\phi_1: C\to A$  be  a unital ${\cal G}_0$-$\dt_0$-multiplicative \morp\,
such that $[\phi_1]$ induces an element in $KL_{loc}(G^{{\cal P}_n\cup K_{n,p}}, \underline{K}(A))$ 
and there is a $\sigma_0 > 0$ where for all 
%$[\phi_1(x)]\in K_0(A)_+\setminus \{0\}$ for all 
$x\in K_{n,p}$ and for all $\tau \in T(A)$,  
\beq
\rho_A([\phi_1(x)])(\tau)>\sigma_0, \makebox{ and }     
\eneq
%there exists a continuous affine map $(\phi_1')_T: T(A)\to T(C_n)$ such that
$[\phi_1](G^{{\cal P}_n\cup K_{n,p}}\cap {\rm ker}\rho_{C,f})\subset {\rm ker}\rho_A$. 
Suppose that  $\kappa_0\in KL_{loc}(G^{{\cal P}_n\cup K_{n,p}}, \underline{K}(A))$ such that 
$\kappa_0(K_{n,p})\subset K_0(A)_+\setminus \{0\},$ ${\kappa_0}(G^{{\cal P}_n\cup K_{n,p}}\cap {\rm ker}\rho_{C,f})
\subset {\rm ker}\rho_A,$
%\textbf{Changed the next line:}
%there is an $\af_0 \in (0, 1/4)$ so that for all 
{{and, for all $x \in K_{n,p}$, 
\beq
\|\rho_A(\kappa_0(x))\|<\af_0\min\{1/2, \sigma_0\}
\eneq
for some $\af_0\in (0, 1/2l(n)),$}}
%\beq
%\|\rho_A(\kappa_0(x))\|<\af_0\min\{1/2, \sigma_0\},
%\eneq
and $p_0\in A$  is a projection such that $[p_0]=[\phi_1(1_C)]-\kappa_0([1_C]).$ 
%
%%%%%%%%
\iffalse
Let $B_0$ be a unital separable simple AH-algebra with slow dimension growth and real rank zero
such that 
\beq
(K_0(B_0), K_0(B_0)_+, [1_{B_0}], K_1(B_0))=(K_0(A), K_0(A)_+, [1_A], K_1(A)).
\eneq
It follows from  Corollary 4.6 of \cite{LinKT} that there exists 
a unital embedding $H: B_0\to A$ such that $[H]$ induces the above identification. 
\fi
%
%%%%%%%%%

Choose integers $m,K\in \N$ such that 
$2\af_0<m/K\le 33\af_0/16.$   Set  
\beq
\eta_0=\min\{(m/K-2\af_0), \af_0, \eta, \sigma_0/2\}>0.
\eneq

Recall that $A$ has tracial rank zero.
Hence, by Lemma 6.10 of \cite{LinTAF}, 
for any $\ep'>0$ and any finite subset ${\cal G}'\subset A,$ 
%With Proposition \ref{Ptrace} in mind, 
there are
projections $e_1, e_2\in A$ 
such that $(1-e_1)A(1-e_1)\cong M_K(e_2 A e_2),$ a finite dimensional \SCA\, $F_a$  of $e_2 A e_2$
with $1_{F_a}=e_2$ and 
a unital ${\cal G}'$-$\ep'$-multiplicative \morp\, $L: A \to F_a$
such that

(1) $\tau(e_1)<\eta_0/64\rforal \tau\in T(A),\\\label{95-T-02}$

(2) $\|[e_i,\, x]\|<\ep'$ for all $x\in {\cal G}'$,  $i=1, 2$, and 

(3) $\|x-(e_1xe_1\oplus \diag(\overbrace{L(x),L(x),...,L(x)}^K))\|<\ep'\rforal x\in {\cal G}'.$

Set $E=\sum_{i=m+1}^K e_{i,i},$ where $\{e_{i,j}\}_{K\times K}$ is a system of matrix units for $M_K.$ 
%Define $L_A:  A \to  EM_K(F)E$ by 
Let
\beq
L_{A,0}(x)&=&e_1xe_1\oplus \diag(\overbrace{L(x),L(x),..., L(x)}^m,0,...,0)\andeqn\\
L_{A,1}(x)&=&0\oplus \diag(0,...,0, \overbrace{L(x),L(x),...,L(x)}^{K-m})\rforal x\in A.
\eneq
Define $\Psi_a: C\to EM_K(F_a)E$ by
$\Psi_a := L_{A,1}\circ \phi_1\circ \Psi_{n+k}$ for some $k\ge 1.$
%\psi_{C_0, C}\circ \psi_{C, C_0}.$ 
With sufficiently large $k,$
%${\cal G}'$ and small $\ep',$ 
we may assume that 
{{$\Psi_a$ is at least  $\G_0$-$\delta_0$-multiplicative. 
Note that  $[\Psi_a\circ \iota_{n, \infty}]|_{{\rm ker}\rho_{C_n}}=0.$ 
Let $\bt=[\phi_1\circ \iota_{n, \infty}]-[\Psi_a\circ \iota_{n, \infty}]-\kappa_0\circ [\iota_{n, \infty}] \in KL(C_n, A)$.}}  
%(Recall that
%we have set $C = C_n$.) 
We compute that, for sufficiently large $k,$
%${\cal G}'$ and small $\ep'$,
for all $\tau \in T(A)$ and $x' \in K'_{n,p}$, if $x = [\iota_{n, \infty}](x') \in K_{n,p}$, then 
\beq
%(\rho_C([\phi_1])(x)-\rho_C([\Psi_1])(x))(t)
\rho_C(\bt(x'))(\tau)&>&\tau(\phi_1(x)) - (1 - \frac{m}{K}) \tau(\phi_1(x))
- \tau(\kappa_0(x))\\
& = & \frac{m}{K} \tau(\phi_1(x)) - \tau(\kappa_0(x)) > 0.  
%\sigma_0-(K-m/K)\sigma_0-\rho_A(\kappa_0(x))-\eta_0/64\\
%&=&(m/K-\af_0)\sigma_0-\eta_0/64>0
\eneq
Moreover  $\bt|_{{\rm ker}\rho_{C_n}}=0.$ 
Let $A'=(e_1\oplus (\sum_{i=1}^m e_{i,i}))A(e_1\oplus (\sum_{i=1}^me_{i,i})).$
Note that $\rho_A(\bt([1_C]))(t)<\rho_A([1_{A'}])(t)$ for all $t\in T(A).$ 

%
%%%%%%%%%%%%%%%%%
\iffalse
Consider the order preserving map 
${\bar \bt}: K_0(C) = K_0(C_n) \to \Aff(T(A'))$  defined 
by $\overline{\beta} := 
\rho_{A'}\circ \bt.$
Recall  that $K_0(C_n)$ satisfies \eqref{Pfinitestage-e}.  By Proposition \ref{Pgexten}, 
there exists an order preserving \hm\, ${\overline{\bar \bt}}: \Aff(T(C_n))\to \Aff(T(A'))$
such that ${\overline{\bar \bt}}\circ \rho_{C_n}={\bar \bt}.$
\fi
%%%%%%%%%%%%

Choose a projection $e_0\in A'$ such that $\bt([1_C])=[e_0].$   
Then we may view $\bt\in KL_e(C_n, e_0A e_0)^{++}.$ 
%%Then $(\bt, {\overline{\bar \bt}})$ 
%is compatible. 
%By {{Lemma 5.2 of \cite{GLNhom2},}}
By Proposition \ref{P2304},
there is a  sequence of \cpc s
%unital \hm\, 
$\Psi_{0,j}: C_n\to e_0 A e_0$
such that
{{\beq
\lim_{j\to\infty}\|\Psi_{0,j}(ab)-\Psi_{0,j}(a)\Psi_{0,j}(b)\|=0\rforal a, b\in C_n\andeqn
 [\{\Psi_{0,j}\}]=\bt.
 \eneq}} 
%(It will not matter what $\Psi_0$ does on
%traces, since $\Psi_0(1_C)$ is ``small".) 

Now define $\Phi_j: C_n\to A$ 
by $\Phi_j(c)=\Psi_{0,j}(c)+\Psi_a(c)$ for $c\in C_n.$ 
Recall that we assume that ${\cal G}_n\subset \iota_{n, \infty}(C_n).$ 
Since $\iota_{n, \infty}$ is injective and $C_n$ is amenable,
there is, for each $n\in \N,$ a 
sequence of \cpc s $\Lambda_{n,k}: C\to C_n$
such that
\beq
\lim_{k\to\infty} \|\Lambda_{n,k}(c)-c\|=0\rforal c\in \iota_{n, \infty}(C_n).
\eneq
We verify that  (for all large $j$)
\beq
[\Phi_j\circ \Lambda_{n,k}]=
%\bt +[\Psi_1]=
([\phi_1]-[\Psi_a]-\kappa_0)+[\Psi_a]=[\phi_1]-\kappa_0.
\eneq
Recall that $\tau(e_1)<\eta_0/16$ for all $\tau\in T(A )$ and $m/K<33\af_0/16.$
We compute that
\beq
|\tau(\Phi_j(x))-\tau(\phi_1(x))|<2\af_0+\eta \rforal  x\in {\cal H}_1\andeqn \tau\in T(A ).
\eneq
Since $C = C_n$ and since $n \geq 1$ {{is}} arbitrary,
this implies that condition (c) of part (3) of Definition \ref{DefA1} holds. 
\end{proof}

\begin{cor}\label{Cor9}
Suppose that {\emph{either}} $C$ is a unital separable simple amenable ${\cal Z}$-stable \CA\, in the UCT class
\emph{or} $C$ is a unital AH-algebra with a faithful tracial state, and  {{suppose}} 
%\in {\cal A}_1$ be a unital separable \CA\, and 
%let 
$B$ is a 
%$\sigma$-unital 
non-unital  {{separable}} simple \CA\, with tracial rank zero 
and continuous scale. 
%^ {\red{(Should $T(B)$ be a metrizable Choquet
%simplex?  Or should we assume that $B$ is separable?)}} 
Let  $\Phi, \Psi: C\to M(B)$ be  two unital \hm s such that $\pi\circ \Phi$ and $\pi\circ \Psi$ are  both 
injective.

Suppose that $\tau\circ \Phi=\tau\circ \Psi$ for all $\tau\in T(B).$ 
Then there exists a sequence $\{ u_n \}$ of unitaries in
$M(B)$ such that
\beq
&&\lim_{n\to\infty}\|u_n^*\Phi(c)u_n-\Psi(c)\|=0\rforal c\in C\andeqn\\
&&u_n^*\Phi(c)u_n-\Psi(c)\in B\rforal n\in \N \makebox{ and } c\in C.
\eneq
The converse also holds.
\end{cor}

\end{document}